\documentclass[aos]{imsart}

\RequirePackage{amsthm,amsmath,amsfonts,amssymb}
\RequirePackage[authoryear]{natbib} %
\RequirePackage[colorlinks,citecolor=blue,urlcolor=blue]{hyperref}
\RequirePackage{graphicx}

\startlocaldefs

\newtheorem{theorem}{Theorem}
\newtheorem{lemma}{Lemma}
\theoremstyle{remark}

\usepackage{enumitem}
\usepackage{booktabs}
\usepackage{bbm}

\newtheorem{corollary}{Corollary}
\newtheorem{proposition}{Proposition}

\def\bs{\boldsymbol}
\def\dotsim{\dot{\sim}}
\newcommand{\E}{\mathbb{E}}
\newcommand{\PP}{\mathbb{P}}
\def\cov{\textrm{Cov}}
\def\var{\textrm{Var}}
\def\wy{y}
\def\cC{\mathcal{C}}
\newtheorem{condition}{Condition}
\def\converge{\rightarrow}
\def\convergep{\stackrel{\PP}{\longrightarrow}}
\def\convergeas{\stackrel{a.s.}{\longrightarrow}}
\def\converged{\stackrel{d}{\longrightarrow}}
\def\op{\textrm{op}}
\def\kron{\mathbbm{1}}
\def\I{\mathbbm{1}}
\def\st{\textit{s.t.}}
\newcommand{\R}{\mathbb{R}}
\def\tr{\textrm{tr}}
\newcommand{\deri}{\text{d}}

\usepackage{diagbox}

\usepackage[dvipsnames]{xcolor}
\RequirePackage[normalem]{ulem}

\RequirePackage[textsize=tiny, textwidth = 2cm, shadow]{todonotes}

\endlocaldefs

\begin{document}

\begin{frontmatter}
\title{Rerandomization with Diminishing Covariate Imbalance and Diverging Number of Covariates}
\runtitle{Rerandomization}

\begin{aug}
\author[A]{\fnms{Yuhao} \snm{Wang}\ead[label=e1]{yuhaow@tsinghua.edu.cn}}
\and
\author[B]{\fnms{Xinran} \snm{Li}\ead[label=e2]{xinranli@illinois.edu}}
\address[A]{Institute for Interdisciplinary Information Sciences, Tsinghua University,
\printead{e1}}

\address[B]{Department of Statistics, University of Illinois at Urbana-Champaign,
\printead{e2}}
\end{aug}

\begin{abstract}
Completely randomized experiments have been the gold standard for drawing causal inference because they can balance all potential confounding on average. 
However, they 
may 
suffer from unbalanced covariates for realized treatment assignments. 
Rerandomization, a design that 
rerandomizes the treatment assignment 
until a prespecified covariate balance criterion is met, 
has recently got attention due to its easy implementation, improved covariate balance and more efficient inference.  
Researchers have then suggested to 
use 
the treatment assignments that minimize the covariate imbalance, namely the optimally balanced design.  
This has 
caused again the long-time controversy between two philosophies for designing experiments: randomization versus optimal and thus almost deterministic designs. 
Existing literature argued that rerandomization with overly balanced observed covariates can lead to highly imbalanced unobserved covariates, making it vulnerable  to model misspecification. 
On the contrary, rerandomization with 
properly balanced 
covariates can provide robust inference for treatment effects while sacrificing %
some
efficiency compared to the %
ideally
optimal design. 
In this paper, we show it is possible that, by making the covariate imbalance diminishing at a proper rate as the sample size increases, 
rerandomization can achieve its ideally optimal precision that one can expect with perfectly balanced covariates, while still maintaining its robustness. 
We further 
investigate conditions 
on the number of covariates 
for achieving the desired optimality.
Our results rely on a more delicate asymptotic analysis for rerandomization, allowing both diminishing covariate imbalance  threshold (or equivalently the acceptance probability) and diverging number of covariates. 
The derived theory for rerandomization provides a deeper understanding of its large-sample property and can better guide its practical implementation. Furthermore,  it  also helps  reconcile the controversy between 
randomized and optimal designs in an asymptotic sense. 
\end{abstract}



\end{frontmatter}

\section{Introduction}

Since the seminal work of \citet{Fisher:1935}, randomized experiments have become the gold standard for drawing causal inference, since they can balance all potential confounding factors, no matter observed or unobserved, on average. 
Moreover, they allow assumption-free inference for causal effects that uses only the randomization of the treatment assignment as the reasoned basis, without imposing any model or distributional assumption on the experimental units, such as independent and identically distributed (i.i.d.) sampling from some (often hypothetical) superpopulation or some model assumptions for the dependence of potential outcomes on covariates. 
This is often called randomization-based or design-based inference, as well as the finite population inference emphasizing its focus on the finite population of experimental units in hand; 
see  \citet{Fisher:1935} and \citet{N23} for origins of this 
inferential framework.

However, as pointed out by \citet{MR12}, the covariate distribution between two treatment groups are likely to be imbalanced for a realized treatment assignment, and a practical remedy hinted by Fisher is to simply rerandomize. 
The idea of rerandomization is intuitive and has a long history in the literature, traced back to Fisher \citep[see][Page 88]{S62}, \citet{S38} and \citet{C82}; see also \citet{MR12} and references therein.   
Besides, it is often used implicitly in the design of experiments when the allocated treated and control groups exhibit undesired imbalances \citep[see, e.g.,][]{BM09, Heckman2021}, although it is often not well-documented.  
Nevertheless, the rerandomization design was formally proposed, analyzed and advocated until recently by \citet{MR12}. 
As noted by the authors, one main explanation for the less popularity of rerandomization is that it brings additional difficulty in analyzing the experiments, 
and in practice people often ignore rerandomization and analyze the experiments as if they were, say, completely randomized experiments. 
\citet{MR12} then proposed randomization tests for sharp null hypotheses (e.g., the treatment has no effect for any unit) taking into account rerandomization. 
More recently, 
\citet{LDR18}  studied the repeated sampling property of the difference-in-means estimator under rerandomization, which exhibits a non-Gaussian distribution, 
and further demonstrated that the estimator can be more precise under rerandomization than that under complete randomization. 
Importantly, rerandomization still allows assumption-free randomization-based inference as the complete randomization, 
and moreover it provides more efficient difference-in-means estimator and shorter confidence intervals for the average treatment effect. 

Researchers, e.g., \citet{K16} and \citet{K18}, have then suggested 
rerandomization with as small threshold as possible for the covariate imbalance, i.e., an optimally balanced design, 
whose idea can be traced back to \citet{S38}, \citet{K59} and \citet{T74}. 
With general continuous covariates, 
there is likely 
only one acceptable treatment assignment or two if the two treatment groups have equal sizes, resulting in an almost deterministic design.  
Obviously, due to the lack of randomization in the treatment assignment, randomization-based inference is no longer applicable or becomes powerless, since it is generally impossible to asymptotically approximate the randomization distribution of a certain estimator 
(which is a discrete distribution whose support consists of one or two points) and the minimum possible $p$-value from a randomization test is either 1 or $0.5$ \citep{MR12,JRS21}.
Therefore, the statistical inference for an optimally balanced design is often driven by additional distributional assumptions on the experimental units, such as i.i.d.\ sampling of units from some superpopulation that is usually hypothetical \citep{JRS21}, and the criteria for choosing the optimal assignments are often based on some model assumptions for the dependence of potential outcomes on covariates. 

Not surprisingly, there is a long-time debate between these two philosophies, randomized versus optimal (and thus almost deterministic) designs, for conducting experiments. 
Intuitively, it is similar in spirit to the 
classical
trade-off between efficiency and  robustness. 
Randomized design allows assumption-free and robust inference for treatment effects, 
while 
the optimal design tries to maximize the inference efficiency 
under 
some hypothesized data-generating models. 
Specifically, randomized design and its inference can use only the randomization of treatment assignments as the ``reasoned basis'' \citep{Fisher:1935}, with all the potential outcomes being conditioned on or equivalently viewed as fixed constants. 
The optimal design often imposes some probabilistic models on the potential outcomes, and its efficiency and inference relies crucially on 
the randomness in the potential outcomes. 
Thus, these two designs use quite different sources of randomness. 
The randomized design has the advantage that the randomness in the treatment assignment is fully controlled by the experimenter and can be readily available for analysis. 
However, the optimal design relies on the randomness of potential outcomes as well as their dependence on covariates, which may be misspecified in practice. 
For example, 
\citet{KKSSA20} demonstrated that the ``perfect'' allocation with minimum covariate imbalance can endanger the estimation precision because unobserved covariates can be highly imbalanced, 
and 
\citet{BCMS20} suggested that targeting a fixed quantile of balance is safer than targeting an absolute balance objective from an ambiguity-averse decision-making perspective.
Indeed, rerandomization can be viewed as a design standing between the completely randomized design and the optimally balanced design. 
More precisely, 
instead of ignoring covariate imbalance as in the completely randomized design or pursuing the minimum possible covariate imbalance as in the optimally balanced design, rerandomization repeatedly randomize treatment assignments until the induced covariate imbalance is below a certain threshold, which is 
chosen carefully
to ensure there is sufficient randomness in the treatment assignment. 
As demonstrated in \citet{LDR18}, 
under rerandomization 
with a fixed and positive covariate imbalance threshold, 
we can still conduct large-sample randomization-based inference; moreover, the difference-in-means estimator will be more precise (at least asymptotically) than that under complete randomization, which can further lead to shorter confidence intervals for the average treatment effect.

Nevertheless, there is still a gap in the current theory of rerandomization. 
Specifically, \citet{LDR18} showed that, 
the smaller the covariate imbalance threshold or equivalently the 
\textit{acceptance probability} (namely the probability that the covariate imbalance for a completely randomized treatment assignment is below the threshold) is, 
the more precise the difference-in-means estimator will be under rerandomization. 
However, this does not mean we should use 
as small threshold as possible, 
since it essentially leads to the optimal design for which randomization-based inference is not feasible or powerless. 
Technically, this is because the current asymptotic theory for rerandomization in  \citet{LDR18} requires a fixed and positive covariate imbalance threshold that does not change with the sample size. This then raises the theoretical question that if we can conduct asymptotic analysis for rerandomization allowing a sample-size dependent covariate imbalance threshold, especially with the acceptance probability diminishing towards zero. 
Philosophically, we are interested in whether, by diminishing the acceptance probability to zero as sample size increases,  
we can asymptotically achieve the \emph{ideally optimal precision} that one would expect with perfectly balanced covariates while still allowing robust randomization-based inference.

To answer the above questions, 
we 
will
conduct more delicate finite population asymptotic analysis for rerandomization, allowing 
the covariate balance criterion including both the threshold and number of involved covariates
to vary with the sample size. 
Our asymptotic analysis relies on a Berry-Essen-type bound for the finite population central limit theorem. 
We derive the asymptotic distribution of the difference-in-means estimator under rerandomization and construct large-sample confidence intervals for the average treatment effect, 
which 
extends \citet{LDR18} that requires a fixed positive threshold and a fixed number of covariates for rerandomization. 
Moreover, we investigate
whether rerandomization can achieve the ideally optimal precision. 
Specifically, we demonstrate that, 
when the number of covariates satisfies certain conditions (generally being a smaller order of the logarithm of the sample size), 
we can diminish the covariate imbalance threshold such that the corresponding acceptance probability converges to zero at a proper rate and the resulting 
difference-in-means estimator achieves the ideally optimal precision and becomes asymptotically Gaussian distributed, under which we can use the usual Wald-type confidence intervals.
Note that this does not contradicts with the general asymptotic non-Gaussianity for rerandomization established in \citet{LDR18}; it is because the non-Gaussian part in the limiting distribution can be asymptotically ignorable when we diminish the acceptance probability as the sample size increases. 

The paper proceeds as follows. 
Section \ref{sec:framework} introduces the framework and reviews existing results. 
Section \ref{sec:mul_berry_bound} studies the multivariate Berry--Esseen-type bound for the finite population central limit theorem under complete randomization, which serves as the basis for studying the asymptotic properties of rerandomization in Section \ref{sec:asymptotic}. 
Section \ref{sec:optimal} studies 
whether rerandomization can achieve the ideally optimal precision that we can expect with perfectly balanced covariates. 
Section \ref{sec:ci} constructs large-sample confidence intervals for the average treatment effect under rerandomization. 
Section \ref{sec:reg_cond} investigates all the involved regularity conditions and discusses their practical implications, including both the covariate trimming and computational cost.
Section \ref{sec:discuss} concludes with a short discussion.

\section{Framework, Notation and Literature Review}\label{sec:framework}

\subsection{Potential outcomes and treatment assignment}
We consider an experiment with two treatment arms (labeled as treatment and control) and $n$ units, among which $n_1$ units will be assigned to the treatment group and the remaining $n_0 = n-n_1$ units will be assigned to the control group, where $n_1$ and $n_0$ are predetermined fixed integers. 
We invoke the potential outcome framework \citep{N23,R74} to define treatment effects, where each unit $i$ has two potential outcomes $Y_i(1)$ and $Y_i(0)$ depending on its treatment assignment. 
The individual treatment effect for unit $i$ is then $\tau_i = Y_i(1) - Y_i(0)$, 
and the corresponding average treatment effect for all units is $\tau = n^{-1} \sum_{i=1}^n \tau_i$, which is our estimand of interest.  
We use $\bs{X}_i\in \mathbb{R}^K$ to denote the available $K$-dimensional covariate vector for each unit $i$, 
and $Z_i$ to denote the treatment assignment indicator, 
where $Z_i=1$ if the unit receives treatment and $Z_i=0$ otherwise. 
The observed outcome for unit $i$ is then one of its two potential outcomes depending on its treatment assignment, i.e., 
$Y_i = Z_i Y_i(1) + (1-Z_i) Y_i(0)$. 

Throughout the paper, we will conduct the finite population inference where all potential outcomes and covariates are viewed as fixed constants and the randomness in the observed data (e.g., $Y_i$'s) comes solely from the random treatment assignments $Z_i$'s. 
This is 
equivalent to conditional inference conditioning on all potential outcomes and covariates; see \citet{LD17} for a review of finite population inference with emphasis on applications to causal inference. 
The finite population inference has the advantage of imposing no model or distributional assumptions on the potential outcomes or covariates. 
Consequently, the distribution of the treatment assignments for all units, namely the treatment assignment mechanism, plays a crucial role for statistical inference. 
In a completely randomized experiment (CRE), 
the probability that 
the treatment assignment vector $\bs{Z}\equiv (Z_1, Z_2, \ldots, Z_n)^\top$ takes a particular value $\bs{z}\equiv (z_1, z_2, \ldots, z_n)^\top$ is $n_1!n_0!/n!$ if $z_i\in \{0,1\}$ for all $i$ and $\sum_{i=1}^n z_i = n_1$, and zero otherwise.  

For descriptive convenience, we introduce several finite population quantities. 
For $z=0,1$,
let $\bar{Y}(z)$ and $\bar{\bs{X}}$ be the finite population averages of potential outcome and covariates, 
$S_z^2 = (n-1)^{-1} \sum_{i=1}^n \{Y_i(z) - \bar{Y}(z)\}^2$, $\bs{S}^2_{\bs{X}} = (n-1)^{-1} \sum_{i=1}^n (\bs{X}_i - \bar{\bs{X}}) (\bs{X}_i - \bar{\bs{X}})^\top$ 
and 
$\bs{S}_{z, \bs{X}} = \bs{S}_{\bs{X}, z}^\top = (n-1)^{-1} \sum_{i=1}^n  \{Y_i(z) - \bar{Y}(z)\} (\bs{X}_i - \bar{\bs{X}})^\top$
be the finite population variance and covariances for potential outcomes and covariates.
For the individual treatment effect, 
we define analogously its finite population variance $S^2_{\tau} = (n-1)^{-1} \sum_{i=1}^n (\tau_i-\tau)^2$ and its finite population covariance with covariates $\bs{S}_{\tau, \bs{X}} = \bs{S}_{\bs{X}, \tau}^\top = (n-1)^{-1} \sum_{i=1}^n  (\tau_i-\tau) (\bs{X}_i - \bar{\bs{X}})^\top$. 

\subsection{Covariate imbalance and rerandomization}\label{sec:cov_imb}

Under the CRE, the units are completely randomized into the two treatment arms, which guarantees that all pretreatment covariates, no matter observed or unobserved, are balanced \textit{on average} between the two treatment groups. 
However, as pointed out by \citet{MR12}, the covariate imbalance is likely to occur for a realized treatment assignment. 
The classical literature in experimental design \citep[see, e.g.,][]{BHH05} suggests blocking on pretreatment covariates, which, however, is not obvious to implement when the covariates are continuous. 
Recently, \citet{MR12} formally proposed a design called rerandomization to actively avoid the unlucky covariate imbalance, 
by discarding those treatment assignments with unacceptable covariate imbalance. 
A general rerandomization design consists of the following steps. 
\begin{enumerate}
	\item Collect the covariate data for the experimental units, and specify a covariate balance criterion. 
	\item Randomly assign $n_1$ units to treatment group and the remaining $n_0$ units to control group.
	\item Check the covariate balance for the treatment assignment from Step 2. 
	If the balance criterion is satisfied, proceed to Step 4; otherwise return to Step 2.
	\item Conduct the experiment using the acceptable treatment assignment from Step 3. 
\end{enumerate}
The balance criterion in Step 1 is an accept-reject function of the treatment assignment vector $\bs{Z}$ and the pretreatment covariates $\bs{X}_i$'s. 
\citet{MR12} suggested to use the Mahalanobis distance between covariate means in two treatment groups as the covariate balance criterion, which enjoys the affinely invariant property. 
Specifically, 
the difference-in-means of covariates between the two treatment groups is 
\begin{align}\label{eq:tau_hat_X}
	\hat{\bs{\tau}}_{\bs{X}} & \equiv \bar{\bs{X}}_1 - \bar{\bs{X}}_0  =  \frac{1}{n_1} \sum_{i=1}^n Z_i \bs{X}_i - \frac{1}{n_0} \sum_{i=1}^n (1-Z_i) \bs{X}_i, 
\end{align}
where $\bar{\bs{X}}_z$ denotes the covariate mean for units under treatment arm $z$, 
and the corresponding Mahalanobis distance for measuring covariate imbalance is 
\begin{align*}
	M & = \hat{\bs{\tau}}_{\bs{X}}^\top   \cov^{-1}\left( \hat{\bs{\tau}}_{\bs{X}}\right)  \hat{\bs{\tau}}_{\bs{X}}
	=  \hat{\bs{\tau}}_{\bs{X}}^\top 
	\left( \frac{n}{n_1 n_0} \bs{S}^2_{\bs{X}} \right)^{-1}
	\hat{\bs{\tau}}_{\bs{X}}
	\\
	& = 
	\frac{n_1n_0}{n} 
	\left(\bar{\bs{X}}_1 - \bar{\bs{X}}_0\right)^\top 
	\left( \bs{S}^2_{\bs{X}} \right)^{-1}
	\left(\bar{\bs{X}}_1 - \bar{\bs{X}}_0\right). 
\end{align*}
Under rerandomization using the Mahalanobis distance (ReM) with a predetermined threshold $a$, a treatment assignment $\bs{Z}$ is acceptable if and only if the corresponding Mahalanobis distance is less than or equal to the threshold $a$, i.e., $M\le a$. 
Throughout the paper, we will focus on ReM to illustrate our theory. 
Our results can be generalized to other covariate balance criteria as well.

\subsection{Recent results and  challenges}\label{sec:recent_result}

Rerandomization has a long history
and has been utilized a lot in practice, although often implicitly.  
A formal proposition of rerandomization does not appear 
until 
\citet{MR12}, 
likely due to the critique that the classical Gaussian distribution theory is no longer valid for rerandomization; see \citet{MR12} and references therein. 
Recently, \citet{LDR18} demonstrated that the usual difference-in-means estimator, 
\begin{align}\label{eq:tau_hat_Y}
	\hat{\tau} & \equiv \bar{Y}_1 - \bar{Y}_0 = \frac{1}{n_1} \sum_{i=1}^n Z_i Y_i - \frac{1}{n_0}  \sum_{i=1}^n (1-Z_i) Y_i, 
\end{align}
is indeed asymptotically non-Gaussian distributed, 
where $\bar{Y}_z$ denotes the average observed outcome for units under treatment arm $z$.
Specifically, they proved that, with a fixed positive threshold $a$ and a fixed number of covariates $K$ that do not vary with the sample size $n$, 
the asymptotic distribution for $\hat{\tau}$ under ReM has the following form: 
\begin{align}\label{eq:rem_ldr}
	\sqrt{n} (\hat{\tau} - \tau) \mid M\le a 
	\ \dotsim \ 
	\sqrt{n V_{\tau\tau}}\left( \sqrt{1 - R^2} \cdot \varepsilon + R \cdot L_{K, a} \right), 
\end{align}
where $\dotsim$ means that the distributions on both sides of  \eqref{eq:rem_ldr} converge weakly to the same distribution. 
In \eqref{eq:rem_ldr}, 
$\varepsilon\sim \mathcal{N}(0,1)$ follows a standard Gaussian distribution, 
$L_{K, a} \sim D_1 \mid \bs{D}^\top \bs{D} \leq a$ follows a constrained Gaussian distribution with $\bs{D} = (D_1, D_2, \ldots, D_K)^\top \sim \mathcal{N}(\bs{0}, \bs{I}_K)$, 
and $\varepsilon$ and $L_{K, a}$ are mutually independent. 
Besides, 
$V_{\tau\tau}$ is the variance of $\hat{\tau}$ under the CRE, 
$R^2$ is the squared multiple correlation between $\hat{\tau}$ and  $\hat{\bs{\tau}}_{\bs{X}}$ under the CRE, 
and 
we defer the explicit expression for $V_{\tau\tau}$ and $R^2$ to Section \ref{sec:berry_motivation}. 
From \eqref{eq:rem_ldr}, 
the difference-in-means estimator  under ReM is asymptotically distributed as the convolution of a Gaussian and a constrained Gaussian random variables. 
Intuitively, the $\varepsilon$ component represents the part of $\hat{\tau}$ that cannot be explained by $\hat{\bs{\tau}}_{\bs{X}}$, 
and the $L_{K,a}$ component represents the part that can be explained by $\hat{\bs{\tau}}_{\bs{X}}$ and thus it depends on both the threshold $a$ for balance criterion and  the number $K$ of involved covariates.

The asymptotic derivation for \eqref{eq:rem_ldr} requires that both the threshold $a$ and number of covariates $K$ for the Mahalanobis distance criterion are fixed and do not change as the sample  size $n$ increases. 
However, both requirements are likely to be violated in practice. 
We first consider the choice of threshold for rerandomization. 
Generally, smaller threshold can provide us better covariate balance and more precise treatment effect estimator as indicated by \eqref{eq:rem_ldr}; see \citet[Theorem 2]{LDR18}. 
Therefore,  
researchers \citep{K16, K18}
have suggested to use as small threshold as possible, say the minimum Mahalanobis distance between covariate means in the two treatment groups. 
However, as argued by \citet{MR12} and \citet{LDR18}, 
too small threshold can lead to powerless randomization tests and inaccurate asymptotic approximations. 
For example, with general continuous covariates and using the minimum Mahalanobis distance as the covariate imbalance threshold, 
very likely there is  only one (or two when $n_1 = n_0$) acceptable treatment assignment, and the corresponding minimum $p$-value that we can get from randomization tests is 1 (or $0.5$ when $n_1=n_0$), indicating no power to reject any hypothesis at a reasonable significance level. 
Besides, the resulting difference-in-means estimator is either deterministic or having only two possible values, under which it is impossible for the estimator to converge to any continuous distribution, and thus the asymptotic approximation derived in \eqref{eq:rem_ldr} no longer holds. 
Based on these observations, \citet{LDR18} suggested  to use small, but not overly small threshold, for conducting rerandomization, which not  only provides better covariate balance but also allows large-sample valid inference for the average treatment effect that bases only on the randomization of the treatment assignments.

Nevertheless, there is still a theoretical gap for the choice of  the rerandomization threshold. 
The existing study assumes a fixed threshold $a$ that does not vary with the sample size $n$. It is then natural to ask: can we decrease the threshold with the sample size such that the difference-in-means estimator under ReM converges weakly to a Gaussian distribution as the right-hand side of \eqref{eq:rem_ldr} with $a=0$, 
the ideally optimal precision we expect when the covariates are perfectly balanced? 
This 
essentially requires a theoretical understanding of rerandomization when the threshold $a$ (or the acceptance probability) varies and especially converges to zero as the sample size goes to infinity. 

We then consider the number of covariates for rerandomization. 
With the rapidly growing ability for collecting data, it is common to have a large number of covariates for the experimental units. 
For example, 
\cite{BLZ16}, \cite{WDTT2016} and \citet{LeiD20} studied regression adjustment for the CRE in the analysis stage when the experiments were completed. 
However, only a few studies have considered a large number of covariates in the design stage of an experiment; two examples are \citet{BS21} and \citet{ZYR21} where the authors proposed ridge and PCA rerandomizations to deal with collinearity among covariates, an issue that becomes increasingly serious as the number of covariates increases with the sample size. 
There is even fewer studies on the theoretical property of rerandomization when the amount of covariate information increases as the sample size grows. 
Note that  
practitioner often tends to balance as many covariates as possible with the hope to get more precise estimator. This is also hinted by the previous asymptotic result \eqref{eq:rem_ldr} in which the asymptotic distribution becomes more concentrated around zero as $R^2$ (a measure for the association between covariates and potential outcomes) increases. 
Therefore, it is important to establish a theory for rerandomization allowing 
diverging number of covariates, which can also provide guidelines on 
how to choose 
covariates 
for 
rerandomization in practice.

\section{A Multivariate Berry--Esseen-type Bound for the  Finite Population Central Limit Theorem}\label{sec:mul_berry_bound}

\subsection{Motivation and finite population central limit theorem for a fixed dimension} \label{sec:berry_motivation}

The key for deriving the asymptotic property of ReM in \eqref{eq:rem_ldr} includes the following facts. 
First, the distribution of the difference-in-means estimator under ReM is essentially the same as its conditional distribution under the CRE given that the treatment assignment satisfies the Mahalanobis distance criterion, as indicated by the left-hand side of \eqref{eq:rem_ldr}. 
This then motivates us to investigate the joint distribution of the difference-in-means of the outcome and covariates in \eqref{eq:tau_hat_X} and \eqref{eq:tau_hat_Y} under the CRE. 
Second, 
by the finite population central limit theorem~\citep{Hajek60, LD17}, 
the joint distribution of  $(\hat{\tau}, \hat{\bs{\tau}}_{\bs{X}}^\top)^\top$ under the CRE is asymptotically Gaussian with mean and covariance matrix the same as its sampling mean and covariance matrix under the CRE: 
$
\E(\hat{\tau}, \hat{\bs{\tau}}_{\bs{X}}^\top)^\top = 
(\tau, \bs{0}_K^\top)^\top
$
and 
\begin{align}\label{eq:V}
	\cov
	\begin{pmatrix}
		\hat{\tau}\\
		\hat{\bs{\tau}}_{\bs{X}}
	\end{pmatrix}
	= 
	\begin{pmatrix}
		n_1^{-1} S^2_1 + n_0^{-1} S^2_0 - n^{-1} S^2_{\tau} & n_1^{-1} \bs{S}_{1, \bs{X}} + n_0^{-1} \bs{S}_{0, \bs{X}}\\
		n_1^{-1} \bs{S}_{\bs{X}, 1} + n_0^{-1} \bs{S}_{\bs{X}, 0} & 
		n/(n_1n_0) \cdot \bs{S}^2_{\bs{X}}
	\end{pmatrix}
	\equiv 
	\bs{V}
	\equiv
	\begin{pmatrix}
		V_{\tau\tau} & \bs{V}_{\tau\bs{x}}\\
		\bs{V}_{\bs{x}\tau} & \bs{V}_{\bs{x}\bs{x}}
	\end{pmatrix}, 
\end{align}
where we introduce $\bs{V}$ to denote  the covariance matrix of $(\hat{\tau}, \hat{\bs{\tau}}_{\bs{X}}^\top)^\top$ under the CRE and $V_{\tau\tau}$ to denote the variance of $\hat{\tau}$ 
used in \eqref{eq:rem_ldr}. 
Specifically, 
$
\sqrt{n} (\hat{\tau} - \tau, \hat{\bs{\tau}}_{\bs{X}}^\top)^\top 
\  \dotsim \ 
\mathcal{N}(\bs{0}_{K+1}, n\bs{V}), 
$
recalling that $\dotsim$ means that the two distributions have the same  weak limits. 
Based  on these observations, 
\citet{LDR18} demonstrated that  the asymptotic distribution of the difference-in-means estimator under ReM is essentially a conditional distribution from a multivariate Gaussian  distribution, which simplifies to \eqref{eq:rem_ldr} and depends crucially on the squared multiple correlation between potential outcomes and covariates (or more precisely between $\hat{\tau}$ and  $\hat{\bs{\tau}}_{\bs{X}}$ under the CRE): 
\begin{align}\label{eq:R2}
	R^2 & = 
	\text{Corr}^2(\hat{\tau}, \hat{\bs{\tau}}_{\bs{X}}) 
	= 
	\frac{\bs{V}_{\tau\bs{x}} \bs{V}_{\bs{x}\bs{x}}^{-1} \bs{V}_{\bs{x}\tau}}{V_{\tau\tau}}
	= 
	\frac{n_1^{-1} S^2_{1\mid \bs{X}} +  n_0^{-1} S^2_{0\mid \bs{X}} - n^{-1} S^2_{\tau \mid \bs{X}}}{n_1^{-1} S^2_{1} +  n_0^{-1} S^2_{0} - n^{-1} S^2_{\tau}}, 
\end{align}
where $S^2_{z\mid \bs{X}} = \bs{S}_{z, \bs{X}} \bs{S}_{\bs{X}}^{-2} \bs{S}_{\bs{X}, z}$ and $S^2_{\tau\mid \bs{X}} = \bs{S}_{\tau, \bs{X}} \bs{S}_{\bs{X}}^{-2} \bs{S}_{\bs{X}, \tau}$ are the finite population variances of the linear projections of potential outcomes $Y_i(z)$'s and individual effects $\tau_i$'s on the covariates $\bs{X}_i$'s. 

Apparently, the above arguments require a fixed number of covariates $K$. Moreover, the weak convergence from the joint distribution to the conditional distribution requires that the probability of the conditioning event $\PP(M \leq a)$ has a positive limit, which implies a positive and non-diminishing threshold $a$ for the Mahalanobis distance criterion. Otherwise, the original derivation in \cite{LDR18} will involve ratios between terms of order $o(1)$, e.g., $\PP\{ \sqrt{n}(\hat{\tau} - \tau) \le c \mid M\le a\} = \PP\{\sqrt{n}(\hat{\tau} - \tau) \le c, M\le a\}/\PP(M\le a)$, of which the limits are unclear. 

From the above discussion, it is obvious that the original form of finite population central limit theorem is not enough for studying the asymptotic property of rerandomization with %
a diminishing threshold and a diverging number of covariates. 
Furthermore, it lefts the question that whether rerandomization with  
threshold or acceptance probability diminishing
at a certain rate can lead to difference-in-means estimator with the ideally optimal precision. 
We will address these concerns in the remaining of the paper.

\subsection{Gaussian approximation under the completely randomized experiment}\label{sec:berry_rate}

We first study the convergence rate for the finite population central limit theorem under the CRE. 
More precisely, we will focus on  the convergence rate for the Gaussian approximation of the joint distribution of the difference-in-means of the outcome and covariates under the CRE, and investigate explicitly how the convergence rate depends on the finite population including the dimension of the covariates. 

Let $r_1 = n_1/n$ and $r_0 = n_0/n$ be the proportions of treated and control units, 
and for each unit $1\le i \le n$, let 
$\bs{u}_i = (r_0 Y_i(1) + r_1 Y_i(0), \bs{X}_i^\top)^\top \in \mathbb{R}^{K+1}$ be a vector consisting of a weighted average of the two potential outcomes and the covariates. 
By the definitions in \eqref{eq:tau_hat_X} and \eqref{eq:tau_hat_Y}, 
we can verify that 
the difference-in-means vector 
$(\hat{\tau}, \hat{\bs{\tau}}_{\bs{X}}^\top)^\top$ 
has the following equivalent form:
\begin{align}\label{eq:did_srs_u}
	\begin{pmatrix}
		\hat{\tau}\\
		\hat{\bs{\tau}}_{\bs{X}}
	\end{pmatrix}
	=
	\frac{n}{n_1n_0}
	\sum_{i=1}^n Z_i 
	\bs{u}_i
	- 
	\frac{n}{n_0}
	\begin{pmatrix}
		\bar{Y}(0)\\
		\bar{\bs{X}}
	\end{pmatrix}, 
\end{align}
which, up to a linear transformation, is essentially the summation of a simple random sample of size $n_1$ from the finite population $\mathcal{U}_n \equiv \{\bs{u}_i: i=1,2,\ldots,n\}$. 
Thus, the sampling property of the difference-in-means $(\hat{\tau}, \hat{\bs{\tau}}_{\bs{X}}^\top)^\top$ can be fully characterized by the population $\mathcal{U}_n$. 
Let $\bar{\bs{u}} = n^{-1} \sum_{i=1}^n \bs{u}_i$ and $\bs{S}_{\bs{u}}^2 = (n-1)^{-1} \sum_{i=1}^n (\bs{u}_i - \bar{\bs{u}}) (\bs{u}_i - \bar{\bs{u}})^\top$ be the finite population mean and covariance matrix for $\mathcal{U}_n$, 
and let $\bs{S}_{\bs{u}}^{-1}$ denote the inverse of the positive semidefinite square root of $\bs{S}_{\bs{u}}^2$.
Define 
\begin{align}\label{eq:gamma_n}
	\gamma_n \equiv 
	\frac{(K+1)^{1/4}}{\sqrt{n r_1r_0}} \frac{1}{n} \sum_{i=1}^n \left\| \bs{S}_{\bs{u}}^{-1} (\bs{u}_i-\bar{\bs{u}}) \right\|_2^3, 
\end{align}
which is the third moment of the standardized finite population $\{\bs{S}_{\bs{u}}^{-1} (\bs{u}_i-\bar{\bs{u}}): i=1,2,\ldots,n\}$ up to a certain scale. 
For descriptive convenience, 
we define $\gamma_n$ to be infinity when $r_1$ or $r_0$ equals zero or $\bs{S}^2_{\bs{u}}$ is singular. 
In \eqref{eq:gamma_n}, we use the subscript $n$ to emphasize the dependence of $\gamma_n$ on the finite population $\mathcal{U}_n$ of size $n$. 
Note that 
$\gamma_n$ is uniquely determined by $r_1, r_0$ and the potential outcomes and covariates of the $n$ experimental units.

Below we consider the Berry--Esseen-type bound for the Gaussian approximation of the difference-in-means vector in \eqref{eq:did_srs_u} under the CRE. 
Note that, under the CRE, $(\hat{\tau}, \hat{\bs{\tau}}_{\bs{X}}^\top)^\top$ has mean $(\tau, \bs{0}_{1\times K})^\top$ and covariance matrix $\bs{V}$ as in \eqref{eq:V}. 
Let $\cC_{K+1}$ denote the collection of all measurable convex sets in $\mathbb{R}^{K+1}$. 
We then focus on bounding the supremum of the absolute difference between the probabilities of being in any measurable convex set for the standardized difference-in-means vector and the standard Gaussian random vector: 
\begin{align}\label{eq:Delta_n}
	\Delta_n & \equiv 
	\sup_{\mathcal{Q} \in \cC_{K+1}}
	\left|
	\PP
	\left\{\bs{V}^{-1/2}
	\begin{pmatrix}
		\hat{\tau} - \tau\\
		\hat{\bs{\tau}}_{\bs{X}}
	\end{pmatrix}
	\in \mathcal{Q}
	\right\}
	- 
	\PP \left( \bs{\varepsilon} \in \mathcal{Q} \right)
	\right|.
\end{align}
By some algebra, $\bs{V} = (nr_1r_0)^{-1} \bs{S}^2_{\bs{u}}$, and thus $\Delta_n$ is well-defined as long as $\gamma_n < \infty$. 
For descriptive convenience, we define $\Delta_n$ to be 1 when $r_1$ or $r_0 $ equals zero or $\bs{S}^2_{\bs{u}}$ is singular. 
The bound for \eqref{eq:Delta_n} is a natural multivariate extension of the classical univariate Berry--Esseen bound for the absolute difference between two distribution functions. 
More importantly, it suffices for our asymptotic analysis of rerandomization, noticing that the acceptance region for $\hat{\bs{\tau}}_{\bs{X}}$ under the Mahalanobis distance criterion is indeed a convex set in $\mathbb{R}^K$. 
From \eqref{eq:did_srs_u}, 
we essentially need to 
understand the Berry--Esseen-type bound for the central limit theorem under simple random sampling, which itself is also a special case of the combinatorial central limit theorem. 
Below we give a brief literature review.

\citet{berry1941accuracy} and \citet{esseen1942liapunov} independently discovered the original Berry--Esseen theorem when studying the convergence rate for Gaussian approximation of summations of independent univariate random variables. 
\citet{bentkus2003dependence,B05} and \citet{R19} then extended it to the multivariate case, considering the Gaussian approximation for probabilities of being in any measurable convex sets. 
Recently, \citet{chernozhukov2017central}, \citet{chernozhukov2020nearly} and  \citet{fang2021high} achieved tighter bounds by focusing only on Gaussian approximation for probabilities of being in hyperrectangles (or more generally sparsely convex sets), where the bounds can vanish even when the dimension of random vectors is much larger than the sample size in the summation. 
Note that all of these results are for independent summands. 

In the context of combinatorial central limit theorem (including the central limit theorem for simple random sampling as a special case), 
the summands become weakly dependent. 
\citet{B69} and \citet{H78} studied the corresponding Berry--Esseen-type bound in the univariate case. 
However, there has been much less study for the multivariate case, in contrast to the rich literature for independent summands. 
One exception is \citet{BG93}, who established the Berry--Esseen-type bound for the multivariate combinatorial central limit theorem. 
Based on their results, 
we can show that there exists an absolute constant $C_K$ that depends only on the dimension $K$ such that $\Delta_n \le C_K \gamma_n$, with $\gamma_n$ and $\Delta_n$ defined in \eqref{eq:gamma_n} and \eqref{eq:Delta_n}. 
However, the authors did not characterize how the constant $C_K$ may increase with the dimension $K$ and thus the bound is not sufficient for studying rerandomization with diverging number of covariates.
To the best of our knowledge, 
there has not been any formal result of the Berry--Esseen-type bound for the combinatorial central limit theorem with explicit dependence on the dimension, except for an informal result presented by \citet{R15} 
at a workshop. 
Based on the result in \citet{R15}, we can show that there exists an absolute constant $C$ such that $\Delta_n \le C \gamma_n$, 
noting that the definition of $\gamma_n$ in \eqref{eq:gamma_n} involves a term of $(K+1)^{1/4}$ that depends explicitly on the dimension of the difference-in-means vector in \eqref{eq:did_srs_u}.

Since the result in \citet{R15} has not been proved yet,
we also derive a Berry--Esseen-type bound for the central limit theorem  under simple random sampling ourselves. Our proof makes use of the multivariate Berry--Esseen-type bound for sum of independent random vectors \citep[see, e.g.,][]{R19} and the coupling between simple random sampling and Bernoulli independent sampling utilized by~\citet{Hajek60}. 
Based on our proof, we can derive that $\Delta_n \le 174\gamma_n + 7 \gamma_n^{1/3}$, 
where the first term of $\gamma_n$ is from the Bernoulli independent sampling or more generally the Berry--Esseen-type bound for sum of independent random vectors, and the additional term of $\gamma_n^{1/3}$ comes from the coupling between simple random sampling and Bernoulli independent sampling. There is actually a tighter bound than $\gamma_n^{1/3}$ for the coupling, but we present the bound $\gamma_n^{1/3}$ for the ease of understanding; see the Supplementary Material \citep{WLSupp2022} for more  details. 
Obviously, our rate of convergence is slower than that conjectured in \citet{R15}. 
Nevertheless, 
it is still able to reveal the interesting property of rerandomization with diminishing threshold for covariate imbalance and diverging number of covariates, as studied in detail shortly. 
We summarize these results for bounding $\Delta_n$ in the following theorem. 

\begin{theorem}\label{thm:berry_esseen_clt}
	For any $n\ge 2, K\ge 0$, $r_1, r_0\in (0,1)$, 
	and any finite population $\Pi_n \equiv \{(Y_i(1), Y_i(0), \bs{X}_i): i=1,2,\ldots, n\}$ with nonsingular $\bs{V}$ defined as in \eqref{eq:V}, 
	define $\gamma_n$ and $\Delta_n$ as in \eqref{eq:gamma_n} and \eqref{eq:Delta_n}. 
	Then 
	\begin{itemize}
		\item[(i)] 
		there exists an absolute constant $C_K$ that depends only on $K$ such that $\Delta_n \le C_K \gamma_n$; 
		\item[(ii)] if the conjecture in \citet{R15} hold, then there exists a universal constant $C$ such that $\Delta_n \le C \gamma_n$; 
		\item[(iii)] 
		$\Delta_n \le 174\gamma_n + 7 \gamma_n^{1/3}$. 
	\end{itemize}
\end{theorem}

\subsection{Gaussian approximation with stronger moment conditions}\label{sec:clt_higher_order}

Theorem \ref{thm:berry_esseen_clt} provides Berry--Esseen-type bounds for the finite population central limit theorem under complete randomization, with (ii) and (iii) characterizing explicit dependence on the dimension of covariates $K$ and thus crucial for studying rerandomization with diverging number of covariates. 
Compared to that conjectured by \citet{R15}, 
our derived bound in Theorem \ref{thm:berry_esseen_clt}(iii) has an additional term of order $\gamma_n^{1/3}$. 
As discussed before, this additional term is due to the coupling between simple random sampling and Bernoulli independent sampling. 
Intuitively, under simple random sampling (or equivalently complete randomization), the treatment indicators for all units are dependent, because the total number of units assigned to the active treatment is constrained to be $n_1$. 
Such a dependence among these indicators makes it more challenging to bound the error for Gaussian approximation.  
Ignoring the dependence on $K$, 
$\gamma_n$ is of order $n^{-1/2}$, making $\gamma_n^{1/3}$ of order $n^{-1/6}$ and thus the bound in Theorem \ref{thm:berry_esseen_clt}(iii) larger than usual Berry-Essen-type bounds. 

Below we also provide more accurate bounds for the coupling between simple random sampling and Bernoulli independent sampling and consequently improve the Berry--Esseen-type bound for the Gaussian approximation in Theorem \ref{thm:berry_esseen_clt}(iii), at least in terms of the explicit dependence on the sample size $n$, 
with stronger moment conditions on the centered finite population $\{\bs{S}_{\bs{u}}^{-1} (\bs{u}_i-\bar{\bs{u}}): 1\le i \le n\}$. 
We summarize the results in the following theorem.

\begin{theorem}\label{thm:Berry--Esseen-higher-order}
	Under the same setting as Theorem~\ref{thm:berry_esseen_clt}, 
	\begin{align*}
	    \Delta_n
			\le 
			180 \gamma_n + \frac{3  (\log n)^{3/4} (K+1)^{3/4}}{ n^{1/4} \sqrt{r_1r_0}} \cdot \max_{1 \le i \le n} \left\| \bs{S}_{\bs{u}}^{-1} (\bs{u}_i-\bar{\bs{u}})\right\|_\infty,
	\end{align*}
	and, for any $\iota\ge 2$, there exists a universal constant $C_{\iota}$ depending only on $\iota$ such that 
	\begin{align*}
		\Delta_n
		\le 
		174 \gamma_n 
        +  
        \frac{ C_{\iota} (K+1)^{3\iota/\{4(\iota+1)\}}}{n^{\iota/\{4(\iota+1)\}} \{r_1r_0\}^{\iota/2}} \cdot \frac{1}{n}\sum_{i = 1}^n 
        \left\| \bs{S}_{\bs{u}}^{-1} (\bs{u}_i-\bar{\bs{u}}) \right\|_{\iota}^\iota.
		\end{align*}
\end{theorem}

From Theorem \ref{thm:Berry--Esseen-higher-order} and ignoring the dependence on $K$, 
if all coordinates of the centered finite population have bounded $\iota$th moments for some $\iota\ge 2$, 
then the additional term in the Berry--Esseen-type bound is of order $n^{- \iota/\{4(\iota+1)\}}$; 
if further they are bounded, then the additional term becomes of order $n^{-1/4}$ except for a $\log n$ term. 
These bounds are more accurate than that of order $n^{-1/6}$ in Theorem \ref{thm:berry_esseen_clt}(iii), but there are still gaps between them and that of order $n^{-1/2}$ conjectured by \citet{R15}, which requires future work. 
Nevertheless, the derived bounds are already sufficient to discover interesting properties of rerandomization and provide almost the same quantitative message for the design of rerandomization (see, e.g., Table \ref{tab:rate} and its discussion).

It is worth mentioning that there are other approaches for deriving Berry--Esseen-type bound, such as Stein's method that was actually used by \citet{BG93} and helps justify Theorem \ref{thm:berry_esseen_clt}(i); see also the recent work by \citet{shi2022berry}. 
Here we use the coupling approach, mainly because we can utilize the recent results on Berry--Esseen-type bounds for Gaussian approximations of summations of independent random vectors \citep{B05, R19}. 
It will be interesting to investigate whether other approaches can give tighter Berry--Esseen-type bounds or even prove the conjectured rate in \citet{R15}.
In addition, the coupling approach also justifies central limit theorems for stratified randomized experiments \citep{Bickel1984, liu2021factorial}, and our results can be useful for deriving the corresponding Berry--Esseen-type bounds. 
We leave these for future work.

\section{Asymptotic Property of Rerandomization with Sample-size Dependent Mahalanobis Distance Criterion}\label{sec:asymptotic}

Throughout the paper, we conduct finite population asymptotic analysis for rerandomization. 
Specifically, we embed the finite population of size $n$ into a  sequence of finite populations with increasing sizes. 
Importantly, we allow both the threshold $a$ and dimension of covariates $K$ for the Mahalanobis distance criterion to depend on the sample size $n$, 
and will write them explicitly as $a_n$ and $K_n$, using the subscript $n$ to emphasize such dependence. 
We further define $p_n \equiv  \PP(\chi^2_{K_n} \le a_n)$, 
where $\chi^2_{K_n}$ denotes a random variable following the chi-square distribution with degrees of freedom $K_n$. 
By the definition of $\Delta_n$ in \eqref{eq:Delta_n}, 
we can derive that the acceptance probability of a completely randomized treatment assignment under ReM, $\PP(M \le a_n)$, is bounded between $p_n - \Delta_n$ and $p_n + \Delta_n$. Thus, we can intuitively understand $p_n$ as the approximate acceptance probability for rerandomization; 
specifically, $\PP(M\le a_n) / p_n = 1 + o(1)$ when $p_n \gg \Delta_n$. 
In practice, given the number of covariates $K_n$, the choice of the threshold $a_n$ is often based on the approximate acceptance probability $p_n$, 
i.e., $a_n$ is the $p_n$th quantile of the 
chi-square distribution 
with degrees of freedom $K_n$. 
For example, 
\citet{MR12} and \citet{LDR18} suggested to choose small but not overly small approximate acceptance probablity, e.g., $p_n = 0.001$. 
Therefore, in the remaining discussion, we will mainly focus on the approximate acceptance probability $p_n$ and number of covariates $K_n$, since they are more relevant for the practical implementation of ReM, and view the threshold $a_n$ as a deterministic function of $p_n$ and $K_n$. 
For descriptive convenience, we sometimes call $p_n$ simply as the acceptance probability, while emphasizing $\PP(M \le a_n)$ as the actual acceptance probability.

\subsection{Asymptotic distribution under ReM}

We first invoke the following regularity condition on the sequence of finite populations, which, by Theorem \ref{thm:berry_esseen_clt}, implies the Gaussian approximation for the difference-in-means of the outcome and covariates under the CRE.  

\begin{condition}\label{cond:gamma_n}
	As $n\rightarrow \infty$, 
	the sequence of finite populations satisfies that 
	$\gamma_n \rightarrow 0$. 
\end{condition} 
Recall the definition of $\gamma_n$ in \eqref{eq:gamma_n}.  
Condition \ref{cond:gamma_n} requires that, for sufficiently large sample size $n$,  
there are positive proportions of units in both treatment and control groups (i.e., $r_1 > 0$ and $r_0 > 0$), 
and the covariance matrix $\bs{V}$ in \eqref{eq:V} for the difference-in-means vector $(\hat{\tau}, \hat{\bs{\tau}}_{\bs{X}}^\top)^\top$ is nonsingular. 
The latter essentially requires that the covariates are not collinear, which can be guaranteed by our design, and that the potential outcomes cannot be pefectly explained by covariates, in the sense that the corresponding $R^2$ in \eqref{eq:R2} is strictly less than 1, which is likely to hold in most applications. 
Besides, 
as demonstrated in the Supplementary Material \citep{WLSupp2022}, 
\begin{align}\label{eq:gamma_n_lower}
	\gamma_n \ge 2^{-3/2} (nr_1r_0)^{-1/2} (K_n+1)^{7/4}. 
\end{align}
Thus, a necessary condition for Condition \ref{cond:gamma_n} is $K_n = o((nr_1r_0)^{2/7}) = o(n^{2/7})$, 
i.e., the number of covariates increases at most a polynomial rate of the sample size as $n$ goes to infinity. 
If the standardized finite population $\{\bs{S}_{\bs{u}}^{-1} (\bs{u}_i-\bar{\bs{u}}):1\le i\le n\}$ is coordinate-wise bounded, and the proportions of treated and control units $r_1$ and $r_0$ are bounded away from zero, 
	then $K_n = o(n^{2/7})$ is also sufficient for Condition \ref{cond:gamma_n}; see the Supplementary Material \citep{WLSupp2022} for details. 
We defer more detailed discussions about Condition \ref{cond:gamma_n} to Section \ref{sec:reg_cond}.

We then invoke the following regularity condition on the choice of the acceptable probability, 
which is coherent with 
our intuition that too small threshold can prevent asymptotic approximation for ReM based on Gaussian and constrained Gaussian distributions, as discussed in Section \ref{sec:recent_result}. 

\begin{condition}\label{cond:p_n}
	As $n\rightarrow \infty$, $p_n/\Delta_n \rightarrow \infty$. 
\end{condition}

We are now ready to present our formal result for the asymptotic distribution of the difference-in-means estimator $\hat{\tau}$ under ReM. 
Recall that $V_{\tau\tau}$ in \eqref{eq:V} is the variance of  $\hat{\tau}$ under the CRE,  $R^2$ in \eqref{eq:R2} is the squared multiple correlation between the difference-in-means of the outcome and covariates under the CRE, and $\varepsilon_0$ and $L_{K_n,a_n}$ are independent standard Gaussian and constrained Gaussian random variables defined as in Section \ref{sec:recent_result}. 
To emphasize its dependence on the sample size $n$,  we will write $R^2$ explicitly as  $R_n^2$. 

\begin{theorem}\label{thm:dim_rem}
	Under ReM and Conditions \ref{cond:gamma_n} and \ref{cond:p_n}, 
	as $n\rightarrow \infty$, 
	\begin{align}\label{eq:tau_hat_rem}
		\sup_{c\in \mathbb{R}}\left| \PP \big\{ V_{\tau\tau}^{-1/2}( \hat{\tau} - \tau) \le c \mid M \le a_n \big\} - \PP\big( \sqrt{1-R^2_n}\ \varepsilon_0  + \sqrt{R^2_n} \ L_{K_n, a_n} \le c \big)
		\right| 
		\rightarrow 0.
	\end{align}
\end{theorem}

From Theorem \ref{thm:dim_rem}, 
under ReM, the difference between the difference-in-means estimator and the true average treatment effect, $\hat{\tau} - \tau$, follows asymptotically the distribution of $V_{\tau\tau}^{1/2}(\sqrt{1-R^2_n}\varepsilon_0  + \sqrt{R^2_n} L_{K_n, a_n} )$, a convolution of standard Gaussian and constrained Gaussian random variables, with coefficients depending on $V_{\tau\tau}$ and $R^2_n$.
Compared to \eqref{eq:rem_ldr}, 
the asymptotic distribution in Theorem \ref{thm:dim_rem} has the same form as that in \citet[][Theorem 1]{LDR18}, 
which is not surprising given that both theorems focus on the same estimator $\hat{\tau}$ and the same design ReM. 
However, Theorem \ref{thm:dim_rem} is more general 
and it covers  \citet[][Theorem 1]{LDR18} as a special case. 
Specifically, when both $K_n$ and $a_n > 0$ are fixed and do not change with $n$, 
$p_n$ is a fixed positive constant 
and Condition \ref{cond:p_n} holds immediately from Condition \ref{cond:gamma_n} and Theorem \ref{thm:berry_esseen_clt}. 
Besides, Condition \ref{cond:gamma_n} is 
almost implied by
the regularity condition involved in \citet{LDR18}; see the Supplementary Material \citep{WLSupp2022} for more details.

More importantly, Theorem \ref{thm:dim_rem}  allows the dimension of covariates and the acceptance probability, as well as the rerandomization threshold, to vary with the sample size. 
Note that from Theorem \ref{thm:berry_esseen_clt}, $\gamma_n \to 0$ implies $\Delta_n \to 0$. 
Thus, Condition \ref{cond:p_n} holds naturally if we choose $p_n$ to be any fixed positive number. 
Moreover, we can also let $p_n$ decrease with $n$ and eventually converge to zero as $n\rightarrow \infty$, while maintaining that $p_n \gg \Delta_n$, under which the actual acceptance probability $\PP(M\le a_n) = p_n (1+o(1))$ also converges to zero as  $n\rightarrow \infty$. 
In simple words, 
Theorem \ref{thm:dim_rem}  allows the acceptance probability to converge to zero as the sample size goes to infinity. 

Note that all potential confounding factors, no matter observed or unobserved, can always be viewed as potential outcomes unaffected by the treatment. Therefore, Theorem \ref{thm:dim_rem} also implies that any potential confounding factor is asymptotically balanced between the two treatment groups.

\subsection{Asymptotic improvement from ReM}

We now investigate the improvement from ReM compared to the CRE, and in particular how such improvement depends on the acceptance probability and the covariate information. %
Note that the CRE can be viewed as a special case of ReM with $\bs{X} = \emptyset$ and $a_n = \infty$. 
By the same logic as Theorem \ref{thm:dim_rem}, we can derive that
\begin{align}\label{eq:tau_hat_cre}
	\sup_{c\in \mathbb{R}}\left| \PP \big\{ V_{\tau\tau}^{-1/2}( \hat{\tau} - \tau) \le c 
	\big\} - \PP\big( \varepsilon_0  \le c \big)
	\right|
	\converge 0, 
\end{align}
i.e., $\hat{\tau} - \tau$ is asymptotically Gaussian distributed with mean zero and variance $V_{\tau\tau}$ under the CRE. 
Obviously, the asymptotic distribution in \eqref{eq:tau_hat_cre} is a special form of that in \eqref{eq:tau_hat_rem} with $R_n^2=0$, 
which is intuitive in the sense that ReM with irrelevant covariates is asymptotically equivalent to the CRE. 
However, when $R_n^2>0$, which is likely to hold in practice, 
we expect ReM to provide more precise difference-in-means estimator than the CRE, 
as discussed in detail below. 

From Theorem \ref{thm:dim_rem} and by the same logic as \citet[][Corollaries 1--3]{LDR18}, we can derive the following asymptotic properties of ReM, demonstrating its advantage over the CRE. 
For $\alpha\in (0,1)$, 
let $\nu_{\alpha, K, a}(R^2)$ denote the $\alpha$th quantile of the distribution of $\sqrt{1-R^2}\varepsilon_0  + \sqrt{R^2} L_{K, a}$, 
and $z_\alpha$ denote the $\alpha$th quantile of the standard Gaussian distribution. 
We further introduce $v_{K, a} = \var(L_{K,a})$ to denote the variance of the constrained Gaussian random variable. 
From \citet{MR12}, $v_{K,a} = \PP(\chi^2_{K+2}\le a)/\PP(\chi^2_{K}\le a)$, 
where $\chi^2_{K+2}$ and $\chi^2_{K}$ follow chi-square distributions with degrees of freedom $K+2$ and $K$, respectively.

\begin{corollary}\label{cor:improve}
	Under ReM and Conditions \ref{cond:gamma_n} and \ref{cond:p_n}, 
	the asymptotic distribution of $V_{\tau\tau}^{-1/2}(\hat{\tau} - \tau)$, $\sqrt{1-R^2_n}\varepsilon_0  + \sqrt{R^2_n} L_{K_n, a_n}$, as shown in \eqref{eq:tau_hat_rem} is symmetric and unimodal around zero. 
	Compared to the asymptotic distribution 
	in \eqref{eq:tau_hat_cre}  under the CRE, 
	the percentage reductions in asymptotic variance and length of asymptotic $1-\alpha$ symmetric quantile range for $\alpha \in (0,1)$ are, respectively, 
	\begin{align}\label{eq:percentage}
		(1-v_{K_n, a_n}) R^2_n
		\quad \text{and} \quad 
		1 - \frac{\nu_{1-\alpha/2, K_n, a_n}(R^2_n)}{z_{1-\alpha/2}}. 
	\end{align}
	Both percentage reductions in \eqref{eq:percentage} are nonnegative and are uniquely determined by $(R^2_n, p_n, K_n)$, and they are nondecreasing in $R^2_n$ and nonincreasing in $p_n$ and $K_n$. 
\end{corollary}

First, from Corollary \ref{cor:improve}, the difference-in-means estimator $\hat{\tau}$ is asymptotically unbiased for the average treatment effect $\tau$. 
As pointed out by \citet{MR12}, when the treated and control groups have different sizes (i.e., $n_1 \ne n_0$), the difference-in-means estimator is generally biased. Corollary \ref{cor:improve} shows that the bias goes away as the sample size goes to infinity when Condition \ref{cond:gamma_n} holds, which requires that $r_1$ and $r_0$ are not too close to zero as implied by the definition in \eqref{eq:gamma_n}. 
Second, the improvement from ReM on estimation precision is nondecreasing in the strength of the association between potential outcomes and covariates measured by $R^2_n$. 
Generally, the more covariates involved in rerandomization, the larger the $R^2_n$ will be. 
However, this does not mean that we should use as many covariates as possible. 
If the additional covariates provide little increment for $R^2_n$, the gain from ReM will deteriorate since the percentage reductions in \eqref{eq:percentage} are nonincreasing in $K_n$ with fixed $R^2_n$ and $p_n$. 
Third, 
both percentage reductions in \eqref{eq:percentage} are nonincreasing in 
the acceptance probability $p_n$, 
and approach their maximum values $R^2_n$ and $1-\sqrt{1-R^2_n}$ when $p_n$ equals $0$. 
Again, this does not mean that we should use as small threshold as possible, 
since the asymptotic derivation in Corollary \ref{cor:improve} requires that $p_n/\Delta_n \rightarrow \infty$ as $n \rightarrow \infty$. 
This then raises the question that if we can choose $p_n$ 
such that both percentage reductions in \eqref{eq:percentage} achieve their maximum values 
as $n\rightarrow \infty$ or the asymptotic distribution of $V_{\tau\tau}^{-1/2}(\hat{\tau} - \tau)$ under ReM becomes essentially Gaussian with mean zero and variance $1-R^2_n$ (i.e., the asymptotic distribution in \eqref{eq:tau_hat_rem} with $L_{K_n, a_n}$ replaced by zero), while still maintaining $p_n / \Delta_n \to \infty$. 
In other words, can we choose the acceptable probability such that rerandomization achieves its ideally optimal precision?
We will answer this question in the next section. 

\section{Optimal Rerandomization with Diminishing Acceptance Probability}\label{sec:optimal}

In this section, we investigate whether rerandomization can achieve its ideally optimal precision by diminishing the acceptance probability to zero at a proper rate as the sample size increases. 
Specifically, we wonder if the asymptotic approximation in \eqref{eq:tau_hat_rem} can hold with $L_{K_n, a_n}$ replaced by zero, and what conditions we need to impose on the sequence of finite populations as well as the choice of acceptance probability. 
These questions rely crucially on the 
asymptotic behavior of the constrained Gaussian random variable $L_{K_n, a_n}$, and in particular its dependence on the acceptance probability $p_n$. 
Below we will first study the asymptotic property of $L_{K_n, a_n}$, and then investigate whether we are able to achieve the ideally optimal rerandomization. 

\subsection{Asymptotic properties of the constrained Gaussian random variable}\label{sec:asym_con_Gauss}

In this subsection, we will mainly focus on the asymptotic behavior of the variance of $L_{K_n, a_n}$, which as mentioned earlier has the following equivalent form: $v_{K_n, a_n} \equiv \var(L_{K_n, a_n}) = \PP(\chi^2_{K_n+2}\le a_n)/\PP(\chi^2_{K_n}\le a_n)$, due to the following two reasons. 
First, as $n\rightarrow \infty$, 
$L_{K_n, a_n} \stackrel{\PP}{\longrightarrow} 0$ if and only if $v_{K_n, a_n} \rightarrow 0$. 
This is because the random variables $L^2_{K_n,a_n}$ for all $n$ are always uniformly integrable, regardless of how $K_n$ and $p_n$ vary with $n$, as demonstrated in the Supplementary Material \citep[Proposition~A2]{WLSupp2022}. 
Second, as shown in Corollary  \ref{cor:improve}, the percentage reduction in asymptotic variance under ReM is $(1-v_{K_n, a_n})R^2_n$, and its relative difference from the ideally optimal percentage reduction is $1 - (1-v_{K_n, a_n})R^2_n/R^2_n = v_{K_n, a_n}$. 
Thus, the variance of  $L_{K_n, a_n}$ characterizes how different ReM is from the ideally optimal one in terms of the improvement on estimation precision, 
and such a difference will become asymptotically negligible if and only if $v_{K_n, a_n} \rightarrow 0$ as $n\rightarrow \infty$.

The following theorem shows the asymptotic behavior of $v_{K_n, a_n}$, which depends crucially on the asymptotic behavior of the  ratio between $\log(p_n^{-1})$ and $K_n$. 

\begin{theorem}\label{thm:v_Ka}
	As $n\rightarrow \infty$, 
	\begin{itemize}
		\item[(i)] if $\log(p_n^{-1}) / K_n \rightarrow \infty$, then $v_{K_n, a_n} \rightarrow 0$; 
		\item[(ii)] if $\limsup_{n\rightarrow \infty}\log(p_n^{-1}) / K_n < \infty$, then $\liminf_{n\rightarrow \infty} v_{K_n, a_n}  > 0$; 
		\item[(iii)] if $\liminf_{n\rightarrow \infty}\log(p_n^{-1}) / K_n >0$, then $\limsup_{n\rightarrow \infty} v_{K_n, a_n}  < 1$; 
		\item[(iv)] if $\log(p_n^{-1}) / K_n \rightarrow 0$, then $v_{K_n, a_n} \rightarrow 1$. 
	\end{itemize}
\end{theorem}

From Theorem \ref{thm:v_Ka},
the smaller the $p_n$ and $K_n$, 
the larger the ratio $\log(p_n^{-1}) / K_n$, and the smaller $v_{K_n, a_n}$ tends to be. 
This is intuitive noting that Corollary \ref{cor:improve} implicitly implies that $v_{K_n, a_n}$, viewed as a function of $(p_n, K_n)$, is nondecreasing in $p_n$ and $K_n$. 
Theorem \ref{thm:v_Ka} has the following implications. 
First, $v_{K_n, a_n}$ or equivalently $L_{K_n, a_n}$ becomes asymptotically negligible if and only if $\log(p_n^{-1})/K_n \rightarrow \infty$ as $n\rightarrow \infty$. 
Intuitively, this means that the constrained Gaussian term in the asymptotic distribution in \eqref{eq:tau_hat_rem} becomes asymptotically negligible if and only if the acceptance probability $p_n$ decreases super-exponentially with respect to the dimension of covariates, i.e., $p_n = \exp(-c_n K_n)$ with $c_n \rightarrow \infty$ as $n\rightarrow \infty$. 
Second, if $\log(p_n^{-1}) / K_n \rightarrow 0$, then the variance of the constrained Gaussian variable $L_{K_n, a_n}$ becomes asymptotically equivalent to that of the unconstrained standard Gaussian random variable, and, from Corollary \ref{cor:improve}, ReM asymptotically provides no gain compared to the CRE in the sense that the percentage reduction in asymptotic variance converges to zero as $n\rightarrow \infty$. 
Thus, when the acceptance probability is too large and in particular decreases  sub-exponentially with respect to  the dimension of the covariates, i.e., $p_n = \exp(-o(1)\cdot K_n)$, then ReM is essentially equivalent to the CRE in large samples. 
Third, 
when the acceptance probability decreases exponentially with respect to the dimension of covariates, i.e., $p_n = \exp(-c_n K_n)$ with $c_n$ being of constant order, 
and assume that the squared multiple correlation $R^2_n$ is of constant order and thus does not diminish to zero, 
asymptotically, ReM provides strictly more precise difference-in-means estimator than the CRE, although there is 
still a gap 
from 
the ideally optimal one. 

From the above, the performance of ReM, in particular its asymptotic improvement over the CRE, depends crucially on the ratio between $\log (p_n^{-1})$ and $K_n$, 
as well as $R^2_n$ measuring the association between potential outcomes and covariates. 
To minimize $v_{K_n, a_n}$, 
Corollary \ref{cor:improve} and Theorem \ref{thm:v_Ka} suggest to use as small acceptance probability and number of covariates as possible. 
However, there is trade-off for the choice of both of them. 
First, although smaller $K_n$ implies smaller $v_{K_n, a_n}$, it at the same time reduces the outcome-covariates association $R^2_n$. 
Second, although smaller $p_n$ decreases $v_{K_n, a_n}$, it at the same time renders the asymptotic approximation inaccurate and may eventually
invalidate the asymptotic approximation in~\eqref{eq:tau_hat_rem}; see Condition \ref{cond:p_n}. 
While the former trade-off for $K_n$ involves more subjective judgments concerning the unknown outcome-covariates dependence structure, the latter trade-off for $p_n$ lies more on the technical side and will be studied in more detail in the next subsection.

\subsection{Optimal rerandomization and its implication} 

From Theorem \ref{thm:v_Ka}, to achieve the ideally optimal rerandomization with given number of covariates, we want to choose the acceptance probability such that the following condition holds. 

\begin{condition}\label{cond:k_np_n}
	As $n \to \infty$, $\log(p_n^{-1})/K_n \rightarrow \infty$.
\end{condition}

We further assume that the outcome-covariates association $R^2_n$ is bounded away from 1, under which the first term $\sqrt{1-R^2_n}\ $ in the asymptotic approximation in  \eqref{eq:tau_hat_rem} is not negligible as $n\rightarrow \infty$.
This condition on $R^2_n$ is likely to hold 
in practice, 
since we generally do not expect that the covariates can perfectly explain the potential outcomes, which is too ideal 
for most applications. 
Moreover, if $R^2_n$ indeed converges to 1 as $n\rightarrow \infty$, then the asymptotic approximation in \eqref{eq:tau_hat_rem} can be of $o_{\PP}(1)$ itself, implying that 
$\hat{\tau}-\tau$
can converge to zero faster than the usual $n^{-1/2}$ rate, under which the causal effect estimation becomes much simpler. 

\begin{condition}\label{cond:rsup}
	As $n\rightarrow \infty$, $\limsup_{n \to \infty} R^2_n < 1$.
\end{condition}

The following theorem shows that, under certain regularity conditions, ReM can achieve its ideally optimal precision. 

\begin{theorem}\label{thm:rem_gaussian}
	Under ReM and Conditions \ref{cond:gamma_n}--\ref{cond:rsup}, 
	\begin{align}\label{eq:tau_dim_hat_rem}
		\sup_{c\in \mathbb{R}} \Big| \PP \big\{ V_{\tau\tau}^{-1/2}( \hat{\tau} - \tau) \le c \mid M \le a_n \big\} - \PP\big( \sqrt{1-R^2_n}\ \varepsilon_0 \le c \big)
		\Big|
		\converge 0. 
	\end{align}
\end{theorem}

From Theorem \ref{thm:rem_gaussian}, $\hat{\tau}$ can be asymptotically Gaussian distributed under ReM. 
Moreover, its asymptotic distribution is the same as that of the regression-adjusted estimator under the CRE \citep{Lin13, LeiD20}. 
Therefore, rerandomization and regression adjustment are essentially dual of each other \citep{LD20}, 
and Theorem \ref{thm:rem_gaussian} closes the previous gap between them that is due to the constrained Gaussian random variable $L_{K_n, a_n}$. 
This new insight is important for practitioners who may worry about efficiency loss of rerandomization compared to regression adjustment.
Moreover, compared to regression adjustment in the analysis stage, 
rerandomization in the design stage is blind of outcomes and has the advantage of avoiding data snooping \citep{Lin13}. 
Besides, the difference-in-means estimator is simpler and provides more transparent analysis for treatment effects \citep{cox2007, Freedman2008chance, Rosenbaum:2010}.

Theorem \ref{thm:rem_gaussian} has important implications. 
First, it shows that, by diminishing the imbalance threshold as the sample size grows, rerandomization can achieve its ideally optimal precision that we can expect from an optimal design with even perfectly balanced covariates. 
Second, we should not choose too small rerandomization threshold (or acceptance probability), as implied by Condition \ref{cond:p_n}, so that it is still possible to conduct robust randomization-based inference as completely randomized experiments, without imposing any distributional assumptions on potential outcomes and covariates. 
These imply that rerandomization with properly diminishing covariate imbalance threshold can achieve optimal efficiency as an ideal optimal design while maintaining robustness as a randomized design, i.e., such an optimal rerandomization can enjoy advantages from both optimal and randomized designs. 
Therefore, Theorem \ref{thm:rem_gaussian} helps reconcile the long-time controversies between the two philosophies (randomized versus optimal) for conducting experiments in an asymptotic sense.

Theorem \ref{thm:rem_gaussian} also provides important insights for practitioners that are more used to optimal designs. 
First, 
the usual optimal designs and the corresponding inference are often sensitive to their model assumptions. 
Our theory shows that the optimal rerandomization can always achieve the ideally optimal precision, while still being robust to model misspecification. 
Second, it 
helps mitigate the computation burden for conducting optimal designs. 
In particular, Theorem \ref{thm:rem_gaussian} suggests that we should not pursue the best allocation minimizing the covariate imbalance between the two treatment groups, which is generally NP-hard. 
Instead, we only need to randomly choose one from the best approximately $p_n$ proportion of all assignments, with $p_n$ satisfying both Conditions \ref{cond:p_n} and \ref{cond:k_np_n}. 
As discussed later in Section \ref{sec:reg_cond}, $p_n$ can often decrease at a polynomial order of the sample size $n$. This indicates that, in expectation, the computational complexity for getting one acceptable assignment is often of polynomial order of the sample size; see Section \ref{sec:comp_cost} for the more explicit rate. 
In sum, the optimal rerandomization can maintain the efficiency gain as an ideal optimal design, while being more robust and 
requiring
less computation. In the Supplementary Material \citep{WLSupp2022}, we provide more discussions on its connection with optimal designs
under certain model assumptions.

The validity of the asymptotic Gaussian approximation for rerandomization depends crucially on Conditions  \ref{cond:gamma_n}--\ref{cond:rsup}, among which Conditions \ref{cond:p_n} and \ref{cond:k_np_n} involves the choice of the acceptance probability. 
Below we assume that the number of covariates $K_n$ involved in rerandomization has been given and that Conditions \ref{cond:gamma_n} and \ref{cond:rsup} hold, and focus on  investigating the existence of choice of acceptance probability $p_n$ such that both Conditions \ref{cond:p_n} and \ref{cond:k_np_n} hold, or equivalently such that rerandomization can achieve its ideally optimal precision 
as in \eqref{eq:tau_dim_hat_rem}. 
It turns out the existence of such a choice of $p_n$ relies crucially on the ratio between $K_n$ and $\log(\Delta_n^{-1})$, recalling the definition of $\Delta_n$ in \eqref{eq:Delta_n}.

\begin{theorem}\label{thm:K_n_Delta_n}
	Under ReM and Conditions \ref{cond:gamma_n} and \ref{cond:rsup}, 
	\begin{itemize}
		\item[(i)] if and only if $\log(\Delta_n^{-1})/K_n \rightarrow \infty$, there exists a sequence $\{p_n\}$ such that both Conditions \ref{cond:p_n} and \ref{cond:k_np_n} hold, under which ReM achieves its ideally optimal precision and the asymptotic Gaussian approximation in \eqref{eq:tau_dim_hat_rem} holds; 
		\item[(iii)] if $\limsup_{n \to \infty} \log(\Delta_n^{-1})/K_n < \infty$, then 
		for any sequence $\{p_n\}$ satisfying Condition \ref{cond:p_n} such that the asymptotic approximation in \eqref{eq:tau_hat_rem} holds, $\liminf_{n \to \infty} \nu_{K_n, a_n} > 0$;
		\item[(iii)] if $\liminf_{n \to \infty} \log(\Delta_n^{-1})/K_n > 0$, then there exists a sequence $\{p_n\}$ satisfying Condition \ref{cond:p_n} such that~\eqref{eq:tau_hat_rem} holds
		and $\limsup_{n \to \infty} \nu_{K_n, a_n} < 1$;
		\item[(iv)] 
		if $\log(\Delta_n^{-1})/K_n \rightarrow 0$, then 
		for any sequence $\{p_n\}$ satisfying 
		Condition \ref{cond:p_n} such that~\eqref{eq:tau_hat_rem} holds, 
		the corresponding 
		$v_{K_n, a_n} \rightarrow 1$ as $n\rightarrow \infty$, 
		under which ReM asymptotically provides no gain on estimation precision compared to the CRE. 
	\end{itemize}
\end{theorem}

From Theorem \ref{thm:K_n_Delta_n}, the optimal precision that ReM can achieve 
depends crucially on the asymptotic behavior of the ratio $\log(\Delta_n^{-1})/K_n$.
For any fixed $n$, in general,
	as $K_n$ increases, $\Delta_n$ will increase, and thus the ratio $\log(\Delta_n^{-1})/K_n$ will decrease. Therefore,
intuitively, the smaller the number of involved covariates for rerandomization, 
the more likely we are able to 
satisfy the condition in Theorem \ref{thm:K_n_Delta_n}(i), and consequently, to achieve the ideally optimal precision.
Moreover, when ReM involves more and more covariates, it will eventually lose its advantage over the CRE. 
Therefore, 
Theorem \ref{thm:K_n_Delta_n} suggests that we should not use too  many covariates when performing rerandomization. 
In practice, we should try to use a moderate number of covariates that are most relevant for the potential outcomes of interest, as measured by the corresponding $R^2_n$. 
For example, when $\gamma_n$ has the same order as its lower bound in \eqref{eq:gamma_n_lower}, and $r_1$ and $r_0$ are bounded away from zero, 
then we can choose $K_n = O(\log n)$ number of covariates, under which $\liminf_{n \to \infty} \log(\Delta_n^{-1})/K_n$ must be positive and ReM can provide non-negligible gain over the CRE as implied by Theorem \ref{thm:K_n_Delta_n}(iii). 
As $\gamma_n$ becomes further from its lower bound, we generally want to use fewer covariates. 
We defer more detailed discussion on the rate of $\log(\Delta_n^{-1})/K_n$ to Section \ref{sec:reg_cond}.

\section{Large-sample Inference under Rerandomization}\label{sec:ci}

Theorems \ref{thm:dim_rem} and \ref{thm:rem_gaussian} provide asymptotic approximations for the distribution of the  difference-in-means estimator $\hat{\tau}$ under ReM, based on which we can construct large-sample confidence intervals for the average treatment effect $\tau$. 
From the asymptotic approximations in \eqref{eq:tau_hat_rem} and \eqref{eq:tau_dim_hat_rem}, the asymptotic distribution of $\hat{\tau}$ under ReM depends on the variance $V_{\tau\tau}$ for the CRE and the squared multiple correlation $R_n^2$, both of which are determined by the finite population variances of the potential outcomes and individual effects as well as their linear projections on covariates. 
For each treatment group $z\in \{0,1\}$, 
recall that $\bar{Y}_z$ and $\bar{\bs{X}}_z$ are the average observed outcome and covariates, 
and define 
$s_z^2 = (n_z-1)^{-1} \sum_{i:Z_i=z} (Y_i - \bar{Y}_z)^2$ 
and 
$
\bs{s}_{z, \bs{X}} = \bs{s}_{\bs{X},z}^\top =  (n_z-1)^{-1} \sum_{i:Z_i=z} (Y_i - \bar{Y}_z) (\bs{X}_i - \bar{\bs{X}}_z )^\top
$
as the sample variance and covariance for the observed outcome and covariates. 
We further introduce 
$s^2_{z \setminus \bs{X}} = s_z^2 - \bs{s}_{z, \bs{X}} \bs{S}_{\bs{X}}^{-2} \bs{s}_{\bs{X}, z}$ 
and 
$s_{\tau \mid \bs{X}}^2 =(\bs{s}_{1, \bs{X}} - \bs{s}_{0, \bs{X}}) \bs{S}_{\bs{X}}^{-2} (\bs{s}_{\bs{X}, 1} - \bs{s}_{\bs{X}, 0})$. 
We can then estimate $V_{\tau\tau}$ and $R^2_n$ in \eqref{eq:V} and \eqref{eq:R2} through replacing the finite population variances in their definitions by the corresponding sample analogues: 
\begin{align}\label{eq:VR2_hat}
	\hat{V}_{\tau\tau} = n_1^{-1} s_1^2 + n_0^{-1} s_0^2 - n^{-1} s_{\tau \mid \bs{X}}^2, 
	\qquad 
	\hat{R}^2_n =  
	1 - 
	\hat{V}_{\tau\tau}^{-1} \big(n_1^{-1} s_{1 \setminus \bs{X}}^2 + n_0^{-1} s_{0 \setminus \bs{X}}^2 \big).
\end{align}
Note that the finite population variance of individual effects $S_\tau^2$ is generally not identifiable, since we can never observe the individual effect for any unit.
Thus, 
we do not expect any consistent estimator for it as well as $V_{\tau\tau}$ \citep{N23}. 
However, because  $S_\tau^2$ is bounded below by $S_{\tau \mid \bs{X}}^2$ that has sample analogue  $s_{\tau \mid \bs{X}}^2$, we can still estimate $V_{\tau\tau}$  conservatively as in \eqref{eq:VR2_hat}.

Based on the asymptotic approximation established in Theorem \ref{thm:dim_rem}  and the estimators in \eqref{eq:VR2_hat}, 
for any $\alpha\in (0,1)$, 
we can then construct the following $1-\alpha$ confidence interval for $\tau$: 
\begin{align}\label{eq:C_hat}
	\hat{\mathcal{C}}_\alpha =  \big[\hat\tau - \hat{V}_{\tau\tau}^{1/2} \cdot \nu_{1-\alpha / 2, K_n, a_n} (\hat{R}_n^2), \ \ \hat\tau + \hat{V}_{\tau\tau}^{1/2} \cdot \nu_{1-\alpha/2, K_n, a_n} (\hat{R}_n^2)\big],
\end{align}
recalling that $\nu_{1-\alpha/2, K, a} (R^2)$ is the $(1-\alpha/2)$th quantile of the 
distribution of $\sqrt{1 - R^2} \ \varepsilon_0 + \sqrt{R^2} \ L_{K, a}$.

To ensure the asymptotic validity of the confidence interval in \eqref{eq:C_hat}, we invoke the following regularity condition. We defer more detailed discussion of the condition to Section \ref{sec:reg_cond}.  
For $z=0,1$, 
let $S^2_{z\setminus \bs{X}} = S^2_{z} - S^2_{z\mid \bs{X}}$ denote the finite population variance of the residuals from the linear projection of potential outcomes $Y_i(z)$'s on covariates $\bs{X}_i$'s.

\begin{condition}\label{cond:infer}
	As $n\rightarrow \infty$, 
	\begin{align}\label{eq:infer_cond}
		\frac{\max_{z\in \{0,1\}}\max_{1\le i \le n}\{Y_i(z) - \bar{Y}(z)\}^2}{r_0 S^2_{1\setminus \bs{X}} + r_1 S^2_{0\setminus \bs{X}}}
		\cdot 
		\frac{\max\{K_n, 1\}}{r_1r_0}
		\cdot 
		\sqrt{ \frac{\max\{1, \log K_n, - \log p_n\} }{n} } 
		\converge 0. 
	\end{align}
\end{condition}

The following theorem shows that the confidence interval in \eqref{eq:C_hat} is asymptotically conservative, and becomes asymptotically exact when $S^2_{\tau\setminus \bs{X}} \equiv S^2_{\tau} - S^2_{\tau\mid \bs{X}}$ is asymptotically negligible.

\begin{theorem}\label{thm:inf}
	Under ReM and Conditions \ref{cond:gamma_n}, \ref{cond:p_n} and \ref{cond:infer}, as $n\rightarrow \infty$, 
	\begin{itemize}
		\item[(i)] the estimators in \eqref{eq:VR2_hat} are asymptotically conservative in the sense that 
		\begin{align*}
			& \quad 
			\max\big\{ 
			|\hat{V}_{\tau\tau}(1-\hat{R}_n^2) - 
			V_{\tau\tau}(1-R^2_n) - S^2_{\tau \setminus \bs{X}}/n|, 
			\ 
			|\hat{V}_{\tau\tau} \hat{R}^2_n - 
			V_{\tau\tau} R^2_n|
			\big\}
			\\ & 
			= 
			o_{\PP}\left( V_{\tau\tau}(1-R_n^2) + S^2_{\tau\setminus \bs{X}} /n \right); 
		\end{align*}
		\item[(ii)] for any $\alpha\in (0,1)$, 
		the resulting $1-\alpha$ confidence interval in \eqref{eq:C_hat} is asymptotically conservative, 
		in the sense that 
		$
		\liminf_{n\rightarrow \infty}
		\PP
		(
		\tau \in \hat{\mathcal{C}}_{\alpha} \mid M \le a_n
		)
		\ge 
		1 - \alpha; 
		$
		\item[(iii)] if further $S^2_{\tau \setminus \bs{X}} = n V_{\tau\tau}(1-R_n^2)\cdot o(1),$ the $1-\alpha$ confidence interval in \eqref{eq:C_hat} becomes asymptotically exact, 
		in the sense that 
		$
		\lim_{n\rightarrow \infty}
		\PP
		(
		\tau \in \hat{\mathcal{C}}_{\alpha} \mid M \le a_n
		)
		= 
		1 - \alpha. 
		$
	\end{itemize}
\end{theorem}

Note that in Theorem \ref{thm:inf}, we do not make the typical assumption that $R_n^2$ and $V_{\tau\tau}$ have limiting values as the sample size goes to infinity \citep[see, e.g.,][]{LDR18}, 
since such an assumption may not be very satisfying given that we allow the number of covariates $K_n$ to vary with the sample size (which consequently affects $R_n^2$). 
This brings additional challenge to the proof of Theorem \ref{thm:inf}. 
Theorem \ref{thm:inf} also highlights the conservativeness in the finite population inference that dates back to \citet{N23}'s analysis for the CRE. 
The conservativeness of the confidence interval comes mainly from $S^2_{\tau \setminus \bs{X}}$ as indicated in Theorem \ref{thm:inf}(i), which characterizes the individual effect heterogeneity after taking into account the covariates. 
The intervals become asymptotically exact when the individual effect heterogeneity is asymptotically linearly explained by the covariates, as shown in Theorem \ref{thm:inf}(iii).

If further Conditions \ref{cond:k_np_n} and \ref{cond:rsup} hold as in Theorem \ref{thm:rem_gaussian}, 
not surprisingly, we can then use the usual Wald-type confidence intervals based on Gaussian quantiles, ignoring the term involving the constrained Gaussian random variable. 
Specifically, for any $\alpha\in (0,1)$, let 
\begin{align}\label{eq:C_tilde}
	\tilde{\mathcal{C}}_\alpha = \Big[\hat\tau - \sqrt{\hat{V}_{\tau\tau} (1 - \hat{R}_n^2)} \cdot z_{1-\alpha / 2}, \ \  \hat\tau + \sqrt{\hat{V}_{\tau\tau} (1 - \hat{R}_n^2)} \cdot z_{1-\alpha/2}\Big].
\end{align}
The following theorem shows the asymptotic validity of the above Wald-type confidence intervals. 

\begin{theorem}\label{thm:inf_gaussian}
	Under ReM and Conditions \ref{cond:gamma_n}--\ref{cond:infer},
	Theorem \ref{thm:inf}(i)--(iii) still hold with $\hat{\mathcal{C}}_{\alpha}$ in \eqref{eq:C_hat} replaced by $\tilde{\mathcal{C}}_\alpha$ in \eqref{eq:C_tilde}. 
\end{theorem}

Theorem \ref{thm:inf_gaussian} shows that the usual Wald-type confidence intervals 
can be asymptotically valid for 
inferring 
the average treatment effect under ReM. 
However, we want to emphasize that it does involve additional regularity conditions, in particular Condition \ref{cond:k_np_n} on the choice of acceptance probability. 
While Theorem \ref{thm:inf_gaussian} provides confidence intervals of more convenient forms, 
we still recommend constructing confidence intervals based on Theorem  \ref{thm:inf}, which takes into account explicitly the constrained Gaussian random variable that is determined by the acceptance probability and the number of covariates for rerandomization. 
The intervals from Theorem \ref{thm:inf} not only requires 
fewer
regularity conditions, but also becomes asymptotically equivalent to the ones from Theorem \ref{thm:inf_gaussian} when the additional Conditions \ref{cond:k_np_n} and \ref{cond:rsup} hold, under which the constrained Gaussian random variable $L_{K_n, a_n}$ is of order $o_{\PP}(1)$ and is thus negligible asymptotically.

Below we give a brief remark on improving the finite-sample performance of the confidence intervals. 
Note that $s_{1 \setminus \bs{X}}^2$ and $s_{0 \setminus \bs{X}}^2$ are almost equivalent to the sample variances of the residuals from the linear projection of observed outcome on covariates in treated and control groups, respectively. 
Inspired by the regression analysis \citep[see, e.g.,][]{mackinnon2013thirty, LeiD20}, we can consider rescaling the residuals to improve their finite-sample performance. 
For example, 
letting $\hat{e}_i$ denote the residual from the linear projection of observed outcome on covariates for unit $i$,  
we can rescale $\hat{e}_i$ to be $\kappa_i\hat{e}_i$, with 
$\kappa_i = 1$ for HC0, 
$\kappa_i = \sqrt{(n_{z_i}-1)/(n_{z_i}-K_n-1)}$ for HC1, 
$\kappa_i = 1/\sqrt{1-H_{z_i,ii}}$ for HC2, 
and 
$\kappa_i = 1/(1-H_{z_i,ii})$ for HC3, 
where $H_{z,ii}$ is the leverage of unit $i$ for the covariate matrix consisting of the intercept and $K_n$ covariates under treatment arm $z$.

Finally, we 
remark that throughout Sections~\ref{sec:asymptotic} to~\ref{sec:ci} we assume Conditions~\ref{cond:gamma_n} and~\ref{cond:p_n}, which are stronger than that required by completely randomized experiments. When these conditions fail, the asymptotic properties of rerandomization may fail, making it challenging to conduct robust randomization-based inference for rerandomization; 
see 
Appendix A1 in the Supplementary Material \citep{WLSupp2022}
for a finite-sample worst-case analysis and its application for practical diagnosis of rerandomization.

\section{Regularity Conditions and  Practical Implications}\label{sec:reg_cond}

From the previous discussion, the asymptotic approximation of the difference-in-mean estimator $\hat{\tau}$ under ReM relies on Conditions \ref{cond:gamma_n} and \ref{cond:p_n}. Once 
further  Conditions \ref{cond:k_np_n} and \ref{cond:rsup} hold, 
we 
can achieve 
the ``optimal'' rerandomization, 
under which $\hat{\tau}$ becomes asymptotically Gaussian distributed. 
The 
large-sample inference for the average treatment effect $\tau$
under ReM relies additionally on Condition \ref{cond:infer}.
These regularity conditions depend crucially on the convergence rate of $\Delta_n$ or its upper bound $\gamma_n$. 
In the following, we first investigate the convergence rate of $\gamma_n$ when units are i.i.d.\ samples from some superpopulation, and discuss its implication on the validity of these regularity conditions. 
We then consider practical strategies that can make these regularity conditions more likely to hold, 
utilizing the advantage of finite population inference that requires no model or distributional assumptions on the potential outcomes and covariates. 
Finally, we discuss the computational cost of (optimal) rerandomization.

\subsection{Regularity conditions under i.i.d. sampling and their implications} \label{sec:iidconditions}

Throughout this subsection, we assume that the potential outcomes and covariates $(Y_i(1), Y_i(0), \bs{X}_i)$, for $i=1,2,\ldots, n$, are i.i.d.\  from a superpopulation that depends implicitly on the sample size $n$, 
noting that the dimension of covariates is allowed to vary with $n$. 
Recall that $\bs{u}_i = (r_0 Y_i(1) + r_1 Y_i(0), \bs{X}_i^\top)^\top \in \mathbb{R}^{K_n+1}$, and $\gamma_n$ in \eqref{eq:gamma_n} is uniquely determined by the treatment group proportions $r_1, r_0$, the dimension of covariates $K_n$, the sample size $n$, and the finite population $\{\bs{u}_1, \bs{u}_2, \ldots, \bs{u}_n\}$, which is further assumed to consist of i.i.d.\ draws from a superpopulation. 
Below we impose some moment conditions on the sequence of superpopulations as $n\rightarrow \infty$.

\begin{condition}\label{cond:moment}
	For each sample size $n$, $\bs{u}_1, \bs{u}_2, \ldots, \bs{u}_n$ are i.i.d.\ random vectors in $\mathbb{R}^{K_n+1}$ with finite and nonsingular covariance matrix. 
	The standardized random vector $\bs{\xi}_i = \cov(\bs{u}_i)^{-1/2} (\bs{u}_i - \E\bs{u}_i) \in \mathbb{R}^{K_n+1}$ satisfies that 
	$\sup_{\bs{\nu}\in \mathbb{R}^{K_n+1}: \bs{\nu}^\top \bs{\nu} = 1} \E|\bs{\nu}^\top \bs{\xi}_i|^\delta = O(1)$ for some $\delta > 2$, i.e., there exists an absolute constant $C<\infty$ such that $\E|\bs{\nu}^\top \bs{\xi}_i|^\delta \le C$ for all $n$ and all unit vector $\bs{\nu}$ in $\mathbb{R}^{K_n+1}$. 
\end{condition}

From \citet[][Proposition F.1]{LeiD20}, a sufficient condition for the uniform boundedness of $\sup_{\bs{\nu}\in \mathbb{R}^{K_n+1}: \bs{\nu}^\top \bs{\nu} = 1} \E|\bs{\nu}^\top \bs{\xi}_i|^\delta$ in Condition \ref{cond:moment} is that $\bs{\xi}_i$ has independent coordinates whose absolute $\delta$-th moment is uniformly bounded by a certain finite constant. 
The following proposition, which is a direct corollary of \citet[][Lemma H.1]{LeiD20}, gives a stochastic upper bound of $\gamma_n$ under Condition \ref{cond:moment}.

\begin{proposition}\label{prop:iidrate}
	If Condition \ref{cond:moment} holds and the dimension of covariates $K_n = O(n^{\omega})$ for some $\omega\in (0,1)$, 
	then 
	\begin{align*}
		\gamma_n 
		= 
		O_{\PP} \left( 
		\frac{1}{\sqrt{r_1r_0}} \frac{(K_n+1)^{7/4}}{n^{1/2-1/\delta}} \right). 
	\end{align*}
\end{proposition}

Note that
$\gamma_n \ge 2^{-3/2} (n r_1r_0)^{-1/2}(K_n+1)^{7/4}$ as shown in \eqref{eq:gamma_n_lower}. Therefore, the rate in Proposition \ref{prop:iidrate} is precise up to an $n^{1/\delta}$ factor;  when $\delta = \infty$, the rate matches the lower bound rate in \eqref{eq:gamma_n_lower}.
From Proposition \ref{prop:iidrate}, we can immediately derive the following implications. 
Recall that $r_1$ and $r_0$ are the proportions of treated and control units. 
In practice, both treatment groups are likely to have non-negligible proportions of units, and thus it is reasonable to assume that  both $r_1^{-1}$ and $r_0^{-1}$ are of order $O(1)$. 
Recall that $R_n^2$ in \eqref{eq:R2} denotes the finite population squared multiple correlation between potential outcomes and covariates. 
We introduce $R_{\sup, n}^2 = \text{Corr}^2(r_1Y(0)+r_0Y(1), \bs{X})$ to denote the superpopulation analogue with $(Y(1), Y(0), \bs{X})$ following the superpopulation distribution at sample size $n$. 

\begin{corollary}\label{cor:reg_cond}
	If Condition \ref{cond:moment} holds, $r_z^{-1} = O(1)$ for $z=0,1$, and 
	$K_n = o(n^{2/7 - 4/(7\delta)})$, then 
	\begin{enumerate}[label=(\roman*)]
		\item $\gamma_n = o_{\PP}(1)$, and thus $\Delta_n$ in \eqref{eq:Delta_n} is of order $o_{\PP}(1)$; 
		\item the finite population and superpopulation squared multiple correlations $R_n^2$ and $R^2_{\sup, n}$ are asymptotically equivalent, in the sense that $R_n^2 - R^2_{\sup, n} = o_{\PP}(1)$; 
		\item 
		if further the standardized potential outcomes 
		have bounded $b$th moments for 
some $b > 4$, 
		both $\var(Y(1))$ and $\var(Y(0))$ are of the same order as $\var(r_0 Y(1) + r_1 Y(0))$, 
		$\limsup_{n\rightarrow \infty} R^2_{\sup, n} < 1$, 
		and $K_n = O(n^c)$ and $-\log p_n = o(n^{1-4/b-2c})$ for some $c<1/2-2/b$, 
		then the quantity on the left hand side of  \eqref{eq:infer_cond} is of order $o_{\PP}(1)$.
	\end{enumerate}
\end{corollary}

Corollary \ref{cor:reg_cond}(i) implies that Condition \ref{cond:gamma_n} holds with high probability, 
which, based on Theorem \ref{thm:dim_rem}, further implies that the asymptotic approximation for the difference-in-means estimator under ReM using Gaussian and constrained Gaussian distributions is valid with high probability, given that the acceptance probability is chosen to satisfy Condition \ref{cond:p_n} with high probability (i.e., $\Delta_n / p_n = o_{\PP}(1)$). 
With additional regularity conditions on the superpopulation variance and squared multiple correlations, 
Corollary \ref{cor:reg_cond}(iii) implies that Condition \ref{cond:infer} holds with high probability, which further implies that the large-sample inference for the average treatment effect discussed in Section \ref{sec:ci} is valid with high probability. 
From the above, 
when the experimental units are i.i.d.\ from some superpopulation satisfying the moment conditions in Condition \ref{cond:moment} and Corollary \ref{cor:reg_cond}(iii), the number of covariates is not too large, 
and the acceptance probability is not too small, 
with high probability, 
we are able to asymptotically approximate the distribution of $\hat{\tau}$ under ReM, as well as constructing asymptotically valid confidence intervals for the average treatment effect $\tau$.

Furthermore, Corollary \ref{cor:reg_cond}(ii) implies that Condition \ref{cond:rsup} holds with high probability when $R^2_{\sup, n} \le 1- c$ for some absolute constant $c>0$. 
Thus, based on Theorem \ref{thm:rem_gaussian}, 
as long as the choice of acceptance probability satisfies Condition \ref{cond:k_np_n}, with high probability, we can approximate the distribution of $\hat{\tau}$ under ReM  by a Gaussian distribution, under which ReM achieves its ideally optimal precision. 
Based on Theorem \ref{thm:rem_gaussian}, we can then derive the following corollary. 
Recall that $\mathcal{U}_n=\{\bs{u}_1, \bs{u}_2, \ldots, \bs{u}_n\}$. 
	In the corollary below, we will write the conditioning on $\mathcal{U}_n$ explicitly to emphasize that we are considering the randomization distribution of $\hat{\tau}$ under ReM.

\begin{corollary}\label{cor:iidrate}
	If Condition \ref{cond:moment} holds,  $r_z^{-1} = O(1)$ for $z=0,1$, $R^2_{\sup, n} \le 1- c$ for some absolute constant $c>0$, and $K_n = o(\log n)$, then there exists a sequence of acceptance probabilities $p_n$ (or equivalently a sequence of thresholds $a_n$) such that, with probability converging to $1$, the distribution of the difference-in-means estimator under ReM can be asymptotically approximated by a Gaussian distribution with mean zero and variance $V_{\tau\tau}(1-R^2_n)$, i.e., 
	\begin{align*}
		\sup_{c\in \mathbb{R}} \Big| \PP \big\{ V_{\tau\tau}^{-1/2}( \hat{\tau} - \tau) \le c \mid M \le a_n , \mathcal{U}_n \big\} - \PP\big( \sqrt{1-R^2_n}\ \varepsilon_0 \le c \mid \mathcal{U}_n \big)
		\Big|
		= o_{\PP}(1).
	\end{align*}
\end{corollary}

Below we further consider the choice of acceptance probability under various cases for the number of covariates involved in rerandomization. 
We assume that Condition \ref{cond:moment} holds and both treatment groups have non-negligible proportions of units, i.e., $r_z^{-1} = O(1)$ for $z=0, 1$, and consider the cases where the number of covariates $K_n$ increases sub-logarithmically, logarithmically and polynomially with the sample size $n$, as shown in the first column of Table \ref{tab:rate}.  
Proposition \ref{prop:iidrate} then gives an upper bound of the stochastic rate of $\gamma_n$, as shown in the second column of Table \ref{tab:rate}, all of which are of order $o_{\PP}(1)$.  
Note that in the polynomial increase case, the rate of $K_n$ is restricted to $K_n \asymp n^{\zeta}$ with $\zeta \in (0, \frac{2}{7} - \frac{4}{7\delta})$; otherwise $\gamma_n$ may not converge to zero in probability, under which Condition \ref{cond:gamma_n} may fail and the asymptotic inference will become difficult. 
From Theorem \ref{thm:berry_esseen_clt}, 
under all the three cases, we can choose the acceptance probability $p_n$ to decrease polynomially with the sample size such that Condition \ref{cond:p_n} holds with high probability, although there are various constraints on the exact polynomial decay rate. 
This is shown in the third column of Table \ref{tab:rate}, where $\kappa$ can at least take value $1/3$ based on our result in Theorem \ref{thm:berry_esseen_clt}(iii) and it can take value 1 if \citet{R15}'s conjecture as in Theorem \ref{thm:berry_esseen_clt}(ii) holds. 
The last column of Table \ref{tab:rate} then shows the corresponding asymptotic behavior of the variance $v_{K_n, a_n}$ of the constrained Gaussian random variable. 
Recall that the gain from ReM on estimation precision depends crucially on $v_{K_n,a_n}$ as shown in Corollary \ref{cor:improve}.  
From Theorem \ref{thm:v_Ka}, 
(i) when $K_n$ increases sub-logarithmically with $n$, $v_{K_n, a_n}$ converges to zero and rerandomization achieves its ideally optimal precision; 
(ii) when $K_n$ increases logarithmically with $n$, $v_{K_n, a_n}$ is strictly between $0$ and 1 when $n$ is sufficiently large
and thus rerandomization provides gain in estimation precision compared to the CRE, although there is still a gap from the ideally optimal gain; 
(iii) when $K_n$ increases polynomially with $n$, $v_{K_n, a_n}$ converges to 1 as $n\rightarrow \infty$ and rerandomization provides no gain over the CRE. 
Therefore, in practice, we do not recommend using too many covariates, under which we will essentially lose the advantage from rerandomization. 
As a side note, these observations hold no matter which Berry-Essen bound we use for the asymptotic Gaussian approximation (either the conjectured (ii), our derived (iii) in Theorem \ref{thm:berry_esseen_clt}, or the ones in Theorem~\ref{thm:Berry--Esseen-higher-order}), except for the explicit rate of the polynomially decaying acceptance probability.

\begin{table}
	\centering
	\caption{
		Asymptotic properties for rerandomization under various rates for the number of involved covariates. 
		Column 1 shows the asymptotic rate of the number of covariates $K_n$ as $n$ increases. 
		Column 2 shows the corresponding stochastic rate of $\gamma_n$ based on Proposition \ref{prop:iidrate}. 
		Column 3 shows the choice of acceptance probability that is sufficient for Condition \ref{cond:p_n} with high probability. 
		Column 4 shows the asymptotic rate for the variance $v_{K_n, a_n}$ of the corresponding constrained Gaussian random variable. 
	}\label{tab:rate}
	\begin{tabular}{llll}
		\toprule
		$K_n$ & $\gamma_n$ & $p_n$ & $v_{K_n, a_n}$ \\
		\midrule
		$K_n = o(\log n)$ & 
		$o_{\PP}\left( \frac{(\log n)^{7/4}}{n^{1/2 - 1/\delta}} \right)$ 
		& $p_n \asymp n^{-\beta}$, $\beta \in (0, \frac{\kappa}{2} - \frac{\kappa}{\delta})$ & $o(1)$
		\\
		\midrule
		$K_n \asymp \log n$ & $O_{\PP}\left( \frac{(\log n)^{7/4}}{n^{1/2 - 1/\delta}} \right)$ & $p_n \asymp n^{-\beta}$, $\beta \in (0, \frac{\kappa}{2} - \frac{\kappa}{\delta})$ & 
		$(0, 1)$
		\\
		\midrule
		$K_n \asymp n^{\zeta}$, $\zeta \in (0, \frac{2}{7} - \frac{4}{7\delta})$ & 
		$O_{\PP}\left(n^{-(\frac{1}{2}-\frac{1}{\delta}-\frac{7\zeta}{4})}\right)$
		& 
		$p_n \asymp n^{-\beta}$, $\beta \in (0, \frac{\kappa}{2} - \frac{\kappa}{\delta} - \frac{7\zeta \kappa}{4})$ & $1-o(1)$
		\\
		\bottomrule
	\end{tabular}
\end{table}

\subsection{Practical implication}\label{sec:implication}

From the previous discussion, 
the asymptotic approximation of rerandomization depends crucially on $\gamma_n$ (which involves $K_n$) and $p_n$.  
Below we provide suggestions and guidance on how to leverage these factors to improve the performance of rerandomization in practice.

We first consider $K_n$. 
From Theorem \ref{thm:K_n_Delta_n} and the discussion in Section \ref{sec:iidconditions}, 
we suggest to choose at most $O(\log n)$ covariates, otherwise we may lose the advantage of rerandomization. 
Note that our finite population inference does not impose any model or distributional assumptions on the potential outcomes and covariates as well as their dependence structure. 
Thus, we have the flexibility to pre-process the covariates in an arbitrary way.
Note that the covariates, although do not affect our asymptotic inference as long as the corresponding regularity conditions hold, do affect the improvement from rerandomization as indicated by the outcome-covariates association $R_n^2$.  
Therefore, we suggest to choose a moderate subset of covariates or a moderate dimensional transformation of original covariates to conduct rerandomization. 
For example, we can choose a subset of covariates of size $O(\log n)$ based on our subjective knowledge about the importance of these covariates for predicting/explaining the potential outcomes or based on some pilot studies. 
Recently, 
in the presence of high-dimensional covariates, 
\citet{ZYR21} proposed to use principle component analysis to find some proper subspace of covariates to conduct rerandomization.

We then consider $\gamma_n$. Below we give an equivalent form of $\gamma_n$ that can be more informative for its practical implications. 
For $z=0,1$ and $1\le i \le n$, 
let $e_i(z) = Y_i(z) - \bar{Y}(z) - \bs{S}_{z, \bs{X}} (\bs{S}_{\bs{X}}^2)^{-1} (\bs{X}_i - \bar{\bs{X}})$ 
be the residual from the linear projection of potential outcome $Y_i(z)$ on covariates $\bs{X}_i$. 
Then $ r_0 e_i(1) + r_1 e_i(0)$ is the residual from the projection of $r_0 Y_i(1) + r_1 Y_i(0)$ on $\bs{X}_i$. 
We further introduce $e_i$ to denote the corresponding standardized residual for unit $i$, i.e.,  $r_0 e_i(1) + r_1 e_i(0)$ standardized by its finite population 
mean and 
standard deviation. 
Let $\tilde{\bs{X}} = (\bs{X}_1, \bs{X}_2 \ldots, \bs{X}_n)^\top \in \mathbb{R}^{n\times K_n}$ be the matrix consisting of the covariates for all $n$ units, $\bs{H} = \tilde{\bs{X}} (\tilde{\bs{X}}^\top \tilde{\bs{X}})^{-1} \tilde{\bs{X}}^\top \in \mathbb{R}^{n\times n}$ be the corresponding projection or hat  matrix, and $H_{ii}$ be the $i$th diagonal element of $\bs{H}$, which is usually called the leverage score for unit $i$ in regression analysis. 
As demonstrated in the Supplementary Material \citep{WLSupp2022}, $\gamma_n$ in \eqref{eq:gamma_n} can be bounded by
\begin{align*}
	\gamma_n &
	= 
	\frac{(K_n+1)^{1/4}}{\sqrt{n r_1r_0}} \frac{1}{n} \sum_{i=1}^n 
	\left( e_i^2 + (n-1) H_{ii} \right)^{3/2} 
	\in 
	\left[
	\frac{1}{4\sqrt{2}}
	\tilde{\gamma}_n,   \sqrt{2} \tilde{\gamma}_n \right], 
\end{align*}
with 
\begin{align}\label{eq:bound_gamma_n}
	\tilde{\gamma}_n = \frac{(K_n+1)^{1/4}}{\sqrt{r_1r_0 n}} \frac{1}{n}\sum_{i=1}^n |e_i|^3  + 
	\frac{(K_n+1)^{1/4}}{\sqrt{r_1r_0}}\sum_{i=1}^n H_{ii}^{3/2}. 
\end{align}
Obviously, $\gamma_n$ and $\tilde{\gamma}_n$ are of the same order, and it is thus equivalent to consider $\tilde{\gamma}_n$. 
From \eqref{eq:bound_gamma_n}, 
$\tilde{\gamma}_n$ depends crucially on the absolute third moment of the standardized residuals $n^{-1} \sum_{i=1}^n |e_i|^3$ 
and the summation of the leverage scores to the power of $3/2$ over all $n$ units. 
Note that $e_i$'s depend on the true potential outcomes and are generally unknown in the design stage of an experiment. 
When the potential outcomes $Y_i(1)$ and $Y_i(0)$, or more precisely their residuals $e_i(1)$ and $e_i(0)$ are not too heavy-tailed, we expect $n^{-1} \sum_{i=1}^n |e_i|^3$ to be of constant order, under which the first term in \eqref{eq:bound_gamma_n} is likely to be well-controlled. 
The second term in \eqref{eq:bound_gamma_n} seems to be more complicated, but fortunately it is known in the design stage since it depends only on the pretreatment covariates. 
Thus, we can and should check the leverage scores for all units before conducting rerandomization. 
If $\sum_{i=1}^n H_{ii}^{3/2}$ is small, then we expect the asymptotic approximation for rerandomization to work well;  
otherwise, we need to be careful. 
As discussed before, under the finite population inference framework, we have the flexibility to pre-process the covariates in an arbitrary way. 
In particular, we can try to shrink the leverage scores via trimming, a practical strategy that is also recommended when conducting regression adjustment for completely randomized experiments \citep{LeiD20}. 
Note that too much trimming may reduce the outcome-covariates association $R_n^2$ and thus deteriorate the improvement from rerandomization. Besides,
as demonstrated in the Supplementary Material \citep{WLSupp2022}, 
$\sum_{i=1}^n H_{ii}^{3/2}$ is always lower bounded by $K_n^{3/2}/\sqrt{n}$. %
In practice, we may consider performing the minimum possible trimming such that the resulting  $\sum_{i=1}^n H_{ii}^{3/2}$ is close to its minimum value $K_n^{3/2}/\sqrt{n}$. 
As a side note, both ridge and PCA rerandomizations recently proposed by \citet{BDR16} and \citet{ZYR21} can also help reduce leverage scores. 
Thus, our theory also 
provides some justification for these two designs. 
Moreover, compared to their pre-processing on the covariates, trimming can be more robust to outliers. 
Besides, it 
does not change the original covariates (as well as their meaning) much and may be more helpful in preserving the explainability of original covariates (more precisely, their squared multiple correlation with the potential outcomes).

Finally, we consider $p_n$. 
The discussion in Section \ref{sec:iidconditions} suggests choosing the acceptance probability $p_n$ such that it decays polynomially with the sample size $n$. However, such results may not be helpful for the choice of $p_n$ under a finite sample size. 
Below we give some practical guideline on the choice of $p_n$ with a moderate number of covariates $K_n$. 
First, from Corollary \ref{cor:improve}, the gap between rerandomization with a certain $p_n$ and the ideally optimal one is characterized by $v_{K_n, a_n}$, the variance of the constrained Gaussian random variable. Therefore, we suggest choosing $p_n$ such that the corresponding $v_{K_n, a_n}$ is small, say, $0.01$. 
Second, with a given $p_n$, we can check the asymptotic approximation by using some pseudo potential outcomes as proxies,  e.g., some linear/nonlinear combinations of covariates based on some prior knowledge. 
Third, note that, with acceptance probability $p_n$, the number of randomizations needed for getting an acceptable treatment assignment is about $1/p_n$.  
In practice, our choice of $p_n$ can also take into account this computation cost; see also Section \ref{sec:comp_cost}.

\subsection{Computational cost of rerandomization}\label{sec:comp_cost}

Below we briefly discuss the computational cost for getting one acceptable treatment assignment from ReM with $n$ units, $K$-dimensional covariates ($K\le n$) and acceptance probability $p>0$. 
To facilitate the computation, we can first standardize the covariates, i.e., getting $\bs{S}_{\bs{X}}^{-1}(\bs{X}_i-\bar{\bs{X}})$ for each $i$, which has a complexity of $O(nK^2)$. 
Then the Mahalanobis distance $M$ is equivalently the Euclidean norm of the difference-in-means of standardized covariates up to some scale, whose computation has complexity of $O(nK)$ and in expectation needs to be done approximately $p^{-1}$ times. Besides, in each iteration, one has to
do a complete randomization of experimental units 
into treatment and control groups, which has 
complexity of $O(n)$ \citep{Fan62}; see also \citet{Meng13}. %
Consequently, in expectation, the computational complexity for getting one acceptable assignment from ReM is approximately of $O(nK(K+p^{-1}))$.
From the discussion in Section \ref{sec:iidconditions} and in particular Table \ref{tab:rate}, 
to achieve the ideally optimal rerandomization, 
we will choose $K = o(\log n)$ and can choose $p^{-1} = n^{\beta}$ for sufficiently small $\beta$, 
under which
the computational complexity
for getting an acceptable assignment from ReM is in expectation approximately of order $o(n^{1+\beta}\log n)$ (i.e., polynomial in $n$ with exponent slightly greater than 1). 
This implies that the optimal rerandomization is computationally
feasible even for relatively large sample size.

\section{Conclusion and Discussion}\label{sec:discuss}
There is a long-time controversy between the two philosophies for designing an experiment: randomization versus optimal and thus often deterministic assignment. 
In the context of balancing covariates in randomized experiments, 
the optimal design tries to find the treatment assignment minimizing the covariate imbalance (say the Mahalanobis distance), 
while rerandomization tries to restrict the covariate imbalance and at the same time maintain sufficient randomness of the design, which is necessary for robust inference of treatment effects.
In this paper, we demonstrated that, by letting the acceptance probability diminish to zero at a property rate (e.g., a polynomial rate), 
rerandomization can still have sufficient randomness for robust randomization-based inference of treatment effects, 
and more importantly, it can achieve the ideally optimal precision that one can expect from the optimally balanced design.
Note that our theory also helps mitigate the computation burden for usual optimal designs. 
In particular, to achieve the ideally optimal precision, 
we only need to randomly select assignments from a small proportion (which generally decreases with the sample size polynomially) of all assignments with the best covariate balance, whose computational complexity is generally of a polynomial order of the sample size with exponent slightly greater than 1 as discussed in Section \ref{sec:comp_cost}.

The derived theory for rerandomization also allows for a diverging number of covariates. In particular, 
we found that, when the number of covariates is too large, not only will the asymptotic approximation for rerandomization become inaccurate, but also rerandomization will lose its gain on efficiency. 
Therefore, we suggest practitioners to use a moderate number of covariates, especially those important and useful for explaining the potential outcomes, and to also perform trimming  
as suggested in Section \ref{sec:reg_cond}. 
Importantly, because our finite population inference imposes no distributional assumptions on potential outcomes and covariates, we are free to adjust the covariates in an arbitrary way, such as trimming or transforming using, e.g., principle components.  
In the Supplementary Material \citep{WLSupp2022}, we also provide additional finite-sample diagnosis tools, discuss the choice of covariates and threshold for rerandomization, and conduct a simulation study.

In this paper, we mainly focused on the asymptotic properties of rerandomized treatment-control experiments using the Mahalanobis distance criterion; see the Supplementary Material \citep{WLSupp2022} for extension to regression adjustment under rerandomization. 
Beyond that, the derived theory, including both the finite population central limit theorem and the asymptotic behavior of the constrained Gaussian random variable, can also be useful for analyzing other covariate balance criteria, such as the Mahalanobis distance criterion with tiers of covariates \citep{MR15,LDR18}.
It will also be interesting to extend the theory to rerandomization in more complex experiments, such as blocked experiments \citep{Sch19,WWL20}, factorial experiments \citep{BDR16,LDR20} and sequential experiments~\citep{ZEM18}.
Besides, we mainly considered finite population inference focusing on the average treatment effect of the experimental units in hand. 
It will be interesting to also consider superpopulation inference of some population average treatment effect when the units are randomly sampled from some superpopulation \citep{SJ20}.

\begin{acks}[Acknowledgments]
We thank the Editor, the Associate Editor and three reviewers 
for constructive comments. Yuhao Wang is also affiliated with Shanghai Qi Zhi Institute.
\end{acks}

\begin{funding}
The work of Yuhao Wang was supported by Tsinghua New Faculty Start-up Fund and the 2030 Innovation Megaprojects of China (Programme on New Generation Artificial Intelligence) Grant No. 2021AAA0150000.
\end{funding}

\begin{supplement}
\stitle{Supplement to ``Rerandomization with Diminishing Covariate Imbalance and Diverging Number of Covariates''}
\sdescription{
First, we provide additional finite-sample diagnosis tools for rerandomization, and conduct a simulation study. 
Second, we extend the asymptotic theory to regression adjustment under rerandomization. 
Third, we study the Berry--Esseen-type bound for finite population central limit theorem under simple random sampling. 
Fourth, we prove all the theorems, corollaries and propositions. 
Fifth, we connect rerandomization with usual optimal designs. 
}
\end{supplement}

\bibliographystyle{imsart-nameyear}
\bibliography{reference}

\newpage

\renewcommand {\theproposition} {A\arabic{proposition}}
\renewcommand {\thefigure} {A\arabic{figure}}
\renewcommand {\thetable} {A\arabic{table}}
\renewcommand {\theequation} {A\arabic{section}.\arabic{equation}}
\renewcommand {\thelemma} {A\arabic{lemma}}
\renewcommand {\thesection} {A\arabic{section}}
\renewcommand {\thetheorem} {A\arabic{theorem}}
\renewcommand {\thecorollary} {A\arabic{corollary}}
\renewcommand {\thecondition} {A\arabic{condition}}

\renewcommand {\thepage} {A\arabic{page}}
\setcounter{page}{1}

\setcounter{equation}{0}
\setcounter{section}{0}
\setcounter{figure}{0}
\setcounter{proposition}{0}
\setcounter{corollary}{0}
\setcounter{theorem}{0}
\setcounter{lemma}{0}
\setcounter{table}{0}
\setcounter{condition}{0}

\begin{center}
	\bf \large
	\uppercase{Supplement to ``Rerandomization with Diminishing Covariate Imbalance and Diverging Number of Covariates''}
\end{center}

Appendix \ref{sec:finite_diag_simu} provides some finite-sample diagnosis tools, discusses the choice of covariates and threshold for rerandomization, and conducts a simulation study.

Appendix \ref{sec:reg_adj} studies regression adjustment after rerandomization.

Appendix \ref{sec:berry_srs} studies Berry--Esseen-type bound for finite population central limit theorem in simple random sampling.

Appendix \ref{sec:asym_cre_rerand} studies asymptotic properties for completely randomized and rerandomized experiments. 
It includes the proofs of Theorems \ref{thm:berry_esseen_clt}--\ref{thm:dim_rem}, Corollary \ref{cor:improve}; and technical details about the comments of~\eqref{eq:gamma_n_lower}.

Appendix \ref{sec:limit_constrained_Gaussian} studies the limiting behavior of the constrained Gaussian random variable. It includes the proof of Theorem \ref{thm:v_Ka}. 

Appendix \ref{sec:asym_optimal_rerand} studies asymptotics for  the optimal rerandomization. It includes the proofs of Theorems \ref{thm:rem_gaussian} and \ref{thm:K_n_Delta_n}. 

Appendix \ref{sec:asym_CI} studies large-sample inference for rerandomization. It includes the proofs of Theorems \ref{thm:inf} and \ref{thm:inf_gaussian}. 

Appendix \ref{sec:regularity_diagnosis} studies the regularity conditions and finite-sample diagnoses for rerandomization. 
It also provides the technical details for the comments on $\gamma_n$ and $\sum_{i=1}^n H_{ii}^{3/2}$ in Section \ref{sec:implication}.

Appendix~\ref{sec:proof_regrem} studies the asymptotic properties of regression adjustment under rerandomization with a diverging number of covariates. It includes the proof of Theorem~\ref{thm:regrem}.

Appendix~\ref{sec:opt_comp} studies connections to optimal design under certain hypothesized model of the potential outcomes.

\section{ Finite-sample Diagnoses and Simulation Studies}\label{sec:finite_diag_simu}

\subsection{Finite-sample diagnoses for rerandomization}\label{sec:fp_rerand}

Our theoretical results are mostly concerned with the asymptotic properties of rerandomization designs. 
In this section, 
we further provide some additional tools for the diagnosis of rerandomization in finite samples. 
Our diagnosis is based on the bias and mean squared error (MSE) of the difference-in-means estimator  $\hat{\tau}$ for estimating the average treatment effect $\tau$.
\citet{KKSSA20} and \citet{NS20} also considered the MSE of $\hat{\tau}$ under rerandomization, but they focused mainly on the case with equal treatment group sizes (i.e., $r_1 = r_0 = 1/2$). 
In the following, we consider a general design $\mathcal{D}$ that randomly assigns $r_1$ proportion of units to treatment and the remaining $r_0$ proportion to control. 
For descriptive convenience, we introduce $\E_{\mathcal{D}}(\cdot)$ and $\var_{\mathcal{D}}(\cdot)$ to denote the mean and variance under the design $\mathcal{D}$. 
The bias and MSE of the difference-in-means estimator  $\hat{\tau}$ under the design $\mathcal{D}$ can then be written as $\E_{\mathcal{D}}(\hat{\tau} - \tau)$ and  $\E_{\mathcal{D}}\{(\hat{\tau} - \tau)^2\}$.

Recall that $\bs{Z}=(Z_1, \ldots, Z_n)^\top$ is the treatment assignment vector for all $n$ units. 
Let  
$\bs{\pi} \equiv \E_{\mathcal{D}} (\bs{Z})\in \mathbb{R}^n$ and $\bs{\Omega} \equiv \cov_{\mathcal{D}}(\bs{Z})\in \mathbb{R}^{n\times n}$ be its mean and covariance matrix under the design $\mathcal{D}$. 
Let $\wy_i = r_0 Y_i(1) + r_1 Y_i(0)$ denote the weighted average of potential outcomes for unit $i$, $\bar{\wy}=n^{-1} \sum_{i=1}^n \wy_i$ be the corresponding finite population average, 
and $\tilde{\bs{\wy}} = (\wy_1 - \bar{\wy}, \ldots, \wy_n - \bar{\wy})^\top \in \mathbb{R}^n$ be the vector consisting of the centered weighted averages of potential outcomes for all units. 
As demonstrated in 
Appendix \ref{sec:regularity_diagnosis}, the bias and MSE of the  difference-in-means estimator $\hat{\tau}$ have the following forms:
\begin{align}\label{eq:bias_mse}
	\E_{\mathcal{D}}(\hat{\tau} - \tau) = \frac{ (\bs{\pi}-r_1 \bs{1}_n)^\top \tilde{\bs{\wy}}}{nr_1 r_0}, 
	\quad 
	\E_{\mathcal{D}}\{(\hat{\tau} - \tau)^2\} = \frac{\tilde{\bs{\wy}}^\top
		\big\{ \bs{\Omega} + (\bs{\pi}-r_1 \bs{1}_n) (\bs{\pi}-r_1 \bs{1}_n)^\top \big\}
		\tilde{\bs{\wy}}}{(nr_1 r_0)^2},
\end{align}
where $\bs{1}_n$ denotes an $n$-dimensional vector with all elements being 1.
In \eqref{eq:bias_mse}, $\tilde{\bs{\wy}}$ depends on the potential outcomes and is generally unknown in the design stage of an experiment. 
However, the other quantities in \eqref{eq:bias_mse} are fully determined by the design $\mathcal{D}$, and can be computed or at least approximated by Monte Carlo method before actually conducting the experiment. 
Below we consider the worst-case behavior of the design in terms of the estimation bias and MSE in \eqref{eq:bias_mse} over the unknown potential outcomes. 
Recall that $V_{\tau\tau}$ in \eqref{eq:V} is the variance of $\hat{\tau}$ under the CRE. 
As verified in 
Appendix \ref{sec:regularity_diagnosis}, we can equivalently write $V_{\tau\tau}$ as $V_{\tau\tau} = \tilde{\bs{\wy}}\tilde{\bs{\wy}}^\top/\{n (n-1) r_1 r_0\}$. 

\begin{proposition}\label{prop:bias_mse}
	For any design $\mathcal{D}$ that randomly assign $r_1$ proportion of units to treatment and the remaining $r_0$ to control,
	the maximum absolute bias and the maximum root MSE of the difference-in-means estimator $\hat{\tau}$ under $\mathcal{D}$, standardized by the corresponding standard deviation of $\hat{\tau}$ under the CRE, have the following forms:
	\begin{align}
		\label{eq:max_bias}
		\max_{\tilde{\bs{\wy}} \ne \bs{0}}V_{\tau\tau}^{-1/2} \left| \E_{\mathcal{D}}(\hat{\tau} - \tau) \right|
		& = 
		\sqrt{\frac{n-1}{n r_1 r_0}} \cdot \left\| \bs{\pi} - r_1 \bs{1}_n \right\|_2 \ge 0, 
		\\
		\label{eq:max_mse}
		\max_{\tilde{\bs{\wy}} \ne \bs{0}}V_{\tau\tau}^{-1/2}  \sqrt{\E_{\mathcal{D}}\{(\hat{\tau} - \tau)^2\} }
		& = 
		\sqrt{\frac{n-1}{n r_1 r_0}} \cdot \lambda_{\max}^{1/2}\left(  \bs{\Omega} + (\bs{\pi}-r_1 \bs{1}_n) (\bs{\pi}-r_1 \bs{1}_n)^\top \right) \ge 1, 
	\end{align}
	where $\tilde{\bs{\wy}}$, $\bs{\pi}$ and $\bs{\Omega}$ are the same as defined before, and $\lambda_{\max}(\cdot)$ denotes the largest eigenvalue of a matrix. 
\end{proposition}

Proposition \ref{prop:bias_mse} characterizes the maximum bias and root MSE under any given design. 
It is not difficult to see that, 
when the design $\mathcal{D}$ is the CRE, 
the maximum mean in \eqref{eq:max_bias} achieves its minimum value 0, 
and the maximum root MSE in \eqref{eq:max_mse} achieves its minimum value 1. 
This implies that the CRE is minimax optimal; see also \citet{W1981}. 
However, this does not contradict with our Corollary \ref{cor:improve}, which shows that the difference-in-means estimator under ReM always has smaller or equal variance and shorter or equal symmetric quantile ranges than that under the CRE asymptotically. 
The reason is that Proposition \ref{prop:bias_mse} considers all possible configurations of potential outcomes, including the case with $R_n^2=0$, i.e., the potential outcomes are uncorrelated with the covariates. 
In this case, the asymptotic distribution of $\hat{\tau}$ under ReM reduces to that under the CRE.

More importantly, Proposition \ref{prop:bias_mse} can help us conduct some finite-sample diagnoses for rerandomization. 
Given any acceptance probability $p$ and covariates $\bs{X}$ for each unit,  
we can estimate $\bs{\pi}$ and $\bs{\Omega}$ by simulating treatment assignments from the corresponding ReM, 
based on which we can then investigate the maximum bias and root MSE in \eqref{eq:max_bias} and \eqref{eq:max_mse}. 
In practice, we may consider several choices of $p$ and $\bs{X}$, and compare them taking into account both the improvement they can bring as shown in Corollary \ref{cor:improve} and the finite-sample biases they may cause as shown in  Proposition \ref{prop:bias_mse}; see the next subsection for details.

\subsection{Choice of covariates and acceptance probability for rerandomization}\label{sec:choice_Kp}

Below we consider some practical strategy to choose the covariates and imbalance threshold for the design of rerandomization in practice. 
From Corollary \ref{cor:improve}, 
if the sample size is large and the asymptotic approximation works well, 
then rerandomization can reduce the MSE of the difference-in-means estimator by $100 (1-v_{K_n, a_n}) R_n^2$ percent, or equivalently the standardized MSE is approximately $1 - (1-v_{K_n, a_n}) R_n^2$. 
On the contrary, 
from Appendix \ref{sec:fp_rerand}, with a finite sample size, 
the worst-case MSE of the difference-in-means estimator under rerandomization is no less than that under the CRE, and their ratio is the square of the quantity in \eqref{eq:max_mse}. 
Obviously, there is a trade-off for the choice of covariates and imbalance threshold. 
First, 
when the threshold (or equivalently the acceptance probability) decreases and the covariates are fixed, the asymptotic percentage reduction in MSE will increase (due to the decreasing $v_{K_n, a_n}$), while the finite-sample worst-case MSE is likely to increase. 
When the number of covariates increases and the acceptance probability is fixed, 
the asymptotic percentage reduction in MSE may increase or decrease (due to the increasing $v_{K_n,a_n}$ and $R_n^2$), while the finite-sample worst-case MSE is likely to increase.

Inspired by the above trade-off, we propose the following measure for the choice of covariates $\bs{X}$ and acceptance probability $p$ for rerandomization, 
which takes a geometric mean of the standardized MSEs in the worst case and the best case (in which the asymptotics works well): 
\begin{align}\label{eq:measure_Xp}
    c(\bs{X}, p) \equiv \widetilde{\text{MSE}}(\bs{X}, p) \times \{1 - (1-v_{K, a}) R^2_{\bs{X}}\}; 
\end{align}
other possible measures taking into account the best- and worst-case MSEs can also be considered for practical diagnosis.
In~\eqref{eq:measure_Xp}, $\widetilde{\text{MSE}}(\bs{X}, p)$ denotes the worst-case standardized MSE under rerandomization with covariates $\bs{X}$ and acceptance probability $p$, 
$K$ is the dimension of covariates, $a$ is the $p$-th quantile of the chi-squared distribution with degrees of freedom $K$, 
and $R^2_{\bs{X}}$ denotes the squared multiple correlation between potential outcomes and covariates $\bs{X}$. 
We can then use \eqref{eq:measure_Xp} as a measure for comparing different rerandomization designs, and can choose the one with minimum value of \eqref{eq:measure_Xp} for the actual implementation of the experiment. Note that $R^2_{\bs{X}}$ in \eqref{eq:measure_Xp} depends on the potential outcomes and is thus unknown in the design stage of experiments. 
In practice, 
we can use some domain knowledge or some prior (pilot) studies to estimate $R^2_{\bs{X}}$. 

Below we illustrate the use of the measure in \eqref{eq:measure_Xp} using the dataset from the Student Achievement and Retention (STAR) Project \citep[STAR,][]{Angrist2009}, a randomized evaluation of academic services and incentives conducted at a Canadian university. 
We focus on the treatment group where the students were offered some academic support (including peer-advising service) and scholarships for meeting targeted grades, and the control group receiving neither of these. 
Similar to \citet{LDR18}, we dropped the students with missing covariates, resulting a treated group of size $n_1 = 118$ and $n_0 = 856$. 
We generate 200 covariates for each unit, where the first 
five are from the STAR project, i.e., high-school GPA, age, gender and indicators for whether lives at home and whether rarely puts off studying for tests, and the rest $195$ are drawn independently from the $t$ distribution with degrees of freedom 2;  
once generated, these covariates are kept fixed, mimicking the finite population inference. 
We consider rerandomization with the first $K=5, 10, 50, 100, 200$ covariates, and consider acceptance probability $p = 0.001, 0.05, 0.01, 0.1, 0.5$. 
The left half of Table \ref{tab:choice_Kp} shows the worst-case MSEs, standardized by that under the CRE, under various choices of $(K,p)$, where each worst-case MSE is estimated based on at least about $10^5$ randomly generated treatment assignments from each design. 
It shows that the worst-case MSE generally increases as the number of covariates increases and the acceptance probability decreases. 
We further hypothesize that $R^2_{\bs{X}}$ takes values $0.4, 0.5, 0.6, 0.7, 0.8$, respectively, when $K$ increases from $5$ to $200$. 
Intuitively, this implies that the additional gain from including more covariates decreases with the number of included covariates. 
The right half of Table \ref{tab:choice_Kp} shows the value of the measure in \eqref{eq:measure_Xp}, which suggests to use rerandomization with $K=5$ covariates and acceptance probability $p=0.01$.

\begin{table}[htb]
	\centering
	\caption{The worst-case mean squared error (standardized by that under the CRE) and the measure in \eqref{eq:measure_Xp} for choosing and diagnosing rerandomization designs under various choices of covariates (whose number is denoted by $K$) and acceptance probability (denoted by $p$). 
	}\label{tab:choice_Kp}
	\resizebox{0.9\columnwidth}{!}{%
		\begin{tabular}{ccccccccccccc}
			\toprule
			& & \multicolumn{5}{c}{worst-case mean squared error} & & \multicolumn{5}{c}{The measure in \eqref{eq:measure_Xp}}
			\\
			\diagbox[height=1.5\line]{$K$}{$p$}
			& & 0.5 & 0.1 & 0.05 & 0.01 & 0.001 & & 0.5 & 0.1 & 0.05 & 0.01 & 0.001
			\\
			\midrule
			5&& 1.012&1.025&1.033&1.068&1.216& & 0.819&0.704&0.685& {\bf 0.674} &0.744 \\
            10&& 1.015&1.095&1.147&1.264&1.414& & 0.839&0.754&0.752&0.762&0.792 \\
            50&& 1.023&1.340&1.523&1.935&2.477&& 0.925&1.083&1.188&1.411&1.676\\
            100&&1.029&1.448&1.684&2.225&2.96&& 0.948&1.212&1.367&1.702&2.114\\
            200&& 1.038&1.495&1.752&2.356&3.189&& 0.972&1.294&1.479&1.892&2.417
			\\
			\bottomrule
		\end{tabular}%
	}
\end{table}

\subsection{A simulation study}\label{sec:simu}
We now conduct a simulation study
to demonstrate the potential gain from trimming, as well as investigating the inference for rerandomization in finite samples. 
We use again the dataset from the STAR project, and
consider in total nine rerandomization designs for the $n_1 + n_0 = 974$ units, with number of covariates $K$ ranging from 0 to 200 and acceptance probability fixed at $p = 0.001$, 
where the covariates are generated in the same way as in Appendix \ref{sec:choice_Kp}. 
Note that when $K=0$, rerandomization without any covariate essentially reduces to the CRE.
We then simulate $10^5$ treatment assignments from each of these designs.

\begin{table}[htb]
	\centering 
	\caption{
		Properties of rerandomization with fixed acceptance probability $p_a = 0.001$ and varying number of covariates $K$.  
		The 1st column shows the number of covariates for each design, the 2nd column shows the corresponding value of $1-v_{K,a}$. 
		The 3rd to 6th columns show the maximum standardized bias and mean squared error in \eqref{eq:max_bias} and \eqref{eq:max_mse}, the summation of leverages to the power of $3/2$, and the maximum leverage over all units. 
		The 7th--10th columns show the analogous quantities for the design using trimmed covariates. 
		The 11th and 12th columns show the minimum possible values of $\sum_{i=1}^n H_{ii}^{3/2}$ (which equals $K^{3/2} / \sqrt{n}$) and $\max_i H_{ii}$ (which equals $K_n/n$) for each $K$.
	}\label{tab:simulation}
	\smallskip
	\resizebox{1\columnwidth}{!}{%
		\begin{tabular}{ccccccccccccccc}
			\toprule
			$K$ & $1-v_{K,a}$ & & \multicolumn{4}{c}{Covariates without trimming} & &  \multicolumn{4}{c}{Trimmed covariates} & &  \multicolumn{2}{c}{Minimal}
			\\
			& & & Bias & RMSE & $\sum_{i=1}^n H_{ii}^{3/2}$ & $\max_{i} H_{ii}$ & & Bias & RMSE & $\sum_{i=1}^n H_{ii}^{3/2}$ & $\max_{i} H_{ii}$ && $\sum_{i=1}^n H_{ii}^{3/2}$ & $\max_{i} H_{ii}$ 
			\\
			\midrule
			0&0&&0.10&1.10&NA&NA&&0.10&1.10&NA&NA&&NA&NA\\
            5&0.97&&0.13&1.10&0.39&0.02&&0.11&1.10&0.39&0.02&&0.36&0.01\\
            9&0.90&&0.76&1.18&1.78&0.66&&0.16&1.10&0.94&0.03&&0.87&0.01\\
            15&0.80&&0.91&1.25&3.48&0.66&&0.20&1.10&1.97&0.05&&1.86&0.02\\
            24&0.70&&1.28&1.46&8.00&0.98&&0.23&1.11&3.92&0.06&&3.77&0.02\\
            37&0.60&&1.33&1.52&12.71&0.98&&0.25&1.11&7.41&0.09&&7.21&0.04\\
            60&0.50&&1.47&1.62&23.45&0.98&&0.27&1.11&15.13&0.13&&14.89&0.06\\
            100&0.41&&1.56&1.72&43.87&0.98&&0.29&1.12&32.33&0.16&&32.04&0.10\\
            200&0.30&&1.62&1.79&104.57&0.99&&0.31&1.13&91.01&0.28&&90.63&0.21\\
			\bottomrule
		\end{tabular}%
	}
\end{table}

Table \ref{tab:simulation} reports the simulation results. The first column shows the number of covariates involved in the nine rerandomization designs, and the second column shows the value of $1-v_{K,a}$ under various values of $K$ and fixed acceptance probability $p = 0.001$.
From Corollary \ref{cor:improve}, if the additional covariates do not increase the squared multiple correlation $R^2$ between potential outcomes and covariates by a relatively large amount, the improvement from rerandomization may decrease as the number of covariates increases. 
The 3rd--6th columns in Table \ref{tab:simulation} show the maximum absolute bias in \eqref{eq:max_bias}, the maximum root MSE in \eqref{eq:max_mse}, 
the summation of leverages to the power of $3/2$ as in \eqref{eq:bound_gamma_n} and the maximum leverage over all units, 
where the first two are estimated based on the $10^5$ simulated assignments from each of these designs. 
Note that when $K=0$, the design is the CRE, and from the discussion after Proposition \ref{prop:bias_mse}, the maximum mean is 0 and the maximum root MSE is 1. 
Thus, there is some variability for estimating the maximum mean and MSE; in practice, we may increase the number of simulated assignments to improve the precision. 
Nevertheless, the 3rd--6th columns in Table \ref{tab:simulation} show the trend that, as the number of covariates increases, 
the maximum bias and MSE under rerandomization will increase, which may render our treatment effect estimation inaccurate, and the leverages will increase as well, which may make the asymptotic approximation less accurate
as discussed shortly. 
We further perform trimming on the covariates, as suggested in Section \ref{sec:implication}. 
Specifically, we trim each covariate at both its $2.5\%$ and  $97.5\%$  quantiles. 
From the 7th--10th columns in Table \ref{tab:simulation}, trimming significantly reduces the maximum biases, maximum MSEs and leverages. 
Moreover, 
compared to the 11th and 12th columns, the values of $\sum_{i=1}^n H_{ii}^{3/2}$ and $\max_i H_{ii}$ after trimming become quite close to their minimal possible values. 
This agrees with the suggestion we gave in Section~\ref{sec:implication}.

\begin{table}[htb]
	\centering
	\caption{
		Asymptotic approximation and coverage property under the nine rerandomization designs in Table \ref{tab:simulation}. 
		The top half uses the first-year GPA from the STAR dataset as the potential outcomes, and the bottom half use the average propensity score from the nine design without trimming, after a quantile transformation using $t$ distribution with degree of freedom 3, as the potential outcomes. 
		The $K$ column shows the number of covariates in each design, Bias column shows the absolute empirical bias standardized by $V_{\tau\tau}^{1/2}$, Ratio column shows the ratio between empirical and asymptotic mean squared errors, and HC0--3 columns show the coverage probabilities in percent of the $95\%$ confidence intervals using the methods HC0--3 described in Section \ref{sec:ci}.  
	}\label{tab:simu_approx}
	\smallskip
	\resizebox{0.9\columnwidth}{!}{%
		\begin{tabular}{ccccccccccccccc}
			\toprule
			$K$ & &  \multicolumn{6}{c}{Covariates without trimming} & & \multicolumn{6}{c}{Trimmed covariates} 
			\\
			& & Bias & Ratio & HC0 & HC1 & HC2 & HC3 & & Bias & Ratio & HC0 & HC1 & HC2 & HC3
			\\
			\midrule
			0&&0.001&1.00&97.5&97.5&97.5&97.5&&0.000&1.00&97.4&97.4&97.4&97.4\\
            5&&0.002&1.00&97.3&97.5&97.5&97.8&&0.000&1.00&97.2&97.4&97.5&97.7\\
            9&&0.012&1.01&96.9&97.3&97.3&97.8&&0.004&1.01&97.0&97.5&97.5&97.9\\
            15&&0.012&1.02&96.9&97.5&97.5&98.1&&0.011&1.02&96.7&97.3&97.3&97.9\\
            24&&0.018&1.02&96.7&97.4&97.4&98.2&&0.006&1.01&96.5&97.3&97.4&98.0\\
            37&&0.003&1.03&96.2&97.3&97.3&98.0&&0.006&1.02&96.1&97.2&97.2&97.9\\
            60&&0.012&1.04&95.5&97.1&97.2&97.7&&0.020&1.04&95.3&97.1&97.1&97.6\\
            100&&0.008&1.05&94.4&96.8&96.8&97.4&&0.014&1.04&94.3&96.8&96.8&97.3\\
            200&&0.006&1.07&93.9&94.0&94.0&94.4&&0.015&1.06&93.7&93.8&93.8&93.9
			\\
			\midrule
			0&&0.005&0.99&97.5&97.5&97.5&97.5&&0.002&1.00&97.5&97.5&97.5&97.5\\
            5&&0.013&1.01&97.2&97.3&97.3&97.4&&0.003&1.00&97.2&97.4&97.4&97.4\\
            9&&0.346&1.03&93.2&93.9&94.3&95.0&&0.021&1.01&96.6&97.0&97.1&97.5\\
            15&&0.513&1.16&89.5&90.9&91.5&92.6&&0.032&1.02&96.1&96.8&96.8&97.4\\
            24&&0.838&1.57&79.2&82.1&83.4&86.0&&0.048&1.02&95.6&96.5&96.6&97.4\\
            37&&0.923&1.73&74.9&79.0&80.4&83.1&&0.056&1.03&95.1&96.4&96.4&97.4\\
            60&&1.022&1.92&68.8&74.9&76.1&78.7&&0.062&1.03&94.6&96.4&96.5&97.3\\
            100&&1.094&2.11&62.7&70.8&71.5&74.5&&0.054&1.04&93.4&96.0&96.0&96.8\\
            200&&1.065&2.11&63.1&63.6&64.1&68.3&&0.067&1.07&92.5&92.6&92.6&92.8
			\\
			\bottomrule
		\end{tabular}%
	}
\end{table}

We then consider the asymptotic approximation and coverage probabilities of confidence intervals under these rerandomization designs with different numbers of covariates. 
We first consider the case where both potential outcomes are the same as the observed first year GPA from the STAR dataset. 
The top half of Table \ref{tab:simu_approx} shows the absolute empirical bias standardized by $V_{\tau\tau}^{1/2}$, the ratio between empirical MSE and the corresponding asymptotic variance, and the empirical coverage probabilities of $95\%$ confidence intervals using methods HC0--3 described in Section \ref{sec:ci}. 
Note that when $K=200$, the number of covariates are greater than the size of treated group, under which we can only perform HC1--3 for the control group.
From Table \ref{tab:simu_approx}, under all the nine designs, the biases are close to 0, the ratios between empirical and asymptotic MSEs are close to 1, and the coverage probabilities are close to the nominal level, all of which indicate that the asymptotic approximation for rerandomization works quite well.  
These in some sense show the robustness of rerandomization. 
To illustrate the potential drawback of rerandomization with a large number of covariates, we also consider potential outcomes constructed in the following way: 
we first estimate the propensity scores for all units under these nine designs, then calculate the average of them for each unit, and finally take a quantile transformation using the $t$ distribution with degrees of freedom 3 to get both potential outcomes. 
The bottom half of Table \ref{tab:simu_approx} shows analogously the standardized absolute empirical bias, the ratio between empirical and asymptotic MSEs and the coverage probabilities of $95\%$ confidence intervals using HC0--3. 
From Table \ref{tab:simu_approx}, as $K$ increases, the standardized biases increases, the ratio becomes further from 1, and the coverage probabilities becomes much smaller than the nominal level, all of which indicates poor asymptotic approximation under rerandomization with a large number of covariates.
Comparing Tables \ref{tab:simulation} and \ref{tab:simu_approx}, we can find that the ratio and the coverage probabilities become further off from their ideal values as the maximum standardized bias and root MSE in Table \ref{tab:simulation} get larger, which indicates that~\eqref{eq:max_bias} and~\eqref{eq:max_mse} can be used as viable tools to help assist 
the design of 
ReM 
in practice. 
Finally, we also consider rerandomization with trimmed covariates for both cases.
From the right half of Table \ref{tab:simu_approx}, trimming helps improve the finite-sample performance of rerandomization in terms of both point and interval estimates. 
Moreover, compared to HC0, HC1--3 help improve the coverage probabilities of the confidence intervals, especially when the number of covariates is relatively large.

\section{Regression adjustment under rerandomization}\label{sec:reg_adj}

Regression adjustment is a popular approach to adjusting for covariate imbalance between two treatment groups after the experiments were conducted. Below we consider linearly regression adjusted estimator after ReM, which can be particularly useful when the analyzer is able to observe more covariate information after conducting the experiment.  
Let $\bs{W}_i \in \mathbb{R}^{J_n}$ denote the available covariate vector for unit $i$ in analysis, 
and 
$\hat{\bs{\tau}}_{\bs{W}}$ denote the corresponding difference-in-means of covariates between treatment and control groups. 
Following \citet{LD20}, a general linearly regression adjusted estimator has the following form: 
\begin{align}\label{eq:reg}
    \hat{\tau}(\bs{\beta}_1, \bs{\beta}_0)
    & = 
    \frac{1}{n_1} \sum_{i=1}^n Z_i \{ Y_i - \bs{\beta}_1^\top (\bs{W}_i-\bar{\bs{W}}) \}
    - 
    \frac{1}{n_0} \sum_{i=1}^n (1-Z_i) \{ Y_i - \bs{\beta}_0^\top (\bs{W}_i - \bar{\bs{W}}) \}
    \\
    & = \hat{\tau} - ( r_0 \bs{\beta}_1 + r_1 \bs{\beta}_0 )^\top \hat{\bs{\tau}}_{\bs{W}},
    \nonumber
\end{align}
where $\bs{\beta}_1$ and $\bs{\beta}_0$ are the covariate adjustment coefficients. From \eqref{eq:reg}, 
the regression adjusted estimator $\hat{\tau}(\bs{\beta}_1, \bs{\beta}_0)$ is essentially the difference-in-means estimator with adjusted treatment and control potential outcomes $Y_i(1) - \bs{\beta}_1^\top (\bs{W}_i-\bar{\bs{W}})$'s and $Y_i(0) - \bs{\beta}_0^\top (\bs{W}_i-\bar{\bs{W}})$'s. 
Therefore, its asymptotic property can be similarly derived as Theorems \ref{thm:dim_rem} and \ref{thm:rem_gaussian}. 
Below we focus on a specific regression adjusted estimator $\hat{\tau}(\tilde{\bs{\beta}}_1, \tilde{\bs{\beta}}_0)$,
which enjoys certain optimalities \citep[see, e.g.,][]{Lin13, LD20} and uses the following least squares coefficients 
for covariate adjustment: 
\begin{align*}
    \tilde{\bs{\beta}}_z = \arg\min_{\bs{\beta}} \sum_{i=1}^n \{Y_i(z) - \bar{Y}(z) - \bs{\beta}_z^\top (\bs{W}_i-\bar{\bs{W}})\}^2 = ( \bs{S}^2_{\bs{W}} )^{-1} \bs{S}_{\bs{W}, z}, 
    \quad (z=0,1)
\end{align*}
where $\bs{S}^2_{\bs{W}}$ and $\bs{S}_{\bs{W}, z}$ denote the finite population covariance matrices for covariates and potential outcomes.

For each unit $i$ and $z=0,1$, 
let $\tilde{e}_i(z) = Y_i(z) - \tilde{\bs{\beta}}_z^\top (\bs{W}_i - \bar{\bs{W}})$ denote the adjusted potential outcome, and $\tilde{\bs{u}}_i = (r_0 \tilde{e}_i(1) + r_1 \tilde{e}_i(0), \bs{X}_i^\top)^\top$. 
Define $\tilde{\gamma}_n$ and $\tilde{\Delta}_n$ analogously as \eqref{eq:gamma_n} and \eqref{eq:Delta_n}, 
but with $Y_i(z)$ replaced by $\tilde{e}_i(z)$, $\bs{u}_i$ replaced by $\tilde{\bs{u}}_i$ and $\hat{\tau}$ replaced by $\hat{\tau}(\tilde{\bs{\beta}}_1, \tilde{\bs{\beta}}_0)$. 
Analogous to \eqref{eq:R2}, we define $\tilde{R}_n^2$ as the squared multiple correlation between adjusted potential outcomes and covariates $\bs{X}_i$'s, 
and $\rho_n^2$ as the squared multiple correlation between original (unadjusted) potential outcomes and covariates $\bs{W}_i$'s. 
We first invoke the following regularity condition, which essentially assumes Conditions \ref{cond:gamma_n} and \ref{cond:p_n} for the adjusted potential outcomes. 

\begin{condition}\label{cond:regrem_Delta}
    Conditions \ref{cond:gamma_n} and \ref{cond:p_n} hold with 
     $\gamma_n$ and $\Delta_n$ replaced by $\tilde{\gamma}_n$ and $\tilde{\Delta}_n$. 
\end{condition}

Note that both adjustment coefficients $\tilde{\bs{\beta}}_1$ and $\tilde{\bs{\beta}}_0$ depend on all potential outcomes and are thus unknown. 
In practice, we can estimate them using the sampling analogues $\hat{\bs{\beta}}_z =  (\bs{S}_{\bs{W}}^2)^{-1} \bs{s}_{z, \bs{W}}$ for $z=0,1$, 
where $\bs{s}_{z, \bs{W}}$ is the sample covariance between observed outcomes and covariates for units under treatment arm $z$, 
and use the regression adjusted estimator $\hat{\tau}(\hat{\bs{\beta}}_1, \hat{\bs{\beta}}_0)$ with the estimated coefficients. 
We then invoke the following regularity condition, which ensures that the regression adjusted estimators with true and estimated coefficients have the same asymptotic distribution.  

\begin{condition}\label{cond:regrem}
	As $n \to \infty$, 
	\begin{align*}
	    \frac{\max_{z \in \{0,1\}} \max_{1 \leq i \leq n} |Y_i(z) - \bar{Y}(z)|}{\sqrt{V_{\tau\tau} (1-\rho^2_{n}) \{ 1-\tilde{R}_n^2\}}}  \cdot J_n \cdot \frac{\max\{1, \log J_n, - \log p_n\}}{n r_1^2 r_0^2} \to 0.
	\end{align*}
\end{condition}

We summarize the asymptotic distribution of $\hat{\tau}(\hat{\bs{\beta}}_1, \hat{\bs{\beta}}_0)$ under ReM in the theorem below.

\begin{theorem}\label{thm:regrem}
	Under ReM and Conditions \ref{cond:regrem_Delta} and \ref{cond:regrem}, as $n \to \infty$,
		\begin{align*}%
		\sup_{c\in \mathbb{R}} \left| \PP \left\{ 
		\frac{\hat{\tau}(\hat{\bs{\beta}}_1, \hat{\bs{\beta}}_0) - \tau}{
		\sqrt{V_{\tau\tau}(1-\rho_n^2)}}
		\le c \mid M \le a_n \right\} 
		- \PP\left\{ \sqrt{1-\tilde{R}^2_n}\ \varepsilon_0  + \sqrt{\tilde{R}^2_n} \ L_{K_n, a_n} \le c \right\}
		\right| 
		\rightarrow 0; 
	\end{align*}
	If further 
    Condition \ref{cond:k_np_n} holds and  $\limsup_{n \to \infty} \tilde{R}_n^2 < 1$, then 
	\begin{align*}
		\sup_{c\in \mathbb{R}} \left| \PP \left\{ 
		\frac{\hat{\tau}(\hat{\bs{\beta}}_1, \hat{\bs{\beta}}_0) - \tau}{
		\sqrt{V_{\tau\tau}(1-\rho_n^2)}}
		\le c \mid M \le a_n \right\} 
		- \PP\left\{ \sqrt{1-\tilde{R}^2_n}\ \varepsilon_0  \le c \right\}
		\right| 
		\rightarrow 0. 
	\end{align*}
\end{theorem}

Theorem \ref{thm:regrem} implies that we can still perform covariate adjustment under ReM with diminishing covariate imbalance threshold as well as diverging numbers of covariates in both design and analysis, which extends the discussion in \citet{LD20} with fixed threshold and fixed numbers of covariates. 
Moreover, with covariate imbalance diminishing at a proper rate, 
the regression adjusted estimator becomes asymptotically Gaussian distributed, and its improvement over the CRE is nondecreasing in $\tilde{R}_n^2$.

\section{Berry--Esseen-type Bound for Finite Population Central Limit Theorem in Simple Random Sampling}\label{sec:berry_srs}

\subsection{Main theorem}

\begin{theorem} \label{thm:berry_esseen_clt_full}
	Consider any finite population $\{\bs{u}_1, \bs{u}_2, \ldots, \bs{u}_N\}$ with $\bs{u}_i \in \mathbb{R}^d$, with $\bar{\bs{u}} = N^{-1} \sum_{i=1}^N \bs{u}_i$ 
	and 
	$\bs{S}^2 = (N-1)^{-1} \sum_{i=1}^N (\bs{u}_i - \bar{\bs{u}})(\bs{u}_i - \bar{\bs{u}})^\top $
	denoting the finite population average and covariance matrix. 
	Let $(Z_1, Z_2, \ldots, Z_N)$ denote the indicators for a simple random sample of size $m$, 
	i.e., the probability that $\bs{Z}$ takes a particular value $\bs{z} = (z_1, \ldots, z_N)\in \{0,1\}^N$ is $m!(N-m)!/N!$ if $\sum_{i=1}^N z_i = m$ and zero otherwise, 
	$f \equiv m/N$ be the fraction of sampled units, 
	and 
	$$
	\bs{W} = 
	\frac{1}{\sqrt{Nf(1-f)}} \bs{S}^{-1}
	\left\{  \sum_{i=1}^N Z_i \bs{u}_i - m \bar{\bs{u}}  \right\}. 
	$$
	Let $\bs{\varepsilon}\sim \mathcal{N}(\bs{0}, \bs{I}_d)$ denote a $d$-dimensional standard Gaussian random vector, and 
	define 
	\begin{align*}
		\gamma \equiv \frac{1}{\sqrt{Nf(1-f)}} \frac{d^{1/4}}{N} \sum_{i=1}^N \left\| \bs{S}^{-1} (\bs{u}_i-\bar{\bs{u}}) \right\|_2^3. 
	\end{align*}
	\begin{itemize}
		\item[(i)] 
		There exists $C_d$ that depends only on $d$ such that for any $N\ge 2$, $f\in (0,1),$ any finite population $\{\bs{u}_i: 1\le i \le N\}$ with nonsingular finite population covariance $\bs{S}^2$, 
		and any measurable convex set $\mathcal{Q} \subset \mathbb{R}^d$, 
		\begin{align*}
			\left| \PP(\bs{W} \in \mathcal{Q}) - \PP(\bs{\varepsilon} \in \mathcal{Q}) \right|
			\le 
			C_d \gamma
		\end{align*}
		\item[(ii)] 
		If the theorem in \citet{R15} holds, then 
		there exists a universal constant $C$ such that for any $N\ge 2$, $d\ge 1$, $f\in (0,1)$,  any finite population $\{\bs{u}_i: 1\le i \le N\}$ with nonsingular finite population covariance $\bs{S}^2$, 
		and any measurable convex set $\mathcal{Q} \subset \mathbb{R}^d$, 
		\begin{align*}
			\left| \PP(\bs{W} \in \mathcal{Q}) - \PP(\bs{\varepsilon} \in \mathcal{Q}) \right|
			\le 
			C \gamma
		\end{align*}
		\item[(iii)]
		For any $N\ge 2$, $d\ge 1$, $f\in (0,1)$,  any finite population $\{\bs{u}_i: 1\le i \le N\}$ with nonsingular finite population covariance $\bs{S}^2$, 
		and any measurable convex set $\mathcal{Q} \subset \mathbb{R}^d$, 
		\begin{align*}
			\left| \PP(\bs{W} \in \mathcal{Q}) - \PP(\bs{\varepsilon} \in \mathcal{Q}) \right|
			\le 
			174 \gamma + 3 \cdot 2^{2/3} \frac{d^{1/2}}{\{Nf(1-f)\}^{1/6}}
			\le 174\gamma + 7 \gamma^{1/3}. 
		\end{align*}

		\item[(iv)]
		For any $N\ge 2$, $d\ge 1$, $f\in (0,1)$,  any finite population $\{\bs{u}_i: 1\le i \le N\}$ with nonsingular finite population covariance $\bs{S}^2$, 
		and any measurable convex set $\mathcal{Q} \subset \mathbb{R}^d$, 
		\begin{align*}
			\left| \PP(\bs{W} \in \mathcal{Q}) - \PP(\bs{\varepsilon} \in \mathcal{Q}) \right|
			\le 
			180 \gamma + \frac{3  (\log N)^{3/4}d^{3/4}}{ N^{1/4} \sqrt{f(1-f)}} \cdot \max_{1 \le i \le n} \left\|\left(\bs{S}_{\bs{u}}^2\right)^{-1/2} (\bs{u}_i-\bar{\bs{u}})\right\|_\infty. 
		\end{align*}
		\item[(v)]
		For any $N\ge 2$, $d\ge 1$, $f\in (0,1)$,  any finite population $\{\bs{u}_i: 1\le i \le N\}$ with nonsingular finite population covariance $\bs{S}^2$, 
		any $\iota \ge 2$, 
		and any measurable convex set $\mathcal{Q} \subset \mathbb{R}^d$, 
		\begin{align*}
			\left| \PP(\bs{W} \in \mathcal{Q}) - \PP(\bs{\varepsilon} \in \mathcal{Q}) \right|
			\le 
			174 \gamma 
        +  
        \frac{ C_{\iota} d^{3\iota/\{4(\iota+1)\}}}{N^{\iota/\{4(\iota+1)\}} \{f(1-f)\}^{\iota/2}} \frac{1}{N}\sum_{i = 1}^N \| \left(\bs{S}_{\bs{u}}^2\right)^{-1/2} (\bs{u}_i-\bar{\bs{u}}) \|_{\iota}^\iota, 
		\end{align*}
		where $C_{\iota}$ is a universal constant depending only on $\iota$. 
	\end{itemize}
\end{theorem}

\subsection{Proof of Theorem \ref{thm:berry_esseen_clt_full}(i) and (ii) based on combinatorial central limit theorem}

To prove Theorem \ref{thm:berry_esseen_clt_full}(i), we need the following lemma, which follows immediately from \citet[][Theorem 1]{BG93}. 

\begin{lemma}\label{lemma:perm_clt_fix_dim}
	Consider any integer $N\ge 1$ and any constant vector $\bs{a}(i,j)\in \mathbb{R}^d$ for all $1\le i,j\le N$ satisfying that 
	\begin{align}\label{eq:aij_constraint}
		\sum_{j=1}^N \bs{a}(i,j) = \bs{0}, 1\le i \le N 
		\ \ \text{and} \ \ 
		\frac{1}{N-1} \sum_{i=1}^N \sum_{j=1}^N \bs{a}(i,j) \bs{a}(i,j)^\top - 
		\frac{1}{N(N-1)} \sum_{j=1}^N \bs{b}(j) \bs{b}(j)^\top = \bs{I}_d, 
	\end{align}
	where $\bs{b}(j) \equiv \sum_{i=1}^N \bs{a}(i,j)$. 
	Let $\pi$ denote a uniformly distributed random permutation of $\{1,2,\ldots,N\}$, 
	$\bs{W} = \sum_{i=1}^N \bs{a}(i, \pi(i))$, and $\bs{\varepsilon}\sim \mathcal{N}(\bs{0}, \bs{I}_d)$ denote a $d$-dimensional standard Gaussian random vector. 
	Then there exists a constant $C_d$ that depends only on $d$ such that for any $N\ge 2$, any $\{\bs{a}(i,j): 1\le i,j\le N\}$ satisfying \eqref{eq:aij_constraint}, and 
	any measurable convex set $\mathcal{Q} \subset \mathbb{R}^d$, 
	\begin{align*}
		\left| \PP(\bs{W} \in \mathcal{Q}) - \PP(\bs{\varepsilon} \in \mathcal{Q}) \right|
		\le 
		C_d \frac{1}{N} \sum_{i=1}^N \sum_{j=1}^N \|\bs{a}(i,j)\|_2^3. 
	\end{align*}
\end{lemma}

\begin{proof}[\bf Proof of Theorem \ref{thm:berry_esseen_clt_full}(i)]
	Let $\bs{S}$ denote the positive definite square root of $\bs{S}^2$, 
	and 
	define 
	\begin{align*}
		\bs{a}(i,j) = 
		\begin{cases}
			\{m(1-f)\}^{-1/2} \bs{S}^{-1} (1-f) ( \bs{u}_i - \bar{\bs{u}}), & \text{if } 1\le j \le m, \\
			- \{m(1-f)\}^{-1/2} \bs{S}^{-1} f ( \bs{u}_i - \bar{\bs{u}}), & \text{if } m < j \le N,
		\end{cases}
		\quad (1 \le i \le N). 
	\end{align*}
	We can then verify that 
	\begin{align*}
		\sum_{j=1}^N \bs{a}(i,j) & = 
		\{m(1-f)\}^{-1/2} \bs{S}^{-1} \cdot
		\left\{
		m(1-f)  - (N-m) f 
		\right\}( \bs{u}_i - \bar{\bs{u}})
		= \bs{0}, \quad (1\le i \le N)
		\\
		\bs{b}(j)  & \equiv \sum_{i=1}^N \bs{a}(i,j) = \bs{0}, \quad (1\le j \le N)
	\end{align*}
	and 
	\begin{align*}
		& \quad \ 
		\frac{1}{N-1} \sum_{i=1}^N \sum_{j=1}^N \bs{a}(i,j) \bs{a}(i,j)^\top - 
		\frac{1}{N(N-1)} \sum_{j=1}^N \bs{b}(j) \bs{b}(j)^\top
		\\
		& = 
		\frac{m}{N-1} \frac{(1-f)^2}{m(1-f)}
		\bs{S}^{-1} \sum_{i=1}^N ( \bs{u}_i - \bar{\bs{u}}) ( \bs{u}_i - \bar{\bs{u}})^\top \bs{S}^{-1}
		+ 
		\frac{N-m}{N-1} \frac{f^2}{m(1-f)} \bs{S}^{-1} \sum_{i=1}^N ( \bs{u}_i - \bar{\bs{u}}) ( \bs{u}_i - \bar{\bs{u}})^\top \bs{S}^{-1}
		\\
		& = (1-f) \bs{S}^{-1} \bs{S}^2 \bs{S}^{-1} + f \bs{S}^{-1} \bs{S}^2 \bs{S}^{-1}
		= \bs{I}_d, 
	\end{align*}
	i.e., $\{\bs{a}(i,j): 1\le i,j\le N\}$ satisfies the conditions in \eqref{eq:aij_constraint}. 
	Besides, $\sum_{i=1}^N \sum_{j=1}^N |\bs{a}(i,j)|^3$ simplifies to 
	\begin{align*}
		\sum_{i=1}^N \sum_{j=1}^N \|\bs{a}(i,j)\|_2^3
		& = 
		m 
		\frac{(1-f)^3}{\{m(1-f)\}^{3/2}}
		\sum_{i=1}^N \left\|\bs{S}^{-1} ( \bs{u}_i - \bar{\bs{u}}) \right\|_2^3 + 
		(N-m) 
		\frac{f^3}{\{m(1-f)\}^{3/2}}
		\sum_{i=1}^N \left\| \bs{S}^{-1} ( \bs{u}_i - \bar{\bs{u}}) \right\|_2^3 \\
		& = 
		\frac{(1-f)^2+f^2}{\sqrt{m(1-f)}} \sum_{i=1}^N \left\| \bs{S}^{-1} ( \bs{u}_i - \bar{\bs{u}}) \right\|_2^3
		\le 
		\frac{\{f+(1-f)\}^2}{\sqrt{Nf(1-f)}} \sum_{i=1}^N \left\| \bs{S}^{-1} ( \bs{u}_i - \bar{\bs{u}})\right\|_2^3
		\\
		& = 
		\frac{1}{\sqrt{Nf(1-f)}} \sum_{i=1}^N \left\| \bs{S}^{-1} ( \bs{u}_i - \bar{\bs{u}}) \right\|_2^3
		= \frac{N}{d^{1/4}} \gamma, 
	\end{align*}
	where the last equality holds by definition. 
	
	Let $\pi$ denote a uniformly distributed random permutation of $\{1,2,\ldots,N\}$, 
	and 
	$\tilde{\bs{W}} = \sum_{i=1}^N \bs{a}(i, \pi(i))$. 
	Then, by definition, 
	\begin{align*}
		\tilde{\bs{W}} & 
		= 
		\sum_{i=1}^N \kron(\pi(i)\le m)\{m(1-f)\}^{-1/2} \bs{S}^{-1} (1-f) \bs{u}_i
		-
		\sum_{i=1}^N \kron(\pi(i) > m)
		\{m(1-f)\}^{-1/2} \bs{S}^{-1} f \bs{u}_i\\
		& = 
		\{m(1-f)\}^{-1/2} \bs{S}^{-1} \sum_{i=1}^N
		\left\{  \kron(\pi(i)\le m) - f  \right\} \bs{u}_i
		= 
		\frac{1}{\sqrt{Nf(1-f)}} \bs{S}^{-1}
		\left\{  \sum_{i=1}^N \kron(\pi(i)\le m) \bs{u}_i - m \bar{\bs{u}}  \right\} 
		\\
		& \sim 
		\frac{1}{\sqrt{Nf(1-f)}} \bs{S}^{-1}
		\left\{  \sum_{i=1}^N Z_i \bs{u}_i - m \bar{\bs{u}}  \right\}
		= 
		\bs{W}, 
	\end{align*}
	where the last $\sim$ holds because $( \kron(\pi(1)\le m), \kron(\pi(2)\le m)), \ldots, \kron(\pi(N)\le m) )$ follows the same distribution as $(Z_1, Z_2, \ldots, Z_N)$. 
	Applying Lemma \ref{lemma:perm_clt_fix_dim}, we can know that there exists $C_d$ that depends only on $d$ such that, for any measurable convex set $\mathcal{Q}\subset \mathbb{R}^d$, 
	\begin{align*}
		\left| \PP(\bs{W} \in \mathcal{Q}) - \PP(\bs{\varepsilon} \in \mathcal{Q}) \right|
		\le 
		C_d \frac{1}{N} \sum_{i=1}^N \sum_{j=1}^N \|\bs{a}(i,j)\|_2^3
		= C_d d^{-1/4}  \cdot \gamma.
	\end{align*}
	This immediately implies that Theorem \ref {thm:berry_esseen_clt_full}(i) holds. 
\end{proof}

To prove Theorem \ref {thm:berry_esseen_clt_full}(ii), we need the following lemma from \citet{R15}. However, the author did not provide a formal proof there. 

\begin{lemma}\label{lemma:perm_clt_vary_dim}
	Consider the same setting as in Lemma \ref{lemma:perm_clt_fix_dim}. 
	There exists a universal constant $C$ such that for any $N\ge 2$, any $d\ge 1$, any $\{\bs{a}(i,j): 1\le i,j\le N\}$ satisfying \eqref{eq:aij_constraint}, and 
	any measurable convex set $\mathcal{Q} \subset \mathbb{R}^d$, 
	\begin{align*}
		\left| \PP(\bs{W} \in \mathcal{Q}) - \PP(\bs{\varepsilon} \in \mathcal{Q}) \right|
		\le 
		C d^{1/4} \frac{1}{N} \sum_{i=1}^N \sum_{j=1}^N |\bs{a}(i,j)|^3. 
	\end{align*}
\end{lemma}

\begin{proof}[\bf Proof of Theorem \ref {thm:berry_esseen_clt_full}(ii)]
	Theorem \ref {thm:berry_esseen_clt_full}(ii) follows from Lemma \ref{lemma:perm_clt_vary_dim}, by almost the same logic as the proof of Theorem \ref {thm:berry_esseen_clt_full}(i). Therefore, we omit its proof here. 
\end{proof}

\subsection{Proof of Theorem \ref{thm:berry_esseen_clt_full}(iii) based on H\'ajek coupling}

\subsubsection{Technical lemmas}

To prove Theorem \ref{thm:berry_esseen_clt_full}(iii), we need the following eight lemmas. 

\begin{lemma}\label{lemma:couple_srs_bs}
	Consider a finite population $\{\bs{u}_1, \bs{u}_2, \ldots, \bs{u}_N\}$ for $N$ units, 
	where $\bs{u}_i \in \mathbb{R}^d$ for all $i$. 
	Let $\bar{\bs{u}} = N^{-1} \sum_{i=1}^N \bs{u}_i$ denote the finite population average. 
	There must exist a pair of random vectors $\bs{Z}$ and $\bs{T}$ in $\{0,1\}^N$ such that 
	\begin{itemize}
		\item[(i)] $\bs{Z} = (Z_1, Z_2, \ldots, Z_N)$ is an indicator vector for a simple random sample of size $m$, i.e., the probability that $\bs{Z}$ takes a particular value $\bs{z} = (z_1, \ldots, z_N)\in \{0,1\}^N$ is $m!(N-m)!/N!$ if $\sum_{i=1}^N z_i = m$ and zero otherwise; 
		\item[(ii)] $\bs{T} = (T_1, T_2, \ldots, T_N) \in \{0, 1\}^N$ is an indicator vector for a Bernoulli random sample with equal probability $m/N$ for all units, i.e., $T_i \overset{i.i.d.}{\sim} \text{Bern}(m/N)$;
		\item[(iii)] 
		the covariances for $\bs{A} \equiv \sum_{i=1}^N Z_i \bs{u}_i$, $\bs{B} \equiv \sum_{i=1}^N T_i ( \bs{u}_i - \bar{\bs{u}} ) + m  \bar{\bs{u}}$ and their difference satisfy 
		$\cov(\bs{B}) = 
		(1-N^{-1}) \cdot \cov(\bs{A})$ and 
		\begin{align*}
			\cov^{-1/2}(\bs{B}) \cdot \E \big\{ (\bs{B} - \bs{A}) (\bs{B} - \bs{A})^\top \big\} 
			\cdot 
			\cov^{-1/2}(\bs{B}) 
			\le \sqrt{\frac{1}{m} + \frac{1}{N-m}} \cdot \bs{I}_d. 
		\end{align*}
	\end{itemize}
\end{lemma}

\begin{lemma}[\citet{R19}]\label{lemma:ind_berry}
	Let $\bs{\xi}_1, \ldots, \bs{\xi}_N$ be $N$ independent $d$-dimensional random vectors, satisfying 
	$\E \bs{\xi}_i = \bs{0}$ for all $1\le i \le N$ and 
	$\sum_{i=1}^N \cov(\bs{\xi}_i)= \bs{I}_d$, 
	and $\bs{\varepsilon} \sim \mathcal{N}(\bs{0}, \bs{I}_d)$ be a $d$-dimensional standard Gaussian random vector. 
	Define $\bs{W} = \sum_{i=1}^N \bs{\xi}_i$. 
	Then for any measurable convex set $\mathcal{Q} \subset \mathbb{R}^d$, 
	\begin{align*}
		\left|
		\PP(\bs{W} \in \mathcal{Q}) - \PP(\bs{\varepsilon} \in \mathcal{Q})
		\right|
		\le 
		58 d^{1/4} \sum_{i=1}^N \E \|\bs{\xi}_i\|_2^3. 
	\end{align*}
\end{lemma}

\begin{lemma}\label{lemma:bern_sample_berry}
	Let $\{\bs{u}_1, \bs{u}_2, \ldots, \bs{u}_N\}$ be a finite population of $N$ units, 
	with $\bs{u}_i \in \mathbb{R}^d$ for all $i$, 
	and 
	$\bs{T} = (T_1, T_2, \ldots, T_N) \in \{0, 1\}^N$ be an indicator vector for a Bernoulli random sample with equal probability $f \equiv m/N$ for all units, i.e., $T_i \overset{i.i.d.}{\sim} \text{Bern}(f)$. 
	Define $\bar{\bs{u}} = N^{-1} \sum_{i=1}^N \bs{u}_i$, 
	$\bs{S}^2 = (N-1)^{-1} \sum_{i=1}^N (\bs{u}_i - \bar{\bs{u}})(\bs{u}_i - \bar{\bs{u}})^\top$, 
	and $\bs{B} \equiv \sum_{i=1}^N T_i ( \bs{u}_i - \bar{\bs{u}} ) + m  \bar{\bs{u}}$. 
	Let $\bs{\varepsilon}$ be a $d$-dimensional standard Gaussian random vector. 
	Then for any $N\ge 2$, $d\ge 1$, $f\in (0,1)$, and any finite population $\{\bs{u}_i: 1\le i \le N\}$ with nonsingular finite population covariance $\bs{S}^2$, 
	we have, for any measurable convex set $\mathcal{Q} \subset \mathbb{R}^d$, 
	\begin{align*}
		\big|
		\PP\big\{ \cov^{-1/2}(\bs{B}) \cdot (\bs{B} - \E\bs{B})  \in \mathcal{Q} \big\} - \PP(\bs{\varepsilon} \in \mathcal{Q})
		\big|
		& \le 
		\frac{165}{\sqrt{Nf(1-f)}} \frac{d^{1/4}}{N}
		\sum_{i=1}^N \big\| \bs{S}^{-1} (\bs{u}_i - \bar{\bs{u}}) \big\|_2^3. 
	\end{align*}
\end{lemma}

\begin{lemma}\label{lemma:convex_set}
	Let $\mathcal{Q} \subset \mathbb{R}^d$ be a convex set in $\mathbb{R}^d$.  
	\begin{itemize}
		\item[(i)] For any $c>0$, 
		$\overline{\mathcal{Q}}_c \equiv 
		\{\bs{x} \in \mathbb{R}^d : \exists \bs{x}' \in \mathcal{Q} \ \st \ \|\bs{x} - \bs{x}'\|_2 < c\}$ is a convex set in $\mathbb{R}^d$. 
		\item[(ii)] For any $c>0$, 
		$\underline{\mathcal{Q}}_c \equiv 
		\{\bs{x} \in \mathbb{R}^d: \|\bs{x}' - \bs{x}\|_2 \ge c \ \forall \bs{x}' \not\in \mathcal{Q}\}$ is a convex set in $\mathbb{R}^d$. 
		\item[(iii)] 
		For any matrix $\bs{\Delta} \in \mathbb{R}^{d \times d}$, 
		$\tilde{\mathcal{Q}} \equiv \{\bs{x} \in \mathbb{R}^d: \bs{\Delta} \bs{x} \in \mathcal{Q} \}$ is a convex set in $\mathbb{R}^d$. 
	\end{itemize}
\end{lemma}

\begin{lemma}\label{lemma:set_incre_decre}
	For any set $\mathcal{Q} \subset \mathbb{R}^d$ and any $c>0$, define 
	\begin{align*}
		\overline{\mathcal{Q}}_c \equiv 
		\{\bs{x} \in \mathbb{R}^d : \exists \bs{x}' \in \mathcal{Q} \ \st \ \|\bs{x} - \bs{x}'\|_2 < c\}, 
		\quad 
		\underline{\mathcal{Q}}_c \equiv 
		\{\bs{x} \in \mathbb{R}^d: \|\bs{x}' - \bs{x}\|_2 \ge c \ \forall \bs{x}' \not\in \mathcal{Q}\}. 
	\end{align*}
	\begin{itemize}
		\item[(i)] For any set $\mathcal{Q} \subset \mathbb{R}^d$ and any positive $c, h$, 
		$\overline{(\overline{\mathcal{Q}}_c)}_h = \overline{\mathcal{Q}}_{c+h}$. 
		\item[(ii)] For any set $\mathcal{Q} \subset \mathbb{R}^d$ and any $c>0$, 
		$(\underline{\mathcal{Q}}_c)^\complement = \overline{\mathcal{B}}_c$, where $\mathcal{B} = \mathcal{Q}^\complement$.
		\item[(iii)] For any set $\mathcal{Q} \subset \mathbb{R}^d$ and any positive $c, h$, 
		$\underline{(\underline{\mathcal{Q}}_c)}_h = \underline{\mathcal{Q}}_{c+h}$. 
	\end{itemize}
\end{lemma}

\begin{lemma}\label{lemma:maximum_surface}
	Let $\bs{\varepsilon}$ be a $d$-dimensional standard Gaussian random variable, 
	$\phi_d(\cdot)$ be the probability density function of $\bs{\varepsilon}$, 
	and $\cC_d$ be the collection of convex sets in $\R^d$.
	We have that
	\begin{equation*}%
		\sup_{c > 0, \mathcal{Q} \in \cC_d} \frac{\int_{\overline{\mathcal{Q}}_c \setminus \mathcal{Q}} \phi_d (\bs{\varepsilon}) \text{d} \bs{\varepsilon}}{c} \leq 4 d^{\frac{1}{4}}
	\end{equation*}
	and 
	\begin{equation*}%
		\sup_{c > 0, \mathcal{Q} \in \cC_d} \frac{\int_{\mathcal{Q} \setminus \underline{\mathcal{Q}}_{c} } \phi_d (\bs{\varepsilon}) \text{d} \bs{\varepsilon}}{c} \leq 4 d^{\frac{1}{4}}. 
	\end{equation*}
\end{lemma}

\begin{lemma}\label{lemma:berry_bound_coupling}
	Let $\bs{B}$ and $\bs{A}$ be two $d$-dimensional random vectors with equal means $\E \bs{B} = \E \bs{A}$
	and nonsingular covariance matrices $\bs{\Sigma}_{\bs{B}}$ and $\bs{\Sigma}_{\bs{A}}$, 
	and $\bs{\varepsilon}$ be a $d$-dimensional standard Gaussian random vector.
	Let $\cC_d$ denote the collection of convex sets in $\R^d$.
	If 
	$\bs{\Sigma}_{\bs{B}} = (1-l)^2 \bs{\Sigma}_{\bs{A}}$ for some $l\in (0,1)$, 
	\begin{align*}
		\sup_{\mathcal{Q} \in \cC_d} 
		\big|
		\PP 
		\big( 
		\bs{\Sigma}_{\bs{B}}^{-1/2} (\bs{B} - \E\bs{B}) \in \mathcal{Q} 
		\big) - \PP(\bs{\varepsilon} \in \mathcal{Q})
		\big| \leq a \quad \text{for some finite $a>0$}
	\end{align*}
	and 
	\begin{align*}
		\bs{\Sigma}_{\bs{B}}^{-1/2} \cdot \E \big\{ (\bs{B} - \bs{A}) (\bs{B} - \bs{A})^\top \big\} 
		\cdot
		\bs{\Sigma}_{\bs{B}}^{-1/2}
		\le b^2 \bs{I}_d 
		\quad \text{for some $b\in (0,1)$,}
	\end{align*}
	then for any positive constants $c$ and $h$, 
	\begin{align*}
		& \quad \ \sup_{\mathcal{Q} \in \cC_d} 
		\big|
		\PP \big( \bs{\Sigma}_{\bs{A}}^{-1/2} (\bs{A} - \E \bs{A} ) \in \mathcal{Q} \big) - \PP(\bs{\varepsilon} \in \mathcal{Q})
		\big| \\
		& \leq 
		4 d^{1/4} (c+h) + 2 d \cdot \exp\left( - \frac{h^2}{2 d l^2}  \right) + a + 
		\PP\big( \| \bs{\Sigma}_{\bs{A}}^{-1/2}(\bs{A} - \bs{B}) \|_2 \ge c \big)
		\\
		& \leq  
		4 d^{1/4} (c+h) + 2 d \cdot \exp\left( - \frac{h^2}{2 d l^2}  \right) + a + \frac{(1-l)^2 b^2 d}{c^2}.
	\end{align*}
\end{lemma}

\begin{lemma}\label{lemma:gamma_srs_bound}
	Under the same setting as in Theorem \ref{thm:berry_esseen_clt_full}, 
	\begin{itemize}
		\item[(i)] $\sum_{i=1}^N \| \bs{S}^{-1} (\bs{u}_i - \bar{\bs{u}}) \|_2^2 = (N-1) d$; 
		\item[(ii)] $N^{-1} \sum_{i=1}^N \| \bs{S}^{-1} (\bs{u}_i - \bar{\bs{u}}) \|_2^3 \ge (d/2)^{3/2}$; 
		\item[(iii)] $\gamma$ defined in Theorem \ref{thm:berry_esseen_clt_full} satisfies that 
		\begin{align*}
			\gamma \equiv \frac{1}{\sqrt{Nf(1-f)}} \frac{d^{1/4}}{N}
			\sum_{i=1}^N \big\| \bs{S}^{-1} (\bs{u}_i - \bar{\bs{u}}) \big\|_2^3
			\ge 
			\frac{d^{7/4}}{2^{3/2}\sqrt{Nf(1-f)}}. 
		\end{align*}
	\end{itemize}

\end{lemma}

\subsubsection{Proofs of the lemmas}

\begin{proof}[Proof of Lemma \ref{lemma:couple_srs_bs}]
	Let $\bs{T} = (T_1, T_2, \ldots, T_N) \in \{0, 1\}^N$ be an indicator vector for a Bernoulli random sample with equal probability $n/N$ for all units. 
	Below we construct the indicator vector $\bs{Z}$ for a simple random sample of size $n$ based on $\bs{T}$. 
	We consider the following three different cases depending on the size $\tilde{m}$ of the set $\mathcal{T} \equiv\{i: T_i = 1, 1\le i \le N\}$, i.e., $\tilde{m} \equiv |\mathcal{T}|$. 
	\begin{itemize}
		\item[(i)] If $\tilde{m} = m$ by accident, we define $\bs{Z} = \bs{T}$; 
		\item[(ii)] If $\tilde{m} > m$, we introduce a set $\mathcal{D}$ to be a simple random sample of size $\tilde{m} - m$ from $\mathcal{T}$, and define $Z_i = 1$ if $i \in \mathcal{T}\setminus \mathcal{D}$ and 0 otherwise;
		\item[(iii)] If $\tilde{m} < m$, we introduce a set $\mathcal{D}$ to be a simple random sample of size $m-\tilde{m}$ from $\{1,2,\ldots,N\}\setminus \mathcal{T}$, and define $Z_i = 1$ if $i \in \mathcal{T} \cup \mathcal{D}$ and 0 otherwise.
	\end{itemize}
	We can verify
	that $\bs{Z}$ must be an indicator vector for a simple random sample of size $m$.  
	This is essentially the coupling between simple random sampling and Bernoulli random sampling used in \citet{Hajek60}.

	By definition, $\bs{B} - \bs{A} = \sum_{i=1}^N (T_i - Z_i) (\bs{u}_i - \bar{\bs{u}})$. 
	By the construction of $\bs{T}$ and $\bs{Z}$, 
	conditioning on $\tilde{m}$, 
	the difference between $\bs{B}$ and $\bs{A}$ is essentially the summation of a simple random sample of size 
	\begin{align}\label{eq:couple_srs}
	    \begin{cases}
	        \tilde{m} - m \text{ from the population } \{\bs{u}_1 - \bar{\bs{u}}, \bs{u}_2 - \bar{\bs{u}}, \ldots, \bs{u}_N - \bar{\bs{u}}\}, & \text{if } \tilde{m} \ge m;
	        \\
	        m - \tilde{m} \text{ from the population } \{-(\bs{u}_1 - \bar{\bs{u}}), -(\bs{u}_2 - \bar{\bs{u}}), \ldots, -(\bs{u}_N - \bar{\bs{u}})\}, & \text{if } \tilde{m} <  m. 
	    \end{cases}
	\end{align}
	Let 
	$\Delta = |\tilde{m} - m|$.
	By the property of simple random sampling, this difference satisfies 
	$\E(\bs{B} - \bs{A} \mid \tilde{m}) = \bs{0}$ and 
	$$
	\cov (\bs{B} - \bs{A} \mid \tilde{m}) = \frac{\Delta }{N} \cdot \frac{N-\Delta}{N-1} \sum_{i=1}^N (\bs{u}_i - \bar{\bs{u}} )(\bs{u}_i - \bar{\bs{u}} )^\top 
	\le  \frac{\Delta }{N} \sum_{i=1}^N (\bs{u}_i - \bar{\bs{u}} )(\bs{u}_i - \bar{\bs{u}})^\top. 
	$$
	By the law of total expectation and total variance, and the fact that $\{\E(\Delta)\}^2 \le \E(\Delta^2) = \var(\tilde{m})$, we have
	$\E(\bs{B} - \bs{A}) = \E\{ \E(\bs{B} - \bs{A} \mid \tilde{m}) \} = \bs{0}$, and 
	\begin{align}\label{eq:bound_cov_diff_B_S}
		\cov(\bs{B} - \bs{A}) 
		& = 
		\E \left\{ \cov(\bs{B} - \bs{A} \mid \tilde{m}) \right\}+ \cov\left\{ \E (\bs{B} - \bs{A} \mid \tilde{m} ) \right\}
		\leq  \frac{\E(\Delta) }{N} \sum_{i=1}^N (\bs{u}_i - \bar{\bs{u}} )(\bs{u}_i - \bar{\bs{u}} )^\top
		\nonumber
		\\
		& \leq   
		\frac{ \sqrt{ \var(\tilde{m}) }  }{N} \sum_{i=1}^N (\bs{u}_i - \bar{\bs{u}} )(\bs{u}_i - \bar{\bs{u}} )^\top 
			\nonumber
			\\
		& = \sqrt{  N \frac{m}{N} \left(  1 - \frac{m}{N} \right) }   \frac{ 1   }{N}  \sum_{i=1}^N (\bs{u}_i - \bar{\bs{u}} )(\bs{u}_i - \bar{\bs{u}} )^\top. 
	\end{align}
	By the property of bernoulli sampling and simple random sampling, 
	we can derive that 
	\begin{align}\label{eq:cov_B}
		\cov(\bs{B}) & = \frac{m}{N} \left(  1 - \frac{m}{N} \right) \sum_{i=1}^N (\bs{u}_i - \bar{\bs{u}} )(\bs{u}_i - \bar{\bs{u}} )^\top. 
	\end{align}
	and 
	\begin{align*}
		\cov(\bs{A}) & = \frac{m(N-m)}{N(N-1)} \sum_{i=1}^N (\bs{u}_i - \bar{\bs{u}} )(\bs{u}_i - \bar{\bs{u}} )^\top
		= 
		\frac{N}{N-1} \cov(\bs{B}). 
	\end{align*}
	\eqref{eq:bound_cov_diff_B_S},  \eqref{eq:cov_B} and the fact that $\E(\bs{B} - \bs{A}) = \bs{0}$
	immediately imply that 
	\begin{align*}
		\cov^{-1/2}(\bs{B}) \cdot 
		\E \big\{ (\bs{B} - \bs{A}) (\bs{B} - \bs{A})^\top \big\} 
		\cdot \cov^{-1/2}(\bs{B})
		& = \cov^{-1/2}(\bs{B}) \cdot \cov(\bs{B} - \bs{A}) \cdot \cov^{-1/2}(\bs{B})
		\\
		& \le \sqrt{\frac{1}{m} + \frac{1}{N-m}} \cdot \bs{I}_d. 
	\end{align*}
	From the above, Lemma \ref{lemma:couple_srs_bs} holds. 
\end{proof}

\begin{proof}[Proof of Lemma \ref{lemma:ind_berry}]
	Lemma \ref{lemma:ind_berry} follows immediately from \citet[][Theorem 1.1]{R19}. 
\end{proof}

\begin{proof}[Proof of Lemma \ref{lemma:bern_sample_berry}]
	By definition, we can derive that 
	$$
	\cov^{-1/2}(\bs{B}) \cdot (\bs{B} - \E\bs{B})
	= \cov^{-1/2}(\bs{B}) \cdot \sum_{i=1}^N T_i ( \bs{u}_i - \bar{\bs{u}} )
	= \cov^{-1/2}(\bs{B}) \sum_{i=1}^N \big( T_i - f \big) (\bs{u}_i - \bar{\bs{u}}), 
	$$
	where the last equality holds due to the centering of the $\bs{u}_i$'s. 
	Define $\bs{\xi}_i = (T_i-f) \cov^{-1/2}(\bs{B}) \cdot (\bs{u}_i - \bar{\bs{u}})$.  
	We can verify that $\bs{\xi}_i$'s satisfy the condition in Lemma \ref{lemma:ind_berry}. 
	Thus, from Lemma \ref{lemma:ind_berry}, for any measurable convex set $\mathcal{Q}\subset \mathbb{R}^d$, 
	\begin{align}\label{eq:bern_berry1}
		\big|
		\PP\big\{ \cov^{-1/2}(\bs{B}) \cdot (\bs{B} - \E\bs{B})  \in \mathcal{Q} \big\} - \PP(\bs{\varepsilon} \in \mathcal{Q})
		\big|
		& 
		\le 58 d^{1/4} \sum_{i=1}^N \E \|\bs{\xi}_i\|_2^3. 
	\end{align}
	By definition, 
	$
	\E \{ |T_i - f|^3 \} = 
	f(1-f)\{f^2 + (1-f)^2 \}.  
	$
	From \eqref{eq:cov_B}, $\cov(\bs{B}) = f(1-f) (N-1) \bs{S}^2$. 
	We can then simplify $\sum_{i=1}^N \E \|\bs{\xi}_i\|_2^3$ as
	\begin{align*}%
		\sum_{i=1}^N \E \|\bs{\xi}_i\|_2^3
		& = 
		\sum_{i=1}^N \frac{\E \big\{ |T_i - f|^3 \big\}}{\{f(1-f) (N-1) \}^{3/2}}
		\big\| \bs{S}^{-1} (\bs{u}_i - \bar{\bs{u}}) \big\|_2^3
		= 
		\frac{f^2 + (1-f)^2}{(N-1)^{3/2} \sqrt{f(1-f)}}
		\sum_{i=1}^N \big\| \bs{S}^{-1} (\bs{u}_i - \bar{\bs{u}}) \big\|_2^3
		\nonumber
		\\
		& \le 
		\frac{2^{3/2}}{\sqrt{Nf(1-f)}} \frac{1}{N}
		\sum_{i=1}^N \big\| \bs{S}^{-1} (\bs{u}_i - \bar{\bs{u}}) \big\|_2^3, 
	\end{align*}
	where the last inequality holds becomes 
	$f^2 + (1-f)^2 \le \{f+(1-f)\}^2 = 1$ 
	and 
	$N-1 \ge N/2$. 
	From \eqref{eq:bern_berry1}, we then have, for any measurable convex set $\mathcal{Q}\subset \mathbb{R}^d$, 
	\begin{align*}
		\big|
		\PP\big\{ \cov^{-1/2}(\bs{B}) \cdot (\bs{B} - \E\bs{B})  \in \mathcal{Q} \big\} - \PP(\bs{\varepsilon} \in \mathcal{Q})
		\big|
		& 
		\le 58 d^{1/4}  \frac{2^{3/2}}{\sqrt{Nf(1-f)}} \frac{1}{N}
		\sum_{i=1}^N \big\| \bs{S}^{-1} (\bs{u}_i - \bar{\bs{u}}) \big\|_2^3
		\\
		& \le 
		\frac{165}{\sqrt{Nf(1-f)}} \frac{d^{1/4}}{N}
		\sum_{i=1}^N \big\| \bs{S}^{-1} (\bs{u}_i - \bar{\bs{u}}) \big\|_2^3
	\end{align*}
	Therefore, Lemma \ref{lemma:bern_sample_berry} holds. 
\end{proof}

\begin{proof}[Proof of Lemma \ref{lemma:convex_set}]
	We first prove (i). 
	Consider any $\bs{x}, \bs{y} \in \overline{\mathcal{Q}}_c$ and any $\lambda \in (0,1)$. 
	By definition, there must exist $\bs{x}', \bs{y}' \in \mathcal{Q}$ such that $\|\bs{x} - \bs{x}'\|_2 < c$ and $\|\bs{y} - \bs{y}'\|_2 < c$. 
	Because $\mathcal{Q}$ is convex, $\lambda \bs{x}' + (1-\lambda) \bs{y}' \in \mathcal{Q}$. 
	Moreover, by the triangle inequality, 
	\begin{align*}
		\| \lambda \bs{x} + (1-\lambda) \bs{y} - \{\lambda \bs{x}' + (1-\lambda) \bs{y}'\} \|_2 
		\le 
		\lambda \|\bs{x} - \bs{x}'\|_2 + (1-\lambda) \|\bs{y} - \bs{y}'\|_2 < c. 
	\end{align*}
	Thus, we must have $\lambda \bs{x} + (1-\lambda) \bs{y} \in \overline{\mathcal{Q}}_c$. 
	Therefore, $\overline{\mathcal{Q}}_c$ must be a convex set. 
	
	We then prove (ii). 
	Consider any $\bs{x}, \bs{y} \in \underline{\mathcal{Q}}_c$ and any $\lambda \in (0,1)$. 
	We prove that $\bs{z}\equiv \lambda \bs{x} + (1-\lambda) \bs{y} \in \underline{\mathcal{Q}}_c$ by contradiction. 
	Suppose that $\bs{z} \notin \underline{\mathcal{Q}}_c$. By definition, there must exist $\bs{z}' \notin \mathcal{Q}$ such that $\|\bs{z} - \bs{z}'\|_2 < c$. 
	Define $\bs{x}' = \bs{x} + \bs{z}' - \bs{z}$ and $\bs{y}' = \bs{y} + \bs{z}' - \bs{z}$. 
	Because $\|\bs{x} - \bs{x}'\|_2 = \|\bs{y} - \bs{y}'\|_2 = \|\bs{z} - \bs{z}'\|_2 < c$ and $\bs{x}, \bs{y} \in \underline{\mathcal{Q}}_c$, by definition, 
	we must have $\bs{x}' , \bs{y}' \in \mathcal{Q}$. 
	Due to the convexity of $A$, this further implies that $\lambda \bs{x}' + (1-\lambda) \bs{y}' \in \mathcal{Q}$. 
	By some algebra, we can show that $\bs{z}' = \lambda \bs{x}' + (1-\lambda) \bs{y}' \in \mathcal{Q}$, which contradicts with $\bs{z}' \notin \mathcal{Q}$. 
	Therefore, we must have $\bs{z} =  \lambda \bs{x} + (1-\lambda) \bs{y} \in \underline{\mathcal{Q}}_c$. 
	Consequently, $\underline{\mathcal{Q}}_c$ is a convex set. 
	
	Finally, we prove (iii). 
	Consider any $\bs{x}, \bs{y} \in \tilde{\mathcal{Q}}$ and any $\lambda \in (0,1)$. 
	By definition, $\bs{\Delta} \bs{x}, \bs{\Delta} \bs{y}  \in \mathcal{Q}$. 
	By the convexity of  $\mathcal{Q}$, this implies that 
	$
	\bs{\Delta} \{\lambda \bs{x} + (1-\lambda) \bs{y}\}
	= 
	\lambda \bs{\Delta} \bs{x} + (1-\lambda) \bs{\Delta} \bs{y} \in \mathcal{Q}.
	$
	Consequently, $\lambda \bs{x} + (1-\lambda) \bs{y} \in \tilde{\mathcal{Q}}$. 
	Therefore, $\tilde{\mathcal{Q}}$ must be a convex set.
	
	From the above, Lemma \ref{lemma:convex_set} holds. 
\end{proof}

\begin{proof}[Proof of Lemma \ref{lemma:set_incre_decre}]
	We first prove (i). 
	We first prove $\overline{(\overline{\mathcal{Q}}_c)}_h \subset \overline{\mathcal{Q}}_{c+h}$. 
	For any $\bs{x} \in \overline{(\overline{\mathcal{Q}}_c)}_h$, by definition, 
	there exists $\bs{x}' \in \overline{\mathcal{Q}}_c$ such that $\|\bs{x} - \bs{x}'\|_2 < h$. 
	By the same logic, there exists $\bs{x}'' \in \mathcal{Q}$ such that $\|\bs{x}' - \bs{x}''\|_2 < c$. 
	By the triangle inequality,
	$\|\bs{x} - \bs{x}''\|_2 \le \|\bs{x} - \bs{x}'\|_2 + \|\bs{x}' - \bs{x}''\|_2 < c+h$, 
	which then implies that $\bs{x} \in \overline{\mathcal{Q}}_{c+h}$. 
	Therefore, we must have $\overline{(\overline{\mathcal{Q}}_c)}_h \subset \overline{\mathcal{Q}}_{c+h}$. 
	We then prove $\overline{(\overline{\mathcal{Q}}_c)}_h \supset \overline{\mathcal{Q}}_{c+h}$. For any $\bs{x} \in \overline{\mathcal{Q}}_{c+h}$, 
	by definition, there exists $\bs{x}' \in \mathcal{Q}$ such that $\|\bs{x} - \bs{x}'\|_2 < c+h$. 
	Let $\lambda = c/(c+h)$, and $\bs{x}'' = \bs{x}' + \lambda (\bs{x} - \bs{x}')$. 
	We then have 
	$\|\bs{x}' - \bs{x}''\|_2 = \lambda \|\bs{x} - \bs{x}'\|_2 < c$, 
	and 
	$\|\bs{x}'' - \bs{x}\|_2 = (1-\lambda) \|\bs{x} - \bs{x}'\|_2 < h$. 
	Consequently, $\bs{x}'' \in \overline{\mathcal{Q}}_c$, and $\bs{x} \in \overline{(\overline{\mathcal{Q}}_c)}_h$. 
	Therefore, we must have $\overline{\mathcal{Q}}_{c+h} \subset \overline{(\overline{\mathcal{Q}}_c)}_h$. 
	From the above, we have $\overline{(\overline{\mathcal{Q}}_c)}_h = \overline{\mathcal{Q}}_{c+h}$. 
	
	We then prove (ii). By definition, 
	\begin{align*}
		\big( \underline{\mathcal{Q}}_c \big)^\complement & = 
		\{\bs{x} \in \mathbb{R}^d: \|\bs{x}' - \bs{x}\|_2 \ge c \ \forall \bs{x}' \not\in \mathcal{Q}\}^\complement
		= \{\bs{x} \in \mathbb{R}^d: \|\bs{x}' - \bs{x}\|_2 \ge c \ \forall \bs{x}' \in \mathcal{B}\}^\complement\\
		& = \{\bs{x} \in \mathbb{R}^d : \exists \bs{x}' \in \mathcal{B} \ \st \ \|\bs{x} - \bs{x}'\|_2 < c\}
		= \overline{\mathcal{B}}_c. 
	\end{align*}
	
	Finally, we prove (iii) using (i) and (ii).  
	From (ii), we have 
	$\underline{(\underline{\mathcal{Q}}_c)}_h = ( \overline{\mathcal{D}}_h )^\complement$, 
	where $\mathcal{D} = (\underline{\mathcal{Q}}_c)^\complement$. 
	By the same logic, $\mathcal{D} = (\underline{\mathcal{Q}}_c)^\complement = \overline{\mathcal{B}}_c$, where $\mathcal{B} = \mathcal{Q}^\complement$. 
	Consequently, using (i) and (ii), 
	we have 
	\begin{align*}
		\underline{(\underline{\mathcal{Q}}_c)}_h
		& = 
		( \overline{\mathcal{D}}_h )^\complement
		= \left( \overline{\left(\overline{\mathcal{B}}_c \right)}_h \right)^\complement
		= \left( \overline{\mathcal{B}}_{c+h} \right)^\complement
		= \underline{\mathcal{Q}}_{c+h}. 
	\end{align*}
	
	From the above, Lemma \ref{lemma:set_incre_decre} holds. 
\end{proof}

\begin{proof}[Proof of Lemma \ref{lemma:maximum_surface}]
	This is a direct consequence of (1.3)-(1.4) of~\citet{B05}. See also~\citep{B93,N03}
\end{proof}

\begin{proof}[Proof of Lemma \ref{lemma:berry_bound_coupling}]
	Let $\bs{\zeta} \equiv \bs{\Sigma}_{\bs{B}}^{-1/2}(\bs{A} - \bs{B})$ and $\bs{\Delta} \equiv \bs{\Sigma}_{\bs{A}}^{-1/2} \bs{\Sigma}_{\bs{B}}^{1/2} = (1-l) \bs{I}_d$. 
	Then, 
	by definition, 
	$\E(\bs{\zeta}\bs{\zeta}^\top) \le b^2 \bs{I}_d$,  and 
	\begin{align*}
		\bs{\Sigma}_{\bs{A}}^{-1/2} (\bs{A} - \E \bs{A} )
		& = 
		\bs{\Delta}
		\bs{\Sigma}_{\bs{B}}^{-1/2} \{ \bs{B} - \E \bs{B} + (\bs{A} - \bs{B} ) \}
		= 
		\bs{\Delta}
		\bs{\Sigma}_{\bs{B}}^{-1/2} (\bs{B} - \E\bs{B})
		+ 
		\bs{\Delta} \bs{\zeta}. 
	\end{align*}
	
	First, 
	for any convex set $\mathcal{Q}\subset \mathbb{R}^d$ and any $c>0$, define 
	\begin{align*}
		\overline{\mathcal{Q}}_c \equiv 
		\{\bs{x} \in \mathbb{R}^d : \exists \bs{x}' \in \mathcal{Q} \ \st \ \|\bs{x} - \bs{x}'\|_2 < c\}
		\quad \text{and} \quad 
		\underline{\mathcal{Q}}_c \equiv 
		\{\bs{x} \in \mathbb{R}^d: \|\bs{x}' - \bs{x}\|_2 \ge c \ \forall \bs{x}' \not\in \mathcal{Q}\}.
	\end{align*}
	Intuitively, 
	$\overline{\mathcal{Q}}(c)$ contains all the points whose distance from $\mathcal{Q}$ is at most $c$, and $\overline{\mathcal{Q}}(c)$ contains all the points whose distance from $\mathcal{Q}^\complement$ is at least $c$. 
	Then, by definition, 
	\begin{align}\label{eq:S_upper_bound1}
		\PP \big\{ \bs{\Sigma}_{\bs{A}}^{-1/2} (\bs{A} - \E \bs{A} ) \in \mathcal{Q} \big\}
		& \le 
		\PP \big\{ \bs{\Sigma}_{\bs{A}}^{-1/2} (\bs{A} - \E \bs{A} ) \in \mathcal{Q},  \| \bs{\Delta} \bs{\zeta} \|_2 < c \big\}
		+ 
		\PP\big( \| \bs{\Delta} \bs{\zeta} \|_2 \ge c \big)
		\nonumber
		\\
		& \le 
		\PP \big\{
		\bs{\Delta}
		\bs{\Sigma}_{\bs{B}}^{-1/2} (\bs{B} - \E\bs{B}) \in \overline{\mathcal{Q}}_c
		\big\} + \PP\big( \| \bs{\Delta} \bs{\zeta} \|_2 \ge c \big), 
	\end{align}
	and 
	\begin{align}\label{eq:S_upper_bound2}
		\PP \big\{ \bs{\Sigma}_{\bs{A}}^{-1/2} (\bs{A} - \E \bs{A} ) \in \mathcal{Q} \big\}
		& \ge 
		\PP \big\{ \bs{\Sigma}_{\bs{A}}^{-1/2} (\bs{A} - \E \bs{A} ) \in \mathcal{Q}, \| \bs{\Delta} \bs{\zeta} \|_2 < c \big\}
		\nonumber
		\\
		& \ge 
		\PP \big\{ \bs{\Delta} 
		\bs{\Sigma}_{\bs{B}}^{-1/2} (\bs{B} - \E\bs{B}) \in \underline{\mathcal{Q}}_c, \| \bs{\Delta} \bs{\zeta} \|_2 < c \big\}
		\nonumber
		\\
		& \ge 
		\PP \big\{ \bs{\Delta}
		\bs{\Sigma}_{\bs{B}}^{-1/2} (\bs{B} - \E\bs{B}) \in \underline{\mathcal{Q}}_{c} \big\}
		- \PP\big( \| \bs{\Delta} \bs{\zeta} \|_2 \ge c \big). 
	\end{align}
	From Lemma \ref{lemma:convex_set} and the condition in Lemma \ref{lemma:berry_bound_coupling}, these imply that 
	\begin{align*}
		\PP \big\{ \bs{\Sigma}_{\bs{A}}^{-1/2} (\bs{A} - \E \bs{A} ) \in \mathcal{Q} \big\}
		& \le \PP \big(
		\bs{\Delta} 
		\bs{\varepsilon} \in \overline{\mathcal{Q}}_c
		\big) + a + \PP\big( \| \bs{\Delta} \bs{\zeta} \|_2 \ge c \big),
	\end{align*}
	and 
	\begin{align*}
		\PP \big\{ \bs{\Sigma}_{\bs{A}}^{-1/2} (\bs{A} - \E \bs{A} ) \in \mathcal{Q} \big\}
		& \ge 
		\PP \big( \bs{\Delta} \bs{\varepsilon} \in \underline{\mathcal{Q}}_c \big) - a - \PP\big( \| \bs{\Delta} \bs{\zeta} \|_2 \ge c \big).
	\end{align*}
	
	Second, by definition,  
	$
	\big\| \big( \bs{\Delta} - \bs{I}_d \big) \bs{\varepsilon} \big\|_2 
	= 
	l \big\| \bs{\varepsilon} \big\|_2. 
	$
	By the Gaussian tail bound, for any $h>0$, 
	\begin{align*}
		\PP\big\{
		\big\| \big( \bs{\Delta}_N - \bs{I}_d \big) \bs{\varepsilon} \big\|_2 
		\ge h
		\big\}
		& \le 
		\PP\left\{
		l
		\| \bs{\varepsilon} \|_2
		\ge h \right\}
		= 
		\PP\left(
		\sum_{k=1}^d \varepsilon_k^2 
		\ge \frac{h^2}{l^2} 
		\right)
		\le 
		\sum_{k=1}^d \PP\left(
		\varepsilon_k^2 
		\ge \frac{h^2 }{d \cdot l^2} 
		\right)
		\\
		& = 2 d \cdot
		\PP\left(
		\varepsilon_k
		\ge \frac{h}{d^{1/2} \cdot l} 
		\right)
		\le 
		2 d \cdot \exp\left( - \frac{h^2}{2 d l^2}  \right). 
	\end{align*}
	By the same logic as \eqref{eq:S_upper_bound1} and using Lemma \ref{lemma:set_incre_decre}, 
	\begin{align*}
		\PP \big(
		\bs{\Delta}
		\bs{\varepsilon} \in \overline{\mathcal{Q}}_c \big)
		& \le 
		\PP \big(
		\bs{\Delta} 
		\bs{\varepsilon} \in \overline{\mathcal{Q}}_c, \big\| \big( \bs{\Delta} - \bs{I}_d \big) \bs{\varepsilon} \big\|_2 
		< h \big) + 
		\PP \big(
		\big\| \big( \bs{\Delta} - \bs{I}_d \big) \bs{\varepsilon} \big\|_2 
		\ge h \big)
		\\
		& \le \PP \big(
		\bs{\varepsilon} \in \overline{\mathcal{Q}}_{c+h} \big) + 
		2 d \cdot \exp\left( - \frac{h^2}{2 d l^2}  \right), 
	\end{align*}
	and by the same logic as \eqref{eq:S_upper_bound2} and using Lemma \ref{lemma:set_incre_decre}, 
	\begin{align*}
		\PP \big( \bs{\Delta}_N \bs{\varepsilon} \in \underline{\mathcal{Q}}_{c} \big) 
		& \ge 
		\PP \big( \bs{\Delta} \bs{\varepsilon} \in \underline{\mathcal{Q}}_{c}, 
		\big\| \big( \bs{\Delta} - \bs{I}_d \big) \bs{\varepsilon} \big\|_2 
		< h
		\big)
		\ge 
		\PP \big(  \bs{\varepsilon} \in \underline{\mathcal{Q}}_{c+h}, 
		\big\| \big( \bs{\Delta} - \bs{I}_d \big) \bs{\varepsilon} \big\|_2 
		< h
		\big)\\
		& \ge 
		\PP \big(  \bs{\varepsilon} \in \underline{\mathcal{Q}}_{c+h}
		\big)
		- 
		\PP \big(
		\big\| \big( \bs{\Delta} - \bs{I}_d \big) \bs{\varepsilon} \big\|_2 
		\ge h \big)
		\\
		& \ge 
		\PP \big(  \bs{\varepsilon} \in \underline{\mathcal{Q}}_{c+h}
		\big)
		- 2 d \cdot \exp\left( - \frac{h^2}{2 d l^2}  \right). 
	\end{align*}
	These imply that 
	\begin{align*}
		\PP \big\{ \bs{\Sigma}_{\bs{A}}^{-1/2} (\bs{A} - \E \bs{A} ) \in \mathcal{Q} \big\}
		& \le \PP \big(
		\bs{\Delta}
		\bs{\varepsilon} \in \overline{\mathcal{Q}}_{c}
		\big) + a + \frac{b^2 d}{c^2(1-b)^2}
		\\
		& \le \PP \big(
		\bs{\varepsilon} \in \overline{\mathcal{Q}}_{c+h} \big) + 
		2 d \cdot \exp\left( - \frac{h^2}{2 d l^2}  \right) + a + \PP\big( \| \bs{\Delta} \bs{\zeta} \|_2 \ge c \big), 
	\end{align*}
	and 
	\begin{align*}
		\PP \big\{ \bs{\Sigma}_{\bs{A}}^{-1/2} (\bs{A} - \E \bs{A} ) \in \mathcal{Q} \big\}
		& \ge 
		\PP \big( \bs{\Delta} \bs{\varepsilon} \in \underline{\mathcal{Q}}_{c} \big) - a - \frac{b^2 d}{c^2(1-b)^2}
		\\
		& \ge 
		\PP \big(  \bs{\varepsilon} \in \underline{\mathcal{Q}}_{c+h}
		\big)
		- 2 d \cdot \exp\left( - \frac{h^2}{2 d l^2}  \right) - a - \PP\big( \| \bs{\Delta} \bs{\zeta} \|_2 \ge c \big). 
	\end{align*}
	
	Third, from Lemma \ref{lemma:maximum_surface}, we have 
	\begin{align*}
		\PP \big( \bs{\varepsilon} \in \overline{\mathcal{Q}}_{c+h} \big) \le 
		\PP \big( \bs{\varepsilon} \in \mathcal{Q} \big) + 4 d^{1/4} (c+h)
		\ \ 
		\text{and} 
		\ \ 
		\PP \big(  \bs{\varepsilon} \in \underline{\mathcal{Q}}_{c+h}
		\big) 
		\ge \PP \big(  \bs{\varepsilon} \in \mathcal{Q}
		\big) - 4 d^{1/4} (c+h). 
	\end{align*}
	
	From the above, we must have that, for any $\mathcal{Q}\in \cC_d$, 
	\begin{align*}
		& \quad \ \big| \PP \big\{ \bs{\Sigma}_{\bs{A}}^{-1/2} (\bs{A} - \E \bs{A} ) \in \mathcal{Q} \big\} - 
		\PP \big( \bs{\varepsilon} \in \mathcal{Q} \big) \big|
		\\
		& \le 4 d^{1/4} (c+h) + 2 d \cdot \exp\left( - \frac{h^2}{2 d l^2}  \right) + a + \PP\big( \| \bs{\Delta} \bs{\zeta} \|_2 \ge c \big).
	\end{align*}
    By Chebyshev's inequality, 
    we can bound the tail probability of $\bs{\Delta} \bs{\zeta} = \bs{\Sigma}_{\bs{A}}^{-1/2}(\bs{A} - \bs{B})$ by 
	\begin{align*}%
		\PP( \| \bs{\Delta} \bs{\zeta} \|_2 \ge c)
		& 
		= 
		\PP\left\{ (1-l) \| \bs{\zeta} \|_2 \ge c\right\}
		\le 
		\frac{(1-l)^2}{c^2}
		\E 
		\big(
		\bs{\zeta}^\top \bs{\zeta}
		\big)
		= \frac{(1-l)^2}{c^2}
		\tr\big\{ \E 
		\big(
		\bs{\zeta} \bs{\zeta}^\top 
		\big) \big\}
		\nonumber
		\\
		& \le \frac{(1-l)^2 b^2 d}{c^2}. 
	\end{align*}
	Therefore, Lemma \ref{lemma:berry_bound_coupling} holds. 
\end{proof}

\begin{proof}[Proof of Lemma \ref{lemma:gamma_srs_bound}]
	By definition and some algebra, we can verify that 
	\begin{align*}
		\sum_{i=1}^N \big\| \bs{S}^{-1} (\bs{u}_i - \bar{\bs{u}}) \big\|_2^2
		& =  \sum_{i=1}^N (\bs{u}_i - \bar{\bs{u}})^\top \bs{S}^{-2} (\bs{u}_i - \bar{\bs{u}})  
		= 
		\text{tr}\left( \sum_{i=1}^N  \bs{S}^{-2} (\bs{u}_i - \bar{\bs{u}})(\bs{u}_i - \bar{\bs{u}})^\top   \right)
		\\
		& = 	\text{tr}\left(  \bs{S}^{-2} \cdot (N-1) \bs{S}^2   \right) = \text{tr}\left(    (N-1) \bs{I}_d   \right)
		= (N-1)d. 
	\end{align*}
	By H\"{o}lder's inequality, 
	\begin{align*}
		\frac{1}{N}\sum_{i=1}^N \big\| \bs{S}^{-1} (\bs{u}_i - \bar{\bs{u}}) \big\|_2^3
		& \ge 
		\left( \frac{1}{N}\sum_{i=1}^N \big\| \bs{S}^{-1} (\bs{u}_i - \bar{\bs{u}}) \big\|_2^2 \right)^{3/2}
		\frac{(N-1)^{3/2} d^{3/2}}{N^{3/2}} \ge 
		(d/2)^{3/2}, 
	\end{align*}
	where the last inequality holds because $(N-1)/N \ge 1/2$. 
	Consequently, we have 
	\begin{align*}
		\gamma \equiv \frac{1}{\sqrt{Nf(1-f)}} \frac{d^{1/4}}{N}
		\sum_{i=1}^N \big\| \bs{S}^{-1} (\bs{u}_i - \bar{\bs{u}}) \big\|_2^3
		\ge 
		\frac{d^{1/4}}{\sqrt{Nf(1-f)}} \frac{d^{3/2}}{2^{3/2}}
		= 
		\frac{d^{7/4}}{2^{3/2}\sqrt{Nf(1-f)}}. 
	\end{align*}
	Therefore, Lemma \ref{lemma:gamma_srs_bound} holds. 
\end{proof}

\subsubsection{Proof of Theorem \ref{thm:berry_esseen_clt_full}(iii)}

\begin{proof}[\bf Proof of Theorem \ref{thm:berry_esseen_clt_full}(iii)]
	Let $\bs{Z}$ and $\bs{T}$ be the pair of indicator vectors for simple random sampling and Bernoulli sampling satisfying Lemma \ref{lemma:couple_srs_bs}. 
	Recall that  
	$\bs{A} \equiv \sum_{i=1}^N Z_i \bs{u}_i$ and  
	$\bs{B} \equiv \sum_{i=1}^N T_i ( \bs{u}_i - \bar{\bs{u}} ) + m  \bar{\bs{u}}$, 
	and define further
	$\bs{\Sigma}_{\bs{A}} = \cov(\bs{A})$ and $\bs{\Sigma}_{\bs{B}} = \cov(\bs{B})$. 
	By definition, we can verify that $\bs{W} = \bs{\Sigma}_{\bs{A}}^{-1/2} (\bs{A} - \E \bs{A} )$. 
	
	First,
	from Lemma \ref{lemma:couple_srs_bs}, 
	$\bs{\Sigma}_{\bs{B}} = (1-l)^2 \cdot \bs{\Sigma}_{\bs{A}}$ with $l = 1 - \sqrt{1-N^{-1}}$, 
	and 
	\begin{align*}
		\bs{\Sigma}_{\bs{B}}^{-1/2} \cdot \E \big\{ (\bs{B} - \bs{A}) (\bs{B} - \bs{A})^\top \big\} \cdot \bs{\Sigma}_{\bs{B}}^{-1/2}
		\le b^2 \bs{I}_d
	\end{align*}
	with 
	$b^2 = \sqrt{1/m+1/(N-m)} = 1/\sqrt{Nf(1-f)}$. 
	
	Second, from Lemma \ref{lemma:bern_sample_berry}, 
	\begin{align*}
		\sup_{\mathcal{Q}\in \cC_d}\big|
		\PP\big\{ \bs{\Sigma}_{\bs{B}}^{-1/2}  \cdot (\bs{B} - \E\bs{B})  \in \mathcal{Q} \big\} - \PP(\bs{\varepsilon} \in \mathcal{Q})
		\big|
		& \le a \equiv \frac{165}{\sqrt{Nf(1-f)}} \frac{d^{1/4}}{N}
		\sum_{i=1}^N \big\| \bs{S}^{-1} (\bs{u}_i - \bar{\bs{u}}) \big\|_2^3 
		= 
		165 \gamma,
	\end{align*}
	where the last equality follows from the definition of $\gamma$ in Theorem \ref{thm:berry_esseen_clt_full}. 
	
	Third, let  
	\begin{align*}
		c = \left\{ \frac{1}{2} (1-l)^2 b^2 d^{3/4} \right\}^{1/3}, 
		\ \ \text{and} \ \ 
		h = \left\{ dl^2 \cdot \log N \right\}^{1/2}.
	\end{align*}
	From Lemma \ref{lemma:berry_bound_coupling}, we have 
	\begin{align*}
		& \quad \ \sup_{\mathcal{Q} \in \cC_d} 
		\big|
		\PP \big( \bs{W} \in \mathcal{Q} \big) - \PP(\bs{\varepsilon} \in \mathcal{Q})
		\big| \\
		& \leq   4 d^{1/4} (c+h) + 2 d \cdot \exp\left( - \frac{h^2}{2 d l^2}  \right) + a + \frac{(1-l)^2 b^2 d}{c^2}\\
		& = a + \left\{ 4 d^{1/4} c + \frac{(1-l)^2 b^2 d}{c^2} \right\} + 2 d \cdot \exp\left( - \frac{h^2}{2 d l^2}  \right) + 4d^{1/4} h
		\\
		& = 
		a
		+ 
		3 \cdot 2^{2/3} (1-l)^{2/3} d^{1/2} b^{2/3} + 2 d N^{-1/2} + 4 d^{3/4} l \cdot \sqrt{\log N}. 
	\end{align*}

	Fourth, from Lemma \ref{lemma:gamma_srs_bound} and by definition, we have 
	\begin{align*}
		& \left\{ 3 \cdot 2^{2/3} (1-l)^{2/ 3} d^{1/2} b^{2/3} \right\}^3
		= 
		108 (1-l)^2 d^{3/2} b^2 
		\le \frac{108 d^{3/2}}{\sqrt{N f(1-f)}}
		\le 
		108 \cdot 2^{3/2} \gamma \le 7^3 \gamma, \\
		& 2 d N^{-1/2} 
		\le \frac{d}{\sqrt{N/4}} \le  \frac{d^{7/4}}{\sqrt{Nf(1-f)}} = 2^{3/2} \gamma, \\
		& 4 d^{3/4} l \cdot \sqrt{\log N}
		= 
		4 d^{3/4} \frac{N^{-1}}{1+\sqrt{1-N^{-1}}} \cdot \sqrt{\log N}
		\le 
		\frac{2 d^{7/4}}{\sqrt{Nf(1-f)}} \sqrt{\frac{\log N}{N}}
		\le 2^{5/2} \gamma
	\end{align*}
	These then imply that 
	\begin{align*}
		\sup_{\mathcal{Q} \in \cC_d} 
		\big|
		\PP \big( \bs{W} \in \mathcal{Q} \big) - \PP(\bs{\varepsilon} \in \mathcal{Q})
		\big| 
		& \leq 
		165 \gamma
		+ 2^{3/2} \gamma + 2^{5/2} \gamma + 
		3 \cdot 2^{2/3} (1-l)^{2/3} d^{1/2} b^{2/3}
		\\
		& \le 174 \gamma + 3 \cdot 2^{2/3} \frac{d^{1/2}}{\{Nf(1-f)\}^{1/6}}
		\\
		& \le 174 \gamma + 7 \gamma^{1/3}. 
	\end{align*}
	
	From the above, Theorem \ref{thm:berry_esseen_clt_full}(iii) holds. 
\end{proof}

\subsection{Proof of Theorem \ref{thm:berry_esseen_clt_full}(iv)}

To prove Theorem \ref{thm:berry_esseen_clt_full}(iv), we need the following two lemmas. 

\begin{lemma}\label{lem:hoeffding}
	Let $\mathcal{X} \equiv \{x_i\}_{i=1}^N$ be $N$ zero-centered real valued quantities; and let $X_1, \cdots, X_m$ be $m$ random sample drawn without replacement from $\mathcal{X}$, then for all $\varepsilon > 0$, 
	\[
	\PP\left(\left|\sum_{i = 1}^m X_i \right| \ge \varepsilon\right) \le 2 \exp\left(- \frac{2 \varepsilon^2}{m (\max_{1 \leq i \leq N} x_i - \min_{1 \leq i \leq N} x_i )^2}\right)
	\]
\end{lemma}

\begin{proof}[Proof of Lemma \ref{lem:hoeffding}]
    Lemma \ref{lem:hoeffding} follows immediately from \citet[][Proposition 1.2]{BM15}. 
\end{proof}

\begin{lemma}\label{lem:hoeffdingbnd}
Let $\bs{Z}$ and $\bs{T}$ be the pair of random vectors constructed as in Lemma~\ref{lemma:couple_srs_bs}. 
For any $c, t>0$, 
    \begin{align*}
	    \PP\left\{ \left\| \cov^{-1/2}(\bs{A}) \cdot (\bs{B} - \bs{A}) \right\|_2 \ge c  \right\}
	    \le 
	    2 d \exp 
	\left(
	- \frac{c^2 f(1-f) \sqrt{N} }{ 2 t d\xi^2 }
	\right) + 2 \exp\left( -2t^2 \right), 
	\end{align*}
    where $f = m/N$ and 
	$
	\xi = \max_{1 \le i \le n} \left\|\left(\bs{S}_{\bs{u}}^2\right)^{-1/2} (\bs{u}_i-\bar{\bs{u}})\right\|_\infty. 
	$
\end{lemma}

\begin{proof}[Proof of Lemma \ref{lem:hoeffdingbnd}]
    For $1\le i \le N$ and $1\le k \le d$, 
	let $\bs{v}_i \equiv \left(\bs{S}_{\bs{u}}^2\right)^{-1/2} (\bs{u}_i-\bar{\bs{u}})$, and $v_{ik}$ denote the $k$-th coordinate of $\bs{v}_i$. 
	From the proof of Lemma \ref{lemma:couple_srs_bs}, we then have 
	\begin{align}\label{eq:diff_B_A_v}
	\cov^{-1/2}(\bs{A}) \cdot (\bs{B} - \bs{A}) 
	= 
	\sqrt{\frac{1}{f(1-f)N}} \cdot \sum_{i=1}^N (T_i-Z_i) \bs{v}_i. 
	\end{align}
	
	We first consider bounding the tail probability of $\sum_{i=1}^N (T_i-Z_i) \bs{v}_i$. 
	By definition, $2 \xi = 2 \max_{1\le i \le n} \| \bs{v}_i\|_{\infty} \ge 2 \max_{1\le i \le n} |v_{ik}| \ge \max_i v_{ik} - \min_i v_{ik}$ for $1\le k \le d$. 
	From \eqref{eq:couple_srs} and Lemma~\ref{lem:hoeffding}, we then have, for $1 \le k \le d$ and any $c>0$, 
	\begin{align*}
	    \PP\left(\left|\sum_{i=1}^n (T_i - Z_i) v_{ik} \right| \ge c \mid \tilde{m}\right)
	    & \le 
	    2 
	    \exp 
	    \left(
	    - \frac{2 c^2}{ |\tilde{m}-m| (\max_{1 \leq i \leq N} v_{ik} - \min_{1 \leq i \leq N} v_{ik} )^2}
	    \right)
	    \\
	    & \le 
	    2 
	    \exp 
	    \left(
	    - \frac{c^2}{ 2 |\tilde{m}-m| \xi^2 }
	    \right), 
	\end{align*}
	where $c^2/(2|\tilde{m}-m|\xi^2)$ is defined to be infinity when $\tilde{m} = m$. 
    This further implies that
	\begin{align*}
	\PP\left(\left\|\sum_{i=1}^n (T_i - Z_i) \bs{v}_{i} \right\|_2 \ge c \mid \tilde{m}\right) 
	& \le 
	\sum_{k=1}^d \PP\left(\left|\sum_{i=1}^n (T_i - Z_i) v_{ik} \right| \ge \frac{c}{\sqrt{d}} \mid \tilde{m}\right)
	\\
	& 
	\le 
	2 d \exp 
	\left(
	- \frac{c^2}{ 2 |\tilde{m}-m| d\xi^2}
	\right).
	\end{align*}
	Note that, by Hoeffding's inequality, for any $t > 0$,
	$
	\PP(|\tilde{m} - m| \ge t) \le 2 \exp(- 2 t^2 / N).
	$
	From the above, we can know that, for any $c,t>0$, 
	\begin{align*}
	    & \quad \ \PP\left(\left\|\sum_{i=1}^n (T_i - Z_i) \bs{v}_{i} \right\|_2 \ge c \right) 
	    \\
	    & \le 
	    \PP\left(\left\|\sum_{i=1}^n (T_i - Z_i) \bs{v}_{i} \right\|_2 \ge c \mid |\tilde{m} - m| < t\sqrt{N} \right)
	    + \PP( |\tilde{m} - m| \ge t\sqrt{N} )
	    \\
	    & \le 
	    2 d \exp 
	\left(
	- \frac{c^2}{ 2 t d\xi^2 \sqrt{N}}
	\right) + 2 \exp\left( -2t^2 \right). 
	\end{align*}
	Consequently, for any $c, t>0$, 
	\begin{align*}
	    & \quad \ \PP\left\{ \left\| \cov^{-1/2}(\bs{A}) \cdot (\bs{B} - \bs{A}) \right\|_2 \ge c  \right\}
	    = \PP\left\{ \left\| \sum_{i=1}^N (T_i-Z_i) \bs{v}_i \right\|_2 \ge c \sqrt{f(1-f)N}  \right\}
	    \\
	    & \le 
	    2 d \exp 
	\left(
	- \frac{ c^2 f(1-f) \sqrt{N}}{ 2 t d\xi^2}
	\right) + 2 \exp\left( -2t^2 \right). 
	\end{align*}
    From the above, Lemma \ref{lem:hoeffdingbnd} holds. 
\end{proof}

\begin{proof}[\bf Proof of Theorem \ref{thm:berry_esseen_clt_full}(iv)]
    Let $\bs{Z}$ and $\bs{T}$ be the pair of indicator vectors for simple random sampling and Bernoulli sampling satisfying Lemma \ref{lemma:couple_srs_bs}, 
    and adopt the same notation from the proof of Theorem \ref{thm:berry_esseen_clt_full}(iii). 
    From the proof of Lemma \ref{lemma:berry_bound_coupling} and Theorem \ref{thm:berry_esseen_clt_full}(iii), 
    for any $c, h >0$, 
    \begin{align*}
		& \quad \ \sup_{\mathcal{Q} \in \cC_d} 
		\big|
		\PP \big( \bs{W} \in \mathcal{Q} \big) - \PP(\bs{\varepsilon} \in \mathcal{Q})
		\big| \\
		& \leq   4 d^{1/4} (c+h) + 2 d \cdot \exp\left( - \frac{h^2}{2 d l^2}  \right) + a + \frac{(1-l)^2 b^2 d}{c^2}\\
		& = a + \left\{ 2 d \cdot \exp\left( - \frac{h^2}{2 d l^2}  \right) + 4d^{1/4} h \right\}
		+ \left[ 4 d^{1/4} c + \PP\left\{ \left\| \bs{\Sigma}_{\bs{A}}^{-1/2} (\bs{B} - \bs{A}) \right\|_2 \ge c  \right\} \right], 
	\end{align*}
    where $a=165\gamma$ and $l = 1-\sqrt{1-N^{-1}}$. 
    Letting $h = \{ dl^2 \cdot \log N \}^{1/2}$ and from the proof of Theorem \ref{thm:berry_esseen_clt_full}(iii), we can know that, for any $c>0$, 
    \begin{align*}
		\sup_{\mathcal{Q} \in \cC_d} 
		\big|
		\PP \big( \bs{W} \in \mathcal{Q} \big) - \PP(\bs{\varepsilon} \in \mathcal{Q})
		\big| \le 174 \gamma 
		+ \left[ 4 d^{1/4} c + \PP\left\{ \left\| \bs{\Sigma}_{\bs{A}}^{-1/2} (\bs{B} - \bs{A}) \right\|_2 \ge c  \right\} \right]. 
	\end{align*}
    Applying Lemma \ref{lem:hoeffdingbnd} and letting 
    \begin{align*}
        c = \left\{ \left( \frac{\log N}{N} \right)^{1/2} 
        \frac{1}{f(1-f)} \frac{\log N}{2} d \xi^2 \right\}^{1/2}, 
        \quad 
        t = \left( \frac{\log N}{4} \right)^{1/2}, 
    \end{align*}
    we have 
    \begin{align*}
        4 d^{1/4} c + \PP\left\{ \left\| \bs{\Sigma}_{\bs{A}}^{-1/2} (\bs{B} - \bs{A}) \right\|_2 \ge c  \right\}
        & \le 
        4 d^{1/4} c + 2 d \exp 
	    \left(
	    - \frac{c^2 f(1-f) \sqrt{N} }{ 2 t d\xi^2 }
	    \right) + 2 \exp\left( -2t^2 \right)\\
	    & = 
	    2\sqrt{2} \frac{(\log N)^{3/4}d^{3/4}}{ N^{1/4} \sqrt{f(1-f)}} \cdot \xi 
	    + 
	    2d N^{-1/2} + 2 N^{-1/2}\\
	    & \le 
	    3 \frac{(\log N)^{3/4}d^{3/4}}{ N^{1/4} \sqrt{f(1-f)}} \cdot \xi 
	    + 
	    4d N^{-1/2}, 
    \end{align*}
    where $\xi = \max_{1 \le i \le n} \left\|\left(\bs{S}_{\bs{u}}^2\right)^{-1/2} (\bs{u}_i-\bar{\bs{u}})\right\|_\infty$. 
    From Lemma \ref{lemma:gamma_srs_bound} and the proof of Theorem \ref{thm:berry_esseen_clt_full}(iii), 
    we can know that 
    $4d N^{-1/2} \le 2^{5/2} \gamma \le 6 \gamma$. 
    From the above, 
    \begin{align*}
        \sup_{\mathcal{Q} \in \cC_d} 
		\big|
		\PP \big( \bs{W} \in \mathcal{Q} \big) - \PP(\bs{\varepsilon} \in \mathcal{Q})
		\big| 
		& \le 174 \gamma 
		+ 3 \frac{(\log N)^{3/4}d^{3/4}}{ N^{1/4} \sqrt{f(1-f)}} \cdot \xi + 6 \gamma
		\\
		& 
		\le 
		180 \gamma + 3 \frac{(\log N)^{3/4}d^{3/4}}{ N^{1/4} \sqrt{f(1-f)}} \cdot \max_{1 \le i \le n} \left\|\left(\bs{S}_{\bs{u}}^2\right)^{-1/2} (\bs{u}_i-\bar{\bs{u}})\right\|_\infty. 
    \end{align*}
    Therefore, Theorem \ref{thm:berry_esseen_clt_full}(iv) holds.
\end{proof}

\subsection{Proof of Theorem \ref{thm:berry_esseen_clt_full}(v)}

To prove Theorem \ref{thm:berry_esseen_clt_full}(v), we need the following two lemmas. 

\begin{lemma}\label{lemma:srs_tail_higher_order}
	Let $\mathcal{X} \equiv \{x_i\}_{i=1}^N$ be $N$ zero-centered real valued quantities, and let $X_1, \cdots, X_m$ be $m$ random sample drawn without replacement from $\mathcal{X}$. Then for any $t > 0$ and $\iota \ge 2$, we have
	\[
	\PP\left(\left|\sum_{i = 1}^m X_i \right| \ge t\right) \le R_\iota \frac{\left( f \sum_{i=1}^N x_i^2\right)^{\iota / 2} + f \sum_{i = 1}^N |x_i|^\iota}{t^\iota},
	\]
	where $f=m/N$ and $R_\iota$ is a universal constant depending only on $\iota$.
\end{lemma}

\begin{proof}[Proof of Lemma \ref{lemma:srs_tail_higher_order}]
	From Markov's inequality, for any $\iota \ge 2$, 
	\[
	\PP\left(\left|\sum_{i = 1}^m X_i \right| \ge t\right) = \PP\left(\left|\sum_{i = 1}^m X_i \right|^\iota \ge t^\iota\right) \le \frac{\E \left|\sum_{i = 1}^m X_i \right|^\iota}{t^\iota}.
	\]
	From \citet[Theorem 4]{Hoeffding63}, 
	we further have 
	$
	\E \left|\sum_{i = 1}^m X_i \right|^\iota \le \E |\sum_{i = 1}^m \tilde{X}_i |^\iota,
	$
	where $\tilde{X}_1, \cdots, \tilde{X}_m$ are i.i.d. random samples drawn with replacement from $\mathcal{X}$. 
	From Rosenthal's inequality, 
    we then have 
	\begin{align*}
	    \PP\left(\left|\sum_{i = 1}^m X_i \right| \ge t\right) \le \frac{\E \left|\sum_{i = 1}^m \tilde{X}_i \right|^\iota}{t^\iota} \le R_\iota \frac{\left(\sum_{i=1}^m \E\tilde{X}_i^2 \right)^{\iota / 2} + \sum_{i = 1}^m \E |\tilde{X}_i|^\iota}{t^\iota},
	\end{align*}
	where $R_\iota$ is a universal constant depending only on $\iota$. 
	Note that 
	$
	\E\tilde{X}_i^2 = N^{-1} \sum_{i=1}^N x_i^2 
	$
	and 
	$
	\E |\tilde{X}_i|^\iota = N^{-1} \sum_{i=1}^N |x_i|^\iota. 
	$
	We can then derive Lemma \ref{lemma:srs_tail_higher_order}. 
\end{proof}

\begin{lemma}\label{lem:rosenthualbnd}
    Let $\bs{Z}$ and $\bs{T}$ be the pair of random vectors constructed as in Lemma~\ref{lemma:couple_srs_bs}.
    For any $\iota \geq 2$ and $t > 0$,
	\begin{align*}
	    \PP\left\{ \left\| \cov^{-1/2}(\bs{A}) \cdot (\bs{B} - \bs{A}) \right\|_2 \ge c  \right\}
	    = 
	    \frac{C_\iota d^{\iota/2+1}}{c^\iota N^{\iota/4} \{f(1-f)\}^{\iota/2}}
	     + 
	    \frac{C_\iota d^{\iota/2} \xi_\iota}{c^\iota N^{(\iota-1)/2} \{f(1-f)\}^{(\iota-1)/2}}
	\end{align*}
	where $C_\iota$ is a constant depending only on $\iota$, and 
	$
	    \xi_\iota = N^{-1} \sum_{i = 1}^N \| \left(\bs{S}_{\bs{u}}^2\right)^{-1/2} (\bs{u}_i-\bar{\bs{u}}) \|_{\iota}^\iota. 
	$
\end{lemma}

\begin{proof}[Proof of Lemma \ref{lem:rosenthualbnd}]
    We construct $\bs{Z}$ and $\bs{T}$ in the same way as in the proof of Lemmas \ref{lemma:couple_srs_bs} and \ref{lem:hoeffdingbnd}, and we adopt the same notation as in Lemma \ref{lem:hoeffdingbnd}.
    We further define $\xi_{k,\iota} = N^{-1} \sum_{i = 1}^N |v_{ik}|^\iota$.

    We first consider bounding the tail probability of $\left\|\sum_{i=1}^n (T_i - Z_i) \bs{v}_{i} \right\|_2$. 
    From Lemma \ref{lemma:srs_tail_higher_order}, for any $c > 0$ and $1\le k \le d$, 
    \begin{align*}
        \PP\left(\left|\sum_{i=1}^n (T_i - Z_i) v_{ik} \right| \ge c \mid \tilde{m}\right) 
        & \le R_\iota \frac{\left(|\tilde{m} - m| \cdot N^{-1}\sum_{i=1}^N v_{ik}^2\right)^{\iota / 2} + |\tilde{m} - m| \cdot N^{-1} \sum_{i = 1}^N |v_{ik}|^\iota}{c^\iota}
        \\
        & 
        \le R_\iota \frac{|\tilde{m} - m|^{\iota / 2} + |\tilde{m} - m| \cdot \xi_{k,\iota}}{c^\iota}, 
    \end{align*}
    where the last inequality holds because $N^{-1}\sum_{i=1}^N v_{ik}^2 = (N-1)/N \le 1$. 
    This then implies that 
    \begin{align*}
        & \quad \ \PP\left(\left\|\sum_{i=1}^n (T_i - Z_i) \bs{v}_{i} \right\|_2 \ge c \mid \tilde{m}\right) 
        \le 
        \sum_{k=1}^d 
        \PP\left(\left|\sum_{i=1}^n (T_i - Z_i) v_{ik} \right| \ge c/\sqrt{d} \mid \tilde{m}\right) 
        \\
        & 
        \le R_\iota \sum_{k=1}^d  \frac{ |\tilde{m} - m|^{\iota / 2} + |\tilde{m} - m| \cdot \xi_{k,\iota}}{c^\iota d^{-\iota/2}}
        = 
        R_\iota d^{\iota/2}  \cdot \frac{ d |\tilde{m} - m|^{\iota / 2} + |\tilde{m} - m| \cdot \sum_{k=1}^d  \xi_{k,\iota}}{c^\iota}. 
    \end{align*}
    By the law of iterated expectation, 
    \begin{align*}
        & \quad \ \PP\left(\left\|\sum_{i=1}^n (T_i - Z_i) \bs{v}_{i} \right\|_2 \ge c \right) 
        = \E \left\{ \PP\left(\left\|\sum_{i=1}^n (T_i - Z_i) \bs{v}_{i} \right\|_2 \ge c \mid \tilde{m}\right) \right\}
        \\
        & \le 
        R_\iota d^{\iota/2} \cdot \frac{ d \cdot \E\{ |\tilde{m} - m|^{\iota/2} \} + \E\{|\tilde{m} - m|\} \cdot \sum_{k=1}^d  \xi_{k,\iota}}{c^\iota}. 
    \end{align*}

    We then consider bounding the moments of $|\tilde{m} - m|$. 
    By Hoeffding's inequality, 
    for any $t > 0$,
	$
	\PP(|\tilde{m} - m| \ge t) \le 2 \exp(- 2 t^2 / N).
	$
	Using \citep[Lemma~1.4]{RH15}, this implies that 
	\begin{align*}
	    \E\{ |\tilde{m} - m|^{\iota/2} \}
	    \le \left( \frac{N}{2} \right)^{\iota/4} \cdot (\iota/2) \cdot \Gamma(\iota/4).
	\end{align*}
	Besides, 
	$\E\{|\tilde{m} - m|\} \le \sqrt{ \var(\tilde{m} - m)} = \sqrt{Nf(1-f)}$. 
	
	Finally, we consider bounding the tail probability of $\cov^{-1/2}(\bs{A}) \cdot (\bs{B} - \bs{A})$. 
	From \eqref{eq:diff_B_A_v}, 
    \begin{align*}
	    & \quad \ \PP\left\{ \left\| \cov^{-1/2}(\bs{A}) \cdot (\bs{B} - \bs{A}) \right\|_2 \ge c  \right\}
	    = \PP\left\{ \left\| \sum_{i=1}^N (T_i-Z_i) \bs{v}_i \right\|_2 \ge c \sqrt{f(1-f)N}  \right\}
	    \\
	    & \le 
	    R_\iota d^{\iota/2} \cdot \frac{ d \cdot \E\{ |\tilde{m} - m|^{\iota/2} \} + \E\{|\tilde{m} - m|\} \cdot \sum_{k=1}^d  \xi_{k,\iota}}{c^\iota N^{\iota/2} \{f(1-f)\}^{\iota/2}}
	    \\
	    & \le 
	    \frac{R_\iota d^{\iota/2}}{c^\iota N^{\iota/2} \{f(1-f)\}^{\iota/2}} \left\{ d \cdot \left( \frac{N}{2} \right)^{\iota/4} \cdot (\iota/2) \cdot \Gamma(\iota/4) + \sqrt{Nf(1-f)} \cdot \sum_{k=1}^d  \xi_{k,\iota} \right\}
	    \\
	    & \le 
	    \frac{C_\iota d^{\iota/2}}{c^\iota N^{\iota/2} \{f(1-f)\}^{\iota/2}} \left\{ d \cdot N^{\iota/4} + \sqrt{Nf(1-f)} \cdot N^{-1} \sum_{i = 1}^N \|\bs{v}_i\|_{\iota}^\iota \right\}
	    \\
	    & = 
	    \frac{C_\iota d^{\iota/2+1}}{c^\iota N^{\iota/4} \{f(1-f)\}^{\iota/2}}
	     + 
	    \frac{C_\iota d^{\iota/2}}{c^\iota N^{(\iota-1)/2} \{f(1-f)\}^{(\iota-1)/2}}
	    \cdot \frac{1}{N} \sum_{i = 1}^N \|\bs{v}_i\|_{\iota}^\iota. 
	\end{align*}
	
	From the above, we can immediately derive Lemma \ref{lem:rosenthualbnd}. 
\end{proof}

\begin{proof}[\bf Proof of Theorem \ref{thm:berry_esseen_clt_full}(v)]
    Letting $\bs{Z}$ and $\bs{T}$ be the pair of indicator vectors for simple random sampling and Bernoulli sampling satisfying Lemma \ref{lemma:couple_srs_bs}, 
    and by the same logic as the proof of Theorem \ref{thm:berry_esseen_clt_full}(iv),  
    for any $c>0$, 
    \begin{align*}
		\sup_{\mathcal{Q} \in \cC_d} 
		\big|
		\PP \big( \bs{W} \in \mathcal{Q} \big) - \PP(\bs{\varepsilon} \in \mathcal{Q})
		\big| \le 174 \gamma 
		+ \left[ 4 d^{1/4} c + \PP\left\{ \left\| \bs{\Sigma}_{\bs{A}}^{-1/2} (\bs{B} - \bs{A}) \right\|_2 \ge c  \right\} \right]. 
	\end{align*}
    Applying Lemma \ref{lem:hoeffdingbnd} and letting 
    $
        c = N^{-\iota/\{4(\iota+1)\}} \cdot d^{(2\iota-1)/\{4(\iota+1)\}}, 
    $
    we have 
    \begin{align*}
        & \quad \ 4 d^{1/4} c + \PP\left\{ \left\| \bs{\Sigma}_{\bs{A}}^{-1/2} (\bs{B} - \bs{A}) \right\|_2 \ge c  \right\}
        \\
        & \le 
        4 d^{1/4} c + \frac{C_\iota d^{\iota/2+1}}{c^\iota N^{\iota/4} \{f(1-f)\}^{\iota/2}}
	     + 
	    \frac{C_\iota d^{\iota/2} \xi_\iota}{c^\iota N^{(\iota-1)/2} \{f(1-f)\}^{(\iota-1)/2}}\\
	    & = 
        4 \frac{d^{3\iota/\{4(\iota+1)\}}}{N^{\iota/\{4(\iota+1)\}}}
        + 
        \frac{C_\iota d^{3\iota/\{4(\iota+1)\}} \cdot d}{N^{\iota/\{4(\iota+1)\}} \{f(1-f)\}^{\iota/2}}
	     + 
	    \frac{C_\iota \cdot d^{3\iota/\{4(\iota+1)\}} \xi_\iota}{ N^{(\iota^2-2)/\{4(\iota+1)\}} \{f(1-f)\}^{(\iota-1)/2}}
	    \\
	    & \le  
	    4 \frac{d^{3\iota/\{4(\iota+1)\}}}{N^{\iota/\{4(\iota+1)\}}}
        + 
        \frac{C_\iota d^{3\iota/\{4(\iota+1)\}} \cdot (d+\xi_\iota)}{N^{\iota/\{4(\iota+1)\}} \{f(1-f)\}^{\iota/2}}
        \le 
        \max\{4, C_{\iota}\} \cdot 
        \frac{d^{3\iota/\{4(\iota+1)\}} \cdot (2d+\xi_\iota)}{N^{\iota/\{4(\iota+1)\}} \{f(1-f)\}^{\iota/2}},
    \end{align*}
    where 
    $
	    \xi_\iota = N^{-1} \sum_{i = 1}^N \| \left(\bs{S}_{\bs{u}}^2\right)^{-1/2} (\bs{u}_i-\bar{\bs{u}}) \|_{\iota}^\iota. 
	$
	Adopting the notation from the proof of Lemma \ref{lem:rosenthualbnd}, 
	\begin{align*}
	    \xi_\iota
	    & = 
	    N^{-1} \sum_{i = 1}^N \sum_{k=1}^d |v_{ik}|^{\iota}
	    = 
	    \sum_{k=1}^d \left( N^{-1}   \sum_{i = 1}^N |v_{ik}|^{\iota} \right)
	    \ge 
	    \sum_{k=1}^d \left( N^{-1}   \sum_{i = 1}^N v_{ik}^2 \right)^{\iota/2}
	    = 
	    \sum_{k=1}^d \left( \frac{N-1}{N} \right)^{\iota/2}\\
	    & 
	    \ge 2^{-\iota/2} \cdot d. 
	\end{align*}
	From the above, we then have 
	\begin{align*}
	    & \quad \ \sup_{\mathcal{Q} \in \cC_d} 
		\big|
		\PP \big( \bs{W} \in \mathcal{Q} \big) - \PP(\bs{\varepsilon} \in \mathcal{Q})
		\big| 
		\\
		& \le 174 \gamma 
        + 
        \max\{4, C_{\iota}\} \cdot 
        \frac{d^{3\iota/\{4(\iota+1)\}} \cdot (1+d+\xi_\iota)}{N^{\iota/\{4(\iota+1)\}} \{f(1-f)\}^{\iota/2}}
        \\
        & \le 
        174 \gamma 
        + 
        \max\{4, C_{\iota}\} \cdot 
        \frac{d^{3\iota/\{4(\iota+1)\}} \cdot (2^{\iota/2+1}+1)\xi_\iota}{N^{\iota/\{4(\iota+1)\}} \{f(1-f)\}^{\iota/2}}\\
        & \le 
        174 \gamma 
        + 
        C_{\iota}' \cdot 
        \frac{d^{3\iota/\{4(\iota+1)\}}}{N^{\iota/\{4(\iota+1)\}} \{f(1-f)\}^{\iota/2}} \cdot \frac{1}{N}\sum_{i = 1}^N \| \left(\bs{S}_{\bs{u}}^2\right)^{-1/2} (\bs{u}_i-\bar{\bs{u}}) \|_{\iota}^\iota,
	\end{align*}
	where $C_{\iota}'= \max\{4, C_{\iota}\} \cdot (2^{\iota/2+1}+1)$ is a universal constant depending only on $\iota$. 
	Therefore, Theorem \ref{thm:berry_esseen_clt_full}(v) holds. 
\end{proof}

\section{Asymptotic Distributions in Completely Randomized and Rerandomized Experiments}\label{sec:asym_cre_rerand}

\begin{proof}[\bf Proof of Theorems \ref{thm:berry_esseen_clt} and \ref{thm:Berry--Esseen-higher-order}]
	Following the notation in Section \ref{sec:berry_rate} and from \eqref{eq:did_srs_u}, 
	the difference-in-means vector $(\hat{\tau}, \hat{\bs{\tau}}_{\bs{X}}^\top)^\top$ is essentially the sample total of a simple random sample of  size $n_1$ from the finite population of $\{\bs{u}_i = (r_0 Y_i(1) + r_1 Y_i(0), \bs{X}_i^\top)^\top: i=1,2,\ldots, n\}$, up to some constant scaling and shifting. 
	This  then  implies that 
	\begin{align*}
		\Delta_n & \equiv 
		\sup_{\mathcal{Q} \in \cC_{K+1}}
		\left|
		\PP
		\left\{\bs{V}^{-1/2}
		\begin{pmatrix}
			\hat{\tau} - \tau\\
			\hat{\bs{\tau}}_{\bs{X}}
		\end{pmatrix}
		\in \mathcal{Q}
		\right\}
		- 
		\PP \left( \bs{\varepsilon} \in \mathcal{Q} \right)
		\right|
		= 
		\sup_{\mathcal{Q} \in \cC_{K+1}}
		\left|
		\PP
		\left(\bs{W}
		\in \mathcal{Q}
		\right)
		- 
		\PP \left( \bs{\varepsilon} \in \mathcal{Q} \right)
		\right|, 
	\end{align*}
	where $\bs{W} = \cov^{-1/2}( \sum_{i=1}^n Z_i \bs{u}_i)  \sum_{i=1}^n Z_i \bs{u}_i$ is the standardization of $\sum_{i=1}^n Z_i \bs{u}_i$. 
	By the definition of $\gamma_n$ in \eqref{eq:gamma_n} and the definitions of $r_1, r_0, K$,
	Theorem \ref{thm:berry_esseen_clt} then follows immediately from Theorem \ref{thm:berry_esseen_clt_full} (i - iii); Theorem \ref{thm:Berry--Esseen-higher-order} follows from Theorem \ref{thm:berry_esseen_clt_full} (iv - v). 
\end{proof}

\begin{proof}[\bf Proof of Theorem~\ref{thm:dim_rem}]
	Under Condition \ref{cond:p_n}, there must exist $\underline{n}\ge 2$ such that $p_n > \Delta_n$ for all $n \geq \underline{n}$. Let $(\tilde{\tau}, \tilde{\bs{\tau}}_{\bs{X}}^\top)^\top$ denote a Gaussian random vector with mean $(\tau, \bs{0}_K^\top)$ and covariance matrix $\bs{V}$ in \eqref{eq:V}. 
	By the definition of $\Delta_n$ in \eqref{eq:Delta_n}, we can know that, for any measurable convex set $\mathcal{Q}$ in $\mathbb{R}^{K+1}$, 
	\begin{align*}
		\left|
		\PP
		\left\{
		\begin{pmatrix}
			\hat{\tau} - \tau\\
			\hat{\bs{\tau}}_{\bs{X}}
		\end{pmatrix}
		\in \mathcal{Q}
		\right\}
		- \PP
		\left\{
		\begin{pmatrix}
			\tilde{\tau} - \tau\\
			\tilde{\bs{\tau}}_{\bs{X}}
		\end{pmatrix}
		\in \mathcal{Q}
		\right\}
		\right|
		& = 
		\left|
		\PP
		\left\{
		\bs{V}^{-1/2}
		\begin{pmatrix}
			\hat{\tau} - \tau\\
			\hat{\bs{\tau}}_{\bs{X}}
		\end{pmatrix}
		\in  \bs{V}^{-1/2} \mathcal{Q}
		\right\}
		- 
		\PP
		\left(
		\bs{\varepsilon}
		\in \bs{V}^{-1/2} \mathcal{Q}
		\right)
		\right|\\
		& \le \Delta_n.
	\end{align*}
	This  implies that,
	\[
	\left| 
	\PP\left(\hat{\bs{\tau}}_{\bs{X}}^\top \bs{V}_{\bs{xx}}^{-1} \hat{\bs{\tau}}_{\bs{X}} 
	\le a_n\right) 
	- 
	\PP\left(\tilde{\bs{\tau}}_{\bs{X}}^\top \bs{V}_{\bs{xx}}^{-1} \tilde{\bs{\tau}}_{\bs{X}} \le a_n\right)
	\right|
	= 
	\left|
	\PP(M \le a_n) - \PP(\tilde{M} \le a_n)
	\right|
	\leq \Delta_n, 
	\]
	and for any $c\in \mathbb{R}$,
	\begin{align}\label{eq:bound_numer_rem_proof}
		& \quad \ 
		\left|
		\PP\left(\hat{\tau} - \tau \le c, \hat{\bs{\tau}}_{\bs{X}}^\top \bs{V}_{\bs{xx}}^{-1} \hat{\bs{\tau}}_{\bs{X}} \le a_n\right) 
		-
		\PP\left( \tilde{\tau} - \tau \le c, \tilde{\bs{\tau}}_{\bs{X}}^\top \bs{V}_{\bs{xx}}^{-1} \tilde{\bs{\tau}}_{\bs{X}} \le a_n\right)
		\right|
		\nonumber
		\\
		& 
		= 
		\left|
		\PP\left(\hat{\tau} - \tau \le c,  M \le a_n\right) 
		-
		\PP\left( \tilde{\tau} - \tau \le c, \tilde{M}\le a_n\right)
		\right|
		\le 
		\Delta_n, 
	\end{align}
	where $\tilde{M} \equiv \tilde{\bs{\tau}}_{\bs{X}}^\top \bs{V}_{\bs{xx}}^{-1} \tilde{\bs{\tau}}_{\bs{X}} \sim \chi_{K_n}^2$.
	By definition, $p_n = \PP(\tilde{M}\le a_n)$. 
	Thus, for $n\ge \underline{n}$, we must have $\PP(M\le a_n) \ge p_n - \Delta_n > 0$, 
	and consequently 
	\begin{align}\label{eq:bound_denom_rem_proof}
		\frac{1}{p_n + \Delta_n} \leq 
		\frac{1}{\PP( M \le a_n)} \leq \frac{1}{p_n - \Delta_n}.
	\end{align}
	From \eqref{eq:bound_numer_rem_proof} and \eqref{eq:bound_denom_rem_proof},  we then have, for all $n \geq \underline{n}$ and $c\in \mathbb{R}$, 
	\begin{align*}
		\PP(\hat{\tau} - \tau \le c \mid M \le a_n)
		& = 
		\frac{\PP(\hat{\tau} - \tau \le c,  M \le a_n)}{\PP(M \le a_n)}
		\le 
		\frac{\PP(\tilde{\tau} - \tau \le c, \tilde{M} \le a_n) + \Delta_n}{\PP(\tilde{M}\le a_n) - \Delta_n}
		\\
		& = 
		\frac{\PP(\tilde{\tau} - \tau \le c \mid \tilde{M} \le a_n) + \Delta_n/p_n}{1 - \Delta_n/p_n}
		\\
		& \le 
		\frac{\PP(\tilde{\tau} - \tau \le c \mid \tilde{M} \le a_n) (1-\Delta_n/p_n) + 2 \Delta_n/p_n}{1 - \Delta_n/p_n}\\
		& = 
		\PP(\tilde{\tau} - \tau \le c \mid \tilde{M} \le a_n) + 
		\frac{2 \Delta_n/p_n}{1 - \Delta_n/p_n}, 
	\end{align*}
	and 
	\begin{align*}
		\PP(\hat{\tau} - \tau \le c \mid M \le a_n)
		& = 
		\frac{\PP(\hat{\tau} - \tau \le c, M\le a_n)}{\PP(M \le a_n)}
		\ge 
		\frac{\PP(\tilde{\tau} - \tau \le c, \tilde{M} \le a_n) - \Delta_n}{\PP(\tilde{M} \le a_n) + \Delta_n}
		\\
		& = 
		\frac{\PP(\tilde{\tau} - \tau \le c \mid \tilde{M} \le a_n) - \Delta_n/p_n}{1 + \Delta_n/p_n}
		\\
		& 
		\ge 
		\frac{\PP(\tilde{\tau} - \tau \le c \mid \tilde{M} \le a_n)(1+\Delta_n/p_n) -2 \Delta_n/p_n}{1 + \Delta_n/p_n}\\
		& = 
		\PP(\tilde{\tau} - \tau \le c \mid \tilde{M} \le a_n) 
		-
		\frac{2 \Delta_n/p_n}{1 + \Delta_n/p_n}.
	\end{align*}
	These imply that, for all $n\ge \underline{n}$, 
	\begin{align*}
		& \quad \ 
		\sup_{c \in \R} \left| \PP\left\{V_{\tau\tau}^{-1/2} (\hat{\tau} - \tau) \le c \mid M \le a_n\right\} - \PP\left\{V_{\tau\tau}^{-1/2} (\tilde{\tau} - \tau) \le c \mid \tilde{M} \le a_n\right\}\right|
		\\
		& \le \max\left\{ \frac{2 \Delta_n/p_n}{1 - \Delta_n/p_n}, \frac{2 \Delta_n/p_n}{1 + \Delta_n/p_n} \right\}
		\le \frac{2 \Delta_n/p_n}{1 - \Delta_n/p_n}. 
	\end{align*}
	Under Condition \ref{cond:p_n}, we then have, as $n\rightarrow \infty$, 
	\begin{align*}
		\sup_{c \in \R} \left| \PP\left\{V_{\tau\tau}^{-1/2} (\hat{\tau} - \tau) \le c \mid M \le a_n\right\} - \PP\left\{V_{\tau\tau}^{-1/2} (\tilde{\tau} - \tau) \le c \mid \tilde{M} \le a_n\right\}\right|
		\rightarrow 0. 
	\end{align*}
	Finally, from the proof of  \citet[][Theorem 1]{LDR18}, for any $c\in \mathbb{R}$, 
	\[
	\PP\left\{V_{\tau\tau}^{-1/2} (\tilde{\tau} - \tau) \le c \mid \tilde{M} \le a_n\right\} = \PP\left\{\left(\sqrt{1 - R^2} \varepsilon_0 + \sqrt{R^2} L_{K_n, a_n}\right) \leq c\right\},
	\]
	with $\varepsilon_0$ and $L_{K_n, a_n}$ defined as in Section \ref{sec:asymptotic}. 
	Therefore, we derive Theorem \ref{thm:dim_rem}. 
\end{proof}

\begin{proof}[\bf Comment on Condition \ref{cond:gamma_n} and regularity conditions in \citet{LDR18}]
By the definition in \eqref{eq:gamma_n}, 
\begin{align*}
    \gamma_n & \le 
	\frac{(K+1)^{1/4}}{\sqrt{n r_1r_0}} \cdot \max_{1\le i\le n} \left\| \bs{S}_{\bs{u}}^{-1} (\bs{u}_i-\bar{\bs{u}}) \right\|_2 \cdot \frac{1}{n} \sum_{i=1}^n \left\| \bs{S}_{\bs{u}}^{-1} (\bs{u}_i-\bar{\bs{u}}) \right\|_2^2
	\\
	& = 
	\frac{(K+1)^{1/4}}{\sqrt{n r_1r_0}} \cdot \max_{1\le i\le n} \left\| \bs{S}_{\bs{u}}^{-1} (\bs{u}_i-\bar{\bs{u}}) \right\|_2 \cdot \frac{(n-1)(K+1)}{n}
	\\
	& \le 
	\frac{(K+1)^{5/4}}{\sqrt{r_1r_0}} \cdot \left\| \bs{S}_{\bs{u}}^{-1} \right\|_2 \cdot \frac{1}{\sqrt{n}} \max_{1\le i\le n} \left\| \bs{u}_i-\bar{\bs{u}} \right\|_2. 
\end{align*}
Note that
\begin{align}\label{eq:bound_max_squarenorm}
    \frac{1}{n} \max_{1\le i\le n} \left\| \bs{u}_i-\bar{\bs{u}} \right\|_2^2 
    & = \frac{1}{n} \max_{1\le i \le n}\left[ r_0\{ Y_i(1) - \bar{Y}(1) \} + r_1 \{ Y_i(0) - \bar{Y}(0) \} \right]^2 + 
    \frac{1}{n} \max_{1\le i\le n}\left\| \bs{X}_i - \bar{\bs{X}} \right\|_2^2
    \nonumber
    \\
    & \le \frac{2}{n} \max_{1\le i\le n} \{ Y_i(1) - \bar{Y}(1) \}^2 + \frac{2}{n} \max_{1\le i\le n}  \{ Y_i(0) - \bar{Y}(0) \}^2 +  \frac{1}{n} \left\| \bs{X}_i - \bar{\bs{X}} \right\|_2^2. 
\end{align}
Under \citet[][Condition 1]{LDR18}, as $n\rightarrow \infty$, 
both $r_1$ and $r_0$ have positive limits, $\bs{S}_{\bs{u}}^2$ has a limiting value (in particular, the limit of $\bs{S}^2_{\bs{X}}$ is nonsingular), 
and the quantities on the right hand side of \eqref{eq:bound_max_squarenorm} converge to zero. 
If additionally the limit of $R^2$ is less than 1, then the limit of $\bs{S}_{\bs{u}}^2$ will be invertible, and thus $\gamma_n$ must converge to zero as $n\rightarrow \infty$, i.e., Condition \ref{cond:gamma_n} holds. 
\end{proof}

\begin{proof}[\bf Proof of Corollary \ref{cor:improve}]
	Corollary \ref{cor:improve} follows by the same logic as \citet[][Corollaries 1--3]{LDR18}. 
\end{proof}

\begin{proof}[\bf Comments on the lower bound of $\gamma_n$ in \eqref{eq:gamma_n_lower}]
	The lower bound of $\gamma_n$ follows by the same logic as Lemma \ref{lemma:gamma_srs_bound}.
\end{proof}

\begin{proof}[\bf Comments on the upper bound of $\gamma_n$]
    By the definition in \eqref{eq:gamma_n}, 
    \begin{align*}
        \gamma_n & \equiv 
    	\frac{(K_n+1)^{1/4}}{\sqrt{n r_1r_0}} \frac{1}{n} \sum_{i=1}^n \left\| \bs{S}_{\bs{u}}^{-1} (\bs{u}_i-\bar{\bs{u}}) \right\|_2^3
    	\le 
    	\frac{(K_n+1)^{1/4}}{\sqrt{n r_1r_0}} \frac{1}{n} \sum_{i=1}^n (K_n+1)^{3/2}\left\| \bs{S}_{\bs{u}}^{-1} (\bs{u}_i-\bar{\bs{u}}) \right\|_{\infty}^{3/2}\\
    	& \le 
    	\frac{(K_n+1)^{7/4}}{\sqrt{n r_1r_0}} \max_{1\le i\le n}\left\| \bs{S}_{\bs{u}}^{-1} (\bs{u}_i-\bar{\bs{u}}) \right\|_{\infty}^{3/2}. 
    \end{align*}
	If the standardized finite population $\{\bs{S}_{\bs{u}}^{-1} (\bs{u}_i-\bar{\bs{u}}):1\le i\le n\}$ is coordinate-wise bounded, and the proportions of treated and control units are bounded away from zero, 
	then there exist finite positive constants $c$ and $C$ such that for all $n$ and $1\le i \le n$, $\|\bs{S}_{\bs{u}}^{-1} (\bs{u}_i-\bar{\bs{u}})\|_{\infty} \le C$ and $\min\{r_1, r_0\}> c$.  
	Consequently, 
	\begin{align*}
	     \gamma_n & \le \frac{(K_n+1)^{7/4}}{\sqrt{n r_1r_0}} \max_{1\le i\le n}\left\| \bs{S}_{\bs{u}}^{-1} (\bs{u}_i-\bar{\bs{u}}) \right\|_{\infty}^{3/2}
	     \le 
	     \frac{(K_n+1)^{7/4}}{\sqrt{n c^2}} C^{3/2} = \left( \frac{K_n+1}{n^{2/7}} \right)^{7/4} \frac{C^{3/2}}{c}, 
	\end{align*}
	under which $K_n = o(n^{2/7})$ implies that $\gamma_n = o(1)$. 
\end{proof}

\section{Limiting Behavior of the Constrained Gaussian Random Variable}\label{sec:limit_constrained_Gaussian}

In this section, we prove Theorem \ref{thm:v_Ka} regarding the limiting behavior of the constrained Gaussian random variable $L_{K_n, a_n}$. 
We first give some technical lemmas in Section \ref{sec:lemma_L_limit}, 
and then study the limiting behavior of $L_{K_n, a_n}$ in Sections \ref{sec:limt_L_logpK_inf}--\ref{sec:limt_L_logpK_0} under various relationship between $\log(p_n^{-1})$ and $K_n$.
Sections \ref{sec:limt_L_logpK_inf}--\ref{sec:limt_L_logpK_0} essentially prove Theorem  \ref{thm:v_Ka}(i)--(iv) respectively, as briefly commented in Section \ref{sec:proof_thm_v_Ka}. 
For descriptive convenience, 
we introduce $\chi^2_K$ to denote a random variable following the chi-square distribution with degrees of freedom $K$.

\subsection{Technical lemmas and their proofs}\label{sec:lemma_L_limit}

\subsubsection{Lemmas for the acceptance probability $p=\PP(\chi^2_K \le a)$}

\begin{lemma}\label{lemma:Gamma_fun}
	For any integer $K\ge 1$, 
	$
	\sqrt{\pi K}\{K/(2e)\}^{K/2} \le
	\Gamma(K/2+1)
	\le 2\sqrt{\pi K}\{K/(2e)\}^{K/2}. 
	$
\end{lemma}

\begin{proof}[Proof of Lemma \ref{lemma:Gamma_fun}]
	We can numerically verify that Lemma \ref{lemma:Gamma_fun} holds when $K=1$. 
	Below we consider only the case with $K\ge 2$. 
	From \citet{K01}, the Gamma function can be bounded by 
	\begin{align*}
		\sqrt{\pi} \left( \frac{K}{2e} \right)^{K/2} 
		\left( K^3 + K^2 + \frac{K}{2} + \frac{1}{100} \right)^{1/6}
		\le 
		\Gamma(K/2 + 1) \le \sqrt{\pi} \left( \frac{K}{2e} \right)^{K/2} 
		\left( K^3 + K^2 + \frac{K}{2} + \frac{1}{30} \right)^{1/6}, 
	\end{align*}
	\begin{align*}
		1 \le \left( 1 + \frac{1}{K} + \frac{1}{2K^2} + \frac{1}{100 K^3} \right)^{1/6}
		\le 
		\frac{\Gamma(K/2+1)}{\sqrt{\pi K}\{K/(2e)\}^{K/2}} 
		\le 
		\left( 1 + \frac{1}{K} + \frac{1}{2K^2} + \frac{1}{30 K^3} \right)^{1/6}
		\le 2. 
	\end{align*}
	From the above, Lemma \ref{lemma:Gamma_fun} holds. 
\end{proof}

\begin{lemma}\label{lemma:link_p_a}
	For any integer $K\ge 1$ and $a > 0$, define $p = \PP(\chi^2_K \le a)$. 
	Then 
	\begin{align*}
		\frac{\log(p^{-1})}{K} \le \frac{ \log (4\pi K)}{2K} + 
		\frac{1}{2} \left\{ \frac{a}{K} - 1 - \log\left( \frac{a}{K} \right) \right\}. 
	\end{align*}
	Moreover, if $a/K<1$, then 
	\begin{align*}
		\frac{\log (p^{-1})}{K} \ge 
		\frac{1}{2} \left\{ \frac{a}{K} - 1 - \log\left( \frac{a}{K} \right) \right\} + \frac{\log(\pi K)}{2K}  + \frac{1}{K}\log \left(1-\frac{a}{K}\right). 
	\end{align*}
\end{lemma}

\begin{proof}[Proof of Lemma \ref{lemma:link_p_a}]
	By definition and using integration by parts, we have 
	\begin{align}\label{eq:int_by_part}
		p & =
		\frac{1}{2^{K/2} \Gamma(K/2)} \int_0^a t^{K/2-1} e^{-t/2} \deri t
		= 
		\frac{1}{2^{K/2} \Gamma(K/2)} \cdot \frac{t^{K/2}}{K/2} e^{-t/2} \Big|_0^a + 
		\frac{1}{2^{K/2} \Gamma(K/2)} \cdot
		\int_0^a \frac{t^{K/2}}{K/2} e^{-t/2} \frac{1}{2} \deri t
		\nonumber
		\\
		& = 
		\frac{a^{K/2} e^{-a/2}}{2^{K/2}\Gamma(K/2+1)}  + 
		\frac{1}{K}
		\frac{1}{2^{K/2} \Gamma(K/2)} 
		\int_0^a t \cdot t^{K/2-1} e^{-t/2} \deri t
		\\
		& \le \frac{a^{K/2} e^{-a/2}}{2^{K/2}\Gamma(K/2+1)}  + 
		\frac{a}{K}
		\frac{1}{2^{K/2} \Gamma(K/2)} 
		\int_0^a t^{K/2-1} e^{-t/2} \deri t
		\nonumber
		\\
		& = \frac{a^{K/2} e^{-a/2}}{2^{K/2}\Gamma(K/2+1)} + \frac{a}{K} p,
		\nonumber
	\end{align}
	and 
	\begin{align*}
		p 
		& 
		= 
		\frac{a^{K/2} e^{-a/2}}{2^{K/2}\Gamma(K/2+1)}  + 
		\frac{1}{K}
		\frac{1}{2^{K/2} \Gamma(K/2)} 
		\int_0^a t \cdot t^{K/2-1} e^{-t/2} \deri t
		\ge \frac{a^{K/2} e^{-a/2}}{2^{K/2}\Gamma(K/2+1)}.  
	\end{align*}
	These implies that 
	\begin{align*}
		\left(1-\frac{a}{K}\right) p \le \frac{a^{K/2} e^{-a/2}}{2^{K/2}\Gamma(K/2+1)} 
		= 
		\frac{\left( a/K \cdot e^{1-a/K} \right)^{K/2}}{\sqrt{\pi K}} 
		\frac{\sqrt{\pi K}\{K/(2e)\}^{K/2}}{\Gamma(K/2+1)}
		\le p. 
	\end{align*}
	From Lemma \ref{lemma:Gamma_fun}, we then have 
	\begin{align*}
		p \ge \frac{1}{2} \frac{\left( a/K \cdot e^{1-a/K} \right)^{K/2}}{\sqrt{\pi K}} 
		\quad \text{and} \quad
		\left(1-\frac{a}{K}\right) p \le \frac{\left( a/K \cdot e^{1-a/K} \right)^{K/2}}{\sqrt{\pi K}}.  
	\end{align*}
	Consequently, 
	\begin{align*}
		\frac{\log(p^{-1})}{K} \le \frac{ \log (4\pi K)}{2K} + 
		\frac{1}{2} \left\{ \frac{a}{K} - 1 - \log\left( \frac{a}{K} \right) \right\}. 
	\end{align*}
	and, when $a/K<1$, 
	\begin{align*}
		\frac{\log (p^{-1})}{K} \ge 
		\frac{1}{2} \left\{ \frac{a}{K} - 1 - \log\left( \frac{a}{K} \right) \right\} + \frac{\log(\pi K)}{2K}  + \frac{1}{K}\log \left(1-\frac{a}{K}\right). 
	\end{align*}
	Therefore, Lemma \ref{lemma:link_p_a} holds. 
\end{proof}

\subsubsection{Lemmas for the variance of $L_{K,a}$ and its bounds}

\begin{lemma}\label{lem:varchi}
	For any integer $K > 0$ and $a > 0$,
	\begin{itemize}
		\item[(i)]
		$
		\var(L_{K, a}) = K^{-1} \E (\chi_{K}^2 \mid \chi_{K}^2 \leq a) = \PP(\chi^2_{K+2} \le a) / \PP(\chi^2_{K} \le a). 
		$
		\item[(ii)] 
		$\var(L_{K, a})$ is nondecreasing in $a$ for any given fixed $K\ge 1$. 
	\end{itemize} 
\end{lemma}

\begin{proof}[Proof of Lemma \ref{lem:varchi}]
	Lemma \ref{lem:varchi} follows immediately from \citet[][Theorem 3.1]{MR12} and \citet[][Lemma A5]{LDR18}. 
\end{proof}

\begin{lemma}\label{lem:vlbnd}
	For any integer $K \ge 1$ and $a \ge 0$, it holds that
	\[
	\min\bigg\{\frac{a}{4 K}, \frac{K - 2}{4 K}\bigg\} \leq \var(L_{K, a}) \leq \frac{a}{K}.
	\]
\end{lemma}

\begin{proof}[Proof of Lemma \ref{lem:vlbnd}]
	The upper bound of $\var(L_{K, a})$ is a direct consequence of Lemma~\ref{lem:varchi}(i). 
	The lower bound of $\var(L_{K, a})$ holds obviously when $a=0$ or $K\le 2$. 
	Below we consider only the lower bound of $\var(L_{K, a})$ when $a>0$ and $K\ge 3$. 
	Define $\tilde{a} = \min\{a, K-2\}$. 
	By the property of chi-square distribution, the density function of $\chi^2_K$ is monotonically increasing on the interval $[0, K - 2] \supset [0, \tilde{a}]$. 
	This implies that $\PP(\chi_{K}^2 \leq \tilde{a} / 2) \leq \PP(\tilde{a} / 2\leq \chi_{K}^2 \leq \tilde{a})$ and 
	\begin{align}\label{eq:prob_cond_chisq}
		\PP(\tilde{a} / 2 \leq \chi_{K}^2 \leq \tilde{a} \mid \chi_{K}^2 \leq \tilde{a})
		& = 
		\frac{\PP(\tilde{a} / 2 \leq \chi_{K}^2 \leq \tilde{a})}{\PP(\chi_{K}^2\le \tilde{a}/2) + \PP(\tilde{a} / 2 \leq \chi_{K}^2 \leq \tilde{a})} \ge 1/2. 
	\end{align}
	Consequently, from Lemma \ref{lem:varchi}, 
	the variance of $L_{K, a}$ multiplied by $K$ can be bounded by 
	\begin{align*}
		K\var(L_{K, a}) & \ge K \var(L_{K, \tilde{a}}) = \E[\chi_{K}^2 \mid \chi_{K}^2 \leq \tilde{a}]
		\geq  \PP(\tilde{a} / 2 \leq \chi_{K}^2 \leq \tilde{a} \mid \chi_{K}^2 \leq \tilde{a}) \cdot \E[\chi_{K}^2 \mid \tilde{a} / 2 \leq \chi_{K}^2 \leq \tilde{a}] 
		\\
		& \geq \frac{1}{2} \cdot \frac{\tilde{a}}{2} 
		= \frac{\min\{a, K-2\}}{4}
	\end{align*}
	where the first inequality holds because $a \ge \tilde{a}$ and 
	the last inequality holds due to \eqref{eq:prob_cond_chisq}. 
	From the above, Lemma \ref{lem:vlbnd} holds.
\end{proof}

\begin{lemma}\label{lemma:bound_chisq_ratio1}
	For any integer $K>0$ and $a>0$, with $p=\PP(\chi^2_K\le a)$, 
	\begin{align*}
		-\log\left\{
		1 - \var(L_{K,a})
		\right\}
		& \ge 
		\frac{K}{2}
		\left\{
		\frac{a}{K} - 1 
		- \log\left( \frac{a}{K} \right) 
		-
		\frac{2\log(p^{-1})}{K} + \frac{\log(\pi K)}{K}
		\right\}. 
	\end{align*}
\end{lemma}

\begin{proof}[Proof of Lemma \ref{lemma:bound_chisq_ratio1}]
	From \eqref{eq:int_by_part}, 
	\begin{align*}
		p & = \PP(\chi^2_K \le a)
		= 
		\frac{a^{K/2} e^{-a/2}}{2^{K/2}\Gamma(K/2+1)}  + 
		\frac{1}{K}
		\frac{1}{2^{K/2} \Gamma(K/2)} 
		\int_0^a t^{K/2} e^{-t/2} \deri t
		\\
		& 
		= 
		\frac{a^{K/2} e^{-a/2}}{2^{K/2}\Gamma(K/2+1)}  + 
		\frac{1}{2^{(K+2)/2} \Gamma((K+2)/2)} 
		\int_0^a t^{(K+2)/2-1} e^{-t/2} \deri t\\
		& = 
		\frac{a^{K/2} e^{-a/2}}{2^{K/2}\Gamma(K/2+1)} + \PP(\chi^2_{K+2} \le a). 
	\end{align*}
	From Lemmas \ref{lemma:Gamma_fun} and \ref{lem:varchi}, this implies that 
	\begin{align}\label{eq:1_v_equi}
		1 - \var(L_{K,a}) & = 1 - \frac{\PP(\chi^2_{K+2} \le a)}{\PP(\chi^2_K \le a)}
		= 
		\frac{a^{K/2} e^{-a/2}}{2^{K/2}\Gamma(K/2+1)} \cdot \frac{1}{p}
		\le 
		\frac{a^{K/2} e^{-a/2}}{2^{K/2}\sqrt{\pi K}\{K/(2e)\}^{K/2}} \cdot \frac{1}{p}
		\\
		& = 
		\frac{(a/K)^{K/2} e^{(K-a)/2}}{p\sqrt{\pi K} } . \nonumber
	\end{align}  
	Consequently, 
	\begin{align*}
		-\log\left\{
		1 - \var(L_{K,a})
		\right\}
		& \ge 
		-\frac{K}{2} \log\left( \frac{a}{K} \right) - \frac{K-a}{2} + \log(p) + \frac{1}{2}\log(\pi K)
		\\
		& = 
		\frac{K}{2}
		\left\{
		\frac{a}{K} - 1 
		- \log\left( \frac{a}{K} \right) 
		-
		\frac{2\log(p^{-1})}{K} + \frac{\log(\pi K)}{K}
		\right\}.
	\end{align*}
	Therefore, Lemma \ref{lemma:bound_chisq_ratio1} holds. 
\end{proof}

\begin{lemma}\label{lemma:bound_chisq_ratio2}
	For any $K>2$, $a \in (0, K-2]$ and $\zeta \in (0,1)$, 
	\begin{align*}
		-\log \left\{
		1 - \var(L_{K,a})
		\right\}
		\ge 
		- \log(2) + 
		\log(K\zeta)
		-
		\frac{\zeta}{2(1-\zeta)}
		\left\{
		a\zeta + (K-2-a)
		\right\}. 
	\end{align*}
\end{lemma}

\begin{proof}[Proof of Lemma \ref{lemma:bound_chisq_ratio2}]
	From \eqref{eq:1_v_equi}, 
	\begin{align*}
		1 - \var(L_{K,a}) & 
		= 
		\frac{a^{K/2} e^{-a/2}}{2^{K/2}\Gamma(K/2+1)} \cdot \frac{1}{\PP(\chi^2_K \le a)}
		= \frac{a^{K/2} e^{-a/2}}{2^{K/2}\Gamma(K/2+1)} \cdot
		\frac{2^{K/2} \Gamma(K/2)}{\int_0^a t^{K/2-1} e^{-t/2} \deri t}\\
		& = 
		\frac{2}{K}
		\frac{a^{K/2} e^{-a/2}}{\int_0^a t^{K/2-1} e^{-t/2} \deri t}.
	\end{align*}
	By the property of chi-square distribution, 
	$t^{K/2-1} e^{-t/2}$ is nondecreasing in $t\in [0,K-2] \supset [0,a]$, which implies that 
	\begin{align*}
		\int_0^a t^{K/2-1} e^{-t/2} \deri t
		& \ge 
		\int_{(1-\zeta) a}^a t^{K/2-1} e^{-t/2} \deri t
		\ge 
		\zeta a \cdot \{ (1-\zeta) a\}^{K/2-1} e^{-(1-\zeta)a/2}
		\\
		& = 
		\zeta (1-\zeta)^{K/2-1} a^{K/2} e^{-(1-\zeta) a/2}. 
	\end{align*}
	Thus, 
	\begin{align*}
		1 - \var(L_{K,a})
		& = 
		\frac{2}{K}
		\frac{a^{K/2} e^{-a/2}}{\int_0^a t^{K/2-1} e^{-t/2} \deri t}
		\le 
		\frac{2}{K}
		\frac{a^{K/2} e^{-a/2}}{\zeta (1-\zeta)^{K/2-1} a^{K/2} e^{-(1-\zeta) a/2}}
		= 
		\frac{2}{K}
		\frac{e^{-\zeta a/2}}{\zeta (1-\zeta)^{K/2-1}}, 
	\end{align*}
	and consequently
	\begin{align*}
		-\log\left\{ 1 - \var(L_{K,a}) \right\}    
		& \ge 
		-\log(2) + \log(K\zeta) + \frac{\zeta a}{2} + \left(\frac{K}{2}-1\right)\log(1-\zeta). 
	\end{align*}
	Using the inequality that $\log (1+x) \ge x/(1+x)$ for all $x>-1,$ we have $\log(1-\zeta) \ge - \zeta/(1-\zeta)$, and thus 
	\begin{align*}
		& \quad \ -\log\left\{ 1 - \var(L_{K,a}) \right\}    
		\\
		& \ge 
		-\log(2) + \log(K\zeta) + \frac{\zeta a}{2} -  \left(\frac{K}{2}-1\right) \frac{\zeta}{1-\zeta} 
		= 
		-\log(2) + \log(K\zeta) + \frac{\zeta}{2(1-\zeta)} \left\{
		a - a\zeta - (K-2)
		\right\}\\
		& = 
		-\log(2) + \log(K\zeta) - \frac{\zeta}{2(1-\zeta)} \left\{
		a\zeta + (K-2-a)
		\right\}. 
	\end{align*}
	Therefore, Lemma \ref{lemma:bound_chisq_ratio2} holds. 
\end{proof}

\subsection{Limiting behavior when $\lim_{n\rightarrow \infty}\log(p_n^{-1}) / K_n = \infty$}\label{sec:limt_L_logpK_inf}

\begin{lemma}\label{lem:dimlka}
	As $n \converge \infty$, 
	if $\log(p_n^{-1}) / K_n \converge \infty$, then $a_n/K_n \converge 0$ and 
	$\var(L_{K_n, a_n}) \converge 0$. 
\end{lemma}
\begin{proof}[Proof of Lemma \ref{lem:dimlka}]
	From Lemma \ref{lem:vlbnd}, it suffices to show that $a_n/K_n \converge 0$ as $n\converge \infty$. 
	We prove this by contradiction. 
	If $a_n/K_n$ does not converge to zero, 
	then there must exist a positive constant $c>0$ and a subsequence $\{n_j:j=1,2,\ldots\}$ such that $a_{n_j}/K_{n_j} \ge c$ for all $j\ge 1$. 
	Thus, for any $j\ge 1$,
	$
	p_{n_j} = \PP(\chi^2_{K_{n_j}} \le a_{n_j}) \ge \PP(\chi^2_{K_{n_j}} \le c K_{n_j}). 
	$
	From Lemma \ref{lemma:link_p_a}, this then implies that 
	\begin{align*}
		\frac{\log(p_{n_j}^{-1})}{K_{n_j}} \le 
		\frac{\log\{ \PP(\chi^2_{K_{n_j}} \le c K_{n_j})^{-1} \}}{K_{n_j}}
		\le
		\frac{ \log (4\pi K_{n_j})}{2K_{n_j}} + 
		\frac{c - 1 - \log\left( c \right)}{2}
		\le \frac{ \log (4\pi)}{2} + 
		\frac{c - 1 - \log\left( c \right)}{2}, 
	\end{align*}
	where the last inequality holds because $\log(4\pi K)/(2K)$ is decreasing in $K$ for $K\ge 1$. 
	However, this contradicts the fact that $\log(p_n^{-1}) / K_n \converge \infty$. 
	Therefore, we must have $a_n/K_n \converge 0$ as $n\converge \infty$. 
	From the above, Lemma \ref{lem:dimlka} holds.
\end{proof}

\subsection{Limiting behavior when $\limsup_{n\rightarrow\infty} \log(p_n^{-1})/K_n < \infty$}

\begin{lemma}\label{lem:nodim}
	If $\limsup_{n\rightarrow\infty} \log(p_n^{-1})/K_n < \infty$, then 
	$\liminf_{n\rightarrow \infty} a_n/K_n > 0$ and 
	$
	\liminf_{n\rightarrow \infty}\var(L_{K_n, a_n}) > 0. 
	$
\end{lemma}

\begin{proof}[Proof of Lemma \ref{lem:nodim}]
	We first prove $\liminf_{n\rightarrow \infty} a_n/K_n > 0$ by contradiction. 
	If  $\liminf_{n\rightarrow \infty} a_n/K_n = 0$, then there exists
	a subsequence $\{n_j:j=1,2,\ldots\}$ such that $a_{n_j}/K_{n_j}\converge 0$ as $j\converge \infty$ and 
	$a_{n_j}/K_{n_j} < 1$ for all $j$. 
	From Lemma \ref{lemma:link_p_a}, we then have, for any $j\ge 1$, 
	\begin{align*}
		\frac{\log (p_{n_j}^{-1})}{K_{n_j}} & \ge 
		\frac{1}{2} \left\{ \frac{a_{n_j}}{K_{n_j}} - 1 - \log\left( \frac{a_{n_j}}{K_{n_j}} \right) \right\} + \frac{\log(\pi K_{n_j})}{2K_{n_j}}  + \frac{1}{K_{n_j}}\log \left(1-\frac{a_{n_j}}{K_{n_j}}\right)\\
		& \ge - \frac{1}{2} - \frac{1}{2} \log\left( \frac{a_{n_j}}{K_{n_j}} \right)  +  \log \left(1-\frac{a_{n_j}}{K_{n_j}}\right),
	\end{align*}
	which converges to infinity as $j\rightarrow \infty$. 
	However, this contradicts with the fact that $\limsup_{n\rightarrow\infty} \log(p_n^{-1})/K_n < \infty$. 
	Thus, we must have $\liminf_{n\rightarrow \infty} a_n/K_n > 0$. 
	
	Second, we prove that $
	\liminf_{n\rightarrow \infty}\var(L_{K_n, a_n}) > 0
	$ by contradiction. 
	If $
	\liminf_{n\rightarrow \infty}\var(L_{K_n, a_n}) = 0
	$, 
	then there exists a subsequence $\{n_j:j=1,2,\ldots\}$ such that 
	$
	\var(L_{K_{n_j}, a_{n_j}}) \converge 0
	$
	as $n\rightarrow \infty$
	and 
	$\var(L_{K_{n_j}, a_{n_j}}) < 1/12$ for all $j$. 
	Below we consider two cases, depending on whether $\limsup_{j\rightarrow\infty} K_{n_j}$ is greater than or equal to 3. 
	If $\limsup_{j\rightarrow\infty} K_{n_j} \ge 3$, 
	then there exists a further subsequence $\{m_1, m_2, \ldots\}\subset\{n_1, n_2, \ldots\}$ such that $K_{m_j}\ge 3$ for all $j$. 
	Because $(K_{m_j}-2)/(4K_{m_j}) \ge 1/12$, from Lemma \ref{lem:vlbnd}, we must have $\var(L_{K_{m_j}, a_{m_j}}) \ge a_{m_j} / (4K_{m_j})$. 
	This then implies that 
	\begin{align*}
		0 = \lim_{j\converge\infty}\var(L_{K_{m_j}, a_{m_j}}) \ge \liminf_{j\converge\infty} a_{m_j} / (4K_{m_j})
		\ge \liminf_{n\converge\infty} a_{n} / (4K_{n}) > 0, 
	\end{align*}
	a contradiction. 
	If $\limsup_{j\rightarrow\infty} K_{n_j} < 3$,  then there exists a further subsequence $\{m_1, m_2, \ldots\}\subset\{n_1, n_2, \ldots\}$ such that $K_{m_j}\le 2$ for all $j$. 
	Because 
	\begin{align*}
		0 < \liminf_{n\converge\infty} a_{n} / (4K_{n}) \le \liminf_{j\converge\infty} a_{m_j} / (4K_{m_j}) 
		\le \liminf_{j\converge\infty} a_{m_j} / 4, 
	\end{align*}
	there must exist a positive constant $\underline{a}>0$ such that $a_{m_j} > \underline{a}$ for all $j$. 
	From Lemma \ref{lem:varchi}, this then implies that 
	\begin{align*}
		0 = \lim_{j\rightarrow \infty} \var(L_{K_{m_j}, a_{m_j}}) 
		\ge \liminf_{j\rightarrow \infty} \var(L_{K_{m_j}, \underline{a}}) \ge \min\left\{ \var(L_{1, \underline{a}}), \var(L_{2, \underline{a}}) \right\} > 0, 
	\end{align*}
	a contradiction. 
	Therefore, we must have $
	\liminf_{n\rightarrow \infty}\var(L_{K_n, a_n}) > 0,
	$
	i.e., Lemma \ref{lem:nodim} holds. 
\end{proof}

\subsection{Limiting behavior when $\liminf_{n\converge \infty} \log(p_n^{-1})/K_n > 0$}

\begin{lemma}\label{lem:varub}
	If 
	$\liminf_{n\converge \infty} \log(p_n^{-1})/K_n > 0$, then 
	\begin{itemize}
		\item[(i)]  
		for any subsequence $\{n_j: j=1,2, \ldots\}$ with $\lim_{j\converge \infty}K_{n_j} = \infty$, $\limsup_{j\converge \infty}a_{n_j}/K_{n_j} < 1$. 
		\item[(ii)] 
		$\limsup_{n \to \infty} \var(L_{K_n, a_n}) < 1.$
	\end{itemize}
\end{lemma}

\begin{proof}[Proof of Lemma \ref{lem:varub}]
	First, 
	we prove (i) by contradiction. 
	Suppose that there exists a subsequence $\{n_j: j=1,2, \ldots\}$ such that $K_{n_j} \converge \infty$ as $j\converge \infty$ and $\limsup_{j\converge \infty}a_{n_j}/K_{n_j} \ge 1$. 
	Then there exists a further subsequence $\{m_j:j=1,2,\ldots\} \subset \{n_j:j=1,2,\ldots\}$ such that $\lim_{j \converge \infty} a_{m_j}/K_{m_j} \ge 1$. 
	Define $\tilde{a}_{m_j} = \min\{1, a_{m_j}/K_{m_j}\} \cdot K_{m_j}$. 
	We can then verify that $\tilde{a}_{m_j} \le a_{m_j}$ and $\lim_{j \converge \infty} \tilde{a}_{m_j}/K_{m_j} = 1$. 
	From Lemma \ref{lemma:link_p_a}, for any $j\ge 1$, 
	\begin{align*}
		\frac{\log(p_{m_j}^{-1})}{K_{m_j}} \le 
		\frac{\log\{ \PP(\chi^2_{K_{m_j}} \le \tilde{a}_{m_j} )^{-1} \}}{K}
		\le 
		\frac{ \log (4\pi K_{m_j})}{2K_{m_j}} + 
		\frac{1}{2} \left\{ \frac{\tilde{a}_{m_j}}{K_{m_j}} - 1 - \log\left( \frac{\tilde{a}_{m_j}}{K_{m_j}} \right) \right\},  
	\end{align*}
	where the right hand side converges to 0 as $j\converge \infty$. Consequently, 
	$\log(p_{m_j}^{-1})/ K_{m_j} \converge 0$ as $j\converge \infty$. 
	However, this contradicts with the fact that $\liminf_{n\converge \infty} \log(p_n^{-1})/K_n > 0$. 
	Therefore, for any subsequence $\{n_j: j=1,2, \ldots\}$ with $\lim_{j\converge \infty}K_{n_j} = \infty$, we must have $\limsup_{j\converge \infty}a_{n_j}/K_{n_j} < 1$. 
	
	Second, we prove 
	(ii) by contradiction. 
	If $\limsup_{n \to \infty} \var(L_{K_n, a_n}) = 1$, then there exists a subsequence $\{n_j: j=1,2,\ldots\}$ such that $\var(L_{K_{n_j}, a_{n_j}}) \converge 1$ as $j\converge \infty$. 
	Below we consider two cases, depending on whether $\limsup_{j\converge\infty} K_{n_j}$ is finite. 
	If $\limsup_{j\converge\infty} K_{n_j} = \infty$, then there exists a further subsequence $\{m_j: j=1,2,\ldots\} \subset \{n_j: j=1,2,\ldots\}$ such that $K_{m_j}\converge \infty$ as $j\converge \infty$. 
	From Lemma \ref{lem:vlbnd} and the discussion before, we have 
	\begin{align*}
		1 > \limsup_{j\converge \infty} a_{m_j}/K_{m_j}
		\ge \limsup_{j\converge \infty} \var(L_{K_{m_j}, a_{m_j}}) = 1, 
	\end{align*}
	a contradiction. 
	If $\limsup_{j\converge\infty} K_{n_j} < \infty$, then there exists a finite integer $\overline{K}$ such that $K_{n_j} \le \overline{K}$ for all $j$. 
	Because $\liminf_{n\converge \infty} \log(p_n^{-1})/K_n > 0$, 
	there must exists a positive constant $c>0$ such that $\log(p_n^{-1})/K_n \ge c$ for all $n$. This immediately implies that $\log( p_n^{-1} ) \ge c$ and $p_n \le e^{-c}$ for all $n$. Consequently, for all $j$, we have $a_{n_j} = F_{K_{n_j}}^{-1} (p_{n_j}) \le F_{K_{n_j}}^{-1} (e^{-c}) \le \max_{1\le K\le \overline{K}} F_K^{-1}(e^{-c}) \equiv \overline{a}$, where $F_K^{-1}(\cdot)$ denotes the quantile function for the chi-square distribution with degrees of freedom $K$. 
	From Lemma \ref{lem:varchi}, we then have 
	$$
	1 = \lim_{j\rightarrow \infty}\var(L_{K_{n_j}, a_{n_j}}) \le \limsup_{j\rightarrow \infty} \var(L_{K_{n_j}, \overline{a}})
	\le \max_{1\le K \le \overline{K}} \var(L_{{K}, \overline{a}}) < 1, 
	$$
	a contradiction. 
	Therefore, we must have $\limsup_{n \to \infty} \var(L_{K_n, a_n}) < 1$.
	
	From the above, Lemma \ref{lem:varub} holds. 
\end{proof}

\subsection{Limiting behavior when $\lim_{n\converge \infty} \log(p_n^{-1}) / K_n = 0$}\label{sec:limt_L_logpK_0}

\begin{lemma}\label{lem:infan}
	If $\log(p_n^{-1}) / K_n \converge 0$ as $n\rightarrow \infty$, then 
	$\liminf_{n\rightarrow \infty} a_n / K_n \geq 1$.
\end{lemma}

\begin{proof}[Proof of Lemma \ref{lem:infan}]
	We prove the lemma by contradiction. 
	Assume that $\liminf_{n\rightarrow \infty} a_n / K_n < 1$. 
	Then there must exist a subsequence $\{n_j: j=1,2,\ldots\}$ such that $a_{n_j}/K_{n_j} < 1$ and $a_{n_j}/K_{n_j}$ converges to some $c\in [0,1)$ as $n\rightarrow \infty$. 
	From Lemma \ref{lemma:link_p_a}, for any $j\ge 1$, 
	\begin{align*}
		\frac{\log (p_{n_j}^{-1})}{K_{n_j}} \ge 
		\frac{1}{2} \left\{ \frac{a_{n_j}}{K_{n_j}} - 1 - \log\left( \frac{a_{n_j}}{K_{n_j}} \right) \right\} + \frac{\log(\pi K_{n_j})}{2K_{n_j}}  + \frac{1}{K_{n_j}}\log \left(1-\frac{a_{n_j}}{K_{n_j}}\right). 
	\end{align*}
	If $\limsup_{j\rightarrow \infty} K_{n_j} = \infty$, 
	then there exists a further subsequence $\{m_j: j=1,2,\ldots\}$ such that $K_{m_j} \converge \infty$ as $j\converge \infty$, under which we have 
	\begin{align*}
		\frac{\log (p_{m_j}^{-1})}{K_{m_j}} & \ge 
		\frac{1}{2} \left\{ \frac{a_{m_j}}{K_{m_j}} - 1 - \log\left( \frac{a_{m_j}}{K_{m_j}} \right) \right\} + \frac{\log(\pi K_{m_j})}{2K_{m_j}}  + \frac{1}{K_{m_j}}\log \left(1-\frac{a_{m_j}}{K_{m_j}}\right) 
		\\
		& \converge 
		\frac{1}{2} \left( c - 1 - \log(c) \right) > 0.
	\end{align*}
	However, this contradicts with $\lim_{n\rightarrow \infty}\log(p_n^{-1}) / K_n = 0$. 
	If $\limsup_{j\rightarrow \infty} K_{n_j} < \infty$, 
	then there exists a finite $\overline{K} < \infty$ such that $K_{n_j} \le \overline{K}$ for all $j$. 
	Thus, $a_{n_j} = a_{n_j}/K_{n_j} \cdot K_{n_j} \le \overline{K}$ for all $j$, under which we have
	\begin{align*}
		p_{n_j} = \PP(\chi^2_{K_{n_j}} \le a_{n_j})
		\le \PP(\chi^2_{K_{n_j}} \le \overline{K})
		\le \max_{1\le K \le \overline{K}}  \PP(\chi^2_{K} \le \overline{K}), 
	\end{align*}
	and 
	\begin{align*}
		\frac{\log (p_{n_j}^{-1})}{K_{n_j}} \ge 
		\frac{\log (p_{n_j}^{-1})}{\overline{K}}
		\ge 
		\frac{1}{\overline{K}} 
		\log\left\{ \frac{1}{\max_{1\le K \le \overline{K}}\PP(\chi^2_{K} \le \overline{K})} \right\} > 0.
	\end{align*}
	However, this contradicts with $\lim_{n\rightarrow \infty}\log(p_n^{-1}) / K_n = 0$. 
	From the above, Lemma \ref{lem:infan} holds. 
\end{proof}

\begin{lemma}\label{lem:infpn2}
	If $\log(p_n^{-1}) / K_n  \converge 0$, then $\var(L_{K_n, a_n}) \converge 1$.
\end{lemma}

\begin{proof}[Proof of Lemma \ref{lem:infpn2}]
	We prove the lemma by contradiction. 
	Assume $\liminf_{n\converge \infty} \var(L_{K_n, a_n}) < 1$.
	Then there exist a constant $c \in [0,1)$ and a subsequence $\{n_j: j = 1, 2,\ldots \}$ such that 
	$\var(L_{K_{n_j}, a_{n_j}}) \le c$ for all $j$. 
	Below we consider several cases, depending on the values of
	$\liminf_{j\converge \infty} K_{n_j}$, 
	$\limsup_{j\converge \infty} a_{n_j}/K_{n_j}$ and $\liminf_{j\converge \infty} (K_{n_j}-a_{n_j})/\sqrt{K_{n_j}}$. 
	
	First, we consider the case in which $\liminf_{j\converge \infty} K_{n_j} < \infty$. 
	Then there exist a finite constant $\overline{K} < \infty$ and a subsequence $\{m_j: j=1,2,\ldots\} \subset \{n_j: j = 1, 2,\ldots \}$ such that 
	$
	K_{m_j} \le \overline{K}. 
	$
	Thus, for any $j\ge 1$, we have
	\begin{align*}
		\min_{1\le K\le \overline{K}}\log 
		\left\{
		\frac{1}{\PP(\chi^2_{K} \le a_{m_j})} \right\}
		\le 
		\log 
		\left\{ \frac{1}{\PP(\chi^2_{K_{m_j}} \le a_{m_j})}\right\}
		=
		\log( p^{-1}_{m_j} )
		\le \overline{K} \frac{\log(p_{m_j}^{-1})}{K_{m_j}}  \converge 0, 
	\end{align*}
	which must implies that $a_{m_j} \converge \infty$ as $j\converge \infty$. Consequently, from Lemma \ref{lem:varchi}, 
	\begin{align*}
		1 > c \ge \var ( L_{K_{m_j}, a_{m_j}} ) \ge 
		\min_{1\le K\le \overline{K}} \var ( L_{K, a_{m_j}} )
		= 
		\min_{1\le K\le \overline{K}} \frac{\PP(\chi^2_{K+2} \le a_{m_j})}{\PP(\chi^2_{K} \le a_{m_j})} 
		\converge 1, 
	\end{align*}
	a contradiction.

	Second, we consider the case in which 
	$\liminf_{j\converge \infty} K_{n_j} = \infty$
	and 
	$\limsup_{j\converge \infty} a_{n_j}/K_{n_j} > 1$. 
	Then there exist a constant $\delta >1$ and a further subsequence $\{m_j:j=1,2,\ldots\}\subset \{n_j:j=1,2,\ldots\}$ such that 
	$
	a_{m_j}/K_{m_j} > \delta
	$
	and 
	$K_{m_j} \ge 3$
	for all $j$. 
	Note that the function $x-1-\log(x)$ is increasing in $x\in [1, \infty)$ and takes positive value when $x>1$. 
	From Lemma \ref{lemma:bound_chisq_ratio1}, we then have
	\begin{align*}
		-\frac{2\log(1-c)}{K_{m_j}} 
		& \ge 
		-
		\frac{2\log\{
			1 - \var(L_{K_{m_j},a_{m_j}})
			\}}{K_{m_j}}
		\ge 
		\frac{a_{m_j}}{K_{m_j}} - 1 
		- \log\left( \frac{a_{m_j}}{K_{m_j}} \right) 
		-
		\frac{2\log(p_{m_j}^{-1})}{K_{m_j}} + \frac{\log(\pi K_{m_j})}{K_{m_j}}
		\\
		& \ge \delta - 1 - \log \delta -
		\frac{2\log(p_{m_j}^{-1})}{K_{m_j}} + \frac{\log(\pi K_{m_j})}{K_{m_j}}
		\converge \delta - 1 - \log \delta > 0.
	\end{align*}
	However, 
	as $j\converge \infty$, 
	$K_{m_j} \converge \infty$ and thus 
	$-2\log(1-c)/K_{m_j} \converge 0$, a contradiction. 
	
	Third, 
	we consider the case in which 
	$\liminf_{j\converge \infty} K_{n_j} = \infty$
	and 
	$\liminf_{j\converge \infty} (K_{n_j}-a_{n_j})/\sqrt{K_{n_j}} < \infty$. 
	Then there exists a finite constant $\beta$ and a subsequence $\{m_j:j=1,2,\ldots\} \subset \{n_j:j=1,2,\ldots\}$ such that 
	$(K_{m_j} - a_{m_j})/\sqrt{K_{m_j}} \le \beta$. 
	By the central limit theorem, 
	\begin{align*}
		p_{m_j}
		& = 
		\PP(\chi^2_{K_{m_j}} \le a_{m_j})
		= 
		\PP\left( \frac{\chi^2_{K_{m_j}} - K_{m_j}}{\sqrt{2K_{m_j}}} \le -\frac{K_{m_j}-a_{m_j}}{\sqrt{2K_{m_j}}} \right)
		\ge
		\PP\left( \frac{\chi^2_{K_{m_j}} - K_{m_j}}{\sqrt{2K_{m_j}}} \le -\frac{\beta}{\sqrt{2}} \right)
		\converge \Phi\left( -\frac{\beta}{\sqrt{2}} \right), 
	\end{align*}
	which implies that $\limsup_{j\rightarrow \infty}\log( p_{m_j}^{-1} ) < \infty$. 
	From Lemma \ref{lemma:bound_chisq_ratio1}
	and the inequality $x-1-\log(x)\ge 0$ for all $x>0$, 
	we then have 
	\begin{align*}
		-\log(1-c)
		& \ge 
		-\log\{
		1 - \var(L_{K_{m_j},a_{m_j}})
		\}
		\ge 
		\frac{K_{m_j}}{2}
		\left\{
		\frac{a_{m_j}}{K_{m_j}} - 1 
		- \log\left( \frac{a_{m_j}}{K_{m_j}} \right)  \right\}
		-
		\log(p_{m_j}^{-1}) + \frac{\log(\pi K_{m_j})}{2}
		\\
		& \ge -
		\log(p_{m_j}^{-1}) + \frac{\log(\pi K_{m_j})}{2} \converge \infty, 
	\end{align*}
	a contradiction. 
	
	Finally, we consider the case in which $\liminf_{j\converge \infty} K_{n_j} = \infty$, $\limsup_{j\converge \infty} a_{n_j}/K_{n_j} \le 1$ and 
	$\liminf_{j\converge \infty} (K_{n_j}-a_{n_j})/\sqrt{K_{n_j}} = \infty$. 
	From Lemma \ref{lem:infan}, 
	$\liminf_{j\converge \infty} a_{n_j}/K_{n_j} \ge \liminf_{n\rightarrow \infty} a_n / K_n \ge 1$, 
	which implies that $a_{n_j}/K_{n_j} \converge 1$ as $j\converge \infty$. 
	Moreover, there exists a finite $\underline{j}$ such that $a_{n_j} < K_{n_j}-3$ for all $j \ge \underline{j}$. 
	For any $j\ge \underline{j}$, 
	define $\Delta_{n_j} = (K_{n_j} - 2 - a_{n_j})/K_{n_j}$ and 
	$\zeta_{n_j} = \min\{ \Delta_{n_j}^{-1/2}/K_{n_j}, K_{n_j}^{-3/4} \} \in (0,1)$. 
	We can then verify that, as $j\converge \infty$, 
	$\Delta_{n_j} \converge 0$, $\zeta_{n_j}\converge 0$, 
	$K_{n_j} \zeta_{n_j} = \min\{ \Delta_{n_j}^{-1/2}, K_{n_j}^{1/4} \} \converge \infty$, and 
	\begin{align*}
		a_{n_j} \zeta_{n_j}^2 \le  a_{n_j} K_{n_j}^{-3/2} = K_{n_j}^{-1/2} \frac{a_{n_j}}{K_{n_j}} \converge 0, 
		\quad 
		\zeta_j (K_{n_j} - 2 - a_{n_j})
		& \le \Delta_{n_j}^{-1/2} \frac{K_{n_j} - 2 - a_{n_j}}{K_{n_j}}
		= \Delta_{n_j}^{1/2} \converge 0. 
	\end{align*}
	From Lemma \ref{lemma:bound_chisq_ratio2}, 
	this further implies that, 
	for any $j \ge \underline{j}$, 
	\begin{align*}
		-\log(1-c) 
		& \ge 
		-\log \big\{
		1 -  \var(L_{K_{n_j}, a_{n_j}})
		\big\}
		\ge 
		- \log(2) + 
		\log(K_{n_j} \zeta_{n_j} )
		-
		\frac{a_{n_j}\zeta_{n_j}^2 + \zeta_{n_j}(K_{n_j}-2-a_{n_j})}{2(1-\zeta_{n_j})}\\
		& \converge \infty, 
	\end{align*}
	a contradiction. 
	
	From the above, Lemma \ref{lem:infpn2} holds. 
\end{proof}

\subsection{Proof of Theorem~\ref{thm:v_Ka}
	and an additional proposition 
}\label{sec:proof_thm_v_Ka}

\begin{proof}[\bf Proof of Theorem \ref{thm:v_Ka}]
	(i) is a direct consequence of Lemma~\ref{lem:dimlka}. 
	(ii) is a direct consequence of Lemma~\ref{lem:nodim}. 
	(iii) is a direct consequence of Lemma~\ref{lem:varub}. 
	(iv) is a direct consequence of Lemma~\ref{lem:infpn2}.
\end{proof}

The following proposition establishes the equivalence between convergence in probability and convergence of variance for the constrained Gaussian random variable discussed in Section \ref{sec:asym_con_Gauss}. 

\begin{proposition}\label{prop:convpvar}
	As $n \to \infty$, $L_{K_n, a_n} \convergep 0$ if and only if $\var(L_{K_n, a_n}) \converge 0$.
\end{proposition}

\begin{proof}[Proof of Proposition \ref{prop:convpvar}]
	The ``if'' is a direct consequence of Chebyshev's inequality. Below we focus on the ``only if'' direction. 
	Suppose that $L_{K_n, a_n} \convergep 0$. 
	Note that when $a_n = \infty$, $L_{K_n, a_n}\sim \varepsilon$, a standard Gaussian random variable. From Lemma \ref{lem:varchi}, for all $n\ge 1$, 
	$\E(L_{K_n, a_n}^2) = \var(L_{K_n, a_n}) \le \var(L_{K_n, \infty}) = \var(\varepsilon) = 1$. 
	From \citet[][Theorem 4.6.3]{D19}, 
	to prove that $\var(L_{K_n, a_n}) = \E(L_{K_n, a_n}^2)\converge 0$, it suffices to show that $\{L_{K_n, a_n}^2: n\ge 1\}$ is uniformly integrable. 
	From \citet[][Lemma A5]{LDR18}, 
	$|L_{K_n, a_n}|$ is stochastically smaller than or equal to $|\varepsilon|$. 
	Because $x\I(x>c)$ is a nondecreasing function of $x\in [0, \infty)$ for any given $c>0$, by the property of stochastic ordering, we have 
	$
	\E\{ L_{K_n, a_n}^2\I(L_{K_n, a_n}^2 > c) \}
	\le 
	\E\{ \varepsilon^2\I(\varepsilon^2 > c) \}. 
	$
	By the dominated convergence theorem, this further implies that 
	\begin{align*}
		\lim_{c\converge \infty}\left( \sup_{n \ge 1} \E\{ L_{K_n, a_n}^2\I(L_{K_n, a_n}^2 > c) \} \right)
		\le 
		\lim_{c\converge \infty}
		\E\{ \varepsilon^2\I(\varepsilon^2 > c) \}
		= 0,  
	\end{align*}
	i.e., $\{L_{K_n, a_n}^2: n\ge 1\}$ is uniformly integrable. 
	From the above, Proposition \ref{prop:convpvar} holds. 
\end{proof}

\section{Asymptotics for Optimal Rerandomization}\label{sec:asym_optimal_rerand}

\begin{lemma}\label{lemma:bound_inf_norm_two_rv}
	For any two random variables $\psi$ and $\tilde{\psi}$, and any constant $\delta > 0$, 
	\begin{align*}
		\sup_{c\in \mathbb{R}}
		\big| \PP(\tilde{\psi} \le c) - \PP(\psi \le c) \big|
		& \le 
		\PP(|\tilde{\psi} - \psi| > \delta) 
		+ 
		\sup_{b\in \mathbb{R}}\PP(b < \psi \le b+\delta)
	\end{align*}
\end{lemma}

\begin{proof}[Proof of Lemma \ref{lemma:bound_inf_norm_two_rv}]
	For any $c\in \mathbb{R}$ and $\delta>0$, we have 
	\begin{align*}
		\PP(\psi \le c) 
		& = 
		\PP(\psi \le c - \delta) + \PP(c-\delta < \psi \le c)
		\\
		& 
		\le 
		\PP(\psi \le c - \delta, |\tilde{\psi} - \psi| \le \delta) + 
		\PP(|\tilde{\psi} - \psi| > \delta) + 
		\PP(c-\delta < \psi \le c)
		\\
		& 
		\le 
		\PP(\tilde{\psi} \le c) + \PP(|\tilde{\psi} - \psi| > \delta) + 
		\PP(c-\delta < \psi \le c)
	\end{align*}
	and 
	\begin{align*}
		\PP(\psi \le c) 
		& = 
		\PP(\psi \le c+\delta) - \PP(c<\psi \le c+\delta) 
		\ge 
		\PP(\psi \le c+\delta, |\tilde{\psi}-\psi| \le \delta) - \PP(c<\psi \le c+\delta) 
		\\
		& \ge 
		\PP(\tilde{\psi} \le c, |\tilde{\psi}-\psi| \le \delta) - \PP(c<\psi \le c+\delta) 
		\\
		& \ge 
		\PP(\tilde{\psi} \le c) - \PP(|\tilde{\psi} - \psi| > \delta) - \PP(c<\psi \le c+\delta).
	\end{align*}
	These imply that 
	\begin{align*}
		\big| \PP(\tilde{\psi} \le c) - \PP(\psi \le c) \big|
		& \le 
		\PP(|\tilde{\psi} - \psi| > \delta) 
		+ 
		\sup_{b\in \mathbb{R}}\PP(b < \psi \le b+\delta).
	\end{align*}
	Taking supremum over $c$, we then derive Lemma \ref{lemma:bound_inf_norm_two_rv}.  
\end{proof}

\begin{lemma}\label{lemma:diminish_inf_norm_two_rv}
	Let $\{{\psi}_n\}$ and $\{\tilde{\psi}_n\}$ be two sequence of random variables satisfying that 
	$\psi_n = \beta_n \varepsilon_0 + \zeta_n$ and $\psi_n - \tilde{\psi}_n = o_{\PP}(\beta_n)$, 
	where $\{\beta_n\}$ is a sequence of positive constants, 
	$\{\zeta_n\}$ is a sequence of random variables independent of  
	$\varepsilon_0$, 
	and $\varepsilon_0$ is a random variable with bounded density. 
	Then we have, as $n\rightarrow \infty$, 
	$
	\sup_{c\in \mathbb{R}}
	| \PP(\tilde{\psi}_n \le c) - \PP(\psi_n \le c) | \rightarrow 0.
	$
\end{lemma}

\begin{proof}[Proof of Lemma \ref{lemma:diminish_inf_norm_two_rv}]
	For any constant $\eta > 0$, 
	using Lemma \ref{lemma:bound_inf_norm_two_rv} with $\delta = \beta_n \eta$, we have 
	\begin{align}\label{eq:inf_norm_two_rv_bound}
		\sup_{c\in \mathbb{R}}
		\big| \PP(\tilde{\psi}_n \le c) - \PP(\psi_n \le c) \big|
		& \le 
		\PP(|\tilde{\psi}_n - \psi_n| > \beta_n \eta ) 
		+ 
		\sup_{b\in \mathbb{R}}\PP(b < \psi_n \le b+\beta_n \eta)
		\nonumber
		\\
		& = 
		\PP(|\tilde{\psi}_n - \psi_n|/\beta_n > \eta ) 
		+ 
		\sup_{b\in \mathbb{R}}\PP(b < \psi_n/\beta_n \le b+ \eta). 
	\end{align}
	Below we consider the two terms in \eqref{eq:inf_norm_two_rv_bound}, separately. 
	First, by the fact that $\psi_n - \tilde{\psi}_n = o_{\PP}(\beta_n)$, the first term in \eqref{eq:inf_norm_two_rv_bound} satisfies that 
	$\PP(|\tilde{\psi}_n - \psi_n|/\beta_n > \eta ) \rightarrow 0$ as $n \rightarrow \infty$. 
	Second, let $C$ be the upper bound of the density of $\varepsilon_0$.  
	For any $b\in \mathbb{R}$, 
	we then have,
	\begin{align*}
		\PP(b < \psi_n/\beta_n \le b+ \eta \mid \zeta_n)
		& = 
		\PP(b -  \zeta_n/\beta_n < \varepsilon_0 \le b -  \zeta_n/\beta_n + \eta \mid \zeta_n)
		\le C \eta, 
	\end{align*}
	and thus, by the law of iterated expectation, 
	\begin{align*}
		\PP(b < \psi_n/\beta_n \le b+ \eta)
		& = 
		\E\{\PP(b < \psi_n/\beta_n \le b+ \eta \mid \zeta_n)\}
		\le C \eta. 
	\end{align*} 
	Consequently, we have 
	$\sup_{b\in \mathbb{R}}\PP(b < \psi_n/\beta_n \le b+ \eta) \le C\eta$. 
	
	From the above, for any constant $\eta>0$, 
	\begin{align*}
		\limsup_{n\rightarrow \infty}\sup_{c\in \mathbb{R}}
		\big| \PP(\tilde{\psi}_n \le c) - \PP(\psi_n \le c) \big|
		& \le 
		\limsup_{n\rightarrow \infty}
		\PP(|\tilde{\psi}_n - \psi_n|/\beta_n > \eta ) 
		+ 
		C \eta 
		\le C\eta. 
	\end{align*}
	Because the above inequality holds for any $\eta>0$, we must have 
	$\limsup_{n\rightarrow \infty}\sup_{c\in \mathbb{R}}
	| \PP(\tilde{\psi}_n \le c) - \PP(\psi_n \le c) | = 0$. 
	Therefore, Lemma \ref{lemma:diminish_inf_norm_two_rv} holds. 
\end{proof}

\begin{proof}[\bf Proof of Theorem~\ref{thm:rem_gaussian}]
	From Theorem \ref{thm:v_Ka}(i) and Proposition \ref{prop:convpvar}, under Condition \ref{cond:k_np_n}, we must have $L_{K_n, a_n} = o_{\PP}(1)$. 
	From Condition \ref{cond:rsup}, we can know that for sufficiently large $n$, 
	$1-R_n^2$ is greater than certain positive constant, 
	and $\sqrt{R_n^2} \ L_{K_n, a_n} = o_{\PP}(\sqrt{1-R_n^2})$. 
	Using Lemma \ref{lemma:diminish_inf_norm_two_rv} with $\psi_n =\sqrt{1-R_n^2} \ \varepsilon_0$ and $\tilde{\psi}_n = \sqrt{1-R_n^2} \ \varepsilon_0 + \sqrt{R_n^2} \ L_{K_n, a_n}$, we then have 
	\begin{align*}
		\sup_{c\in \mathbb{R}}\bigg|\PP\left( \sqrt{1-R_n^2} \ \varepsilon_0 \le c \right) -&\; \PP\left\{ \big(\sqrt{1-R_n^2} \ \varepsilon_0  + \sqrt{R_n^2} \ L_{K_n, a_n} \big) \le c \right\} \bigg|  \converge 0.
	\end{align*}
	From Theorem~\ref{thm:dim_rem}, as $n\rightarrow \infty$, 
	\begin{align*}
		& \quad \ \sup_{c\in \mathbb{R}}  \left| \PP \left\{ V_{\tau\tau}^{-1/2}( \hat{\tau} - \tau) \le c \mid M \le a_n \right\} - \PP\left\{ \sqrt{1-R_n^2} \ \varepsilon_0 \le c \right\}
		\right| \\
		& \leq \sup_{c\in \mathbb{R}}\left| \PP \left\{ V_{\tau\tau}^{-1/2}( \hat{\tau} - \tau) \le c \mid M \le a_n \right\} - \PP\left\{ \big(\sqrt{1-R_n^2} \ \varepsilon_0  + \sqrt{R_n^2} \ L_{K_n, a_n} \big) \le c \right\}
		\right| \\
		& + \sup_{c\in \mathbb{R}}\left|\PP\left\{ \sqrt{1-R_n^2} \ \varepsilon_0 \le c \right\} - \PP\left\{ \big(\sqrt{1-R_n^2} \ \varepsilon_0  + \sqrt{R_n^2} \ L_{K_n, a_n} \big) \le c \right\} \right|
		\\
		& \rightarrow 0.
	\end{align*}
	Therefore, Theorem~\ref{thm:rem_gaussian} holds. 
\end{proof}

\begin{proof}[\bf Proof of Theorem~\ref{thm:K_n_Delta_n}]
	Note that Condition \ref{cond:p_n} is that $p_n/\Delta_n \rightarrow \infty$ as $n\rightarrow \infty$, 
	and Condition \ref{cond:k_np_n} is that $\log(p_n^{-1})/K_n \rightarrow \infty$ as $n\rightarrow \infty$. 
	Below we prove Theorem \ref{thm:K_n_Delta_n}(i)--(iv) respectively. 
	
	First, we prove (i). 
	Consider first the ``only if'' part. 
	If both Conditions  \ref{cond:p_n} and \ref{cond:k_np_n} hold for some sequence $\{p_n\}$, then we must have, for sufficiently large $n$, 
	$
	\log (\Delta_n^{-1})/K_n 
	=
	\log (p_n^{-1})/K_n + \log(p_n/\Delta_n)/K_n  \ge \log (p_n^{-1})/K_n, 
	$
	which must imply that $\log (\Delta_n^{-1})/K_n  \rightarrow \infty$ as $n\rightarrow \infty$. 
	Consider then the ``if'' part. 
	Because $\log (\Delta_n^{-1})/K_n  \rightarrow \infty$ as $n\rightarrow \infty$, 
	we can construct a sequence $\{p_n\}$ such that, as $n\rightarrow \infty$,  $\Delta_n/p_n \rightarrow 0$, and $ \log(\Delta_n/p_n)/K_n+ \log (\Delta_n^{-1})/K_n \rightarrow \infty$. 
	For such a choice of $\{p_n\}$, Condition \ref{cond:p_n} holds obviously, and 
	$
	\log(p_n^{-1})/K_n = \log(\Delta_n/p_n)/K_n+ \log (\Delta_n^{-1})/K_n \rightarrow \infty, 
	$
	i.e., Condition \ref{cond:k_np_n} holds. 
	
	Second, we prove (ii). 
	For any sequence $\{p_n\}$ such that Condition \ref{cond:p_n} holds, we have 
	\begin{align*}
		\limsup_{n\rightarrow \infty} \frac{\log(p_n^{-1})}{K_n} 
		\le  
		\limsup_{n\rightarrow \infty} \frac{\log(\Delta_n/p_n)}{K_n} + \limsup_{n\rightarrow \infty} \frac{\log (\Delta_n^{-1})}{K_n}
		\le 
		\limsup_{n\rightarrow \infty} \frac{\log (\Delta_n^{-1})}{K_n}
		< \infty.
	\end{align*}
	Fro Theorem \ref{thm:v_Ka}(ii), this further implies that $\liminf_{n\rightarrow \infty} v_{K_n, a_n} >0$.
	
	Third, we prove (iii). 
	Because Condition \ref{cond:gamma_n} holds, from Theorem \ref{thm:berry_esseen_clt}, 
	we can construct a sequence $\{p_n\}$ such that $\lim_{n\rightarrow \infty} p_n/\Delta_n= \infty$ and 
	$ \limsup_{n\rightarrow \infty}\log(p_n/\Delta_n)/ \log (\Delta_n^{-1}) < 1 - c$ for some $c>0$.  
	This then implies that 
	\begin{align*}
		\liminf_{n\rightarrow\infty}\frac{\log(p_n^{-1})}{K_n} = 
		\liminf_{n\rightarrow\infty} \left[\frac{\log(\Delta_n^{-1})}{K_n} \left\{1 -  \log(p_n/\Delta_n)/\log(\Delta_n^{-1}) \right\}\right]
		\ge 
		c \liminf_{n\rightarrow\infty} \frac{\log(\Delta_n^{-1})}{K_n} > 0. 
	\end{align*}
	From Theorem \ref{thm:v_Ka}(iii), we then have $ \limsup_{n\rightarrow \infty} v_{K_n, a_n} < 1$. 
	
	Fourth, we prove (iv). 
	For any sequence $\{p_n\}$ such that Condition \ref{cond:p_n} holds, we have 
	\begin{align*}
		\limsup_{n\rightarrow \infty} \frac{\log(p_n^{-1})}{K_n} 
		\le  
		\limsup_{n\rightarrow \infty} \frac{\log(\Delta_n/p_n)}{K_n} + \limsup_{n\rightarrow \infty} \frac{\log (\Delta_n^{-1})}{K_n}
		\le 
		\limsup_{n\rightarrow \infty} \frac{\log (\Delta_n^{-1})}{K_n}
		= 0. 
	\end{align*}
	From Theorem \ref{thm:v_Ka}(iv), this further implies that $v_{K_n, a_n} \rightarrow 0$ as $n\rightarrow \infty$. 
\end{proof}

\section{Asymptotic Validity of Confidence Intervals}\label{sec:asym_CI}

\subsection{Technical lemmas}

For descriptive convenience, throughout this section, we define $a/b$ as $+\infty$ when $a>0$ and $b = 0$. 

\begin{lemma}\label{lem:scineq}
	Let $\{u_i\in \mathbb{R}: i=1,2,\ldots,n\}$ be a finite population of $N>0$ units,
	with $\bar{u} = N^{-1} \sum_{i=1}^N u_i$ and 
	$\sigma^2_{u} = N^{-1} \sum_{i=1}^N (u_i - \bar{u})^2$. 
	Let $(Z_1, \cdots, Z_N)$ denote a sampling indicator vector for a simple random sample of size $m>0$, 
	and 
	$\hat{u} = m^{-1} \sum_{i=1}^N Z_i u_i$ denote the corresponding sample average. 
	Define $f=m/N$. 
	Then for any $t > 0$, 
	\begin{align*}
		\PP\left( 
		\left| \hat{u} - \bar{u} \right|
		\ge 
		t 
		\right)
		& 
		\le 
		2 
		\exp\left(
		- \frac{70^2}{71^2} \frac{N f^2 t^2}{\sigma^2_u}
		\right). 
	\end{align*}
\end{lemma}

\begin{lemma}\label{lemma:s_uw_scalar}
	Let $\{ (u_i, w_i)\in \mathbb{R}^2: i = 1, 2, \ldots, N \}$ be a finite population of $N\ge 2$ units, 
	with finite population averages and covariance 
	$\bar{u} \equiv N^{-1} \sum_{i=1}^N u_i$, $\bar{w} = N^{-1} \sum_{i=1}^N w_i$ and 
	$S_{uw} = (N-1)^{-1} \sum_{i=1}^N (u_i - \bar{u}) (w_i - \bar{w})$. 
	Let $(Z_1, \cdots, Z_N)$ denote a sampling indicator vector for a simple random sample of size $m\ge 2$, 
	with corresponding sample averages and covariance
	$\hat{u} = m^{-1} \sum_{i=1}^N Z_i u_i$, $\hat{w} = m^{-1} \sum_{i=1}^N Z_i w_i$ and 
	$s_{uw} = (m-1)^{-1} \sum_{i=1}^N Z_i (u_i - \hat{u}) (w_i - \hat{w})$. 
	Define $f=m/N$,
	\begin{align*}
		\Delta_{u} = \hat{u} - \bar{u}, 
		\quad
		\Delta_{w} = \hat{w} - \bar{w},
		\quad
		\Delta_{uw} = 
		\frac{1}{m}\sum_{i=1}^N Z_i (u_i - \bar{u}) (w_i - \bar{w})
		- 
		\frac{N-1}{N} S_{uw}, 
	\end{align*}
	and 
	\begin{align*}
		\sigma^2_u = \frac{1}{N} \sum_{i=1}^N(u_i-\bar{u})^2, \ \ 
		\sigma^2_w = \frac{1}{N} \sum_{i=1}^N(w_i-\bar{w})^2,
		\ \ 
		\sigma^2_{u\times w} = \frac{1}{N} \sum_{i=1}^N\left\{ 
		(u_i - \bar{u}) (w_i - \bar{w})- \frac{N-1}{N}S_{uw}
		\right\}^2. 
	\end{align*}
	Then 
	$
	\left| s_{uw} -  S_{uw} \right|
	\le 
	2 | \Delta_{u\times w} | + 2 |\Delta_{u}| |\Delta_{w}|
	+ 2(1-f)| S_{uw} |/m,  
	$
	and for any $t > 0$, 
	\begin{align*}
		\PP\left( \left|\Delta_u\right| \ge t \right)
		& \le 
		2 
		\exp\left(
		- \frac{70^2}{71^2} \frac{N f^2 t^2}{\sigma^2_u}
		\right),
		& 
		\PP\left( \left|\Delta_w\right| \ge t \right)
		\le 
		2 
		\exp\left(
		- \frac{70^2}{71^2} \frac{N f^2 t^2}{\sigma^2_w}
		\right), 
		\\
		\PP\left( \left| \Delta_{u\times w}
		\right| \ge t \right)
		& \le 
		2 
		\exp\left(
		- \frac{70^2}{71^2} \frac{N f^2 t^2}{\sigma^2_{u\times w}}
		\right). 
	\end{align*}
\end{lemma}

\begin{lemma}\label{lemma:s_uw_vector}
	Let $\{ (u_i, \bs{w}_i^\top)\in \mathbb{R}^{1+K}: i = 1, 2, \ldots, N \}$ be a finite population of $N\ge 2$ units, 
	with 
	$\bs{w}_i = (w_{1i}, w_{2i}, \ldots w_{Ki})^\top$ and 
	finite population averages and covariance 
	$\bar{u} \equiv N^{-1} \sum_{i=1}^N u_i$, 
	$\bar{\bs{w}} = (\bar{w}_1, \ldots, \bar{w}_K)^\top = N^{-1} \sum_{i=1}^N \bs{w}_i$ and 
	$\bs{S}_{u\bs{w}} = (S_{uw_1}, \ldots, S_{uw_K}) =  (N-1)^{-1} \sum_{i=1}^N (u_i - \bar{u}) (\bs{w}_i - \bar{\bs{w}})^\top$. 
	Let $(Z_1, \cdots, Z_N)$ denote a sampling indicator vector for a simple random sample of size $m\ge 2$, 
	with corresponding sample averages and covariance
	$\hat{u} = m^{-1} \sum_{i=1}^N Z_i u_i$, 
	$\hat{\bs{w}} = m^{-1} \sum_{i=1}^N Z_i \bs{w}_i$ and 
	$\bs{s}_{u\bs{w}} = (s_{uw_1}, \ldots, s_{uw_K}) = (m-1)^{-1} \sum_{i=1}^N Z_i (u_i - \hat{u}) (\bs{w}_i - \hat{\bs{w}})^\top$. 
	Let $f=m/N$, and for $1\le k\le K$, define 
	\begin{align*}
		\Delta_{u} = \hat{u} - \bar{u}, 
		\quad
		\Delta_{w_k} = \hat{w}_k - \bar{w}_k,
		\quad
		\Delta_{uw_k} = 
		\frac{1}{m}\sum_{i=1}^N Z_i (u_i - \bar{u}) (w_{ki} - \bar{w}_k)
		- 
		\frac{N-1}{N} S_{uw_k}, 
	\end{align*}
	and 
	\begin{align*}
		\sigma^2_u = \frac{1}{N} \sum_{i=1}^N(u_i-\bar{u})^2, \ \ 
		\sigma^2_{w_k} = \frac{1}{N} \sum_{i=1}^N(w_{ki}-\bar{w}_k)^2,
		\ \ 
		\sigma^2_{u\times w_k} = \frac{1}{N} \sum_{i=1}^N\left\{ 
		(u_i - \bar{u}) (w_{ki} - \bar{w}_k)- \frac{N-1}{N}S_{uw_k}
		\right\}^2. 
	\end{align*}
	Then 
	\begin{align*}
		\left\| \bs{s}_{u\bs{w}} - \bs{S}_{u\bs{w}}  \right\|^2_2
		\le 
		12 \sum_{k=1}^K \Delta_{u\times w_k}^2 + 12 \Delta_u^2 \sum_{k=1}^K \Delta_{w_k}^2 
		+ 
		\frac{12(1-f)^2}{m^2} \sum_{k=1}^K  S_{uw_k}^2, 
	\end{align*}
	and for any $t > 0$, 
	\begin{align*}
		\PP\left( \Delta_u^2 \ge t \right)
		& \le 
		2 
		\exp\left(
		- \frac{70^2}{71^2} \frac{N f^2 t}{\sigma^2_u}
		\right), 
		\qquad \quad
		\PP\left( \sum_{k=1}^K \Delta_{w_k}^2 \ge t \right)
		\le 
		2 K 
		\exp\left(
		- \frac{70^2}{71^2} \frac{N f^2  t}{\sum_{k=1}^K \sigma^2_{w_k}}
		\right), \\
		\PP\left( \sum_{k=1}^K \Delta_{u\times w_k}^2 \ge t \right)
		& 
		\le 
		2 K 
		\exp\left(
		- \frac{70^2}{71^2} \frac{N f^2  t}{\sum_{k=1}^K \sigma^2_{u\times w_k}}
		\right). 
	\end{align*}
\end{lemma}

\begin{lemma}\label{lemma:s_uw_re}
	Consider the same setting as in Lemma \ref{lemma:s_uw_vector} and any event $\bs{Z} \in \mathcal{E} \subset \{0,1\}^N$ with positive probability $p = \PP(\bs{Z} \in \mathcal{E})$. 
	Define 
	\begin{align*}
		\xi
		& = 
		\frac{\max\{1, \log K, - \log p\}}{Nf^2} \sum_{k=1}^K \sigma^2_{u\times w_k} 
		+
		\frac{ \max\{1, -\log p\} \cdot \max\{1, \log K, - \log p\} }{N^2f^4} \sigma^2_u \sum_{k=1}^K \sigma^2_{w_k}
		\\
		& \quad \ + 
		\frac{(1-f)^2}{N^2 f^2} \sum_{k=1}^K  S_{uw_k}^2. 
	\end{align*}
	Then 
	for any $t \ge 3 \cdot 71^2/70^2$, 
	\begin{align*}
		\PP
		\left( 
		\left\| \bs{s}_{u\bs{w}} - \bs{S}_{u\bs{w}}  \right\|^2_2 > 36 t^2  \xi
		\mid \bs{Z} \in 
		\mathcal{E}
		\right)
		& \le 
		6
		\exp\left(
		- \frac{1}{3} \frac{70^2}{71^2} t
		\right). 
	\end{align*}
\end{lemma}

\begin{lemma}\label{lemma:V_R2_hat_bound}
	Under ReM with threshold $a_n$, 
	along the sequence of finite populations with increasing sample size $n$, 
	if $\min\{n_1, n_0\} \ge 2$ when $n$ is sufficiently large, then 
	the estimators $\hat{V}_{\tau\tau}$ and $\hat{R}^2$ satisfy that 
	\begin{align*}
		\hat{V}_{\tau\tau} - V_{\tau\tau} - n^{-1} S_{\tau\setminus \bs{X}}^2
		& = 
		O_{\PP}\left(
		\frac{\xi_{11}^{1/2}}{n_1} + 
		\frac{\xi_{00}^{1/2}}{n_0} + 
		\frac{\xi_{1\bs{w}} + \xi_{0\bs{w}}}{n} + \left\| S_{1\bs{w}} - S_{0\bs{w}} \right\|_2\frac{\xi_{1\bs{w}}^{1/2} + \xi_{0\bs{w}}^{1/2}}{n} 
		\right), 
	\end{align*}
	and 
	\begin{align*}
		\hat{V}_{\tau\tau} \hat{R}^2_n - V_{\tau\tau} R^2_n
		& = 
		O_{\PP}\left(
		\frac{\xi_{1\bs{w}}}{n_1} 
		+ 
		\frac{\xi_{0\bs{w}}}{n_0}
		+ \left\| S_{1\bs{w}} \right\|_2 \frac{\xi_{1\bs{w}}^{1/2}}{n_1}
		+ \left\| S_{0\bs{w}} \right\|_2 \frac{\xi_{0\bs{w}}^{1/2}}{n_1} 
		+ 
		\left\| S_{1\bs{w}} - S_{0\bs{w}} \right\|_2 \frac{\xi_{1\bs{w}}^{1/2} + \xi_{0\bs{w}}^{1/2}}{n}
		\right), 
	\end{align*}
	where 
	$\bs{w}_i = (w_{1i}, \ldots, w_{K_n i})^\top = \bs{S}_{\bs{X}}^{-1}(\bs{X}_i - \bar{\bs{X}})$ is the standardized covariates, 
	$S_{z\bs{w}} = (S_{zw_1}, \ldots, S_{zw_{K_n}})$ is the finite population covariance between $Y(z)$ and $\bs{w}$, 
	\begin{align*}
		\xi_{zz}
		& = 
		\frac{\max\{1, - \log \tilde{p}_n \}}{n r_z^2} \sigma^2_{z\times z} 
		+
		\frac{ \max\{1, (-\log \tilde{p}_n)^2 \} }{n^2 r_z^4} \sigma^4_z 
		+ 
		\frac{(1-r_z)^2}{n^2 r_z^2} S_{z}^4, 
		\\
		\xi_{z \bs{w}}
		& = 
		\frac{\max\{1, \log K_n, - \log \tilde{p}_n\}}{n r_z^2} \sum_{k=1}^K \sigma^2_{z\times w_k} 
		+
		\frac{ \max\{1, -\log \tilde{p}_n \} \cdot \max\{1, \log K_n, - \log \tilde{p}_n \} }{n^2 r_z^4} \sigma^2_u \sum_{k=1}^{K_n} \sigma^2_{w_k}
		\\
		& \quad \ + 
		\frac{(1-r_z)^2}{n^2 r_z^2} \sum_{k=1}^{K_n} S_{z w_k}^2,  
	\end{align*}
	$\tilde{p}_n = \PP(M \le a_n)$ is the actual acceptance probability under ReM, and 
	\begin{align*}
		\sigma^2_z & = \frac{1}{n} \sum_{i=1}^n \{Y_i(z) - \bar{Y}(z)\}^2 = \frac{n-1}{n} S_z^2, 
		\qquad
		\sigma_{w_k}^2 
		= \frac{1}{n} \sum_{i=1}^n ( w_{ki} - \bar{w}_k )^2  = \frac{n-1}{n}, 
		\\
		\sigma^2_{z\times z} 
		& = 
		\frac{1}{n} \sum_{i=1}^n\Big[ 
		\{Y_i(z) - \bar{Y}(z) \}^2 - \sigma^2_z
		\Big]^2,
		\ \ 
		\sigma^2_{z\times w_k} 
		= 
		\frac{1}{n} \sum_{i=1}^n\Big[ 
		\{Y_i(z) - \bar{Y}(z)\} (w_{ki} - \bar{w}_k)- \frac{n-1}{n} S_{z w_k}
		\Big]^2. 
	\end{align*}
\end{lemma}

\begin{lemma}\label{lemma:V_R2_hat_bound_simp}
	Under the same setting as Lemma \ref{lemma:V_R2_hat_bound}, 
	if $\max\{1, \log K_n, - \log \tilde{p}_n \} = O(nr_1^2 r_0^2)$, 
	then 
	\begin{align*}
		& \quad \ \max\left\{ \big| \hat{V}_{\tau\tau} - V_{\tau\tau} - n^{-1} S_{\tau\setminus \bs{X}}^2 \big|, \ \ \big| \hat{V}_{\tau\tau} \hat{R}^2_n - V_{\tau\tau} R^2_n \big| \right\}
		\\
		& 
		= 
		\max_{z\in \{0,1\}}\max_{1\le i \le n}\{Y_i(z) - \bar{Y}(z)\}^2 \cdot 
		O_{\PP}\left( 
		\max\{K_n, 1\} \cdot \frac{\sqrt{ \max\{1, \log K_n, - \log \tilde{p}_n\} }}{n^{3/2} r_1^2r_0^2} \right). 
	\end{align*}
\end{lemma}

\begin{lemma}\label{lemma:cond_infer}
	Under the same setting as Lemmas \ref{lemma:V_R2_hat_bound} and \ref{lemma:V_R2_hat_bound_simp}, 
	\begin{itemize}
		\item[(i)] if Condition \ref{cond:p_n} holds, then $\max\{1, -\log \tilde{p}_n\} = O(\max\{1, -\log p_n\})$, 
		recalling that $\tilde{p}_n = \PP(M\le a_n)$ is the actually acceptance probability under ReM, while $p_n = \PP(\chi^2_{K_n} \le a_n)$ is the approximate acceptance probability; 
		\item[(ii)]
		$\max_{z\in \{0,1\}}\max_{1\le i \le n}\{Y_i(z) - \bar{Y}(z)\}^2/
		(r_0 S^2_{1\setminus \bs{X}} + r_1 S^2_{0\setminus \bs{X}})
		\ge 1/2$; 
		\item[(iii)] 
		if Conditions \ref{cond:p_n} and \ref{cond:infer} hold, 
		then,  $\max\{1, \log K_n, - \log \tilde{p}_n \} = o(nr_1^2 r_0^2)$. 
	\end{itemize}
\end{lemma}

\subsection{Proofs of the lemmas}

\begin{proof}[Proof of Lemma \ref{lem:scineq}]
	When $\sigma^2_u = 0$, $u_1 = \ldots = u_N = \bar{u}$, and thus $\hat{u} - \bar{u}$ must be a constant zero, under which Lemma \ref{lem:scineq} holds obviously. 
	Below we consider only the case where $\sigma^2_u > 0$.
	From \citet[][Lemma S1]{BLZ16}, for any $t > 0$, 
	\begin{align*}
		\PP\left( 
		\left| \hat{u} - \bar{u} \right|
		\ge 
		t 
		\right)
		& 
		= 
		\PP\left( 
		\hat{u} - \bar{u}
		\ge 
		t 
		\right)
		+ 
		\PP\left\{
		(-\hat{u}) - (-\bar{u})
		\ge 
		t 
		\right\}
		\le 
		2 
		\exp\left(
		- \frac{f m t^2}{(1+c)^2 \sigma^2_u}
		\right), 
	\end{align*}
	where $c \equiv \min \{1 / 70, (3 f)^2 / 70, (3 - 3 f)^2 / 70\} \le 1/70$. 
	This then implies that for any $t > 0$, 
	\begin{align*}
		\PP\left( 
		\left| \hat{u} - \bar{u} \right|
		\ge 
		t 
		\right)
		& 
		\le 
		2 
		\exp\left(
		- \frac{f m t^2}{(1+c)^2 \sigma^2_u}
		\right)
		\le 
		2 
		\exp\left(
		- \frac{70^2}{71^2} \frac{N f^2 t^2}{\sigma^2_u}
		\right), 
	\end{align*}
	i.e., Lemma \ref{lem:scineq} holds.
\end{proof}

\begin{proof}[Proof of Lemma \ref{lemma:s_uw_scalar}]
	First, by definition, the sample covariance between $u$ and $w$ has the following equivalent forms:
	\begin{align*}
		s_{uw} & = 
		\frac{1}{m-1}\sum_{i=1}^N Z_i (u_i - \hat{u}) (w_i - \hat{w}) 
		= 
		\frac{m}{m-1} \frac{1}{m}\sum_{i=1}^N Z_i (u_i - \bar{u}) (w_i - \bar{w}) 
		- 
		\frac{m}{m-1} (\hat{u} - \bar{u}) (\hat{w} -  \bar{w})\\
		& = 
		\frac{m}{m-1} 
		\left\{ \frac{1}{m}\sum_{i=1}^N Z_i (u_i - \bar{u}) (w_i - \bar{w})
		- 
		\frac{N-1}{N} S_{uw}
		\right\}
		- 
		\frac{m}{m-1} (\hat{u} - \bar{u}) (\hat{w} -  \bar{w})
		+  
		\frac{m(N-1)}{(m-1)N} S_{uw}. 
	\end{align*}
	Consequently, we can bound the difference between $s_{uw}$ and $S_{uw}$ by 
	\begin{align*}
		& \quad \ \left| s_{uw} -  S_{uw} \right|
		\\
		& = 
		\left|
		\frac{m}{m-1} 
		\left\{ \frac{1}{m}\sum_{i=1}^N Z_i (u_i - \bar{u}) (w_i - \bar{w})
		- 
		\frac{N-1}{N} S_{uw}
		\right\}
		- 
		\frac{m}{m-1} (\hat{u} - \bar{u}) (\hat{w} -  \bar{w})
		+  
		\frac{1-f}{m-1} S_{uw}
		\right|\\
		& 
		\le 
		2 \left| \Delta_{u\times w}
		\right| + 2
		\left|\Delta_{u}\right| \left|\Delta_{w}\right|
		+ \frac{2(1-f)}{m} \left| S_{uw} \right|, 
	\end{align*}
	where the last inequality holds because $m/(m-1) \le 2$. 
	
	Second, 
	applying Lemma \ref{lem:scineq} to the finite populations of $\{u_i\}_{i=1}^n$, $\{w_i\}_{i=1}^n$ and $\{(u_i - \bar{u}) (w_i - \bar{w})\}_{i=1}^n$, 
	we can immediately derive the probability bounds for $\Delta_{u}$, $\Delta_{w}$ and $\Delta_{u\times w}$.

	From the above, Lemma \ref{lemma:s_uw_scalar} holds. 
\end{proof}

\begin{proof}[Proof of Lemma \ref{lemma:s_uw_vector}]
	First, we consider the bound for $\| \bs{s}_{u\bs{w}} - \bs{S}_{u\bs{w}} \|^2_2$. 
	From Lemma \ref{lemma:s_uw_scalar} and the Cauchy-Schwartz inequality, 
	\begin{align*}
		\left\| \bs{s}_{u\bs{w}} - \bs{S}_{u\bs{w}}  \right\|^2_2 
		& = \sum_{k=1}^K 
		\left( s_{uw_k} - S_{u w_k}  \right)^2
		\le 
		4\sum_{k=1}^K 
		\left( | \Delta_{u\times w_k} | + |\Delta_{u}| |\Delta_{w_k}|
		+ (1-f)| S_{uw_k} |/m \right)^2
		\\
		& \le 
		12 \sum_{k=1}^K 
		\left( \Delta_{u\times w_k}^2 + \Delta_{u}^2 \Delta_{w_k}^2
		+ (1-f)^2 S_{uw_k}^2/m^2 \right)
		\\
		& = 
		12 \sum_{k=1}^K \Delta_{u\times w_k}^2 + 12 \Delta_u^2 \sum_{k=1}^K \Delta_{w_k}^2 
		+ 
		\frac{12(1-f)^2}{m^2} \sum_{k=1}^K  S_{uw_k}^2. 
	\end{align*}
	
	Second, the probability bound for $\Delta_u^2$ follows immediately from Lemma \ref{lemma:s_uw_scalar}. 
	
	Third, we consider the probability bound for $\sum_{k=1}^K \Delta_{w_k}^2$. 
	We consider two cases separately, depending on whether $\sum_{j=1}^K \sigma^2_{w_j}$ is positive. 
	When $\sum_{j=1}^K \sigma^2_{w_j}>0$,  
	we introduce $a_k = \sigma^2_{w_k}/\sum_{j=1}^K \sigma^2_{w_j}$ for $1\le k\le K$. 
	Obviously, $a_k\ge 0$ for all $k$ and $\sum_{k=1}^K a_k = 1$. 
	Note that if $a_k = 0$ for some $1\le k \le K$, then it follows from Lemma~\ref{lem:scineq} that the corresponding $\Delta_{w_k}$ is a constant zero. With this in mind,
	from Lemma \ref{lemma:s_uw_scalar}, we have that for any $t>0$, 
	\begin{align*}
		\PP\left( \sum_{k=1}^K \Delta_{w_k}^2 \ge t \right)
		& = 
		\PP\left( \sum_{k: a_k > 0} \Delta_{w_k}^2 \ge \sum_{k: a_k > 0} a_k t \right)
		\le 
		\sum_{k: a_k > 0} \PP\left( \Delta_{w_k}^2 \ge  a_k t \right)
		\le 
		2 
		\sum_{k: a_k > 0}
		\exp\left(
		- \frac{70^2}{71^2} \frac{N f^2 a_k t}{\sigma^2_{w_k}}
		\right)\\
		& = 
		2 
		\sum_{k: a_k > 0}
		\exp\left(
		- \frac{70^2}{71^2} \frac{N f^2  t}{\sum_{j=1}^K \sigma^2_{w_j}}
		\right)
		\le 
		2 K 
		\exp\left(
		- \frac{70^2}{71^2} \frac{N f^2  t}{\sum_{k=1}^K \sigma^2_{w_k}}
		\right). 
	\end{align*}
	When $\sum_{j=1}^K \sigma^2_{w_j} = 0$, 
	$\sum_{k=1}^K \Delta_{w_k}^2$ is a constant zero, under which the above probability bound holds obviously. 
	
	Fourth, we consider the probability bound for $\sum_{k=1}^K \Delta_{u\times w_k}^2$. 
	By the same logic as the proof above for the probability bound of $\sum_{k=1}^K \Delta_{w_k}^2$, we can derive that, for any $t>0$, 
	\begin{align*}
		\PP\left( \sum_{k=1}^K \Delta_{u\times w_k}^2 \ge t \right)
		& 
		\le 
		2 K 
		\exp\left(
		- \frac{70^2}{71^2} \frac{N f^2  t}{\sum_{k=1}^K \sigma^2_{u\times w_k}}
		\right). 
	\end{align*}
	
	From the above, 
	Lemma \ref{lemma:s_uw_vector} holds. 
\end{proof}

\begin{proof}[Proof of Lemma \ref{lemma:s_uw_re}]
	From Lemma \ref{lemma:s_uw_vector}, we have that for any $t > 0$, 
	\begin{align}\label{eq:bound_s_uw_re}
		& \quad \ \PP
		\left( 
		\left\| \bs{s}_{u\bs{w}} - \bs{S}_{u\bs{w}}  \right\|^2_2 > 36 t^2  \xi
		\mid \bs{Z} \in 
		\mathcal{E}
		\right)
		\nonumber
		\\
		& \le 
		\frac{
			\PP
			\left( 
			\left\| \bs{s}_{u\bs{w}} - \bs{S}_{u\bs{w}}  \right\|^2_2 > 36 t^2  \xi
			\right)
		}{
			\PP
			\left( 
			\bs{Z} \in 
			\mathcal{E}
			\right)
		}
		\le 
		\frac{1}{p}
		\PP
		\left( 
		12 \sum_{k=1}^K \Delta_{u\times w_k}^2 + 12 \Delta_u^2 \sum_{k=1}^K \Delta_{w_k}^2 
		+ 
		\frac{12(1-f)^2}{N^2 f^2} \sum_{k=1}^K  S_{uw_k}^2 > 36 t^2 \xi
		\right)
		\nonumber
		\\
		& 
		\le 
		\frac{1}{p}
		\PP\left( \sum_{k=1}^K \Delta_{u\times w_k}^2 > t^2 \xi \right)
		+ 
		\frac{1}{p}
		\PP\left( 
		\Delta_u^2 \sum_{k=1}^K \Delta_{w_k}^2  > t^2 \xi 
		\right)
		+ 
		\frac{1}{p} 
		\PP\left( 
		\frac{(1-f)^2}{N^2 f^2} \sum_{k=1}^K  S_{uw_k}^2 > t^2 \xi 
		\right). 
	\end{align}
	Below we consider the three terms in \eqref{eq:bound_s_uw_re} separately. 
	
	First, 
	we prove that, for any $t^2 \ge 3 \cdot 71^2/70^2$,  
	\begin{align}\label{eq:bound_s_uw_re_1}
		\frac{1}{p}
		\PP\left( \sum_{k=1}^K \Delta_{u\times w_k}^2 > t^2 \xi \right)
		& \le     
		2 \exp\left( -\frac{1}{3} \frac{70^2}{71^2} t^2 \right). 
	\end{align}
	Note that if $\sum_{k=1}^K \sigma^2_{u\times w_k}=0$, then $\sum_{k=1}^K \Delta_{u\times w_k}^2$ is 
	a constant zero 
	and the above inequality holds obviously.  
	Below we consider only the case where $\sum_{k=1}^K \sigma^2_{u\times w_k}>0$. 
	By definition, for any $t^2 \ge 3 \cdot 71^2/70^2$, 
	\begin{align*}
		\frac{70^2}{71^2}
		\frac{N f^2  t^2 \xi}{\sum_{k=1}^K \sigma^2_{u\times w_k}}
		- \log K + \log p
		& 
		\ge 
		\frac{70^2}{71^2} t^2 \max\{1, \log K, - \log p\} - \log K + \log p
		\\
		& \ge 
		\frac{70^2}{71^2} t^2 \frac{1 + \log K - \log p}{3} - \log K + \log p\\
		& \ge \frac{1}{3} \frac{70^2}{71^2} t^2. 
	\end{align*}
	Thus,  
	from Lemma \ref{lemma:s_uw_vector}, 
	for any $t^2 \ge 3 \cdot 71^2/70^2$, 
	\begin{align*}
		\frac{1}{p}
		\PP\left( \sum_{k=1}^K \Delta_{u\times w_k}^2 > t^2 \xi \right)
		& \le 
		2 \frac{K}{p} 
		\exp\left(
		- \frac{70^2}{71^2} \frac{N f^2  t^2 \xi}{\sum_{k=1}^K \sigma^2_{u\times w_k}}
		\right)
		=
		2 
		\exp\left(
		- \frac{70^2}{71^2} \frac{N f^2  t^2 \xi}{\sum_{k=1}^K \sigma^2_{u\times w_k}} + \log K - \log p
		\right)\\
		& 
		\le 
		2 \exp\left( -\frac{1}{3} \frac{70^2}{71^2} t^2 \right). 
	\end{align*}
	
	Second, we prove that, for any $t \ge 3 \cdot 71^2/70^2 $, 
	\begin{align}\label{eq:bound_s_uw_re_2}
		\frac{1}{p}
		\PP\left( 
		\Delta_u^2 \sum_{k=1}^K \Delta_{w_k}^2  > t^2 \xi 
		\right)
		\le 
		4
		\exp\left(
		- \frac{1}{3} \frac{70^2}{71^2} t
		\right). 
	\end{align}
	Note that if $\sigma^2_u=0$ or $\sum_{k=1}^K \sigma^2_{w_k}=0$, 
	then $\Delta_u^2 \sum_{k=1}^K \Delta_{w_k}^2$ is 
	a constant zero 
	and the above inequality holds obviously. 
	Below we consider only the case where both $\sigma^2_u$ and $\sum_{k=1}^K \sigma^2_{w_k}$ are positive. 
	By definition, 
	for any $t>0$, 
	\begin{align*}
		t^2 \xi
		& 
		\ge 
		t  \frac{\max\{1, -\log p\}}{N f^2} \sigma^2_u  \cdot
		t  \frac{\max\{1, \log K, - \log p\}}{Nf^2} \sum_{k=1}^K \sigma^2_{w_k}. 
	\end{align*}
	From Lemma \ref{lemma:s_uw_vector}, 
	this implies that,  for any $t>0$, 
	\begin{align}\label{eq:deltau_deltaw}
		& \quad \ \frac{1}{p}
		\PP\left( 
		\Delta_u^2 \sum_{k=1}^K \Delta_{w_k}^2  > t^2 \xi 
		\right)
		\\
		& \le 
		\frac{1}{p}
		\PP\left( 
		\Delta_u^2 > t \frac{\max\{1, -\log p\}}{N f^2} \sigma^2_u
		\right)
		+ 
		\frac{1}{p}
		\PP\left(
		\sum_{k=1}^K \Delta_{w_k}^2
		> t \frac{\max\{1, \log K, - \log p\}}{Nf^2} \sum_{k=1}^K \sigma^2_{w_k}
		\right) \nonumber
		\\
		& 
		\le 
		\frac{2}{p} 
		\exp\left(
		- \frac{70^2}{71^2}  t \max\{1, -\log p\}
		\right) 
		+
		\frac{2K}{p}
		\exp\left(
		- \frac{70^2}{71^2} t \max\{1, \log K, - \log p\}
		\right) \nonumber\\
		& = 
		2
		\exp\left(
		- \frac{70^2}{71^2}  t \max\{1, -\log p\} - \log p
		\right) 
		+
		2
		\exp\left(
		- \frac{70^2}{71^2} t \max\{1, \log K, - \log p\}
		+ \log K - \log p
		\right). \nonumber
	\end{align}
	Note that when $t \ge 2 \cdot 71^2/70^2 $, 
	\begin{align*}
		\frac{70^2}{71^2} t \max\{1, -\log p\} + \log p
		\ge \frac{70^2}{71^2}  t \frac{1 - \log p}{2} + \log p 
		\ge \frac{1}{2}\frac{70^2}{71^2}  t, 
	\end{align*}
	and when $t \ge 3 \cdot 71^2/70^2 $, 
	\begin{align}\label{eq:tlowerbnd}
		\frac{70^2}{71^2} t\max\{1, \log K, - \log p\}
		- \log K + \log p
		& \ge 
		\frac{70^2}{71^2} t \frac{1+\log K - \log p}{3} - \log K + \log p\\
		& \ge 
		\frac{1}{3} \frac{70^2}{71^2} t. \nonumber
	\end{align}
	Thus, when $t \ge 3 \cdot 71^2/70^2 $, we have 
	\begin{align*}
		\frac{1}{p}
		\PP\left( 
		\Delta_u^2 \sum_{k=1}^K \Delta_{w_k}^2  > t^2 \xi 
		\right)
		\le 
		2
		\exp\left(
		- \frac{1}{2}\frac{70^2}{71^2}  t
		\right) 
		+
		2
		\exp\left(
		- \frac{1}{3} \frac{70^2}{71^2} t
		\right)
		\le 
		4
		\exp\left(
		- \frac{1}{3} \frac{70^2}{71^2} t
		\right). 
	\end{align*}
	
	Third, 
	by definition, 
	when $t\ge 1$, 
	$
	t^2 \xi \ge \xi \ge  (1-f)^2/(N^2 f^2) \cdot \sum_{k=1}^K  S_{uw_k}^2. 
	$
	This immediately implies that, when $t\ge 1$, 
	\begin{align}\label{eq:bound_s_uw_re_3}
		\frac{1}{p} 
		\PP\left( 
		\frac{(1-f)^2}{N^2 f^2} \sum_{k=1}^K  S_{uw_k}^2 > t^2 \xi 
		\right)
		& = 0. 
	\end{align}
	
	From \eqref{eq:bound_s_uw_re}--\eqref{eq:bound_s_uw_re_3}, we can know that, when $t \ge  3 \cdot 71^2/70^2 $, 
	\begin{align*}
		& \quad \ \PP
		\left( 
		\left\| \bs{s}_{u\bs{w}} - \bs{S}_{u\bs{w}}  \right\|^2_2 > 36 t^2  \xi
		\mid \bs{Z} \in 
		\mathcal{E}
		\right)
		\nonumber
		\\
		& 
		\le 
		\frac{1}{p}
		\PP\left( \sum_{k=1}^K \Delta_{u\times w_k}^2 > t^2 \xi \right)
		+ 
		\frac{1}{p}
		\PP\left( 
		\Delta_u^2 \sum_{k=1}^K \Delta_{w_k}^2  > t^2 \xi 
		\right)
		+ 
		\frac{1}{p} 
		\PP\left( 
		\frac{(1-f)^2}{N^2 f^2} \sum_{k=1}^K  S_{uw_k}^2 > t^2 \xi 
		\right)\\
		& \le 
		2 \exp\left( -\frac{1}{3} \frac{70^2}{71^2} t^2 \right)
		+ 4
		\exp\left(
		- \frac{1}{3} \frac{70^2}{71^2} t
		\right) \le 
		6
		\exp\left(
		- \frac{1}{3} \frac{70^2}{71^2} t
		\right). 
	\end{align*}
	Therefore, Lemma \ref{lemma:s_uw_re} holds. 
\end{proof}

\begin{proof}[Proof of Lemma \ref{lemma:V_R2_hat_bound}]
	By definition, Lemma \ref{lemma:s_uw_re} immediately implies that, under ReM, 
	\begin{align*}
		\left| s_z^2 - S_z^2 \right| = O_{\PP}\left( \xi_{zz}^{1/2}\right), 
		\quad
		\left\| s_{z\bs{w}} - S_{z\bs{w}} \right\|_2 = O_{\PP}\left( \xi_{z\bs{w}}^{1/2}\right).  
	\end{align*}
	This implies that, for $z=0,1$, 
	\begin{align*}
		\left| \left\|s_{z\bs{w}} \right\|_2^2 - \left\|S_{z\bs{w}} \right\|_2^2 \right|
		& = 
		\left|
		\left( s_{z\bs{w}} - S_{z\bs{w}} \right) 
		\left( 
		s_{z\bs{w}} - S_{z\bs{w}} + 2S_{z\bs{w}}
		\right)^\top
		\right|
		= 
		\left|
		\left\| s_{z\bs{w}} - S_{z\bs{w}} \right\|_2^2 + 
		2 \left( s_{z\bs{w}} - S_{z\bs{w}} \right) S_{z\bs{w}}^\top
		\right|
		\\
		& \le 
		\left\| s_{z\bs{w}} - S_{z\bs{w}} \right\|_2^2 + 
		2 \left\| s_{z\bs{w}} - S_{z\bs{w}} \right\|_2 \left\| S_{z\bs{w}} \right\|_2 
		= O_{\PP} \left( 
		\xi_{z\bs{w}} + \left\| S_{z\bs{w}} \right\|_2 \xi_{z\bs{w}}^{1/2}
		\right). 
	\end{align*}
	By the same logic,
	\begin{align*}
		\left| s_{\tau\mid \bs{X}}^2 - S_{\tau\mid \bs{X}}^2 \right|
		& = 
		\left|\left\| \left( s_{1\bs{w}} - s_{0\bs{w}} \right) -  \left( S_{1\bs{w}} - S_{0\bs{w}} \right) \right\|_2^2  
		+ 
		2
		\left\{ \left( s_{1\bs{w}} - s_{0\bs{w}} \right) -  \left( S_{1\bs{w}} - S_{0\bs{w}} \right) \right\} 
		\left( S_{1\bs{w}} - S_{0\bs{w}} \right)^\top \right|
		\\
		& 
		\le 
		2 \left( 
		\left\| s_{1\bs{w}} - S_{1\bs{w}} \right\|_2^2 + 
		\left\| s_{0\bs{w}} - S_{0\bs{w}} \right\|_2^2
		\right)
		+ 
		2
		\left\| S_{1\bs{w}} - S_{0\bs{w}} \right\|_2 
		\left\{ 
		\left\| s_{1\bs{w}} - S_{1\bs{w}} \right\|_2 
		+ 
		\left\| s_{0\bs{w}} -  S_{0\bs{w}} \right\|_2 
		\right\}
		\\
		& = 
		O_{\PP} \left( \xi_{1\bs{w}} + \xi_{0\bs{w}} + \left\| S_{1\bs{w}} - S_{0\bs{w}} \right\|_2 \xi_{1\bs{w}}^{1/2} + \left\| S_{1\bs{w}} - S_{0\bs{w}} \right\|_2  \xi_{0\bs{w}}^{1/2}  \right). 
	\end{align*}
	
	From the above and by definition, we then have 
	\begin{align*}
		\left| \hat{V}_{\tau\tau} - V_{\tau\tau} - n^{-1} S_{\tau\setminus \bs{X}}^2 \right|
		& \le 
		n_1^{-1} \left| s_1^2 - S_1^2 \right| + n_0^{-1} \left| s_0^2 - S_0^2 \right| + n^{-1} \left| s_{\tau \mid \bs{X}}^2 - S_{\tau \mid \bs{X}}^2 \right|
		\\
		& = 
		O_{\PP}\left(
		\frac{\xi_{11}^{1/2}}{n_1} + 
		\frac{\xi_{00}^{1/2}}{n_0} + 
		\frac{\xi_{1\bs{w}} + \xi_{0\bs{w}}}{n}
		+ \left\| S_{1\bs{w}} - S_{0\bs{w}} \right\|_2 \frac{ \xi_{1\bs{w}}^{1/2} + \xi_{0\bs{w}}^{1/2} }{n}
		\right), 
	\end{align*}
	and 
	\begin{align*}
		& \quad \ \left| \hat{V}_{\tau\tau} \hat{R}^2_n - V_{\tau\tau} R^2_n \right|
		\\
		& =
		\left|
		n_1^{-1} \left\|s_{1\bs{w}} \right\|_2^2 + n_0^{-1} \left\|s_{0\bs{w}} \right\|_2^2 - n^{-1} s_{\tau \mid \bs{X}}^2
		- 
		\left(
		n_1^{-1} \left\|S_{1\bs{w}} \right\|_2^2 + n_0^{-1} \left\|S_{0\bs{w}} \right\|_2^2 - n^{-1} S_{\tau \mid \bs{X}}^2
		\right)
		\right|\\
		& \le 
		n_1^{-1} \left| \left\|s_{1\bs{w}} \right\|_2^2 - \left\|S_{1\bs{w}} \right\|_2^2 \right|
		+ 
		n_0^{-1} \left| \left\|s_{0\bs{w}} \right\|_2^2 - \left\|S_{0\bs{w}} \right\|_2^2 \right|
		+ 
		n^{-1}\left| s_{\tau \mid \bs{X}}^2 - S_{\tau \mid \bs{X}}^2 \right|\\
		& = 
		O_{\PP}\left(
		\frac{\xi_{1\bs{w}}}{n_1} + \left\| S_{1\bs{w}} \right\|_2 \frac{\xi_{1\bs{w}}^{1/2}}{n_1}
		+ 
		\frac{\xi_{0\bs{w}}}{n_1} + \left\| S_{0\bs{w}} \right\|_2 \frac{\xi_{0\bs{w}}^{1/2}}{n_0} 
		+ 
		\frac{\xi_{1\bs{w}} + \xi_{0\bs{w}}}{n} + \left\| S_{1\bs{w}} - S_{0\bs{w}} \right\|_2 \frac{ \xi_{1\bs{w}}^{1/2} + \xi_{0\bs{w}}^{1/2} }{n}
		\right)\\
		& = 
		O_{\PP}\left(
		\frac{\xi_{1\bs{w}}}{n_1} 
		+ 
		\frac{\xi_{0\bs{w}}}{n_0}
		+ \left\| S_{1\bs{w}} \right\|_2 \frac{\xi_{1\bs{w}}^{1/2}}{n_1}
		+ \left\| S_{0\bs{w}} \right\|_2 \frac{\xi_{0\bs{w}}^{1/2}}{n_0}
		+ 
		\left\| S_{1\bs{w}} - S_{0\bs{w}} \right\|_2 \frac{ \xi_{1\bs{w}}^{1/2} + \xi_{0\bs{w}}^{1/2} }{n}
		\right). 
	\end{align*}
	Therefore, Lemma \ref{lemma:V_R2_hat_bound} holds. 
\end{proof}

\begin{proof}[Proof of Lemma \ref{lemma:V_R2_hat_bound_simp}]
	First, we consider bounding some finite population quantities.  
	For descriptive convenience, we introduce 
	$\psi = \max_{z\in \{0,1\}}\max_{1\le i \le n}\{Y_i(z) - \bar{Y}(z)\}^2$. 
	By definition, 
	for $z=0,1$ and $1\le k \le K$, 
	\begin{align}\label{eq:sigmawz}
		\sigma^2_z & = \frac{1}{n} \sum_{i=1}^n \{Y_i(z) - \bar{Y}(z)\}^2
		\le \psi, 
		\ \ \ 
		S_z^2 = \frac{n}{n-1} \sigma^2_z \le 2 \psi, 
		\ \ \ 
		\sigma_{w_k}^2 
		= \frac{1}{n} \sum_{i=1}^n ( w_{ki} - \bar{w}_k )^2  = \frac{n-1}{n} \le 1. 
	\end{align}
	and 
	\begin{align*}
		\sigma^2_{z\times z} 
		& = 
		\frac{1}{n} \sum_{i=1}^n\Big[ 
		\{Y_i(z) - \bar{Y}(z) \}^2 - \sigma^2_z
		\Big]^2
		\le \frac{1}{n} \sum_{i=1}^n \{Y_i(z) - \bar{Y}(z) \}^4
		\le \psi^2, 
		\\
		\sigma^2_{z\times w_k} 
		& = 
		\frac{1}{n} \sum_{i=1}^n\Big[ 
		\{Y_i(z) - \bar{Y}(z)\} (w_{ki} - \bar{w}_k)- \frac{n-1}{n} S_{z w_k}
		\Big]^2
		\le 
		\frac{1}{n} \sum_{i=1}^n
		\{Y_i(z) - \bar{Y}(z)\}^2 (w_{ki} - \bar{w}_k)^2
		\\
		& 
		\le 
		\psi \cdot
		\frac{1}{n} \sum_{i=1}^n (w_{ki} - \bar{w}_k)^2
		\le 
		\psi.
	\end{align*}
	Furthermore, by the Cauchy–Schwarz inequality, 
	\begin{align*}
		S_{z w_k}^2
		& = 
		\left[\frac{1}{n-1} \sum_{i=1}^n\{Y_i(z) - \bar{Y}(z)\} (w_{ki} - \bar{w}_k) \right]^2 
		\le 
		\frac{1}{(n-1)^2}
		\sum_{i=1}^n \{Y_i(z) - \bar{Y}(z)\}^2 
		\cdot 
		\sum_{i=1}^n (w_{ki} - \bar{w}_k)^2 \\
		& = 
		\frac{1}{n-1}\sum_{i=1}^n \{Y_i(z) - \bar{Y}(z)\}^2 
		\le 
		2 \psi. 
	\end{align*}
	
	Second, we consider the bounds on $\xi_{zz}$ and $\xi_{z \bs{w}}$ for $z=0,1$. 
	For descriptive convenience, we introduce 
	$
	b_n = \max\{1, - \log \tilde{p}_n \}
	$
	and 
	$
	c_n = \max\{1, \log K_n, - \log \tilde{p}_n\}. 
	$
	By definition and from the bounds we derived above, 
	for $z=0,1$, 
	\begin{align*}
		\xi_{zz}
		& = 
		\frac{b_n}{n r_z^2} \sigma^2_{z\times z} 
		+
		\frac{ b_n^2 }{n^2 r_z^4} \sigma^4_z 
		+ 
		\frac{(1-r_z)^2}{n^2 r_z^2} S_{z}^4
		\le 
		\psi^2 
		\left(
		\frac{b_n}{n r_z^2}
		+
		\frac{ b_n^2 }{n^2 r_z^4}
		+ 
		\frac{4}{n^2 r_z^2}
		\right)
		\le 
		\psi^2 
		\left(
		\frac{2b_n}{n r_z^2}
		+
		\frac{ b_n^2 }{n^2 r_z^4}
		\right)
	\end{align*}
	where the last inequality holds because $b_n\ge 1$ and $n\ge 4$. 
	From the condition in Lemma \ref{lemma:V_R2_hat_bound_simp}, 
	$b_n \le c_n = O(nr_1^2r_0^2)$, and thus 
	\[
	\xi_{zz} = 
	\psi^2 
	\frac{b_n}{n r_z^2}
	\left( 2 + \frac{b_n}{n r_z^2} \right)
	= 
	O\left( \psi^2 
	\frac{b_n}{n r_z^2} \right).
	\] 
	Similarly, we can derive that, for $z=0,1$, 
	\begin{align}\label{eq:xizw}
		\xi_{z \bs{w}}
		& = 
		\frac{c_n}{n r_z^2} \sum_{k=1}^{K_n} \sigma^2_{z\times w_k} 
		+
		\frac{ b_n c_n }{n^2 r_z^4} \sigma^2_z \sum_{k=1}^{K_n} \sigma^2_{w_k}
		+ 
		\frac{(1-r_z)^2}{n^2 r_z^2} \sum_{k=1}^{K_n} S_{z w_k}^2
		\le 
		\psi K_n 
		\left(
		\frac{c_n}{n r_z^2} 
		+
		\frac{ b_n c_n }{n^2 r_z^4}
		+ 
		\frac{2}{n^2 r_z^2} \right) 
		\\
		& \le 
		\psi K_n 
		\left(
		2 \frac{c_n}{n r_z^2} 
		+
		\frac{ b_n c_n }{n^2 r_z^4} \right)
		= 
		\psi K_n
		\frac{c_n}{n r_z^2} 
		\left(
		2 
		+
		\frac{ b_n }{n r_z^2} \right)
		= 
		O\left( \psi K_n
		\frac{c_n}{n r_z^2}  \right). \nonumber
	\end{align}

	Third, we consider the probability bounds for 
	$\hat{V}_{\tau\tau} - V_{\tau\tau} - n^{-1} S_{\tau\setminus \bs{X}}^2$ and $\hat{V}_{\tau\tau} \hat{R}^2_n - V_{\tau\tau} R^2_n$. From the bounds we derived before, for $z=0,1$, 
	\begin{align*}
		\left\| S_{z\bs{w}} \right\|_2 & 
		\le 
		\left( \sum_{k=1}^{K_n} S_{z\bs{w}_k}^2 \right)^{1/2} 
		\le 
		\sqrt{2K_n\psi}, 
		\quad
		\left\| S_{1\bs{w}} - S_{0\bs{w}} \right\|_2
		\le 
		\left\| S_{1\bs{w}} \right\|_2 + \left\| S_{0\bs{w}} \right\|_2 
		\le 
		2 \sqrt{2K_n\psi}. 
	\end{align*}
	Consequently, we have 
	\begin{align*}
		& \quad \ \frac{\xi_{11}^{1/2}}{n_1} + 
		\frac{\xi_{00}^{1/2}}{n_0} + 
		\frac{\xi_{1\bs{w}} + \xi_{0\bs{w}}}{n} + \left\| S_{1\bs{w}} - S_{0\bs{w}} \right\|_2\frac{\xi_{1\bs{w}}^{1/2} + \xi_{0\bs{w}}^{1/2}}{n} 
		\\
		& = 
		\psi 
		\cdot
		O\left(
		\frac{\sqrt{b_n}}{n^{3/2} r_1^2} + \frac{\sqrt{b_n}}{n^{3/2} r_0^2} 
		+ 
		K_n \frac{c_n}{n^2r_1^2} + K_n \frac{c_n}{n^2r_0^2} + 
		K_n \frac{\sqrt{c_n}}{n^{3/2} r_1} + K_n \frac{\sqrt{c_n}}{n^{3/2} r_0}
		\right)\\
		& = 
		\psi \cdot O\left(
		\frac{\sqrt{b_n}}{n^{3/2} r_1^2r_0^2} 
		+ 
		K_n \frac{c_n}{n^2r_1^2r_0^2} + 
		K_n \frac{\sqrt{c_n}}{n^{3/2} r_1r_0} 
		\right)
		= 
		\psi \cdot O\left\{
		\frac{\sqrt{b_n}}{n^{3/2} r_1^2r_0^2} 
		+ 
		K_n \frac{\sqrt{c_n}}{n^{3/2} r_1r_0} 
		\left(1 + 
		\sqrt{\frac{c_n}{nr_1^2 r_0^2}}
		\right)
		\right\}\\
		& 
		=
		\psi \cdot O\left(
		\frac{\sqrt{b_n}}{n^{3/2} r_1^2r_0^2} 
		+ 
		K_n \frac{\sqrt{c_n}}{n^{3/2} r_1r_0} \right)
		=
		\psi \cdot O
		\left(
		\frac{\sqrt{b_n} + K_n \sqrt{c_n}}{n^{3/2} r_1^2r_0^2} 
		\right). 
	\end{align*}
	where the second last equality holds because $c_n = O(nr_1^2 r_0^2)$.  
	Similarly, we can derive that 
	\begin{align*}
		& \quad \ \frac{\xi_{1\bs{w}}}{n_1} 
		+ 
		\frac{\xi_{0\bs{w}}}{n_0}
		+ \left\| S_{1\bs{w}} \right\|_2 \frac{\xi_{1\bs{w}}^{1/2}}{n_1}
		+ \left\| S_{0\bs{w}} \right\|_2 \frac{\xi_{0\bs{w}}^{1/2}}{n_0} 
		+ 
		\left\| S_{1\bs{w}} - S_{0\bs{w}} \right\|_2 \frac{\xi_{1\bs{w}}^{1/2} + \xi_{0\bs{w}}^{1/2}}{n}
		\\
		& = 
		\psi \cdot
		O\left(
		K_n \frac{c_n}{n^2 r_1^3} + K_n \frac{c_n}{n^2 r_0^3} + 
		K_n \frac{\sqrt{c_n}}{n^{3/2} r_1^2} + 
		K_n \frac{\sqrt{c_n}}{n^{3/2} r_0^2} + 
		K_n \frac{\sqrt{c_n}}{n^{3/2} r_1} + K_n \frac{\sqrt{c_n}}{n^{3/2} r_0}
		\right)\\
		& = 
		\psi \cdot
		O\left(
		K_n \frac{c_n}{n^2 r_1^3 r_0^3} + 
		K_n \frac{\sqrt{c_n}}{n^{3/2} r_1^2 r_0^2}
		\right)
		= 
		\psi \cdot 
		O\left\{
		K_n \frac{\sqrt{c_n}}{n^{3/2} r_1^2 r_0^2}
		\left( 1 +
		\sqrt{\frac{c_n}{n r_1^2 r_0^2}}
		\right)
		\right\}
		\\
		& = 
		\psi \cdot 
		O\left(
		K_n \frac{\sqrt{c_n}}{n^{3/2} r_1^2 r_0^2}
		\right). 
	\end{align*}
	Note that, by definition, 
	$
	b_n \le c_n. 
	$
	Thus, we must have 
	\begin{align*}
		\frac{\sqrt{b_n} + K_n \sqrt{c_n}}{n^{3/2} r_1^2r_0^2} 
		\le 
		2 \max\{K_n, 1\} \cdot  \frac{\sqrt{c_n}}{n^{3/2} r_1^2r_0^2}, 
		\qquad
		K_n \frac{\sqrt{c_n}}{n^{3/2} r_1^2 r_0^2} 
		\le \max\{K_n, 1\} \cdot  \frac{\sqrt{c_n}}{n^{3/2} r_1^2r_0^2}. 
	\end{align*}
	From the above and Lemma \ref{lemma:V_R2_hat_bound}, these then imply that 
	\begin{align*}
		\max\left\{ \big|\hat{V}_{\tau\tau} - V_{\tau\tau} - n^{-1} S_{\tau\setminus \bs{X}}^2\big|, \ \ \big|\hat{V}_{\tau\tau} \hat{R}^2_n - V_{\tau\tau} R^2_n\big| \right\}
		= 
		\psi \cdot 
		O_{\PP}\left( \max\{K_n, 1\} \cdot \frac{\sqrt{c_n}}{n^{3/2} r_1^2r_0^2} \right). 
	\end{align*}
	Therefore, Lemma \ref{lemma:V_R2_hat_bound_simp} holds. 
\end{proof}

\begin{proof}[Proof of Lemma \ref{lemma:cond_infer}]
	We first prove (i). 
	By definition, $|\tilde{p}_n - p_n| \le \Delta_n$. 
	Because $\Delta_n/p_n = o(1)$, 
	this implies that 
	\begin{align*}
		-\log \tilde{p}_n 
		& \le 
		-\log \left( p_n - \Delta_n \right) 
		= 
		-\log p_n - \log \left(1 - \Delta_n/p_n \right) 
		= 
		-\log p_n + O(1) = O\left( \max\{1, -\log p_n\} \right). 
	\end{align*}
	Thus, 
	$\max\{1, -\log \tilde{p}_n\} = O(\max\{1, -\log p_n\})$. 
	
	We then prove (ii). 
	By definition, 
	\begin{align*}
		r_0 S^2_{1\setminus \bs{X}} + r_1 S^2_{0\setminus \bs{X}}
		& \le  r_0 S_1^2 + r_1 S_0^2 
		\le 2 (r_0+r_1) \max_{z\in \{0,1\}}\max_{1\le i \le n}\{Y_i(z) - \bar{Y}(z)\}^2
		\\
		& = 
		2 \max_{z\in \{0,1\}}\max_{1\le i \le n}\{Y_i(z) - \bar{Y}(z)\}^2. 
	\end{align*}
	Thus, (ii) holds. 
	
	Last, we prove (iii). From (i) and Condition \ref{cond:infer}, we can verify that 
	\begin{align*}
		& \quad \ \frac{\max_{z\in \{0,1\}}\max_{1\le i \le n}\{Y_i(z) - \bar{Y}(z)\}^2}{r_0 S^2_{1\setminus \bs{X}} + r_1 S^2_{0\setminus \bs{X}}}
		\cdot 
		\frac{\max\{K_n, 1\}}{r_1  r_0}
		\cdot 
		\sqrt{ \frac{\max\{1, \log K_n, - \log \tilde{p}_n\} }{n} } 
		\\
		& = 
		O
		\left(
		\frac{\max_{z\in \{0,1\}}\max_{1\le i \le n}\{Y_i(z) - \bar{Y}(z)\}^2}{r_0 S^2_{1\setminus \bs{X}} + r_1 S^2_{0\setminus \bs{X}}}
		\cdot 
		\frac{\max\{K_n, 1\}}{r_1  r_0}
		\cdot 
		\sqrt{ \frac{\max\{1, \log K_n, - \log p_n\} }{n} } 
		\right)
		\\
		& = o(1). 
	\end{align*}
	From (ii), this then implies that 
	\begin{align*}
		o(1) 
		& =
		\frac{\max_{z\in \{0,1\}}\max_{1\le i \le n}\{Y_i(z) - \bar{Y}(z)\}^2}{nr_1r_0 V_{\tau\tau}(1-R^2)}
		\cdot 
		\frac{\max\{K_n, 1\}}{r_1  r_0}
		\cdot 
		\sqrt{ \frac{\max\{1, \log K_n, - \log \tilde{p}_n\} }{n} } 
		\\
		& \ge 
		\frac{1}{2} \cdot 
		\frac{\max\{K_n, 1\}}{r_1  r_0}
		\cdot 
		\sqrt{ \frac{\max\{1, \log K_n, - \log \tilde{p}_n\} }{n} } 
		\ge 
		\frac{1}{2} 
		\sqrt{ \frac{\max\{1, \log K_n, - \log \tilde{p}_n\} }{n r_1^2 r_0^2} }. 
	\end{align*}
	Consequently, we must have 
	$\max\{1, \log K_n, - \log \tilde{p}_n\} = o(n r_1^2 r_0^2)$, 
	i.e., (iii) holds. 
\end{proof}

\subsection{Proofs of Theorems \ref{thm:inf} and \ref{thm:inf_gaussian}}%

\begin{proof}[\bf Proof of Theorem \ref{thm:inf}(i)]
	From Lemmas \ref{lemma:V_R2_hat_bound_simp} and \ref{lemma:cond_infer},  
	under ReM and Conditions \ref{cond:p_n} and \ref{cond:infer}, we must have 
	\begin{align*}
		& \quad \ \max\left\{ \big| \hat{V}_{\tau\tau} - V_{\tau\tau} - n^{-1} S_{\tau\setminus \bs{X}}^2 \big|, \ \ \big| \hat{V}_{\tau\tau} \hat{R}^2_n - V_{\tau\tau} R^2_n \big| \right\}
		\\
		& 
		= 
		\max_{z\in \{0,1\}}\max_{1\le i \le n}\{Y_i(z) - \bar{Y}(z)\}^2 \cdot 
		O_{\PP}\left( 
		\max\{K_n, 1\} \cdot \frac{\sqrt{ \max\{1, \log K_n, - \log p_n\} }}{n^{3/2} r_1^2r_0^2} \right)
		\\
		& = 
		\frac{\max_{z\in \{0,1\}}\max_{1\le i \le n}\{Y_i(z) - \bar{Y}(z)\}^2}{n r_1 r_0} \cdot 
		O_{\PP}\left( 
		\frac{\max\{K_n, 1\}}{r_1r_0} \cdot \sqrt{ \frac{ \max\{1, \log K_n, - \log p_n\} }{n} } \right)
		\\
		& = 
		\frac{r_0 S^2_{1\setminus \bs{X}} + r_1 S^2_{0\setminus \bs{X}}}{nr_1r_0} \cdot o_{\PP}(1)
		= 
		\left( n_1^{-1}  S^2_{1\setminus \bs{X}} + n_0^{-1}  S^2_{0\setminus \bs{X}} \right) \cdot o_{\PP}(1). 
	\end{align*}
	Note that, by definition, 
	\begin{align}\label{eq:denom_inf_cond_equiv}
		V_{\tau\tau} (1-R^2_n)+n^{-1} S^2_{\tau \setminus \bs{X}} 
		& = 
		\left( n_1^{-1} S^2_{1} +  n_0^{-1} S^2_{0} - n^{-1} S^2_{\tau} \right) 
		- 
		\left( n_1^{-1} S^2_{1\mid \bs{X}} +  n_0^{-1} S^2_{0\mid \bs{X}} - n^{-1} S^2_{\tau \mid \bs{X}} \right)
		+n^{-1} S^2_{\tau \setminus \bs{X}} 
		\nonumber
		\\
		& = n_1^{-1}  S^2_{1\setminus \bs{X}} + n_0^{-1}  S^2_{0\setminus \bs{X}}.
	\end{align}
	From the above, Theorem \ref{thm:inf}(i) holds. 
\end{proof}

To prove Theorem \ref{thm:inf}(ii), we need the following two lemmas. 

\begin{lemma}\label{lemma:quant_converg}
	Let $\{\psi_n\}$ and $\{\tilde{\psi}_n\}$ be two sequences of continuous random variables such that, as $n\rightarrow \infty$, 
	$
	\sup_{c\in \mathbb{R}}\big| \PP(\psi_n \le c ) - \PP(\tilde{\psi}_n \le c ) \big| \rightarrow 0. 
	$
	For any $n$ and $\alpha \in (0,1)$, 
	let $q_{n}(\alpha)$ and $\tilde{q}_{n}(\alpha)$ be the $\alpha$th quantile of $\psi_n$ and $\tilde{\psi}_n$, respectively. 
	Then for any $0 < \alpha < \beta < 1$, 
	$\I \{ \tilde{q}_n(\beta) \le q_n(\alpha) \} \rightarrow 0$ as $n \rightarrow \infty$. 
\end{lemma}

\begin{proof}[Proof of Lemma \ref{lemma:quant_converg}]
	From the condition and definition in Lemma \ref{lemma:quant_converg}, 
	$\PP\{ \psi_n \le q_n(\alpha) \} = \alpha$, $\PP\{ \tilde{\psi}_n \le \tilde{q}_n(\beta) \} = \beta$, 
	and 
	$
	\big| \PP \{ \psi_n \le \tilde{q}_n(\beta) \} - \PP\{ \tilde{\psi}_n \le \tilde{q}_n(\beta) \} \big|
	\le 
	\sup_{c\in \mathbb{R}}\big| \PP(\psi_n \le c ) - \PP(\tilde{\psi}_n \le c ) \big| \rightarrow 0
	$
	as $n \rightarrow \infty$. 
	These imply that $\PP \{ \psi_n \le \tilde{q}_n(\beta) \} \converge \beta  > \alpha$ as $n\rightarrow \infty$. 
	Below we prove Lemma \ref{lemma:quant_converg} by contradiction. 
	
	Suppose that $\I \{ \tilde{q}_n(\beta) \le q_n(\alpha) \}$ does not converge to zero as $n \rightarrow \infty$. 
	Then there exists a subsequence $\{n_j\}$ such that $\tilde{q}_{n_j}(\beta) \le  q_{n_j}(\alpha)$ for all $j$. 
	This implies that, for all $j$, 
	$
	\PP \{ \psi_{n_j} \le \tilde{q}_{n_j}(\beta) \}  \le \PP \{ \psi_{n_j} \le q_{n_j}(\alpha) \}
	= \alpha. 
	$
	Consequently, we must have 
	$\limsup_{j \rightarrow \infty} \PP \{ \psi_{n_j} \le \tilde{q}_{n_j}(\beta) \}  \le \alpha$. 
	However, this contradicts with the fact that 
	$\lim_{n\rightarrow \infty}\PP \{ \psi_n \le \tilde{q}_n(\beta) \} = \beta$.
	From the above, Lemma	\ref{lemma:quant_converg} holds. 
\end{proof}

\begin{lemma}\label{lemma:non_Gaussian_quantile}
	Let $\varepsilon_0 \sim \mathcal{N}(0,1)$,  and $L_{K_n, a_n}$ be the truncated Gaussian random variables defined as in Section \ref{sec:recent_result},
	where $\{K_n\}$ and $\{a_n\}$ are sequences of positive integers and thresholds, 
	and $\varepsilon_0$ is independent of $L_{K_n, a_n}$ for all $n$. 
	Let $\{A_n\}$,  $\{B_n\}$, $\{\tilde{A}_n\}$ and $\{\tilde{B}_n\}$  be sequences of nonnegative constants, 
	and for each $n$, 
	define $\psi_n = A_n^{1/2} \cdot \varepsilon_0 + B_n^{1/2} \cdot L_{K_n, a_n}$ and $\tilde{\psi}_n = \tilde{A}_n^{1/2} \cdot \varepsilon_0 + \tilde{B}_n^{1/2} \cdot L_{K_n, a_n}$. 
	For each $n$ and $\alpha \in (0,1)$, let $q_{n}(\alpha)$ and $\tilde{q}_{n}(\alpha)$ be the $\alpha$th quantile of $\psi_n$ and $\tilde{\psi}_n$, respectively. 
	If $\max\{ | \tilde{A}_n - A_n|, |\tilde{B}_n - B_n|\} = o(A_n)$, 
	then for any $0 < \alpha < \beta < 1$, as $n \rightarrow \infty$, 
	$\I \{ \tilde{q}_n(\beta) \le q_n(\alpha) \} \rightarrow 0$ 
	and 
	$\I \{ q_n(\beta) \le \tilde{q}_n(\alpha) \} \rightarrow 0$. 
\end{lemma}

\begin{proof}[Proof of Lemma \ref{lemma:non_Gaussian_quantile}]
	Because $\var(L_{K_n, a_n}) \le 1$, $L_{K_n, a_n} = O_{\PP}(1)$.
	Using the inequality that $|\sqrt{b}-\sqrt{c}| \le \sqrt{|b-c|}$ for any $b, c\ge 0$, we have
	\begin{align*}
		\tilde{\psi}_n - \psi_n & = 
		\big(\tilde{A}_n^{1/2} - A_n^{1/2} \big) \varepsilon_0 + \big(\tilde{B}_n^{1/2} - B_n^{1/2} \big) L_{K_n, a_n}
		= |\tilde{A}_n - A_n|^{1/2} \cdot O_{\PP}(1) + |\tilde{B}_n - B_n|^{1/2} \cdot O_{\PP}(1)
		\\
		& = A_n^{1/2} \cdot o_{\PP}(1). 
	\end{align*}
	From Lemma \ref{lemma:diminish_inf_norm_two_rv}, this then implies that 
	$\sup_{c\in \mathbb{R}} | \PP(\tilde{\psi}_n  \le c) - \PP(\psi_n \le c) | \rightarrow 0$ as $n\rightarrow \infty$. 
	From Lemma \ref{lemma:quant_converg}, this further implies that, for any $0 < \alpha < \beta < 1$, 
	$\I \{ \tilde{q}_n(\beta) \le q_n(\alpha) \} \rightarrow 0$ as $n \rightarrow \infty$. 
	By summery, we also have $\I \{ q_n(\beta) \le \tilde{q}_n(\alpha) \} \rightarrow 0$ as $n \rightarrow \infty$. 
	Therefore, Lemma \ref{lemma:non_Gaussian_quantile} holds. 
\end{proof}

\begin{proof}[\bf Proof of Theorem \ref{thm:inf}(ii)]
	For descriptive convenience, 
	let $\varepsilon_0$ and $L_{K_n, a_n}$ be two independent standard and constrained Gaussian random variables defined as in Section \ref{sec:recent_result}, which are further constructed to be independent of the treatment assignment $\bs{Z}$. 
	We then define 
	\begin{align*}
		\theta_n & = \sqrt{V_{\tau\tau} (1-R^2_n)} \cdot \varepsilon_0 + \sqrt{V_{\tau\tau} R^2_n} \cdot L_{K_n, a_n}
		\equiv A_n^{1/2} \cdot \varepsilon_0 + B_n^{1/2} \cdot L_{K_n, a_n}, 
		\\ 
		\tilde{\theta}_n & = \sqrt{V_{\tau\tau} (1-R^2_n)+n^{-1} S^2_{\tau \setminus \bs{X}}} \cdot \varepsilon_0 + \sqrt{V_{\tau\tau} R^2_n} \cdot L_{K_n, a_n}
		\equiv \tilde{A}_n^{1/2} \cdot \varepsilon_0 + \tilde{B}_n^{1/2} \cdot L_{K_n, a_n}, 
		\\
		\hat{\theta}_n & = \sqrt{\hat{V}_{\tau\tau} (1-\hat{R}^2_n)} \cdot \varepsilon_0 + \sqrt{\hat{V}_{\tau\tau} \hat{R}^2_n} \cdot L_{K_n, a_n}
		\equiv \hat{A}_n^{1/2} \cdot \varepsilon_0 + \hat{B}_n^{1/2} \cdot L_{K_n, a_n}, 
	\end{align*}
	where $A_n, \tilde{A}_n, \hat{A}_n$ and $B_n, \tilde{B}_n, \hat{B}_n$ denote the squared coefficients of the standard and constrained Gaussian random variables, respectively. 
	We introduce $q_{\alpha}(A, B, K, a)$ to denote the $\alpha$th quantile of $A^{1/2} \varepsilon_0 + B^{1/2} L_{K, a}$, 
	and further define 
	$\hat{q}_{n, \alpha} = q_{\alpha}(\hat{A}_n, \hat{B}_n, K_n, a_n)$, 
	$\tilde{q}_{n, \alpha} = q_{\alpha}(\tilde{A}_n, \tilde{B}_n, K_n, a_n)$
	and 
	$q_{n, \alpha} = q_{\alpha}(A_n, B_n, K_n, a_n)$. 
	
	First, we prove that $\lim_{n\rightarrow \infty}\PP( \hat{q}_{n, \beta}\le \tilde{q}_{n, \alpha} \mid M\le a_n) = 0$ for any $0 < \alpha < \beta < 1$. 
	From Theorem \ref{thm:inf}(i), under ReM, 
	$\max\{ |\hat{A}_n - \tilde{A}_n|, |\hat{B}_n - \tilde{B}_n| \} = o_{\PP}(\tilde{A}_n)$. 
	By \citet[][Theorem 2.3.2]{D19}, under ReM, 
	for any subsequence $\{n_j: j=1,2,\ldots\}$, 
	there exists a further subsequence $\{m_j: j=1,2,\ldots\}\subset \{n_j: j=1,2,\ldots\}$ such that 
	$|\hat{A}_{m_j} - \tilde{A}_{m_j}|/\tilde{A}_{m_j} \convergeas 0$ and 
	$|\hat{B}_{m_j} - \tilde{B}_{m_j}|/\tilde{A}_{m_j} \convergeas 0$ as $j\rightarrow \infty$. 
	From Lemma \ref{lemma:non_Gaussian_quantile}, this immediately implies that, for any $0 < \alpha < \beta < 1$, 
	$\I\{\hat{q}_{m_j, \beta} \le \tilde{q}_{m_j, \alpha}\} \convergeas 0$ as $n\rightarrow 0$. 
	From \citet[][Theorem 2.3.2]{D19}, we can know that, under ReM, for any $0 < \alpha < \beta < 1$, 
	$\I(\hat{q}_{n, \beta}\le \tilde{q}_{n, \alpha}) \convergep 0$ as $n\rightarrow 0$. 
	Consequently, under ReM, for any $0 < \alpha < \beta < 1$, as $n\rightarrow \infty$, 
	\begin{align}\label{eq:conv_quantile_est_tilde}
		\PP( \hat{q}_{n, \beta}\le \tilde{q}_{n, \alpha} \mid M\le a_n)
		& = 
		\E\{\I(\hat{q}_{n, \beta}\le \tilde{q}_{n, \alpha}) \mid M \le a_n\} 
		\rightarrow 0.
	\end{align}

	Second, we prove the asymptotic validity of the confidence interval $\hat{\mathcal{C}}_{\alpha}$ for $\alpha \in (0, 1)$. 
	For any $\alpha \in (0, 1)$ and $\eta\in (0, (1-\alpha)/2)$, the coverage probability of the confidence interval $\hat{\mathcal{C}}_{\alpha}$ can be bounded by 
	\begin{align*}
		\PP(\tau \in \hat{\mathcal{C}}_{\alpha} \mid M\le a_n) 
		& = \PP\{ | \hat{\tau}-\tau|  \le \hat{q}_{n, 1-\alpha/2} \mid M \le a_n \}
		\\
		& \ge 
		\PP\{ | \hat{\tau}-\tau| \le \hat{q}_{n, 1-\alpha/2}, \ 
		\hat{q}_{n, 1-\alpha/2} \ge \tilde{q}_{n, 1-\alpha/2-\eta} 
		\mid M \le a_n \}
		\\
		& 
		\ge 
		\PP\{ | \hat{\tau}-\tau|  \le  \tilde{q}_{n, 1-\alpha/2-\eta}, \  
		\hat{q}_{n, 1-\alpha/2} \ge \tilde{q}_{n, 1-\alpha/2-\eta} 
		\mid M \le a_n \}
		\\
		& 
		\ge 
		\PP\{ | \hat{\tau}-\tau|  \le  \tilde{q}_{n, 1-\alpha/2-\eta}
		\mid M \le a_n \}
		- 
		\PP\{  \hat{q}_{n, 1-\alpha/2} < \tilde{q}_{n, 1-\alpha/2-\eta} 
		\mid M \le a_n \}. 
	\end{align*}
	From \eqref{eq:conv_quantile_est_tilde} and Theorem \ref{thm:dim_rem}, 
	$\PP\{  \hat{q}_{n, 1-\alpha/2} < \tilde{q}_{n, 1-\alpha/2-\eta} 
	\mid M \le a_n \} = o(1)$, 
	and 
	\begin{align*}
		\PP\{ | \hat{\tau}-\tau|  \le  \tilde{q}_{n, 1-\alpha/2-\eta}
		\mid M \le a_n \}
		& = 
		\PP(|\theta_n| \le \tilde{q}_{n, 1-\alpha/2-\eta} ) + o(1)
		\ge \PP(|\tilde{\theta}_n| \le \tilde{q}_{n, 1-\alpha/2-\eta} ) + o(1)
		\\
		& = 1 - \alpha - 2\eta + o(1), 
	\end{align*}
	where the last inequality follows from \citet[][Lemma A3]{LD20}. 
	These then imply that 
	\begin{align*}
		\liminf_{n \rightarrow \infty} \PP(\tau \in \hat{\mathcal{C}}_{\alpha} \mid M\le a_n)  
		\ge 
		1 - \alpha - 2\eta. 
	\end{align*}
	Because the above inequality holds for any $\eta \in (0, (1-\alpha)/2)$, 
	we must have 
	$\liminf_{n \rightarrow \infty} \PP(\tau \in \hat{\mathcal{C}}_{\alpha} \mid M\le a_n)  
	\ge 1 - \alpha$. 
	
	From the above, Theorem \ref{thm:inf}(ii) holds. 
\end{proof}

\begin{proof}[\bf Proof of Theorem \ref{thm:inf}(iii)]
	For any $\alpha \in (0, 1)$ and $\eta \in (0, \alpha/2)$, 
	the coverage probability of the confidence interval $\hat{\mathcal{C}}_{\alpha}$ can be bounded by 
	\begin{align}\label{eq:cover_upper}
		& \quad \ \PP(\tau \in \hat{\mathcal{C}}_{\alpha} \mid M\le a_n) 
		\nonumber
		\\
		& = \PP\{ | \hat{\tau}-\tau|  \le \hat{q}_{n, 1-\alpha/2} \mid M \le a_n \}
		\nonumber
		\\
		& 
		\le \PP\{ |\hat{\tau}-\tau|  \le \hat{q}_{n, 1-\alpha/2}, \  \hat{q}_{n, 1-\alpha/2} \le \tilde{q}_{n, 1-\alpha/2 + \eta} \mid M \le a_n \}
		+ \PP( \hat{q}_{n, 1-\alpha/2} > \tilde{q}_{n, 1-\alpha/2 + \eta} \mid M \le a_n )
		\nonumber
		\\
		& \le 
		\PP\{ |\hat{\tau}-\tau|  \le \tilde{q}_{n, 1-\alpha/2 + \eta} \mid M \le a_n \}
		+ \PP( \hat{q}_{n, 1-\alpha/2} > \tilde{q}_{n, 1-\alpha/2 + \eta} \mid M \le a_n ). 
	\end{align}
	Below we consider the two terms in \eqref{eq:cover_upper}, separately. 
	
	First, from Theorem \ref{thm:dim_rem}, 
	$
	\PP\{ |\hat{\tau}-\tau|  \le \tilde{q}_{n, 1-\alpha/2 + \eta} \mid M \le a_n \} = 
	\PP(|\theta_n| \le \tilde{q}_{n, 1-\alpha/2 + \eta}) + o(1). 
	$
	Because 
	\begin{align*}
		\tilde{\theta}_n - \theta_n 
		& = 
		\big\{ \sqrt{V_{\tau\tau} (1-R^2_n)+n^{-1} S^2_{\tau \setminus \bs{X}}} 
		- 
		\sqrt{V_{\tau\tau} (1-R^2_n)} \big\} \cdot \varepsilon_0
		= 
		\sqrt{n^{-1} S^2_{\tau \setminus \bs{X}}} \cdot O_{\PP}(1)\\
		& = 	\sqrt{V_{\tau\tau} (1-R^2_n)} \cdot o_{\PP}(1), 
	\end{align*}
	from Lemma \ref{lemma:diminish_inf_norm_two_rv}, we must have $\sup_{c\in \mathbb{R}} |\PP(\theta_n \le c) - \PP(\tilde{\theta}_n\le c)| \rightarrow 0$.  
	This then implies that 
	\begin{align*}
		\PP\{ |\hat{\tau}-\tau|  \le \tilde{q}_{n, 1-\alpha/2 + \eta} \mid M \le a_n \} 
		& = 
		\PP(|\theta_n| \le \tilde{q}_{n, 1-\alpha/2 + \eta}) + o(1)
		= 
		\PP(|\tilde{\theta}_n| \le \tilde{q}_{n, 1-\alpha/2 + \eta}) + o(1)
		\\
		& = 1 - \alpha + 2\eta + o(1). 
	\end{align*}
	
	Second, by the same logic as the proof of  \eqref{eq:conv_quantile_est_tilde} in Theorem \ref{thm:inf}(ii), we can derive that, for any $0< \alpha < \beta < 1$, 
	$
	\PP( \tilde{q}_{n, \beta}\le \hat{q}_{n, \alpha} \mid M\le a_n)
	= 
	\E\{\I(\tilde{q}_{n, \beta}\le \hat{q}_{n, \alpha}) \mid M \le a_n\} 
	\rightarrow 0.
	$
	This immediately implies that 
	\begin{align*}
		\PP( \hat{q}_{n, 1-\alpha/2} > \tilde{q}_{n, 1-\alpha/2 + \eta} \mid M \le a_n )
		& = o(1). 
	\end{align*}
	
	From the above, we can know that 
	$
	\limsup_{n\rightarrow\infty}\PP(\tau \in \hat{\mathcal{C}}_{\alpha} \mid M\le a_n) 
	\le 1 - \alpha + 2\eta.
	$
	Because this inequality holds for any $\eta \in (0, \alpha/2)$, we must have 
	$
	\limsup_{n\rightarrow\infty}\PP(\tau \in \hat{\mathcal{C}}_{\alpha} \mid M\le a_n) 
	\le 1 - \alpha.
	$
	From Theorem \ref{thm:inf}(ii), we then have 
	$
	\lim_{n\rightarrow\infty}\PP(\tau \in \hat{\mathcal{C}}_{\alpha} \mid M\le a_n) 
	= 1 - \alpha.
	$
	Therefore, Theorem \ref{thm:inf}(iii) holds. 
\end{proof}

To prove Theorem \ref{thm:inf_gaussian}, we need the following lemma. 

\begin{lemma}\label{lemma:non_Gaussian_quantile_diminish}
	Let $\varepsilon_0 \sim \mathcal{N}(0,1)$,  and $L_{K_n, a_n}$ be the truncated Gaussian random variables defined as in Section \ref{sec:recent_result},
	where $\{K_n\}$ and $\{a_n\}$ are sequences of positive integers and thresholds, 
	and $\varepsilon_0$ is independent of $L_{K_n, a_n}$ for all $n$. 
	Let $\{A_n\}$,  $\{B_n\}$, $\{\tilde{A}_n\}$ and $\{\tilde{B}_n\}$  be sequences of nonnegative constants, 
	and for each $n$, 
	define $\psi_n = A_n^{1/2} \cdot \varepsilon_0 + B_n^{1/2} \cdot L_{K_n, a_n}$ and $\tilde{\psi}_n = \tilde{A}_n^{1/2} \cdot \varepsilon_0 + \tilde{B}_n^{1/2} \cdot L_{K_n, a_n}$. 
	For each $n$ and $\alpha \in (0,1)$, let $q_{n}(\alpha)$ and $\tilde{q}_{n}(\alpha)$ be the $\alpha$th quantile of $\psi_n$ and $\tilde{\psi}_n$, respectively. 
	If $L_{K_n, a_n} = o_{\PP}(1)$, $\tilde{A}_n - A_n = o(A_n)$ and 
	$\tilde{B}_n - B_n = O(A_n)$, 
	then for any $0 < \alpha < \beta < 1$, as $n \rightarrow \infty$, 
	$\I \{ \tilde{q}_n(\beta) \le q_n(\alpha) \} \rightarrow 0$ 
	and 
	$\I \{ q_n(\beta) \le \tilde{q}_n(\alpha) \} \rightarrow 0$. 
\end{lemma}

\begin{proof}[Proof of Lemma \ref{lemma:non_Gaussian_quantile_diminish}]
	Note that $L_{K_n, a_n} = o_{\PP}(1)$.
	Using the inequality that $|\sqrt{b}-\sqrt{c}| \le \sqrt{|b-c|}$ for any $b, c\ge 0$, we then have
	\begin{align*}
		\tilde{\psi}_n - \psi_n & = 
		\big(\tilde{A}_n^{1/2} - A_n^{1/2} \big) \varepsilon_0 + \big(\tilde{B}_n^{1/2} - B_n^{1/2} \big) L_{K_n, a_n}
		= |\tilde{A}_n - A_n|^{1/2} \cdot O_{\PP}(1) + |\tilde{B}_n - B_n|^{1/2} \cdot o_{\PP}(1)
		\\
		& = A_n^{1/2} \cdot o_{\PP}(1). 
	\end{align*}
	From Lemma \ref{lemma:diminish_inf_norm_two_rv}, this then implies that 
	$\sup_{c\in \mathbb{R}} | \PP(\tilde{\psi}_n  \le c) - \PP(\psi_n \le c) | \rightarrow 0$ as $n\rightarrow \infty$. 
	Lemma \ref{lemma:non_Gaussian_quantile_diminish} then follows immediately from Lemma \ref{lemma:quant_converg}. 
\end{proof}

\begin{proof}[\bf Proof of Theorem \ref{thm:inf_gaussian}]
	Following the notation in the proof of Theorem \ref{thm:inf}(ii), we define additionally
	\begin{align*}
		\check{\theta}_n & = \sqrt{\hat{V}_{\tau\tau} (1-\hat{R}^2_n)} \cdot \varepsilon_0 + 0 \cdot L_{K_n, a_n}
		\equiv \check{A}_n^{1/2} \cdot \varepsilon_0 + \check{B}_n^{1/2} \cdot L_{K_n, a_n}, 
	\end{align*}
	and $\check{q}_{n, \alpha} = q_{\alpha}(\check{A}_n, \check{B}_n, K_n, a_n)$. 
	where $\check{A}_n$ nd $\check{B}_n$ denote the squared coefficients of the standard and constrained Gaussian random variables, respectively.
	From Theorem \ref{thm:inf}(i) and Condition \ref{cond:rsup}, 
	$|\check{A}_n - \tilde{A}_n| = |\hat{A}_n - \tilde{A}_n| = o_{\PP}(\tilde{A}_n)$, 
	and, for sufficiently large $n$, 
	$
	|\check{B} - \tilde{B}_n| =  \sqrt{V_{\tau\tau} R_n^2} =  \sqrt{V_{\tau\tau}} \cdot O(1)
	= 
	\sqrt{V_{\tau\tau} (1-R_n^2)}  \cdot O(1) = O(\tilde{A}_n). 
	$
	We can then prove Theorem \ref{thm:inf_gaussian} using almost the same steps as the proof of Theorem \ref{thm:inf}, where we will replace $\hat{q}_{n, \alpha}$ by $\check{q}_{n, \alpha}$ and use Lemma \ref{lemma:non_Gaussian_quantile_diminish} instead of Lemma \ref{lemma:non_Gaussian_quantile}. 
	For conciseness, we omit the detailed proof here. 
\end{proof}

\section{Regularity Conditions and Diagnoses for Rerandomization}\label{sec:regularity_diagnosis}

To prove Proposition \ref{prop:iidrate}, we need the following two lemmas. 

\begin{lemma}%
	\label{lem:lei}
	Let $\bs{W}_1, \bs{W}_2, \ldots$ be i.i.d. random vectors in $\R^{K_n}$ with $\E[\bs{W}_i] = \bs{0}$ and $\cov(\bs{W}_i) = \bs{I}_{K_n}$. Assume that 
	\begin{align*}
		\sup_{\bs{\nu} \in \R^{K_n}: \bs{\nu}^\top \bs{\nu} = 1}\E|\nu^\top \bs{W}_i|^\delta = O(1)
		\ \ \ \text{and} \ \ \ 
		\max_{1\le i \le n} \left|\|\bs{W}_i\|_2^2 - \E\|\bs{W}_i\|_2^2 \right| = O_{\PP}(\omega(n, K_n)), 
	\end{align*}
	for some $\delta > 2$ and some function $\omega(n, K_n)$ increasing in $n$ and $K_n$. 
	If 
	$K_n = O(n^\beta)$ for some $0 < \beta < 1$, then when $n$ is sufficiently large, 
	\begin{align*}
		\|\bs{S}_{\bs{W}}^2 - \bs{I}_{K_n}\|_{\op} = O_{\PP} \left(\frac{\omega(n, K_n)}{n} + \left(\frac{K_n}{n}\right)^{\frac{\delta - 2}{\delta}}\log^4\left(\frac{n}{K_n}\right) + \left(\frac{K_n}{n}\right)^{\frac{\min\{\delta - 2, 2\}}{\min\{\delta, 4\}}}\right)
	\end{align*}
	and 
	\begin{align*}
		\max_{1 \leq i \leq n} \|(\bs{S}_{\bs{W}}^2)^{-1/2} (\bs{W}_i - \bar{\bs{W}})\|_2^2 = O_{\PP} \left(\omega(n, K_n) + \frac{K_n^{\frac{2\delta - 2}{\delta}}}{n^{\frac{\delta - 2}{\delta}}} \log^4\left(\frac{n}{K_n}\right) + n \cdot \left(\frac{K_n}{n}\right)^{\frac{\min\{2 \delta - 2, 6\}}{\min\{\delta, 4\}}} + K_n\right), 
	\end{align*}
	where $\bar{\bs{W}} = n^{-1} \sum_{i=1}^n \bs{W}_i$ and $\bs{S}^2_{\bs{W}} = (n-1)^{-1} \sum_{i=1}^n (\bs{W}_i - \bar{\bs{W}}) (\bs{W}_i - \bar{\bs{W}})^\top$. 
\end{lemma}

\begin{proof}[Proof of Lemma \ref{lem:lei}]
	Lemma \ref{lem:lei} follows immediately from \citet[][Lemma H.1]{LeiD20}. 
	Let $\tilde{\bs{W}} = (\bs{W}_1 - \bar{\bs{W}}, \ldots, \bs{W}_n - \bar{\bs{W}})^\top$, and 
	$\tilde{\bs{H}} = \tilde{\bs{W}} (\tilde{\bs{W}}^\top \tilde{\bs{W}})^{-1} \tilde{\bs{W}}^\top$. 
	We can verify that 
	$\tilde{\bs{W}}^\top \tilde{\bs{W}} = \sum_{i=1}^n (\bs{W}_i - \bar{\bs{W}}) (\bs{W}_i - \bar{\bs{W}})^\top = (n-1) \bs{S}^2_{\bs{W}}$, and 
	the $i$th diagonal element of $\tilde{\bs{H}}$ has the following equivalent forms:
	\begin{align*}
		H_{ii} & = (\bs{W}_i - \bar{\bs{W}})^\top (\tilde{\bs{W}}^\top \tilde{\bs{W}})^{-1} (\bs{W}_i - \bar{\bs{W}}) = 
		(n-1)^{-1} (\bs{W}_i - \bar{\bs{W}})^\top (\bs{S}^2_{\bs{W}})^{-1} (\bs{W}_i - \bar{\bs{W}}) 
		\\
		& = (n-1)^{-1}  \|(\bs{S}_{\bs{W}}^2)^{-1/2} (\bs{W}_i - \bar{\bs{W}})\|_2^2. 
	\end{align*}
	From \citet[][Lemma H.1]{LeiD20}, we can know that 
	\begin{align*}
		\left\| n^{-1} \tilde{\bs{W}}^\top \tilde{\bs{W}}  - \bs{I}_{K_n} \right\|_{\op}
		= 
		O_{\PP} \left(\frac{\omega(n, K_n)}{n} + \left(\frac{K_n}{n}\right)^{\frac{\delta - 2}{\delta}}\log^4\left(\frac{n}{K_n}\right) + \left(\frac{K_n}{n}\right)^{\frac{\min\{\delta - 2, 2\}}{\min\{\delta, 4\}}}\right), 
	\end{align*}
	and 
	\begin{align*}
		\max_{1\le i \le n} |H_{ii} - K_n/n| & = O_{\PP} 
		\left( 
		\frac{\omega(n, K_n)}{n} + \left( \frac{K_n}{n} \right)^{\frac{2\delta-2}{\delta}} \log^4\left( \frac{n}{K_n}  \right)	+ \left( \frac{K_n}{n} \right)^{\frac{\min\{2\delta-2, 6\}}{\min\{\delta, 4\}\}}}	
		\right). 
	\end{align*}
	These  immediately imply that 
	\begin{align*}
		\|\bs{S}_{\bs{W}}^2 - \bs{I}_{K_n}\|_{\op} & = 
		\left\| \frac{n}{n-1}\left( n^{-1} \tilde{\bs{W}}^\top \tilde{\bs{W}} - \bs{I}_{K_n} \right) + \frac{1}{n-1} \bs{I}_{K_n} \right\|_{\op}
		\le \frac{n}{n-1} \left\|n^{-1} \tilde{\bs{W}}^\top \tilde{\bs{W}} - \bs{I}_{K_n}  \right\|_{\op} + \frac{1}{n-1}
		\\
		& \le 2 \left\|n^{-1} \tilde{\bs{W}}^\top \tilde{\bs{W}} - \bs{I}_{K_n}  \right\|_{\op} + \frac{2}{n}
		\\
		& = O_{\PP} \left(\frac{\omega(n, K_n)}{n} + \left(\frac{K_n}{n}\right)^{\frac{\delta - 2}{\delta}}\log^4\left(\frac{n}{K_n}\right) + \left(\frac{K_n}{n}\right)^{\frac{\min\{\delta - 2, 2\}}{\min\{\delta, 4\}}}\right), 
	\end{align*}
	where the last equality holds because $(K_n/n)^{\min\{\delta - 2, 2\}/\min\{\delta, 4\}} \ge (K_n/n)^{1/2} \ge 1/n$ when $n$ is sufficiently large, 
	and 
	\begin{align*}
		\max_{1 \leq i \leq n} \|(\bs{S}_{\bs{W}}^2)^{-1/2} (\bs{W}_i - \bar{\bs{W}})\|_2^2  & = 
		(n-1)\max_{1\le i \le n} H_{ii} \le n \max_{1\le i \le n} |H_{ii} - K_n/n| + K_n
		\\
		& = 
		O_{\PP} 
		\left( 
		\omega(n, K_n) + \frac{K_n^{\frac{2\delta-2}{\delta}} }{n^{\frac{\delta-2}{\delta}} }  \log^4\left( \frac{n}{K_n}  \right)	+ n \left( \frac{K_n}{n} \right)^{\frac{\min\{2\delta-2, 6\}}{\min\{\delta, 4\}\}}} + K_n	
		\right).
	\end{align*}
	Therefore, Lemma  \ref{lem:lei} holds. 
\end{proof}

\begin{lemma}\label{lem:max}
	Assume that $\bs{W}_1, \bs{W}_2, \ldots$ are i.i.d. random vectors in $\R^{K_n}$ with $\max_{1 \leq j \leq K_n} \E |W_{ij}|^\delta \le M$ for some absolute constants $M<\infty$ and $\delta > 2$, where $W_{ij}$ is the $j$th coordinate of $\bs{W}_i$. Then 
	\begin{align*}
		\max_{1\le i \le n} \left|\|\bs{W}_i\|_2^2 - \E\|\bs{W}_i\|_2^2 \right| = O_{\PP}(n^{2/\delta} K_n).  
	\end{align*}
\end{lemma}

\begin{proof}[Proof of Lemma \ref{lem:max}]
	By H\"older's inequality, for $1\le i \le n$ and $1\le j \le K_n$, $\{ \E(W_{ij}^2) \} ^{\delta/2} \le \E |W_{ij}|^{\delta}$, 
	\begin{align*}
		\frac{1}{K_n} \left|\|\bs{W}_i\|_2^2 - \E\|\bs{W}_i\|_2^2 \right| & 
		\le 
		\frac{1}{K_n}
		\sum_{j=1}^{K_n} \left|   W_{ij}^2 - \E(W_{ij}^2) \right|
		\le 
		\left(\frac{1}{K_n} \sum_{j=1}^{K_n} \left|   W_{ij}^2 - \E(W_{ij}^2) \right|^{\delta/2} \right)^{2/\delta}, 
	\end{align*}
	and 
	\begin{align*}
		\frac{1}{2}\left|   W_{ij}^2 - \E(W_{ij}^2) \right| 
		\le 
		\frac{1}{2} \left( W_{ij}^2  + \E(W_{ij}^2) \right)
		\le 
		\left[
		\frac{1}{2} \left\{ |W_{ij}|^{\delta}  + \left( \E(W_{ij}^2) \right) ^{\delta/2} \right\}
		\right]^{2/\delta}. 
	\end{align*}
	These imply that 
	\begin{align*}
		& \quad \ \E\left\{\left|\|\bs{W}_i\|_2^2 - \E\|\bs{W}_i\|_2^2\right|^{\delta/2} \right\} 
		\\
		& = K_n^{\delta/2}\E\left\{\frac{1}{K_n^{\delta/2}}\left|\|\bs{W}_i\|_2^2 - \E\|\bs{W}_i\|_2^2\right|^{\delta/2} \right\} 
		\le 
		K_n^{\delta/2} \E \left(\frac{1}{K_n} \sum_{j=1}^{K_n} \left|   W_{ij}^2 - \E(W_{ij}^2) \right|^{\delta/2} \right)
		\\
		& = K_n^{\delta/2-1} 2^{\delta/2} \sum_{j=1}^{K_n}  \E \left( \frac{1}{2^{\delta/2}}\left|   W_{ij}^2 - \E(W_{ij}^2) \right|^{\delta/2} \right) 
		\le  K_n^{\delta/2-1} 2^{\delta/2} \sum_{j=1}^{K_n} \E \left[
		\frac{1}{2} \left\{ |W_{ij}|^{\delta}  + \left( \E(W_{ij}^2) \right) ^{\delta/2} \right\}
		\right]\\
		& = K_n^{\delta/2-1} 2^{\delta/2-1} \sum_{j=1}^{K_n}  
		\left\{ \E |W_{ij}|^{\delta}  + \left( \E(W_{ij}^2) \right) ^{\delta/2} \right\}
		\le K_n^{\delta/2-1} 2^{\delta/2} \sum_{j=1}^{K_n}  
		\E |W_{ij}|^{\delta} 
		\\
		& 
		\le 2^{\delta/2} K_n^{\delta/2} M. 
	\end{align*}
	Consequently, 
	\begin{align*}
		\E\left\{ \max_{1\le i \le n} \left|\|\bs{W}_i\|_2^2 - \E\|\bs{W}_i\|_2^2\right|^{\delta/2} \right\} 
		\le 
		\sum_{i=1}^n 
		\E\left\{\left|\|\bs{W}_i\|_2^2 - \E\|\bs{W}_i\|_2^2\right|^{\delta/2} \right\} 
		\le 2^{\delta/2} M K_n^{\delta/2}  n. 
	\end{align*}
	By the Markov's inequality, 
	$
	\max_{1\le i \le n} \left|\|\bs{W}_i\|_2^2 - \E\|\bs{W}_i\|_2^2\right| = O_{\PP} (K_n  n^{2/\delta}), 
	$
	i.e., Lemma \ref{lem:max} holds. 
\end{proof}

\begin{proof}[\bf Proof of Proposition~\ref{prop:iidrate}]
	By the same logic as Lemma \ref{lemma:gamma_srs_bound}, we can bound $\gamma_n$ by 
	\begin{align}\label{eq:gamma_nub}
		\gamma_n & \equiv 
		\frac{(K_n+1)^{1/4}}{\sqrt{n r_1r_0}} \frac{1}{n} \sum_{i=1}^n \left\| \left(\bs{S}_{\bs{u}}^2\right)^{-1/2} (\bs{u}_i-\bar{\bs{u}}) \right\|_2^3 \nonumber 
		\\
		& \le 
		\frac{(K_n+1)^{1/4}}{\sqrt{n r_1r_0}} \frac{1}{n} \sum_{i=1}^n \left\| \left(\bs{S}_{\bs{u}}^2\right)^{-1/2} (\bs{u}_i-\bar{\bs{u}}) \right\|_2^2 \cdot \max_{1\le i \le n} \left\| \left(\bs{S}_{\bs{u}}^2\right)^{-1/2} (\bs{u}_i-\bar{\bs{u}}) \right\|_2 \nonumber\\
		& = 
		\frac{(K_n+1)^{1/4}}{\sqrt{n r_1r_0}} \frac{(n-1)(K_n+1)}{n} \cdot \max_{1\le i \le n} \left\| \left(\bs{S}_{\bs{u}}^2\right)^{-1/2} (\bs{u}_i-\bar{\bs{u}}) \right\|_2 \nonumber \\
		& = \frac{(K_n+1)^{5/4}}{\sqrt{n r_1r_0}} \frac{n-1}{n} \cdot \max_{1\le i \le n} \left\| \left(\bs{S}_{\bs{\xi}}^2\right)^{-1/2} (\bs{\xi}_i-\bar{\bs{\xi}}) \right\|_2, 
	\end{align}
	where the last equality holds because the quantity $\| \left(\bs{S}_{\bs{u}}^2\right)^{-1/2} (\bs{u}_i-\bar{\bs{u}}) \|_2$ is invariant under a non-singular linear transformation of $\bs{u}_i$'s. 
	Under Condition \ref{cond:moment} and the fact that $K_n + 1= O(n^{\beta})$ for some $\beta\in (0,1)$, from Lemmas \ref{lem:lei} and \ref{lem:max}, 
	we can know that %
	\begin{align*}
		& \quad \ \max_{1 \leq i \leq n} \|(\bs{S}_{\bs{\xi}}^2)^{-1/2} (\bs{\xi}_i - \bar{\bs{\xi}})\|_2^2 
		\\
		& = O_{\PP} \left(n^{2/\delta} (K_n+1) + \frac{(K_n+1)^{\frac{2\delta - 2}{\delta}}}{n^{\frac{\delta - 2}{\delta}}} \log^4\left(\frac{n}{K_n+1}\right) + n \cdot \left(\frac{K_n+1}{n}\right)^{\frac{\min\{2 \delta - 2, 6\}}{\min\{\delta, 4\}}} + K_n+1\right). 
	\end{align*}
	Note that as $n\rightarrow \infty$, $(K_n+1)/n = o(1)$, 
	\begin{align*}
		\frac{1}{n^{2/\delta} (K_n+1)}
		\frac{(K_n+1)^{\frac{2\delta - 2}{\delta}}}{n^{\frac{\delta - 2}{\delta}}} \log^4\left(\frac{n}{K_n+1}\right)
		& = 
		\frac{1}{(K_n+1)^{2/\delta}}
		\frac{K_n+1}{n} \log^4\left(\frac{n}{K_n+1}\right) = o(1), 
	\end{align*}
	and 
	\begin{align}\label{eq:bound_xi}
		& \quad \ \frac{1}{n^{2/\delta} (K_n+1)} n \cdot \left(\frac{K_n+1}{n}\right)^{\frac{\min\{2 \delta - 2, 6\}}{\min\{\delta, 4\}}}
		\nonumber
		\\
		& = 
		\I(\delta \le 4) \frac{1}{n^{2/\delta} (K_n+1)} n \cdot \left(\frac{K_n+1}{n}\right)^{2 - 2/\delta} 
		+ 
		\I(\delta > 4 ) \frac{1}{n^{2/\delta} (K_n+1)} n \cdot \left(\frac{K_n+1}{n}\right)^{3/2}
		\nonumber
		\\
		& = 
		\I(\delta \le 4) \frac{(K_n+1)^{1-2/\delta}}{n} + \I(\delta > 4 ) \frac{(K_n+1)^{1/2}}{n^{2/\delta+1/2}} = o(1).
	\end{align}
	Thus, we must have 
	\begin{align*}
		\max_{1 \leq i \leq n} \|(\bs{S}_{\bs{\xi}}^2)^{-1/2} (\bs{\xi}_i - \bar{\bs{\xi}})\|_2^2 = O_{\PP} \left( n^{2/\delta} (K_n+1) \right).  
	\end{align*}
	From \eqref{eq:gamma_nub}, this then implies that 
	\begin{align*}
		\gamma_n = \frac{(K_n+1)^{5/4}}{\sqrt{n r_1r_0}} \frac{n-1}{n} \cdot O_{\PP} \left( \sqrt{n^{2/\delta} (K_n+1)} \right)
		= 
		O_{\PP} \left( 
		\frac{1}{\sqrt{r_1r_0}} \frac{(K_n+1)^{7/4}}{n^{1/2-1/\delta}} \right).
	\end{align*}
	Therefore, Proposition \ref{prop:iidrate} holds. 
\end{proof}

\begin{proof}[\bf Proof of Corollary \ref{cor:reg_cond}(i)]
	When Condition \ref{cond:moment} holds,  $r_z^{-1} = O(1)$ and $K_n = o(n^{2/7-4/(7\delta)})$, 
	from Proposition \ref{prop:iidrate}, we have 
	\begin{align*}
		\gamma_n 
		= 
		O_{\PP} \left( 
		\frac{1}{\sqrt{r_1r_0}} \frac{(K_n+1)^{7/4}}{n^{1/2-1/\delta}} \right)
		= 
		\frac{\{n^{2/7-4/(7\delta)}\}^{7/4}}{n^{1/2-1/\delta}} \cdot
		o_{\PP} (1)
		= o_{\PP} (1). 
	\end{align*}
	Therefore, Corollary \ref{cor:reg_cond}(i) holds. 
\end{proof}

To prove Corollary \ref{cor:reg_cond}(ii), we need the following two lemmas.

\begin{lemma}\label{lemma:R2_equiv}
	The squared multiple correlation $R_n^2$ defined as in \eqref{eq:R2} can be equivalently written as 
	\begin{align*}
		R_n^2 = 
		\frac{\bs{S}_{r_0Y(1) + r_1 Y(0), \bs{X}} (\bs{S}_{\bs{X}}^2)^{-1} \bs{S}_{\bs{X}, r_0Y(1) + r_1 Y(0)}}{S^2_{r_0Y(1)+r_1Y(0)}}, 
	\end{align*}
	where $S^2_{r_0Y(1)+r_1Y(0)}$ denotes the finite population variance of $r_0Y(1)+r_1Y(0)$ and $\bs{S}_{r_0Y(1) + r_1 Y(0), \bs{X}}$ denotes the finite population covariance between $r_0Y(1)+r_1Y(0)$  and $\bs{X}$. 
\end{lemma}

\begin{proof}[Proof of Lemma \ref{lemma:R2_equiv}]
	Let $S_{10}$ be the finite population covariance between $Y(1)$ and $Y(0)$. 
	By some algebra, $S^2_{r_0Y(1)+r_1Y(0)}$ has the following equivalent forms:
	\begin{align}\label{eq:V_tautau_equiv}
		S^2_{r_0Y(1)+r_1Y(0)}
		& = 
		r_0^2 S^2_1 + r_1^2 S_0^2 + 2 r_1 r_0 S_{10}
		= 
		(r_0^2+r_1r_0) S^2_1 + (r_1^2+r_1r_0) S_0^2 - r_1r_0(S^2_1 + S_0^2 - 2S_{10})
		\nonumber
		\\
		& = r_0  S^2_1 + r_1 S_0^2 - r_1 r_0 S^2_{\tau} =
		nr_1r_0 (n_1^{-1} S^2_1 + n_0^{-1}  S_0^2 - n^{-1} S^2_{\tau}   )
		= nr_1r_0 V_{\tau\tau}. 
	\end{align}
	By the same logic, we have 
	\begin{align*}
		\bs{S}_{r_0Y(1) + r_1 Y(0), \bs{X}} (\bs{S}_{\bs{X}}^2)^{-1} \bs{S}_{\bs{X}, r_0Y(1) + r_1 Y(0)}
		& = 
		nr_1r_0 (n_1^{-1} S^2_{1\mid \bs{X}} + n_0^{-1}  S_{0\mid \bs{X}}^2 - n^{-1} S^2_{\tau\mid \bs{X}}   ). 
	\end{align*}
	Lemma \ref{lemma:R2_equiv} then follows from the definition in \eqref{eq:R2}. 
\end{proof}

\begin{lemma}\label{lemma:matrix_inv_converge}
	For any sequence of positive integers $\{K_n\}$ and any sequence of matrices $\bs{A}_n \in \mathbb{R}^{K_n}$,  
	if $\|\bs{A}_n - \bs{I}_{K_n}\|_{\op} = o_{\PP}(1)$, then $\|\bs{A}_n^{-1}- \bs{I}_{K_n}\|_{\op} = o_{\PP}(1)$.
\end{lemma}

\begin{proof}[Proof of Lemma \ref{lemma:matrix_inv_converge}]
	Note that 
	\begin{align*}
		\|\bs{A}_n^{-1} - \bs{I}_{K_n}\|_{\op}
		& = 
		\|\bs{A}_n^{-1} (\bs{A}_n - \bs{I}_{K_n})\|_{\op}
		\le 
		\| \bs{A}_n^{-1} \|_{\op} \cdot \| \bs{A}_n - \bs{I}_{K_n} \|_{\op}
		\\
		& = 
		\| \bs{I}_{K_n} + (\bs{A}_n^{-1} - \bs{I}_{K_n}) \|_{\op} \cdot \| \bs{A}_n - \bs{I}_{K_n} \|_{\op}
		\\
		& \le 
		\big( \| \bs{I}_{K_n} \|_{\op} + \| \bs{A}_n^{-1} - \bs{I}_{K_n} \|_{\op} \big) \cdot \| \bs{A}_n - \bs{I}_{K_n} \|_{\op}
		\\
		& \le \| \bs{A}_n - \bs{I}_{K_n} \|_{\op} + \| \bs{A}_n - \bs{I}_{K_n} \|_{\op} \cdot \| \bs{A}_n^{-1} - \bs{I}_{K_n} \|_{\op}. 
	\end{align*}
	Thus, when $\| \bs{A}_n - \bs{I}_{K_n} \|_{\op} < 1$, we have 
	$
	\|\bs{A}_n^{-1} - \bs{I}_{K_n}\|_{\op} \le 
	\| \bs{A}_n - \bs{I}_{K_n} \|_{\op}/(1-\| \bs{A}_n - \bs{I}_{K_n} \|_{\op}). 
	$
	By the property of convergence in probability \citep[e.g.,][Theorem 2.3.2]{D19}, we can immediately derive Lemma \ref{lemma:matrix_inv_converge}. 
\end{proof}

\begin{proof}[\bf Proof of Corollary \ref{cor:reg_cond}(ii)]
	We first prove that $
	\|\bs{S}_{\bs{\xi}}^2 - \bs{I}_{K_n+1}\|_\op = o_{\PP}(1). 
	$
	Under Condition \ref{cond:moment}, from Lemmas \ref{lem:lei} and \ref{lem:max}, 
	\begin{align*}
		\|\bs{S}_{\bs{\xi}}^2 - \bs{I}_{K_n+1}\|_{\op} = O_{\PP} \left(\frac{n^{2/\delta} (K_n+1)}{n} + \left(\frac{K_n+1}{n}\right)^{\frac{\delta - 2}{\delta}}\log^4\left(\frac{n}{K_n+1}\right) + \left(\frac{K_n+1}{n}\right)^{\frac{\min\{\delta - 2, 2\}}{\min\{\delta, 4\}}}\right). 
	\end{align*}
	We can verify that $2/7 - 4/(7\delta) < 1 - 2/\delta$ for $\delta > 2$.  Thus, we must have $(K_n+1)/n^{1-2/\delta} = o(1)$, which further implies that 
	$
	\|\bs{S}_{\bs{\xi}}^2 - \bs{I}_{K_n+1}\|_\op = o_{\PP}(1). 
	$

	We then prove that $R^2_n - R^2_{\sup, n} = o_{\PP}(1)$. 
	By definition, 
	\begin{align*}
		\begin{pmatrix}
			r_0Y_i(1) + r_1 Y_i(0) - \E\{r_0Y_i(1) + r_1 Y_i(0) \}\\
			\bs{X}_i - \E (\bs{X}_i)
		\end{pmatrix}
		& = \bs{u}_i
		= 
		\cov (\bs{u})^{1/2} \bs{\xi}_i 
		= 
		\begin{pmatrix}
			(1, \bs{0}_{K_n}^\top)  \cov (\bs{u})^{1/2} \bs{\xi}_i \\
			(\bs{0}_{K_n}, \bs{I}_{K_n}) \cov (\bs{u})^{1/2} \bs{\xi}_i 
		\end{pmatrix}
		\\
		& \equiv 
		\begin{pmatrix}
			\bs{a}^\top \bs{\xi}_i \\
			\bs{B}^\top \bs{\xi}_i 
		\end{pmatrix}, 
	\end{align*}
	where $\bs{a}^\top = (1, \bs{0}_{K_n}^\top)  \cov (\bs{u})^{1/2}  \in \mathbb{R}^{1 \times (K_n+1)}$ and $\bs{B}^\top = (\bs{0}_{K_n}, \bs{I}_{K_n}) \cov (\bs{u})^{1/2} \in \mathbb{R}^{K_n \times (K_n+1)}$. 
	We can then verify that 
	\begin{align*}
		\var\{ r_0Y(1) + r_1 Y(0) \} = \bs{a}^\top \bs{a}, \ \ \
		\var(\bs{X}_i) = \bs{B}^\top \bs{B},  \ \ \ 
		\cov\{r_0Y(1) + r_1 Y(0), \bs{X}\}= \bs{a}^\top \bs{B}. 
	\end{align*}
	Consequently, the super population squared multiple correlation between $r_0Y_i(1) + r_1 Y_i(0)$ and $\bs{X}_i$ has the following equivalent forms: 
	\begin{align*}
		R^2_{\sup,n} & = 
		\frac{\cov\{r_0Y(1) + r_1 Y(0), \bs{X}_i\} \{\var(\bs{X})\}^{-1} \cov\{\bs{X}, r_0Y(1) + r_1 Y(0)\}}{\var\{ r_0Y(1) + r_1 Y(0) \}}
		= 
		\frac{\bs{a}^\top \bs{B} (\bs{B}^\top \bs{B})^{-1} \bs{B}^\top \bs{a}}{\bs{a}^\top \bs{a}}. 
	\end{align*}
	From Lemma \ref{lemma:R2_equiv},  the finite population squared multiple correlation 
	$R_n^2$
	satisfies that 
	\begin{align*}
		S^2_{r_0Y(1)+r_1Y(0)} R_n^2 = 
		\bs{S}_{r_0Y(1) + r_1 Y(0), \bs{X}} (\bs{S}_{\bs{X}}^2)^{-1} \bs{S}_{\bs{X}, r_0Y(1) + r_1 Y(0)}
		= 	
		\bs{a}^\top \bs{S}^2_{\bs{\xi}} \bs{B} \left( \bs{B}^\top \bs{S}^2_{\bs{\xi}} \bs{B} \right)^{-1} \bs{B}^\top \bs{S}^2_{\bs{\xi}}  \bs{a}. 
	\end{align*}
	Let $\tilde{\bs{a}} = \bs{a}/\sqrt{\bs{a}^\top \bs{a}}$ and $\bs{B} =\bs{Q} \bs{C} \bs{\Gamma}^\top$ be the singular value decomposition of $\bs{B}$, where $\bs{Q}\in \mathbb{R}^{(K_n+1) \times K_n}$, $\bs{Q}^\top \bs{Q} = \bs{I}_{K_n}$, $\bs{C}\in \mathbb{R}^{K_n\times K_n}$ is a diagonal matrix, and $\bs{\Gamma}\in \mathbb{R}^{K_n\times K_n}$ is an orthogonal matrix. 
	We then have 
	\begin{align*}
		& \quad \ \frac{S^2_{r_0Y(1)+r_1Y(0)} }{\var\{ r_0Y(1) + r_1 Y(0) \}} R_n^2 - R^2_{\sup} 
		\\ 
		& =
		\tilde{\bs{a}}^\top \bs{S}^2_{\bs{\xi}} \bs{B} \left( \bs{B}^\top \bs{S}^2_{\bs{\xi}} \bs{B} \right)^{-1} \bs{B}^\top \bs{S}^2_{\bs{\xi}}  \tilde{\bs{a}}
		-
		\tilde{\bs{a}}^\top \bs{B} (\bs{B}^\top \bs{B})^{-1} \bs{B}^\top \tilde{\bs{a}} 
		= 
		\tilde{\bs{a}}^\top \bs{S}^2_{\bs{\xi}} \bs{Q} \left( \bs{Q}^\top \bs{S}^2_{\bs{\xi}} \bs{Q} \right)^{-1} \bs{Q}^\top \bs{S}^2_{\bs{\xi}}  \tilde{\bs{a}}
		- 
		\tilde{\bs{a}}^\top \bs{Q} \bs{Q}^\top \tilde{\bs{a}} 
		\\
		& = 
		\tilde{\bs{a}}^\top \bs{S}^2_{\bs{\xi}} \bs{Q} 
		\left\{ \left( \bs{Q}^\top \bs{S}^2_{\bs{\xi}} \bs{Q} \right)^{-1} - \bs{I}_{K_n}
		\right\}
		\bs{Q}^\top \bs{S}^2_{\bs{\xi}}  \tilde{\bs{a}}
		+ 
		\tilde{\bs{a}}^\top \left( \bs{S}^2_{\bs{\xi}} -\bs{I}_{K_n}\right) \bs{Q} 
		\bs{Q}^\top \bs{S}^2_{\bs{\xi}}  \tilde{\bs{a}}
		+
		\tilde{\bs{a}}^\top \bs{Q} 
		\bs{Q}^\top \left( \bs{S}^2_{\bs{\xi}} - \bs{I}_{K_n}  \right) \tilde{\bs{a}}. 
	\end{align*}
	By the property of operator norm and the fact that $\tilde{\bs{a}}^\top \tilde{\bs{a}} = 1$ and $\bs{Q}^\top \bs{Q} = \bs{I}_{K_n}$, we then have 
	\begin{align*}
		\left| \frac{S^2_{r_0Y(1)+r_1Y(0)} R_n^2}{\var\{ r_0Y(1) + r_1 Y(0) \}}  - R^2_{\sup}  \right|
		\le 
		\Big\| \left( \bs{Q}^\top \bs{S}^2_{\bs{\xi}} \bs{Q} \right)^{-1} - \bs{I}_{K_n}
		\Big\|_{\op}
		\left\| \bs{S}^2_{\bs{\xi}}   \right\|_{\op}^2 
		+ 
		\left\| \bs{S}^2_{\bs{\xi}} -\bs{I}_{K_n} \right\|_{\op} \left( \left\| \bs{S}^2_{\bs{\xi}}   \right\|_{\op} + 1 \right). 
	\end{align*}
	Note that $\| \bs{S}^2_{\bs{\xi}} -\bs{I}_{K_n} \|_{\op} = o_{\PP}(1)$, 
	$\| \bs{S}^2_{\bs{\xi}}   \|_{\op} \le \| \bs{S}^2_{\bs{\xi}} -\bs{I}_{K_n} \|_{\op} + 	\| \bs{I}_{K_n} \|_{\op}
	= 1 + o_{\PP}(1),$
	and 
	$
	\| \bs{Q}^\top \bs{S}^2_{\bs{\xi}} \bs{Q} - \bs{I}_{K_n}
	\|_{\op}
	=  
	\| \bs{Q}^\top (\bs{S}^2_{\bs{\xi}} - \bs{I}_{K_n}) \bs{Q} \|_{\op} 
	\le \| \bs{S}^2_{\bs{\xi}} - \bs{I}_{K_n}  \|_{\op} = o_{\PP}(1). 
	$
	From Lemma \ref{lemma:matrix_inv_converge}, 
	we can then derive that 
	$\| ( \bs{Q}^\top \bs{S}^2_{\bs{\xi}} \bs{Q} )^{-1} - \bs{I}_{K_n}
	\|_{\op} = o_{\PP}(1)$,
	$
	S^2_{r_0Y(1)+r_1Y(0)}/\var\{ r_0Y(1) + r_1 Y(0) \} \cdot R_n^2  - R^2_{\sup} = o_{\PP}(1), 
	$
	\begin{align*}
		\left| \frac{S^2_{r_0Y(1)+r_1Y(0)} }{\var\{ r_0Y(1) + r_1 Y(0) \}} - 1 \right|
		& = \left| \tilde{\bs{a}}^\top \bs{S}^2_{\bs{\xi}} \tilde{\bs{a}} - 1 \right|
		= \left| \tilde{\bs{a}}^\top \left( \bs{S}^2_{\bs{\xi}} - \bs{I}_{K_n} \right) \tilde{\bs{a}} \right|
		\le 
		\| \bs{S}^2_{\bs{\xi}} -\bs{I}_{K_n} \|_{\op} = o_{\PP}(1).
	\end{align*}
	Consequently, 
	\begin{align*}
		R_n^2 -  R^2_{\sup,n}  = 
		\frac{S^2_{r_0Y(1)+r_1Y(0)} }{\var\{ r_0Y(1) + r_1 Y(0) \}} R_n^2  - R^2_{\sup}
		- 
		\left( \frac{S^2_{r_0Y(1)+r_1Y(0)} }{\var\{ r_0Y(1) + r_1 Y(0) \}}  - 1 \right) R_n^2
		= o_{\PP}(1). 
	\end{align*}
	
	From the above, Corollary \ref{cor:reg_cond}(ii) holds. 
\end{proof}

\begin{proof}[\bf Proof of Corollary \ref{cor:reg_cond}(iii)]
	First, from \eqref{eq:denom_inf_cond_equiv} and \eqref{eq:V_tautau_equiv}, we can know that 
	\begin{align*}
		r_0  S^2_{1\setminus \bs{X}} + r_1 S^2_{0\setminus \bs{X}}
		& = n r_1 r_0 \big(  n_1^{-1}  S^2_{1\setminus \bs{X}} + n_0^{-1}  S^2_{0\setminus \bs{X}} \big)
		= 
		n r_1 r_0 \big\{ V_{\tau\tau} (1-R^2_n)+n^{-1} S^2_{\tau \setminus \bs{X}}  \big\}
		\\
		& \ge n r_1 r_0 V_{\tau\tau} (1-R^2_n) 
		= S^2_{r_0Y(1) + r_1 Y(0)} (1-R_n^2). 
	\end{align*}
	From Corollary \ref{cor:reg_cond}(ii) and its proof, and by the conditions in Corollary \ref{cor:reg_cond}(iii), we can know that 
	$
	S^2_{r_0Y(1) + r_1 Y(0)} = \var(r_0Y_i(1) + r_1 Y_i(0)) \cdot (1+o_{\PP}(1))
	$
	and 
	$1-R_n^2 = 1 - R^2_{\sup, n} + o_{\PP}(1) = (1-R^2_{\sup, n}) \cdot (1 + o_{\PP}(1))$. 
	These imply that the quantity on the left hand side of \eqref{eq:infer_cond} satisfies 
	\begin{align*}
		& \quad \ \frac{\max_{z\in \{0,1\}}\max_{1\le i \le n}\{Y_i(z) - \bar{Y}(z)\}^2}{r_0 S^2_{1\setminus \bs{X}} + r_1 S^2_{0\setminus \bs{X}}}
		\cdot 
		\frac{\max\{K_n, 1\}}{r_1r_0}
		\cdot 
		\sqrt{ \frac{\max\{1, \log K_n, - \log p_n\} }{n} } 
		\\
		& = 
		\frac{\max_{z\in \{0,1\}}\max_{1\le i \le n}\{Y_i(z) - \bar{Y}(z)\}^2}{
			\var(r_0Y_i(1) + r_1 Y_i(0)) \cdot (1-R^2_{\sup, n}) 
		}
		\cdot 
		\frac{\max\{K_n, 1\} \cdot \sqrt{\max\{1, \log K_n, - \log p_n\} }}{n^{1/2}} \cdot O_{\PP}(1)\\
		& = 
		\frac{\max_{z\in \{0,1\}}\max_{1\le i \le n}\{Y_i(z) - \bar{Y}(z)\}^2}{
			\var(r_0Y_i(1) + r_1 Y_i(0)) 
		}
		\cdot 
		\frac{\max\{K_n, 1\} \cdot \sqrt{\max\{1, \log K_n, - \log p_n\} }}{n^{1/2}} \cdot O_{\PP}(1), 
	\end{align*}
	where the last equality follows from the condition on $R^2_{\sup, n}$. 
	
	Second, 
	by some algebra, for $1\le i \le n$, 
	\begin{align*}
		|Y_i(z) - \bar{Y}(z)|
		& \le 
		| Y_i(z) - \E(Y(z))| + | \bar{Y}(z) - \E(Y(z)) |
		\le 2 \max_{1\le i \le n} | Y_i(z) - \E(Y(z))| 
	\end{align*}
	and 
	\begin{align*}
		\max_{1\le i \le n} | Y_i(z) - \E(Y(z))|^b
		& 
		\le 
		\sum_{i=1}^n | Y_i(z) - \E(Y(z))|^{b}
		= n \cdot \E \big\{ | Y(z) - \E(Y(z))|^{b}  \big\} \cdot O_{\PP}(1)
		\\
		& = 
		n \cdot \E \left\{ \left| \frac{Y(z) - \E(Y(z))}{\sqrt{\var(Y(z))}} \right|^{b}  \right\} 
		\cdot 
		\{ \var(Y(z)) \}^{b/2}  \cdot O_{\PP}(1)\\
		& = n 	\{ \var(Y(z)) \}^{b/2}  \cdot O_{\PP}(1). 
	\end{align*}
	These imply that 
	$
	\max_{1\le i \le n} | Y_i(z) - \E(Y(z))|^2
	= n^{2/b}  \var(Y(z)) \cdot O_{\PP}(1). 
	$
	Consequently, we can further bound the quantity on the left hand side of \eqref{eq:infer_cond} by 
	\begin{align*}
		& \quad \ \frac{\max_{z\in \{0,1\}}\max_{1\le i \le n}\{Y_i(z) - \bar{Y}(z)\}^2}{r_0 S^2_{1\setminus \bs{X}} + r_1 S^2_{0\setminus \bs{X}}}
		\cdot 
		\frac{\max\{K_n, 1\}}{r_1r_0}
		\cdot 
		\sqrt{ \frac{\max\{1, \log K_n, - \log p_n\} }{n} } 
		\\
		& = 
		\frac{\max_{z\in \{0,1\}}\max_{1\le i \le n}\{Y_i(z) - \bar{Y}(z)\}^2}{
			\var(r_0Y_i(1) + r_1 Y_i(0))  
		}
		\cdot 
		\frac{\max\{K_n, 1\} \cdot \sqrt{\max\{1, \log K_n, - \log p_n\} }}{n^{1/2}} \cdot O_{\PP}(1)
		\\
		& =  
		\frac{n^{2/b} \{\var(Y(1)) + \var(Y(0))\}}{\var(r_0Y_i(1) + r_1 Y_i(0))  }
		\cdot 
		\frac{\max\{K_n, 1\} \cdot \sqrt{\max\{1, \log K_n, - \log p_n\} }}{n^{1/2}} \cdot O_{\PP}(1)
		\\
		& = 
		\frac{\max\{K_n, 1\} \cdot \sqrt{\max\{1, \log K_n, - \log p_n\} }}{n^{1/2-2/b}} \cdot O_{\PP}(1)\\
		& = 
		\frac{\max\{K_n, 1\} \cdot \sqrt{\max\{1, \log K_n\} }}{n^{1/2-2/b}} \cdot O_{\PP}(1)
		+ 
		\frac{\max\{K_n, 1\} \cdot \sqrt{ - \log p_n }}{n^{1/2-2/b}} \cdot O_{\PP}(1). 
	\end{align*}
	
	Third, 
	because $K_n = O(n^c)$ for some $c<1/2-2/b$, we have 
	\begin{align*}
		\frac{\max\{K_n, 1\} \cdot \sqrt{\max\{1, \log K_n\} }}{n^{1/2-2/b}}
		& = 
		\frac{\log n}{n^{1/2-2/b-c}} \cdot O(1) = o(1), 
	\end{align*}
	and 
	\begin{align*}
		\frac{\max\{K_n, 1\} \cdot \sqrt{ - \log p_n }}{n^{1/2-2/b}}
		& = 
		\frac{\sqrt{ - \log p_n }}{n^{1/2-2/b-c}} \cdot O(1)
		= 
		\sqrt{
			\frac{- \log p_n}{n^{1-4/b-2c}}
		}
		= o(1), 
	\end{align*}
	where the last condition holds by the condition on $p_n$. 
	
	From the above, we can know that Corollary \ref{cor:reg_cond} holds. 
\end{proof}

\begin{proof}[\bf Proof of Corollary \ref{cor:iidrate}]
	We choose $p_n \propto n^{-h}$ for some $0<h< (1/2-1/\delta)/3$. 
	Below we verify that Conditions \ref{cond:gamma_n}--\ref{cond:rsup} holds with high probability. 
	
	First, from Proposition \ref{prop:iidrate}, $\gamma_n = o_{\PP}(1)$. 
	Second, from Theorem \ref{thm:berry_esseen_clt} and Proposition \ref{prop:iidrate},  and by the construction of $p_n$, 
	\begin{align*}
		\frac{\Delta_n}{p_n} = 
		\frac{\gamma_n+\gamma_n^{1/3}}{p_n} \cdot O(1)
		= 
		\frac{(\log n)^{(7/12)}}{n^{(1/2-1/\delta)/3}} \cdot n^{h} \cdot o_{\PP}(1)
		= 
		\frac{(\log n)^{(7/12)}}{n^{(1/2-1/\delta)/3- h} } \cdot o_{\PP}(1) = o_{\PP}(1). 
	\end{align*}
	Third, 
	$K_n/\log (p_n^{-1}) = K_n/ \log(n^h) \cdot O(1) = h^{-1} K_n/\log(n) = o(1) $. 
	Fourth, from Corollary \ref{cor:reg_cond}(ii), 
	$R_n^2 = R^2_{\sup, n} + o_{\PP}(1) \le 1-c+o_{\PP}(1)$. 
	
	Therefore, by the property of convergence in probability \citep[e.g.,][Theorem 2.3.2]{D19}, 
	Corollary \ref{cor:iidrate} follows from 
	Theorem \ref{thm:rem_gaussian}. 
\end{proof}

\begin{proof}[\bf Comments on the equivalent form of $\gamma_n$ and its bounds in Section \ref{sec:implication}]
	We first prove the equivalent form of $\gamma_n$. 
	Because  $(e_i, \bs{X}_i)$ is a non-singular linear transformation of $\bs{u}_i$, 
	and the finite population covariance between $e_i$ and $\bs{X}_i$ is zero, 
	we can equivalently write $(\bs{u}_i-\bar{\bs{u}})^\top \bs{S}^{-2}_{\bs{u}} (\bs{u}_i-\bar{\bs{u}})$ as 
	\begin{align*}
		(\bs{u}_i-\bar{\bs{u}})^\top \bs{S}^{-2}_{\bs{u}} (\bs{u}_i-\bar{\bs{u}}) 
		& =  
		(e_i, (\bs{X}_i - \bar{\bs{X}})^\top) 
		\begin{pmatrix}
			S_e^{-2}  & \bs{0} \\
			\bs{0} &  \bs{S}^{-2}_{\bs{X}}
		\end{pmatrix}
		\begin{pmatrix}
			e_i\\
			\bs{X}_i - \bar{\bs{X}}
		\end{pmatrix}
		\\
		& = e_i^2 + (\bs{X}_i - \bar{\bs{X}})^\top \bs{S}^{-2}_{\bs{X}} (\bs{X}_i - \bar{\bs{X}})
		= 
		e_i^2 +(n-1) H_{ii}, 
	\end{align*}
	where the second last equality holds since the finite population variance of $e_i$,  $S^2_{e}$ equals $1$, 
	and the last equality follows from the definition of $H_{ii}$'s. 
	This then implies that 
	\begin{align*}
		\gamma_n = 
		\frac{(K_n+1)^{1/4}}{\sqrt{n r_1r_0}} \frac{1}{n} \sum_{i=1}^n \left\| \left(\bs{S}_{\bs{u}}^2\right)^{-1/2} (\bs{u}_i-\bar{\bs{u}}) \right\|_2^3
		= 
		\frac{(K_n+1)^{1/4}}{\sqrt{n r_1r_0}} \frac{1}{n} \sum_{i=1}^n 
		\left( e_i^2 +(n-1) H_{ii} \right)^{3/2}. 
	\end{align*}
	
	We then prove the bounds for $\gamma_n$. 
	By H\"{o}lder's inequality, 
	\begin{align*}
		\left( e_i^2 +(n-1) H_{ii} \right)^{3/2} 
		& = 
		2^{3/2} \left( \frac{e_i^2 +(n-1) H_{ii}}{2} \right)^{3/2} 
		\le 
		2^{3/2} 
		\frac{|e_i|^3 + (n-1)^{3/2} H_{ii}^{3/2}}{2}
		\\
		& \le 
		\sqrt{2} (|e_i|^3 + n^{3/2} H_{ii}^{3/2}). 
	\end{align*}
	This immediately implies that $\gamma_n \le \sqrt{2} \tilde{\gamma_n}.$ 
	Note that 
	\begin{align*}
		2 \left( e_i^2 +(n-1) H_{ii} \right)^{3/2} 
		& \ge 
		(e_i^2)^{3/2} + \{(n-1) H_{ii} \}^{3/2}
		\ge |e_i|^3 + \frac{n^{3/2}}{2^{3/2}} H_{ii}^{3/2} 
		\\
		& 
		\ge \frac{1}{2^{3/2}} \left( |e_i|^3 + n^{3/2} H_{ii}^{3/2}  \right). 
	\end{align*}
	This immediately implies that $\gamma_n \ge  \tilde{\gamma}_n/2^{5/2} = \tilde{\gamma}_n/(4\sqrt{2})$. 
\end{proof}

\begin{proof}[\bf Comment on $\sum_{i=1}^n H_{ii}^{3/2}$] 
	Because $H_{ii}$'s are the diagonal elements of a projection matrix of rank $K_n$, 
	we have $\sum_{i=1}^n H_{ii} = K_n$. 
	By H\"{o}lder's inequality, we then have 
	\begin{align*}
		\sum_{i=1}^n H_{ii}^{3/2} 
		& = 
		n \cdot \frac{1}{n} \sum_{i=1}^n H_{ii}^{3/2} 
		\ge n \cdot \left( \frac{1}{n} \sum_{i=1}^n H_{ii} \right)^{3/2}
		= \frac{K_n^{3/2}}{\sqrt{n}}. 
	\end{align*}
\end{proof}

\begin{proof}[\bf Comments on the first two moments of $\hat{\tau}-\tau$ under any design]
	From \eqref{eq:did_srs_u} and by definition, 
	\begin{align*}
		\hat{\tau} - \tau & = \frac{n}{n_1n_0} \sum_{i=1}^n Z_i \wy_i  - \frac{n}{n_0} \bar{Y}(0) - \bar{Y}(1) + \bar{Y}(0)
		= 
		\frac{n}{n_1n_0} \sum_{i=1}^n Z_i \wy_i  - \frac{n}{n_0} 
		\{ r_1 \bar{Y}(0) + r_0\bar{Y}(1) \}
		\\
		& 
		=
		\frac{n}{n_1n_0} \sum_{i=1}^n Z_i \wy_i  - \frac{n}{n_0} \bar{\wy}
		= 
		\frac{n}{n_1n_0} \sum_{i=1}^n Z_i \wy_i  - 	\frac{n}{n_1n_0} \sum_{i=1}^n Z_i \bar{\wy}
		= \frac{n}{n_1n_0} \sum_{i=1}^n Z_i (\wy_i  - \bar{\wy})
		\\
		& 
		= \frac{n}{n_1 n_0} \sum_{i=1}^n \left( Z_i - \frac{n_1}{n} \right) \left( \wy_i - \bar{\wy} \right)
		= 
		\frac{n}{n_1 n_0} \left(\bs{Z} - r_1 \bs{1}_n \right)^\top \tilde{\bs{\wy}}. 
	\end{align*}
	Consequently, 
	\begin{align*}
		\E_{\mathcal{D}}(\hat{\tau} - \tau ) & = 
		\frac{n}{n_1 n_0} \left(\E \bs{Z} - r_1 \bs{1}_n \right)^\top \tilde{\bs{\wy}}
		=
		\frac{1}{nr_1r_0} \left( \bs{\pi} - r_1 \bs{1}_n \right)^\top \tilde{\bs{\wy}}, 
	\end{align*}
	and 
	\begin{align}\label{eq:Z_second_moment}
		\E_{\mathcal{D}}\{(\hat{\tau} - \tau )^2\}
		& = 
		\left( \frac{n}{n_1 n_0} \right)^2 
		\tilde{\bs{\wy}}^\top 
		\E\left\{
		\left(\bs{Z} - r_1 \bs{1}_n \right)
		\left(\bs{Z} - r_1 \bs{1}_n \right)^\top \right\}
		\tilde{\bs{\wy}}
		\nonumber
		\\
		& = 
		\frac{1}{(nr_1r_0)^2}
		\tilde{\bs{\wy}}^\top 
		\left\{
		\var_{\mathcal{D}}(\bs{Z}) + (\E\bs{Z} -r_1 \bs{1}_n )(\E\bs{Z} -r_1 \bs{1}_n )^\top
		\right\}
		\tilde{\bs{\wy}}
		\nonumber
		\\
		& = 
		\frac{1}{(nr_1r_0)^2}
		\tilde{\bs{\wy}}^\top 
		\left\{
		\bs{\Omega} + (\bs{\pi} -r_1 \bs{1}_n )(\bs{\pi} -r_1 \bs{1}_n )^\top
		\right\}
		\tilde{\bs{\wy}}. 
	\end{align}
	Therefore, \eqref{eq:bias_mse} holds. 
\end{proof}

\begin{proof}[\bf Proof of Proposition \ref{prop:bias_mse}]
	From \eqref{eq:V_tautau_equiv}, we can know that 
	\begin{align*}
		V_{\tau\tau} = \frac{1}{nr_1r_0}  S^2_{r_0Y(1)+r_1Y(0)} 
		= 
		\frac{1}{nr_1r_0} \frac{1}{n-1} \tilde{\bs{\wy}}^\top \tilde{\bs{\wy}}. 
	\end{align*}
	This implies that 
	\begin{align*}
		V_{\tau\tau}^{-1/2} \E_{\mathcal{D}}(\hat{\tau} - \tau )
		=
		\sqrt{\frac{n-1}{nr_1r_0}} \cdot 
		\frac{ \left( \bs{\pi} - r_1 \bs{1}_n \right)^\top \tilde{\bs{\wy}}}{\|\tilde{\bs{\wy}}\|_2}, 
	\end{align*}
	and 
	\begin{align*}
		V_{\tau\tau}^{-1/2} \sqrt{\E_{\mathcal{D}}\{(\hat{\tau} - \tau )^2\}}
		& = 
		\sqrt{\frac{n-1}{nr_1r_0}} 
		\sqrt{
			\frac{\tilde{\bs{\wy}}^\top 
				\left\{
				\bs{\Omega} + (\bs{\pi} -r_1 \bs{1}_n )(\bs{\pi} -r_1 \bs{1}_n )^\top
				\right\}
				\tilde{\bs{\wy}}}{\tilde{\bs{\wy}}^\top \tilde{\bs{\wy}}}
		}. 
	\end{align*}
		By some matrix properties, we can know that %
	\begin{align*}
		\max_{\tilde{\bs{\wy}}\ne \bs{0}} V_{\tau\tau}^{-1/2} \E_{\mathcal{D}}(\hat{\tau} - \tau )
		=
		\sqrt{\frac{n-1}{nr_1r_0}} \cdot
		\| \bs{\pi} - r_1 \bs{1}_n \|_2 \ge 0, 
	\end{align*}
	and 
	\begin{align*}
		\max_{\tilde{\bs{\wy}}\ne \bs{0}} V_{\tau\tau}^{-1/2} \sqrt{\E_{\mathcal{D}}\{(\hat{\tau} - \tau )^2\}}
		& = 
		\sqrt{\frac{n-1}{nr_1r_0}} \cdot
		\lambda_{\max}^{1/2} \left( \bs{\Omega} + (\bs{\pi} -r_1 \bs{1}_n )(\bs{\pi} -r_1 \bs{1}_n )^\top \right). 
	\end{align*}

	Below we prove the inequality on the right hand side of \eqref{eq:max_mse}. 
	Let $\bs{\Psi} = \bs{\Omega} + (\bs{\pi} -r_1 \bs{1}_n )(\bs{\pi} -r_1 \bs{1}_n )^\top$
	From \eqref{eq:Z_second_moment}, 
	$\bs{\Psi} = \E\{
		(\bs{Z} - r_1 \bs{1}_n )
		(\bs{Z} - r_1 \bs{1}_n )^\top\}$. 
	Because 
	$\bs{1}_n^\top (\bs{Z} - r_1 \bs{1}_n ) = \sum_{i=1}^n Z_i - n_1 = 0$, 
	we must have $\bs{1}_n^\top \bs{\Psi} \bs{1}_n = 0$. 
	This implies that $\bs{\Psi}$ has at most $n-1$ positive eigenvalues. 
	Consequently, 
	\begin{align*}
	    & \quad \ (n-1) \lambda_{\max} (\bs{\Psi}) 
	    \\
	    & \ge \tr(\bs{\Psi}) 
	    = \E \left[ \tr \left( (\bs{Z} - r_1 \bs{1}_n )
		(\bs{Z} - r_1 \bs{1}_n )^\top\right) \right] 
		= 
		 \E \left[ \tr \left( (\bs{Z} - r_1 \bs{1}_n )^\top
		(\bs{Z} - r_1 \bs{1}_n )\right) \right] 
		\\
		& = \E \left\{ \tr \left( 
		\bs{Z}^\top \bs{Z} - 2 r_1 \bs{1}_n^\top\bs{Z} + r_1^2 \bs{1}_n^\top \bs{1}_n
		\right) \right\} 
		= \E (n_1 - 2 r_1 n_1+ r_1^2 n ) = n r_1 r_0, 
	\end{align*}
	i.e., $\lambda_{\max} (\bs{\Psi})  \ge n r_1 r_0/(n-1)$. 
	This immediately implies the inequality on the right hand side of \eqref{eq:max_mse}. 
	
	From the above, Proposition \ref{prop:bias_mse} holds. %
\end{proof}

\section{Asymptotic analysis of regression adjustment under rerandomization}\label{sec:proof_regrem}

To prove Theorem~\ref{thm:regrem}, we need the following four lemmas. 

\begin{lemma}\label{lem:regrem}
	Under ReM and 
    Condition \ref{cond:regrem_Delta}, 
    as $n \to \infty$,
	\begin{align*}%
		\sup_{c\in \mathbb{R}} & \bigg| \PP \big\{
		V_{\tau\tau}^{-1/2} (1-\rho^2_{n})^{-1/2}
		\{ \hat{\tau}(\tilde{\bs{\beta}}_1, \tilde{\bs{\beta}}_0) - \tau\} \le c \mid M \le a_n \big\} \\
		&  - \PP\left( \sqrt{1-R^2_n(\tilde{\bs{\beta}}_1, \tilde{\bs{\beta}}_0)}\ \varepsilon_0  + \sqrt{R^2_n(\tilde{\bs{\beta}}_1, \tilde{\bs{\beta}}_0)} \ L_{K_n, a_n} \le c \right)
		\bigg| 
		\rightarrow 0.
	\end{align*}
\end{lemma}

\begin{proof}[Proof of Lemma \ref{lem:regrem}]
Define $V_{\tau\tau}(\tilde{\bs{\beta}}_1, \tilde{\bs{\beta}}_0)$ analogously as $V_{\tau\tau}$ in \eqref{eq:V}, but using the adjusted potential outcomes with adjustment coefficients $\tilde{\bs{\beta}}_1$ and $\tilde{\bs{\beta}}_0$. 
From \citet[][Proof of Theorem 5]{LD20}, $V_{\tau\tau}(\tilde{\bs{\beta}}_1, \tilde{\bs{\beta}}_0) = V_{\tau\tau}(1-\rho_n^2)$. 
Lemma \ref{lem:regrem} then follows immediately from Theorem~\ref{thm:dim_rem}. 
\end{proof}

\begin{lemma}\label{lemma:deltaw}
	Consider the same setting as in Lemma \ref{lemma:s_uw_vector} and any event $\bs{Z} \in \mathcal{E} \subset \{0,1\}^N$ with positive probability $p \equiv \PP(\bs{Z} \in \mathcal{E})$. 
	Then 
	for any $t \ge 3 \cdot 71^2/70^2$, 
	\begin{align*}
		\PP\left( \sum_{k=1}^K \Delta_{w_k}^2 > t \frac{\max\{1, \log K, - \log p\}}{N f^2} \sum_{k=1}^K \sigma_{w_k}^2 \mid \bs{Z} \in 
		\mathcal{E} \right)
		\le 
		2 
		\exp\left(
		- \frac{1}{3} \frac{70^2}{71^2} t
		\right).
	\end{align*}
\end{lemma}

\begin{proof}[Proof of Lemma \ref{lemma:deltaw}]
    If $\sum_{k=1}^K \sigma_{w_k}^2 = 0$, then $\sum_{k=1}^K \Delta_{w_k}^2$ is constant zero, and Lemma \ref{lemma:deltaw} holds obviously. 
    Below we consider only the case in which $\sum_{k=1}^K \sigma_{w_k}^2 > 0$. 
    From Lemma \ref{lemma:s_uw_vector}, 
    for any $t>0$, 
    \begin{align*}
        & \quad \ \PP\left( \sum_{k=1}^K \Delta_{w_k}^2 > t \frac{\max\{1, \log K, - \log p\}}{N f^2} \sum_{k=1}^K \sigma_{w_k}^2 \mid \bs{Z} \in 
		\mathcal{E} \right)
		\\
		& 
		\le 
		2 \frac{K}{p} \exp\left( - \frac{70^2}{71^2} t \max\{1, \log K, - \log p\} \right)\\
		& = 
		 2 \exp\left(- \frac{70^2}{71^2} t \max\{1, \log K, - \log p\} + \log K - \log p\right). 
    \end{align*}
    From \eqref{eq:tlowerbnd}, when $t \ge 3 \cdot 71^2 / 70^2$, we then have 
    \begin{align*}
        \PP\left( \sum_{k=1}^K \Delta_{w_k}^2 > t \frac{\max\{1, \log K, - \log p\}}{N f^2} \sum_{k=1}^K \sigma_{w_k}^2 \mid \bs{Z} \in 
		\mathcal{E} \right)
		\le 
		2 \exp\left(- \frac{1}{3} \frac{70^2}{71^2} t \right).
    \end{align*}
    Therefore, Lemma \ref{lemma:deltaw} holds. 
\end{proof}

\begin{lemma}\label{lemma:beta_tauw}
	Under ReM with actual acceptance probability $\tilde{p}_n = \PP(M\le a_n)$, 
	if $\min\{n_1,n_0\}\ge 2$ when $n$ is sufficiently large, and 
	$\max\{1, \log J_n, -\log \tilde{p}_n \} = O(n r_1^2 r_0^2)$, 
	then
	\begin{align*}
	& \quad \ \big\{ r_0 (\hat{\bs{\beta}}_1 - \tilde{\bs{\beta}}_1) + r_1 (\hat{\bs{\beta}}_0 - \tilde{\bs{\beta}}_0)\big\}^\top \hat{\bs{\tau}}_{\bs{W}} \\
	& = 
    O_{\PP}\left( \max_{z \in \{0,1\}} \max_{1 \leq i \leq n} |Y_i(z) - \bar{Y}(z)| \cdot J_n \frac{\max\{1, \log J_n, -\log \tilde{p}_n\}}{nr_1^2 r_0^2} \right). 
	\end{align*}
\end{lemma}

\begin{proof}[Proof of Lemma \ref{lemma:beta_tauw}]
    First, we bound the Euclidean norms of $r_0 (\hat{\bs{\beta}}_1 - \tilde{\bs{\beta}}_1)^\top \hat{\bs{\tau}}_{\bs{W}}$ and $r_1 (\hat{\bs{\beta}}_0 - \tilde{\bs{\beta}}_0)^\top \hat{\bs{\tau}}_{\bs{W}}$. 
    By definition, 
    \begin{align*}
        \hat{\bs{\tau}}_{\bs{W}}
        & =
        \bar{\bs{W}}_1 - \bar{\bs{W}}_0 
        = \frac{n}{n_1 n_0} \sum_{i=1}^{n} Z_i (\bs{W}_i - \bar{\bs{W}}). 
    \end{align*}
    We then have  
    \begin{align*}
        \left\| r_0 (\hat{\bs{\beta}}_1 - \tilde{\bs{\beta}}_1)^\top \hat{\bs{\tau}}_{\bs{W}} \right\|_2
        & = 
        \left\| (\bs{s}_{1, \bs{W}} - \bs{S}_{1, \bs{W}})^\top (\bs{S}_{\bs{W}}^2)^{-1} \frac{1}{n_1} \sum_{i=1}^{n} Z_i (\bs{W}_i - \bar{\bs{W}}) \right\|_2
        \\
        & = 
        \left\|  (\bs{s}_{1, \bs{w}} - \bs{S}_{1, \bs{w}})^\top ( \bar{\bs{w}}_1 - \bar{\bs{w}} ) \right\|_2
        \le 
        \| \bs{s}_{1, \bs{w}} - \bs{S}_{1, \bs{w}} \|_2 \| \bar{\bs{w}}_1 - \bar{\bs{w}} \|_2, 
    \end{align*}
    and 
    \begin{align*}
        \left\| r_1 (\hat{\bs{\beta}}_0 - \tilde{\bs{\beta}}_0)^\top \hat{\bs{\tau}}_{\bs{W}} \right\|_2
        & = 
        \left\| (\bs{s}_{0, \bs{W}} - \bs{S}_{0, \bs{W}})^\top (\bs{S}_{\bs{W}}^2)^{-1} \frac{1}{n_0} \sum_{i=1}^{n} Z_i (\bs{W}_i - \bar{\bs{W}}) \right\|_2
        \\
        & = 
        \left\| (\bs{s}_{0, \bs{W}} - \bs{S}_{0, \bs{W}})^\top (\bs{S}_{\bs{W}}^2)^{-1} \frac{1}{n_0} \sum_{i=1}^{n} (1-Z_i) (\bs{W}_i - \bar{\bs{W}}) \right\|_2
        \\
        & = 
        \left\|  (\bs{s}_{0, \bs{w}} - \bs{S}_{0, \bs{w}})^\top ( \bar{\bs{w}}_0 - \bar{\bs{w}} ) \right\|_2
        \le 
        \| \bs{s}_{0, \bs{w}} - \bs{S}_{0, \bs{w}} \|_2 \| \bar{\bs{w}}_0 - \bar{\bs{w}} \|_2, 
    \end{align*}
    where $\bs{w}_i = (\bs{S}_{\bs{W}}^2)^{-1/2} (\bs{W}_i - \bar{\bs{W}})$ denotes the standardized covariate vector for $1\le i \le n$. 
    
    Second, from Lemma \ref{lemma:s_uw_re} and by the same logic as the proof of 
    Lemma \ref{lemma:V_R2_hat_bound_simp} (in particular, \eqref{eq:xizw}), we can know that, for $z=0,1$, 
    \begin{align*}
        \| \bs{s}_{z, \bs{w}} - \bs{S}_{z, \bs{w}} \|_2^2 
        & = 
        \max_{z \in \{0,1\}} \max_{1 \leq i \leq n} \{Y_i(z) - \bar{Y}(z)\}^2 \cdot J_n \frac{\max\{1, \log J_n, -\log \tilde{p}_n\}}{nr_z^2} \cdot O_{\PP}(1).
    \end{align*}
    
    Third, from Lemma \ref{lemma:deltaw}, for $z=0,1$, 
    \begin{align*}
        \| \bar{\bs{w}}_z - \bar{\bs{w}} \|_2^2 
        & = 
        J_n \frac{\max\{1, \log J_n, - \log \tilde{p}_n\}}{n r_z^2} \cdot O_{\PP}(1). 
    \end{align*}

    From the above, for $z=0,1$, 
    \begin{align*}
        \left\| r_{1-z} (\hat{\bs{\beta}}_z - \tilde{\bs{\beta}}_z)^\top \hat{\bs{\tau}}_{\bs{W}} \right\|_2 
        & \le \| \bs{s}_{z, \bs{w}} - \bs{S}_{z, \bs{w}} \|_2 \| \bar{\bs{w}}_z - \bar{\bs{w}} \|_2
        \\
        & = \max_{z \in \{0,1\}} \max_{1 \leq i \leq n} |Y_i(z) - \bar{Y}(z)| \cdot J_n \frac{\max\{1, \log J_n, -\log \tilde{p}_n\}}{nr_z^2} \cdot O_{\PP}(1)\\
        & = \max_{z \in \{0,1\}} \max_{1 \leq i \leq n} |Y_i(z) - \bar{Y}(z)| \cdot J_n \frac{\max\{1, \log J_n, -\log \tilde{p}_n\}}{nr_1^2 r_0^2} \cdot O_{\PP}(1). 
    \end{align*}
    Therefore, Lemma \ref{lemma:beta_tauw} holds. 
\end{proof}

\begin{lemma}\label{lemma:beta_tauw_cond}
	Under ReM with actual acceptance probability $\tilde{p}_n = \PP(M\le a_n)$, 
	\begin{itemize}
	    \item[(i)] if Condition \ref{cond:regrem_Delta} holds, then $\max\{1, -\log \tilde{p}_n\} = O(\max\{1, -\log p_n\})$, 
		recalling that $p_n = \PP(\chi^2_{K_n} \le a_n)$
		is the approximate acceptance probability; 
		\item[(ii)] 
		if Conditions \ref{cond:regrem_Delta} and \ref{cond:regrem} hold, 
		then,  $\max\{1, \log K_n, - \log \tilde{p}_n \} = o(nr_1^2 r_0^2)$. 
	\end{itemize}
\end{lemma}

\begin{proof}[Proof of Lemma \ref{lemma:beta_tauw_cond}]
    In Lemma \ref{lemma:beta_tauw_cond}, (i) follows by the same logic as Lemma \ref{lemma:cond_infer}, and below we focus only on the proof of (ii). 
    From the proof of Lemma \ref{lemma:cond_infer}, 
    $2\max_{z \in \{0,1\}} \max_{1 \leq i \leq n} |Y_i(z) - \bar{Y}(z)|^2 \ge r_0S_1^2 + r_1S_0^2 \ge n r_1r_0 V_{\tau\tau}$. 
    Consequently, from Condition \ref{cond:regrem}, 
    \begin{align*}
        o(1) 
        & = 
        \frac{\max_{z \in \{0,1\}} \max_{1 \leq i \leq n} |Y_i(z) - \bar{Y}(z)|}{\sqrt{V_{\tau\tau} (1-\rho^2_n) \{ 1-R^2_n(\tilde{\bs{\beta}}_1, \tilde{\bs{\beta}}_0)\}}}  \cdot J_n \cdot \frac{\max\{1, \log J_n, - \log p_n\}}{n r_1^2 r_0^2} 
        \\
        & \ge \frac{2^{-1/2} \sqrt{nr_1r_0 V_{\tau\tau}}}{\sqrt{V_{\tau\tau}}} 
        \cdot J_n \cdot \frac{\max\{1, \log J_n, - \log p_n\}}{n r_1^2 r_0^2} \\
        & \ge 2^{-1/2} \cdot \sqrt{n r_1 r_0} \cdot \frac{\max\{1, \log J_n, - \log p_n\}}{n r_1^2 r_0^2}
        \ge  \frac{2^{-1/2} }{\sqrt{n r_1 r_0}}.
    \end{align*}
    This implies that $1 = o(\sqrt{n r_1 r_0})$, and thus  $\max\{1, \log K_n, - \log \tilde{p}_n \} = o(nr_1^2 r_0^2)$.
    From the above, Lemma \ref{lemma:beta_tauw_cond} holds. 
\end{proof}

\begin{proof}[\bf Proof of Theorem \ref{thm:regrem}(i)]
Below we prove the first part of Theorem \ref{thm:regrem}. 
Define 
\begin{align*}
    \tilde{\psi}_n & = V_{\tau\tau}^{-1/2} (1-\rho^2_{n})^{-1/2}
	\{ \hat{\tau}(\tilde{\bs{\beta}}_1, \tilde{\bs{\beta}}_0) - \tau\}, 
	\qquad 
	\hat{\psi}_n = V_{\tau\tau}^{-1/2} (1-\rho^2_{n})^{-1/2}
	\{ \hat{\tau}(\hat{\bs{\beta}}_1, \hat{\bs{\beta}}_0) - \tau\}, \\
	\psi_n & = \sqrt{1-\tilde{R}^2_n}\ \varepsilon_0  + \sqrt{\tilde{R}^2_n} \ L_{K_n, a_n}. 
\end{align*}
Note that Conditions \ref{cond:regrem_Delta} and \ref{cond:regrem} hold. 
From Lemma \ref{lem:regrem}, $n\rightarrow\infty$, 
$$\sup_{c\in \mathbb{R}}
| \PP(\tilde{\psi}_n \le c \mid M\le a_n) - \PP( \psi_n \le c) | \rightarrow 0. $$ 
From Lemmas \ref{lemma:beta_tauw} and \ref{lemma:beta_tauw_cond}, under ReM, 
\begin{align*}
    \tilde{\psi}_n - \hat{\psi}_n
    & = 
    \frac{\hat{\tau}(\tilde{\bs{\beta}}_1, \tilde{\bs{\beta}}_0)
    - 
    \hat{\tau}(\hat{\bs{\beta}}_1, \hat{\bs{\beta}}_0)}{V_{\tau\tau}^{1/2} (1-\rho^2_{n})^{1/2}}
    = 
    \frac{
    \big\{ r_0 (\hat{\bs{\beta}}_1 - \tilde{\bs{\beta}}_1) + r_1 (\hat{\bs{\beta}}_0 - \tilde{\bs{\beta}}_0)\big\}^\top \hat{\bs{\tau}}_{\bs{W}}
    }{
    V_{\tau\tau}^{1/2} (1-\rho^2_{n})^{1/2}
    }
    \\
    & = \sqrt{1-\tilde{R}^2_n} \cdot o_{\PP}(1). 
\end{align*}

For any $\eta>0$, define $\delta_n = \sqrt{1-\tilde{R}^2_n} \cdot  \eta$. 
From Lemma \ref{lemma:bound_inf_norm_two_rv}, 
\begin{align*}
	& \quad \ \sup_{c\in \mathbb{R}}
	\big| \PP(\hat{\psi}_n \le c \mid M\le a_n) - \PP( \tilde{\psi}_n \le c \mid M\le a_n) \big|
	\\
	& \le 
	\PP(|\hat{\psi}_n - \tilde{\psi}_n| > \delta_n \mid M\le a_n) 
	+ 
	\sup_{b\in \mathbb{R}}\PP(b < \tilde{\psi}_n \le b+\delta_n \mid M\le a_n)
	\\
	& \le 
	\PP(|\hat{\psi}_n - \tilde{\psi}_n| > \delta_n \mid M\le a_n) 
	+ 
	\sup_{b\in \mathbb{R}}\PP(b < \psi_n \le b+\delta_n)
	\\
	& \quad \ + 2
	\sup_{c\in \mathbb{R}}
	\big| \PP(\tilde{\psi}_n \le c \mid M\le a_n) - \PP( \psi_n \le c) \big|
\end{align*}
Letting $n\rightarrow \infty$, 
from the discussion before, we have 
\begin{align*}
    \limsup_{n\rightarrow\infty}\sup_{c\in \mathbb{R}}
	\big| \PP(\hat{\psi}_n \le c \mid M\le a_n) - \PP( \tilde{\psi}_n \le c \mid M\le a_n) \big|
	\le 
	\limsup_{n\rightarrow \infty}
	\sup_{b\in \mathbb{R}}\PP(b < \psi_n \le b+\delta_n). 
\end{align*}
By definition, for any $b\in \mathbb{R}$, 
we have, with $b' = b/\sqrt{1-\tilde{R}^2_n}$, 
\begin{align*}
    & \quad \ \PP(b < \psi_n \le b+\delta_n)
    \\
    & = 
    \PP\left( b' <  \varepsilon_0  + \sqrt{ \frac{\tilde{R}^2_n}{1 - \tilde{R}^2_n}} \ L_{K_n, a_n} \le b'+ \eta \right)
    \\
    & =
    \E \left\{
    \PP\left( b' <  \varepsilon_0  + \sqrt{ \frac{\tilde{R}^2_n}{1 - \tilde{R}^2_n}} \ L_{K_n, a_n} \le b'+ \eta 
    \mid L_{K,a}
    \right)
    \right\}\\
    & \le 
    \E \left( 
    \eta/\sqrt{2\pi} 
    \right)
    = \eta/\sqrt{2\pi}, 
\end{align*}
where the last inequality holds because the density of $\varepsilon_0$ is bounded by $1/\sqrt{2\pi}$. 
This then implies that 
\begin{align*}
    \limsup_{n\rightarrow\infty}\sup_{c\in \mathbb{R}}
	\big| \PP(\hat{\psi}_n \le c \mid M\le a_n) - \PP( \tilde{\psi}_n \le c \mid M\le a_n) \big|
	\le 
	\eta/\sqrt{2\pi}.
\end{align*}
Because the above inequality holds for any $\eta>0$, we must have, as $n\rightarrow\infty$, 
\begin{align*}
    \lim_{n\rightarrow\infty}\sup_{c\in \mathbb{R}}
	\big| \PP(\hat{\psi}_n \le c \mid M\le a_n) - \PP( \tilde{\psi}_n \le c \mid M\le a_n) \big| = 0. 
\end{align*}

From the discussion before, 
as $n\rightarrow \infty$, 
\begin{align*}
    & \quad \ \sup_{c\in \mathbb{R}}
    | \PP(\hat{\psi}_n \le c \mid M\le a_n) - \PP( \psi_n \le c) |
    \\
    & \le \sup_{c\in \mathbb{R}}
    | \PP(\tilde{\psi}_n \le c \mid M\le a_n) - \PP( \psi_n \le c) |
    + 
    \sup_{c\in \mathbb{R}}
	\big| \PP(\hat{\psi}_n \le c \mid M\le a_n) - \PP( \tilde{\psi}_n \le c \mid M\le a_n) \big|
\end{align*}
Therefore, the first part of Theorem \ref{thm:regrem} holds. 
\end{proof}

\begin{proof}[\bf Proof of Theorem \ref{thm:regrem}(ii)]
Because Condition \ref{cond:k_np_n} holds and $\limsup_{n \to \infty} \tilde{R}_n^2 < 1$, 
by the same logic as the proof of Theorem \ref{thm:rem_gaussian}, we can know that, as $n\rightarrow \infty$, 
\begin{align*}
    \sup_{c\in \mathbb{R}} \left| 
        \PP\left\{ \sqrt{1-\tilde{R}^2_n}\ \varepsilon_0  + \sqrt{\tilde{R}^2_n} \ L_{K_n, a_n} \le c \right\}
		- 
		\PP\left\{ \sqrt{1-\tilde{R}^2_n}\ \varepsilon_0  \le c \right\}
		\right| 
		\rightarrow 0. 
\end{align*}
From the first part of Theorem \ref{thm:regrem}, we can immediately derive the second part of the theorem. 
\end{proof}

\section{Connection with optimal designs}\label{sec:opt_comp}

\subsection{Optimal design via minimizing Mahalanobis distance}\label{sec:opt_min_M}
In this section, we show that, under certain model assumptions, 
the optimal design that tries to minimize the mean squared error (MSE) of the difference-in-means estimator will seek the assignment minimizing the Mahalanobis distance for covariate imbalance between the two treatment groups. 
For more detailed discussion of optimally balanced designs, we refer the readers to \citet{K16} and \citet{K18}.

Suppose that the potential outcomes satisfy the following model:
\begin{align}\label{eq:linear_model}
	Y_i(z) = \alpha_z + \bs{\beta}_z^\top \bs{X}_i + e_i(z), 
	\quad (z=0,1; i=1,2,\ldots, n),
\end{align}
where $(e_i(1), e_i(0))$'s are mutually independent across all units, and $e_i(z)$'s have mean zero and the same variance $\sigma^2_z$ across all units for $z=0,1$. 
Throughout the discussion in this section, 
the covariates $\bs{X}_1, \cdots, \bs{X}_n$  are fixed constants or equivalently being conditioned on. 
Under model \eqref{eq:linear_model}, 
the expected treatment effect for each unit $i$ is then 
\begin{align*}
	\tau_i^\star = \E\{ \tau_i \} = \E\{Y_i(1) - Y_i(0)\} = \alpha_1 - \alpha_0 + (\bs{\beta}_1 - \bs{\beta}_0)^\top \bs{X}_i, 
	\quad (i=1,2,\ldots, n)
\end{align*}
and its average over all units is 
\begin{align*}
	\tau^\star = \E\{\tau\} = \frac{1}{n} \sum_{i=1}^n \tau_i^\star = \alpha_1 - \alpha_0 + (\bs{\beta}_1 - \bs{\beta}_0)^\top \bar{\bs{X}}. 
\end{align*}

We are interested in estimating the average treatment effect $\tau^\star$ using the difference-in-means estimator $\hat{\tau}$. 
Moreover, we will write the difference-in-means estimator as $\hat{\tau}(\bs{Z})$, to emphasize its dependence on the treatment assignment. 
For any fixed treatment assignment $\bs{z}$, 
the MSE of the corresponding difference-in-means estimator under model \eqref{eq:linear_model} has the following decomposition: 
\begin{align}\label{eq:mse_model}
	\E [ \{ \hat{\tau}(\bs{z}) - \tau^\star \}^2 ] 
	& = 
	[\E\{ \hat{\tau}(\bs{z}) - \tau^\star \}]^2 + 
	\var \{ \hat{\tau}(\bs{z}) - \tau^\star \}
	\\
	\nonumber
	& 
	= 
	\{ \tilde{\bs{\beta}}^\top \hat{\bs{\tau}}_{\bs{X}}(\bs{z}) \}^2
	+ 
	\frac{\sigma^2_1}{n_1} + 
	\frac{\sigma^2_0}{n_0}, 
\end{align} 
where $\tilde{\bs{\beta}} = r_0 \bs{\beta}_1 + r_1 \bs{\beta}_0$ and $\hat{\bs{\tau}}_{\bs{X}}(\bs{z})$ is the difference-in-means of covariates under the treatment assignment $\bs{z}$. 
From \eqref{eq:mse_model}, 
the optimal assignment minimizing the MSE is equivalently the one minimizing the squared bias of $\hat{\tau}$. 
Since $\tilde{\bs{\beta}}$ is unknown, 
similar to 
\citet[][Section 2.3.3]{K18}, 
we consider the worst-case squared bias after some standardization. 
Specifically, 
let $\mu_i \equiv r_0 \E \{Y_i(1)\} + r_1 \E \{ Y_i(0)\}$ be a certain weighted average of expected potential outcomes  
for each unit $i$. 
The finite population variance of $\mu_i$'s across all units 
can be equivalently written as 
$
S^2_{\mu} = \tilde{\bs{\beta}}^\top \bs{S}^2_{\bs{X}}\tilde{\bs{\beta}}
$. 
We then consider the worst-case squared bias of $\hat{\tau}$ standardized by $S^2_{\mu}$, which has the following equivalent forms:
\begin{align}\label{eq:worse_case_M}
	\sup_{\tilde{\bs{\beta}} \ne \bs{0}} \frac{
    \{ \tilde{\bs{\beta}}^\top \hat{\bs{\tau}}_{\bs{X}}(\bs{z}) \}^2		
	}{S^2_{\mu}}
	= 
	\sup_{\tilde{\bs{\beta}} \ne \bs{0}} \frac{
	\{ \tilde{\bs{\beta}}^\top \hat{\bs{\tau}}_{\bs{X}}(\bs{z}) \}^2	
	}{\tilde{\bs{\beta}}^\top \bs{S}^2_{\bs{X}}\tilde{\bs{\beta}}}
	= \hat{\bs{\tau}}_{\bs{X}}(\bs{z})^\top \bs{S}_{\bs{X}}^{-2} \hat{\bs{\tau}}_{\bs{X}}(\bs{z})
	= \frac{n}{n_1n_0} M(\bs{z}), 
\end{align}
where $M(\bs{z})$ is the Mahalanobis distance of covariate means in two treatment groups under the treatment assignment $\bs{z}$, as defined in Section \ref{sec:cov_imb}. 
Consequently, the assignment minimizing the Mahalanobis distance is equivalently the one that minimizes worst-case standardized squared bias.
Therefore, under the proposed model and criterion, minimizing the Mahalanobis distance leads to the optimal design. 

\subsection{Model-based efficiency of rerandomization}\label{sec:opt_efficiency}
We now briefly discuss the efficiency of rerandomization under the  proposed model in \eqref{eq:linear_model}. 
In short, we will show that, under ReM with properly diminishing threshold for covariate imbalance, the design can asymptotically achieve the optimal efficiency.

To gain some intuition, beyond the equal variance assumption, we further assume that $e_i(z)$'s are i.i.d.\  across all units, for $z=0,1$; the i.i.d.\ assumption will be relaxed later. 
We can verify that the difference-in-means estimator has the following decomposition:
\begin{align}\label{eq:decomp_tau_hat_tau_star}
	\hat{\tau} - \tau^\star 
	& = \tilde{\bs{\beta}}^\top \hat{\bs{\tau}}_{\bs{X}} + 
	\Big\{ \frac{1}{n_1} \sum_{i=1}^n Z_i e_i(1) - \frac{1}{n_0} \sum_{i=1}^n (1-Z_i) e_i(0) \Big\}
	\equiv \tilde{\bs{\beta}}^\top \hat{\bs{\tau}}_{\bs{X}} + \hat{\tau}_e, 
\end{align}
where $\tilde{\bs{\beta}} = r_0 \bs{\beta}_1 + r_1 \bs{\beta}_0$ is defined the same as before and 
$\hat{\tau}_e$ is the difference-in-means for the residual potential outcomes. 
For treatment assignment mechanisms depending only on the covariates, such as rerandomization based on $\bs{X}_i$'s, $\hat{\bs{\tau}}_{\bs{X}}$ and $\hat{\tau}_e$ must be mutually independent, 
with $\hat{\bs{\tau}}_{\bs{X}}$ following its randomization distribution and 
\begin{align}\label{eq:dist_tau_hat_e_iid}
	\hat{\tau}_e \sim \frac{1}{n_1} \sum_{i=1}^{n_1} e_i(1) - \frac{1}{n_0} \sum_{i=n_1+1}^n e_i(0). 
\end{align}
This is because the conditional distribution of $\hat{\tau}_e$ given $\bs{Z}$ must follow the distribution on the right hand side of \eqref{eq:dist_tau_hat_e_iid}. 
By the standard central limit theorem, when $e_i(1)$'s and $e_i(0)$'s have finite second moments, and the proportions of treated and control units $r_1$ and $r_0$ have positive limits as $n\rightarrow\infty$, 
$\sqrt{n}\hat{\tau}_e$ will asymptotically converge to a Gaussian distribution with mean zero and variance $r_1^{-1} \sigma^2_1 + r_0^{-1} \sigma^2_0$. 
By Slutsky's theorem, 
as long as $\sqrt{n} \tilde{\bs{\beta}}^\top \hat{\bs{\tau}}_{\bs{X}} = o_{\PP}(1)$, 
$\sqrt{n}(\hat{\tau} - \tau^\star )$ will converge to the same asymptotic distribution as $\sqrt{n}\hat{\tau}_e$, which is actually the optimal efficiency that we can expect, as implied by \eqref{eq:mse_model}.

Below we rigorously study the asymptotic efficiency of rerandomization under model \eqref{eq:linear_model}. 
First, we allow the residuals $e_i(z)$'s to be non-identically distributed for $z=0,1$, but require them to have bounded third absolute moments.
\begin{condition}\label{cond:error}
	There exists some finite constant $C_e$ such that, for all $n$, 
	\[
	\max_{1 \leq i \leq n} \E[|e_i(1)|^3] \leq C_e, \quad \& \quad \max_{1 \leq i \leq n} \E[|e_i(0)|^3] \leq C_e.
	\]
\end{condition}
Second, to conduct the optimal rerandomization with diminishing covariate imbalance, we invoke similar regularity conditions as Conditions \ref{cond:gamma_n}--\ref{cond:k_np_n}. 
Because here we care only the difference-in-means of covariates, 
we redefine the quantifies in the main paper by excluding the potential outcomes there. Specifically, 
analogous to $\gamma_n$ and $\Delta_n$ in \eqref{eq:gamma_n} and \eqref{eq:Delta_n},
define  
\begin{align*}
    \gamma_n^\star \equiv 
	\frac{K_n^{1/4}}{\sqrt{n r_1r_0}} \frac{1}{n} \sum_{i=1}^n \left\| \bs{S}_{\bs{X}}^{-1} (\bs{X}_i-\bar{\bs{X}}) \right\|_2^3, 
	\quad 
	\Delta_n^\star \equiv 
	\sup_{\mathcal{Q} \in \cC_{K_n}}
	\left|
	\PP
	\left(\bs{V}_{\bs{x}\bs{x}}^{-1/2} \hat{\bs{\tau}}_{\bs{X}}
	\in \mathcal{Q}
	\right)
	- 
	\PP \left( \bs{\varepsilon}^\star \in \mathcal{Q} \right)
	\right|,
\end{align*}
where $\bs{\varepsilon}^\star \sim \mathcal{N}(\bs{0}, \bs{I}_{K_n})$. 
We then invoke the following regularity condition. 
\begin{condition}\label{cond:gamma_delta_star}
    Conditions \ref{cond:gamma_n}--\ref{cond:k_np_n}, with $\gamma_n$ and $\Delta_n$ replaced by $\gamma_n^\star$ and $\Delta_n^\star$, hold. 
\end{condition}
Third, we assume the following condition on model \eqref{eq:linear_model} and proportions of treated and control units. 
\begin{condition}\label{cond:model_stable}
    As the sample size $n$ increases, 
    \begin{itemize}
        \item[(i)] the residual variances $\sigma_1^2$ and $\sigma_0^2$ do not vary, and at least one of them is positive;
        \item[(ii)] the proportions of treated and control units satisfy $\sqrt{n}\min\{r_1, r_0\} \rightarrow \infty$ as $n\rightarrow\infty$; 
        \item[(iii)] the weighted average of expected potential outcomes has bounded finite population variance, i.e., 
        $S^2_{\mu} = \tilde{\bs{\beta}}^\top \bs{S}^2_{\bs{X}}\tilde{\bs{\beta}} \le C_{\mu}$ for all $n$ and some finite constant $C_{\mu}$. 
    \end{itemize}
\end{condition}

Under the above conditions, $\sqrt{n} \tilde{\bs{\beta}}^\top \hat{\bs{\tau}}_{\bs{X}}$ will converge in probability to zero. 
This implies that, asymptotically, rerandomization will achieve the optimal efficiency (or equivalently be the optimal design) under model \eqref{eq:linear_model}. 
We summarize the results in the following theorem.

\begin{theorem}\label{thm:remmodel}
	Under ReM and Conditions \ref{cond:error}--\ref{cond:model_stable}, 
    as $n \to \infty$,
    \begin{align*}
        \frac{\hat{\tau} - \tau^\star}{\sqrt{\sigma_1^2/n_1 + \sigma_0^2/n_0}} \mid M \le a_n
        \converged 
        \mathcal{N}(0, 1). 
    \end{align*} 
\end{theorem} 

Theorem~\ref{thm:remmodel} shows that, under model~\eqref{eq:linear_model}, ReM with properly diminishing threshold can asymptotically achieve the optimal efficiency as implied by \eqref{eq:mse_model}.
Both Theorem~\ref{thm:rem_gaussian} and Theorem~\ref{thm:remmodel} show the optimality of ReM with properly diminishing covariate imbalance threshold.
However, their justification is quite different. 
First, the two theorems rely on different sources of randomness. 
Theorem~\ref{thm:rem_gaussian} views all the potential outcomes as fixed constant (or equivalently conditioning on all the potential outcomes), and the randomness comes solely from the treatment assignment; 
while Theorem~\ref{thm:remmodel} assumes additionally that the potential outcomes are random following model \eqref{eq:mse_model}. 
Second, due to the aforementioned difference, 
the estimands for average treatment effects have different forms in the two theorems. 
Theorem~\ref{thm:rem_gaussian} focuses on $\tau = n^{-1} \sum_{i=1}^n \{Y_i(1) - Y_i(0)\}$, 
while Theorem~\ref{thm:remmodel} focuses on $\tau^\star = \E(\tau) = n^{-1} \sum_{i=1}^n \E\{Y_i(1) - Y_i(0)\}$ under model \eqref{eq:mse_model}.

\subsection{Technical details}\label{sec:optimal_tech}

\begin{proof}[Additional details for \eqref{eq:mse_model}--\eqref{eq:decomp_tau_hat_tau_star}]
    First, we prove the decomposition of $\hat{\tau} - \tau^\star$ in \eqref{eq:decomp_tau_hat_tau_star}. 
    By some algebra, 
    \begin{align*}
        \hat{\tau} - \tau^\star 
        & = \frac{1}{n_1} Z_i Y_i(1) - \frac{1}{n_0} Z_i Y_i(0)
        - 
        \{ \alpha_1 - \alpha_0 + (\bs{\beta}_1 - \bs{\beta}_0)^\top \bar{\bs{X}} \}
        \\
        & = 
        \alpha_1 + \bs{\beta}_1^\top \bar{\bs{X}}_1 + \bar{e}_1 - \alpha_0 - \bs{\beta}_0^\top \bar{\bs{X}}_0 - \bar{e}_0 - \{ \alpha_1 - \alpha_0 + (\bs{\beta}_1 - \bs{\beta}_0)^\top (r_1 \bar{\bs{X}}_1 + r_0 \bar{\bs{X}}_0) \}
        \\
        & =  \bs{\beta}_1^\top r_0 (\bar{\bs{X}}_1 - \bar{\bs{X}}_0) 
        + \bs{\beta}_0^\top r_1 (\bar{\bs{X}}_1 - \bar{\bs{X}}_0) + \bar{e}_1 - \bar{e}_0\\
        & = ( r_0 \bs{\beta}_1 + r_1 \bs{\beta}_0 )^\top \hat{\bs{\tau}}_{\bs{X}} + \hat{\tau}_e
        = \tilde{\bs{\beta}}^\top \hat{\bs{\tau}}_{\bs{X}} + \hat{\tau}_e, 
    \end{align*}
    where $\bar{e}_z$ denotes the average of residual potential outcomes $e_i(z)$'s for units under treatment arm $z$.  

    Second, we prove the decomposition of the model-based MSE in \eqref{eq:mse_model}. 
    From the decomposition in \eqref{eq:decomp_tau_hat_tau_star} and the property of model \eqref{eq:linear_model}, 
    $\E\{\hat{\tau}(\bs{z}) - \tau^\star \} = \tilde{\bs{\beta}}^\top \hat{\bs{\tau}}_{\bs{X}}$, and 
    $
    \var\{\hat{\tau}(\bs{z}) - \tau^\star \} = \var\{ \hat{\tau}_e(\bs{z}) \}
        = \sigma_1^2/n_1 + \sigma_0^2/n_0, 
    $
    where we use $\hat{\tau}_e(\bs{z})$ to emphasize that it is the difference-in-means of residual potential outcomes under the treatment assignment $\bs{z}$. 
    These then immediately imply the decomposition in \eqref{eq:mse_model}. 
    
    Third, we prove \eqref{eq:worse_case_M}. By the definition of matrix norm,
    letting $\check{\bs{\beta}} = \bs{S}_{\bs{X}}\tilde{\bs{\beta}}$, 
    we have 
    \begin{align*}
        \sup_{\tilde{\bs{\beta}} \ne \bs{0}} \frac{
    	\{ \tilde{\bs{\beta}}^\top \hat{\bs{\tau}}_{\bs{X}}(\bs{z}) \}^2	
    	}{\tilde{\bs{\beta}}^\top \bs{S}^2_{\bs{X}}\tilde{\bs{\beta}}}
    	& = 
    	\sup_{\check{\bs{\beta}} \ne \bs{0}} \frac{
    	\|  \hat{\bs{\tau}}_{\bs{X}}(\bs{z})^\top \bs{S}_{\bs{X}}^{-1}\check{\bs{\beta}} \|_2^2	
    	}{\check{\bs{\beta}}^\top \check{\bs{\beta}}}
    	= 
    	\| \hat{\bs{\tau}}_{\bs{X}}(\bs{z})^\top \bs{S}_{\bs{X}}^{-1} \|_2^2 
    	= \hat{\bs{\tau}}_{\bs{X}}(\bs{z})^\top \bs{S}_{\bs{X}}^{-2} \hat{\bs{\tau}}_{\bs{X}}(\bs{z}).
    \end{align*}
    This immediately implies \eqref{eq:worse_case_M}. 
\end{proof}

\begin{proof}[Proof of Theorem~\ref{thm:remmodel}]
Below we consider the two terms in the decomposition \eqref{eq:decomp_tau_hat_tau_star} separately.  

First, we consider the limiting distribution of $\tilde{\bs{\beta}}^\top \hat{\bs{\tau}}_{\bs{X}}$. 
By the same logic as the proof of Theorem~\ref{thm:berry_esseen_clt}, under Condition \ref{cond:gamma_delta_star}, 
as $n\rightarrow \infty$, 
\begin{align*}
    \sup_{c \in \R} \left|\PP\left( \sqrt{n} \tilde{\bs{\beta}}^\top \hat{\bs{\tau}}_{\bs{X}} \leq c \mid M \leq a_n\right) - \PP\left(\sqrt{n} \tilde{\bs{\beta}}^\top \tilde{\bs{\tau}}_{\bs{X}}  \leq c \mid \tilde{M} \leq a_n\right)\right| \rightarrow 0,
\end{align*}
By the same logic as the proof of \citet[][Theorem 1]{LDR18}, 
$$
\sqrt{n} \tilde{\bs{\beta}}^\top \tilde{\bs{\tau}}_{\bs{X}} \mid \tilde{M} \leq a_n 
\ \sim \  
\sqrt{n} \tilde{\bs{\beta}}^\top \bs{V}_{\bs{xx}}^{1/2} \bs{\varepsilon}^\star \mid (\bs{\varepsilon}^\star)^\top \bs{\varepsilon}^\star \leq a_n
\ \sim \ 
\sqrt{n} \| \bs{V}_{\bs{xx}}^{1/2} \tilde{\bs{\beta}} \|_2 L_{K_n, a_n}, 
$$
recalling that $\bs{\varepsilon}^\star \sim \mathcal{N}(\bs{0}, \bs{I}_{K_n})$. 
From Conditions \ref{cond:gamma_delta_star} and \ref{cond:model_stable} and using Theorem \ref{thm:v_Ka}(i) and Proposition \ref{prop:convpvar},
$$
\sqrt{n} \| \bs{V}_{\bs{xx}}^{1/2} \tilde{\bs{\beta}} \|_2 L_{K_n, a_n}
= 
\sqrt{
(r_1r_0)^{-1} \tilde{\bs{\beta}}^\top \bs{S}^2_{\bs{X}} \tilde{\bs{\beta}}
}
\ L_{K_n, a_n}
= O(1) \cdot o_{\PP}(1) =  o_{\PP}(1). 
$$
From the above, we can derive that 
$
\sqrt{n} \tilde{\bs{\beta}}^\top \hat{\bs{\tau}}_{\bs{X}} = o_{\PP}(1). 
$

Second, we consider the limiting of $\hat{\tau}_e$. 
For any fixed acceptable assignment $\bs{z}$ under ReM, 
from the standard
univariate Berry--Esseen theorem \citep{esseen1942liapunov}, 
there exists a universal constant $C$ such that
\begin{align*}
& \quad \ \sup_{c \in \R} |\PP\{ \var(\hat{\tau}_e)^{-1/2} \hat{\tau}_e \leq c \mid \bs{Z} \equiv \bs{z}\} - \PP(\varepsilon_0 \leq c)| 
\\
& \leq 
C
\frac{n_1^{-2} \E|e_i^3(1)| + n_0^{-2} \E|e_i^3(0)|}{(n_1^{-1}\sigma_1^2 + n_0^{-1}\sigma_0^2)^{3/2}}
\le 
\frac{C}{\sqrt{n}} \frac{(1/r_1+1/r_0) C_e}{(\sigma_1^2/r_1 + \sigma_0^2/r_0)^{3/2}}
\le 
\frac{C}{\sqrt{n}} \frac{2C_e /\min\{r_1, r_0\}}{(\sigma_1^2 + \sigma_0^2)^{3/2}}
\\
& \le 
\frac{2CC_e (\sigma_1^2 + \sigma_0^2)^{-3/2}}{\sqrt{n}\min\{r_1, r_0\}}. 
\end{align*}
where $\varepsilon_0\sim \mathcal{N}(0,1)$.
This then implies that, for any $c\in \mathbb{R}$, 
\begin{align*}
    & \quad \ |\PP\{ \var(\hat{\tau}_e)^{-1/2}  \hat{\tau}_e \leq c \mid M \leq a_n\} - \PP(\varepsilon_0 \leq c)|
    \\
    & = | \E[ \PP\{ \var(\hat{\tau}_e)^{-1/2}  \hat{\tau}_e \leq c \mid \bs{Z}\}\mid M \leq a_n ]  - \PP(\varepsilon_0 \leq c)|
    \\
    & \le 
    \E[ 
    | \PP\{ \var(\hat{\tau}_e)^{-1/2}  \hat{\tau}_e \leq c \mid \bs{Z}\} - 
    \PP(\varepsilon_0 \leq c)
    |
    \mid M \leq a_n ] 
    \\
    & \le 
    \frac{2CC_e (\sigma_1^2 + \sigma_0^2)^{-3/2}}{\sqrt{n}\min\{r_1, r_0\}}, 
\end{align*}
i.e., 
\begin{align}\label{eq:berry_tau_hat_e}
    \sup_{c \in \R} |\PP\{ \var(\hat{\tau}_e)^{-1/2}  \hat{\tau}_e \leq c \mid M \leq a_n\} - \PP(\varepsilon_0 \leq c)| 
    \le \frac{2CC_e (\sigma_1^2 + \sigma_0^2)^{-3/2}}{\sqrt{n}\min\{r_1, r_0\}}. 
\end{align}

Finally, we study the limiting distribution of $\tilde{\bs{\beta}}^\top \hat{\bs{\tau}}_{\bs{X}} + \hat{\tau}_e$. 
From Lemma \ref{lemma:bound_inf_norm_two_rv}, for any constant $\delta > 0$, 
\begin{align}\label{eq:berry_bound_model}
    & \quad \ \sup_{c\in \mathbb{R}}\left| \PP \left\{
    \frac{\tilde{\bs{\beta}}^\top \hat{\bs{\tau}}_{\bs{X}}+ \hat{\tau}_e}{\var(\hat{\tau}_e)^{1/2}} \le c
    \mid M\le a_n
    \right\}
    - 
    \PP \left\{
    \frac{ \hat{\tau}_e}{\var(\hat{\tau}_e)^{1/2}} \le c
    \mid M\le a_n
    \right\} \right|
    \nonumber
    \\
    & \le 
    \PP  \left\{
    \left|\frac{\tilde{\bs{\beta}}^\top \hat{\bs{\tau}}_{\bs{X}}}{\var(\hat{\tau}_e)^{1/2}} \right| > \delta
    \mid M\le a_n
    \right\} + 
    \sup_{b\in \mathbb{R}} \PP \left\{
    b < \frac{ \hat{\tau}_e}{\var(\hat{\tau}_e)^{1/2}} \le b+\delta
    \mid M\le a_n
    \right\}. 
\end{align}
Because $\sqrt{n}\tilde{\bs{\beta}}^\top \hat{\bs{\tau}}_{\bs{X}} = o_{\PP}(1)$
and $n \var(\hat{\tau}_e) = \sigma_1^2/r_1 + \sigma_0^2/r_0 \ge \sigma_1^2 + \sigma_0^2 > 0$, 
$\tilde{\bs{\beta}}^\top \hat{\bs{\tau}}_{\bs{X}}/\var(\hat{\tau}_e)^{1/2} = o_{\PP}(1)$, and thus the first term in \eqref{eq:berry_bound_model} converges to zero as $n\rightarrow \infty$. 
From \eqref{eq:berry_tau_hat_e}, 
\begin{align*}
    \sup_{b\in \mathbb{R}} \PP \left\{
    b < \frac{ \hat{\tau}_e}{\var(\hat{\tau}_e)^{1/2}} \le b+\delta
    \mid M\le a_n
    \right\}
    & \le 
    \sup_{b\in \mathbb{R}} 
    \PP(b < \varepsilon_0 \le b+\delta) + \frac{4CC_e (\sigma_1^2 + \sigma_0^2)^{-3/2}}{\sqrt{n}\min\{r_1, r_0\}}\\
    & \le \frac{\delta}{2\pi} + \frac{4CC_e (\sigma_1^2 + \sigma_0^2)^{-3/2}}{\sqrt{n}\min\{r_1, r_0\}},
\end{align*}
where, from Condition \ref{cond:model_stable}, the upper bound converges to $\delta/(2\pi)$ as $n\rightarrow \infty$. 
From the above, for any constant $\delta>0$, 
\begin{align*}
    \limsup_{n\rightarrow \infty}\sup_{c\in \mathbb{R}}\left| \PP \left\{
    \frac{\tilde{\bs{\beta}}^\top \hat{\bs{\tau}}_{\bs{X}}+ \hat{\tau}_e}{\var(\hat{\tau}_e)^{1/2}} \le c
    \mid M\le a_n
    \right\}
    - 
    \PP \left\{
    \frac{ \hat{\tau}_e}{\var(\hat{\tau}_e)^{1/2}} \le c
    \mid M\le a_n
    \right\} \right|
    \le \frac{\delta}{2\pi}. 
\end{align*}
This immediately implies that, as $n\rightarrow \infty$, 
\begin{align*}
    \sup_{c\in \mathbb{R}}\left| \PP \left\{
    \frac{\tilde{\bs{\beta}}^\top \hat{\bs{\tau}}_{\bs{X}}+ \hat{\tau}_e}{\var(\hat{\tau}_e)^{1/2}} \le c
    \mid M\le a_n
    \right\}
    - 
    \PP \left\{
    \frac{ \hat{\tau}_e}{\var(\hat{\tau}_e)^{1/2}} \le c
    \mid M\le a_n
    \right\} \right| \rightarrow 0. 
\end{align*}
From \eqref{eq:berry_tau_hat_e} and Condition \ref{cond:model_stable}, we further have 
\begin{align*}
    \sup_{c\in \mathbb{R}}\left| \PP \left\{
    \frac{\tilde{\bs{\beta}}^\top \hat{\bs{\tau}}_{\bs{X}}+ \hat{\tau}_e}{\var(\hat{\tau}_e)^{1/2}} \le c
    \mid M\le a_n
    \right\}
    - 
    \PP \left\{
   \varepsilon_0 \le c
    \right\} \right| \rightarrow 0, 
\end{align*}
i.e., 
\begin{align*}
    \frac{\hat{\tau} - \tau^\star}{\sqrt{\sigma_1^2/n_1 + \sigma_0^2/n_0}} \mid M \le a_n
        \converged 
        \mathcal{N}(0, 1).
\end{align*}
Therefore, Theorem \ref{thm:remmodel} holds. 
\end{proof}

\end{document}